\pdfoutput=1
\RequirePackage{ifpdf}
\ifpdf % We~are running pdfTeX in pdf mode
\documentclass[pdftex]{sigma}%,draft
\else
\documentclass{sigma}
\fi

\numberwithin{equation}{section}

\newtheorem{Theorem}{Theorem}[section]
\newtheorem*{Theorem*}{Theorem}
\newtheorem{Corollary}[Theorem]{Corollary}
\newtheorem{Lemma}[Theorem]{Lemma}
\newtheorem{Proposition}[Theorem]{Proposition}
 { \theoremstyle{definition}

\newtheorem{Remark}[Theorem]{Remark} }

\usepackage[all]{xy}

\newcommand{\bk}{\mathbf{k}}
\newcommand{\bl}{\mathbf{l}}
\newcommand{\bm}{\mathbf{m}}
\newcommand{\bn}{\mathbf{n}}
\newcommand{\bs}{\mathbf{s}}
\newcommand{\BZ}{\mathbb{Z}}
\newcommand{\BR}{\mathbb{R}}
\newcommand{\BC}{\mathbb{C}}
\newcommand{\BH}{\mathbb{H}}
\newcommand{\BO}{\mathbb{O}}
\newcommand{\cB}{\mathcal{B}}
\newcommand{\cF}{\mathcal{F}}
\newcommand{\cH}{\mathcal{H}}

\newcommand{\cO}{\mathcal{O}}
\newcommand{\cP}{\mathcal{P}}

\newcommand{\cU}{\mathcal{U}}
\newcommand{\fa}{\mathfrak{a}}
\newcommand{\fg}{\mathfrak{g}}
\newcommand{\fh}{\mathfrak{h}}
\newcommand{\fk}{\mathfrak{k}}
\newcommand{\fl}{\mathfrak{l}}
\newcommand{\fm}{\mathfrak{m}}
\newcommand{\fn}{\mathfrak{n}}
\newcommand{\fp}{\mathfrak{p}}
\newcommand{\fz}{\mathfrak{z}}

\newcommand{\itinv}{{-\mathit{1}}}
\newcommand{\rj}{\mathrm{j}}
\newcommand{\rK}{\mathrm{K}}

\newcommand{\tr}{\operatorname{tr}}
\newcommand{\eqspace}{\mathrel{\phantom{=}}}
\newcommand{\ds}{\displaystyle}
\newcommand{\Hom}{\operatorname{Hom}}
\newcommand{\End}{\operatorname{End}}
\newcommand{\Det}{\operatorname{Det}}
\newcommand{\Tr}{\operatorname{Tr}}
\newcommand{\Sym}{\operatorname{Sym}}
\newcommand{\Skew}{\operatorname{Skew}}
\newcommand{\Herm}{\operatorname{Herm}}
\newcommand{\Pf}{\operatorname{Pf}}
\newcommand{\rank}{\operatorname{rank}}
\newcommand{\hboxtimes}{\mathbin{\hat{\boxtimes}}}
\newcommand{\hsum}{\sideset{}{^\oplus}\sum}
\renewcommand{\Re}{\operatorname{Re}}

\renewcommand{\det}{\operatorname{det}}
\newcommand{\Proj}{\operatorname{Proj}}
\newcommand{\op}{\mathrm{op}}

\begin{document}
\allowdisplaybreaks

\newcommand{\arXivNumber}{2105.13976}

\renewcommand{\PaperNumber}{033}

\FirstPageHeading

\ShortArticleName{Computation of Weighted Bergman Inner Products on Bounded Symmetric Domains}

\ArticleName{Computation of Weighted Bergman Inner Products\\ on Bounded Symmetric Domains and Restriction\\ to Subgroups}

\Author{Ryosuke NAKAHAMA}

\AuthorNameForHeading{R.~Nakahama}

\Address{Institute of Mathematics for Industry, Kyushu University,\\ 744 Motooka, Nishi-ku Fukuoka 819-0395, Japan}
\Email{\href{mailto:nakahama.ryosuke.252@m.kyushu-u.ac.jp}{nakahama.ryosuke.252@m.kyushu-u.ac.jp}}

\ArticleDates{Received June 14, 2021, in final form April 06, 2022; Published online May 03, 2022}

\Abstract{Let $(G,G_1)$ be a symmetric pair of holomorphic type, and we consider a pair of Hermitian symmetric spaces $D_1=G_1/K_1\subset D=G/K$, realized as bounded symmetric domains in complex vector spaces $\mathfrak{p}^+_1\subset\mathfrak{p}^+$ respectively. Then the universal covering group~$\widetilde{G}$ of~$G$ acts unitarily on the weighted Bergman space $\mathcal{H}_\lambda(D)\subset\mathcal{O}(D)$ on~$D$. Its restriction to the subgroup~$\widetilde{G}_1$ decomposes discretely and multiplicity-freely, and its branching law is given explicitly by Hua--Kostant--Schmid--Kobayashi's formula in terms of the $K_1$-decomposition of the space $\mathcal{P}(\mathfrak{p}^+_2)$ of polynomials on the orthogonal complement $\mathfrak{p}^+_2$ of $\mathfrak{p}^+_1$ in~$\mathfrak{p}^+$. The object of this article is to compute explicitly the inner product $\big\langle f(x_2),{\rm e}^{(x|\overline{z})_{\mathfrak{p}^+}}\big\rangle_\lambda$ for $f(x_2)\in\mathcal{P}(\mathfrak{p}^+_2)$, $x=(x_1,x_2)$, $z\in\mathfrak{p}^+=\mathfrak{p}^+_1\oplus\mathfrak{p}^+_2$. For example, when $\mathfrak{p}^+$, $\mathfrak{p}^+_2$ are of tube type and $f(x_2)=\det(x_2)^k$, we compute this inner product explicitly by introducing a multivariate generalization of Gauss' hypergeometric polynomials ${}_2F_1$. Also, as an application, we construct explicitly $\widetilde{G}_1$-intertwining operators (symmetry breaking operators) $\mathcal{H}_\lambda(D)|_{\widetilde{G}_1}\to\mathcal{H}_\mu(D_1)$ from holomorphic discrete series representations of $\widetilde{G}$ to those of $\widetilde{G}_1$, which are unique up to constant multiple for sufficiently large~$\lambda$.}

\Keywords{weighted Bergman spaces; holomorphic discrete series representations; branching laws; intertwining operators; symmetry breaking operators; highest weight modules}

\Classification{22E45; 43A85; 17C30; 33C67}

\tableofcontents

\section{Introduction}

The purpose of this article is to compute explicitly the weighted Bergman inner product of polynomials on some subspace
of a bounded symmetric domain, and study the decomposition of the restriction of a holomorphic discrete series representation
to some subgroup in detail.

Let $\fp^+$ be a finite-dimensional complex vector space, $D\subset\fp^+$ be a Hermitian symmetric space realized as a bounded symmetric domain
centered at the origin, $G$ be the indentity component of the biholomorphism group of $D$,
and $K\subset G$ be the isotropy subgroup at the origin,
so that $D\simeq G/K$ holds. Let $\widetilde{G}$ and $\widetilde{K}$ denote the universal covering groups of $G$ and $K$ respectively.
We~consider an irreducible unitary representation $(\tau,V)$ of $\widetilde{K}$, and consider the homogeneous vector bundle
$\widetilde{G}\times_{\widetilde{K}} V\to \widetilde{G}/\widetilde{K}\simeq D$.
Then this bundle is trivializable, and the space of holomorphic sections is isomorphic to $\cO(D,V)$,
the space of $V$-valued holomorphic functions on~$D$, on which $\widetilde{G}$ acts by
\[
(\hat{\tau}(g)f)(x):=\tau\big(\kappa\big(g^{-1},x\big)\big)^{-1}f\big(g^{-1}x\big), \qquad
g\in\widetilde{G},\quad x\in D,\quad f\in\cO(D,V),
\]
by using a function $\kappa\colon \widetilde{G}\times D\to \widetilde{K}^\BC$ satisfying the cocycle condition.
Next we consider a function $B\colon D\times\overline{D}\to \widetilde{K}^\BC$ satisfying
\[
B(gx,\overline{gy})=\kappa(g,x)B(x,\overline{y})\kappa(g,y)^*, \qquad
g\in\widetilde{G},\quad x,y\in D.
\]
Then the function $\tau(B(x,\overline{y}))\in\cO\big(D\times\overline{D},\End_\BC(V)\big)\simeq\cO\big(D\times\overline{D},V\otimes\overline{V}\big)$
is invariant under the diagonal action of~$\widetilde{G}$, and if this is positive definite, then there exists a Hilbert space
$\cH_\tau(D,V)\subset\cO(D,V)$ with the reproducing kernel $\tau(B(x,\overline{y}))$, on which $\widetilde{G}$ acts unitarily.
We~call the representation $(\hat{\tau},\cH_\tau(D,V))$ a \textit{unitary highest weight representation} of $\widetilde{G}$.
Such representations are classified by \cite{EHW, J}. Moreover, if the inner product is given by the converging integral
\begin{gather}\label{intro_inner_prod}
\langle f,g\rangle_{\hat{\tau}}:=C_\tau\int_D \big(\tau\big(B(x,\overline{x})^{-1}\big)f(x),g(x)\big)_\tau\Det_{\fp^+}(B(x,\overline{x}))^{-1}{\rm d}x
\end{gather}
(a \textit{weighted Bergman inner product}), then $(\hat{\tau},\cH_\tau(D,V))$ is called a \textit{holomorphic discrete series representation}.
In this paper, when $(\tau,V)$ is written as $(\tau,V)=\big(\chi^{-\lambda}\otimes\tau_0,V\big)$ by using a~fixed character $\chi$
and a fixed representation $(\tau_0,V)$ of $\widetilde{K}$, we also write $\cH_\tau(D,V)=:\cH_\lambda(D,V)$, $\cO(D,V)=:\cO_\lambda(D,V)$.
In addition, if $V=\BC$ then we write $\cH_\lambda(D,\BC)=:\cH_\lambda(D)$ and call it of \textit{scalar type}.
If $\lambda$ is sufficiently large, then $\cH_\lambda(D,V)$ is holomorphic discrete, and the $\widetilde{K}$-finite part
$\cH_\lambda(D,V)_{\widetilde{K}}$ coincides with the space of all polynomials $\cO(D,V)_{\widetilde{K}}=\cP(\fp^+,V)$.
On the other hand, for smaller $\lambda$, $\cO_\lambda(D,V)_{\widetilde{K}}=\cP(\fp^+,V)$ may be reducible as a $\big(\fg,\widetilde{K}\big)$-module,
and $\cH_\lambda(D,V)_{\widetilde{K}}$ may be smaller than $\cP(\fp^+,V)$ even if it exists.
By computing the inner product~(\ref{intro_inner_prod}) for large $\lambda$ and observing the poles of its meromorphic continuation,
we can get some information on submodules of $\cO_\lambda(D,V)_{\widetilde{K}}$ (see, e.g.,~\cite{FK0,N1}).

For example, we consider the bounded symmetric domain $D$ of type $I\!I\!I_r$,
\[
D:=\{x\in\Sym(r,\BC)\mid I-xx^* \text{ is positive definite}\}.
\]
Then the universal covering group $\widetilde{\operatorname{Sp}}(r,\BR)$ of
\[
\operatorname{Sp}(r,\BR):=\left\{ g\in {\rm GL}(2r,\BC)\ \middle|\
{}^t\hspace{-1pt}g \begin{pmatrix} 0&I\\ -I&0 \end{pmatrix}g=\begin{pmatrix} 0&I\\ -I&0 \end{pmatrix}\!, \;
g\begin{pmatrix} 0&I\\ I&0 \end{pmatrix}=\begin{pmatrix} 0&I\\ I&0 \end{pmatrix}\overline{g} \right\}
\]
acts on $\cO(D)=\cO_\lambda(D)$ by
\[
\left(\tau_\lambda\left( \begin{pmatrix} a&b\\c&d \end{pmatrix}^{-1}\right)f\right)(x)
=\det(cx+d)^{-\lambda}f\big((ax+b)(cx+d)^{-1}\big).
\]
We~note that $\det(cx+d)^{-\lambda}$ is not well-defined on $\operatorname{Sp}(r,\BR)\times D$ if $\lambda\notin\BZ$,
but is well-defined on the universal covering space $\widetilde{\operatorname{Sp}}(r,\BR)\times D$.
If $\lambda>r$, then this preserves the weighted Bergman inner product
\[
\langle f,g\rangle_\lambda:=C_\lambda\int_D f(x)\overline{g(x)}\det(I-xx^*)^{\lambda-(r+1)}{\rm d}x, \]
and the corresponding Hilbert space $\cH_\lambda(D)\subset\cO(D)$, with the reproducing kernel $\det(I-xy^*)^{-\lambda}$,
gives a holomorphic discrete series representation of scalar type of $\widetilde{\operatorname{Sp}}(r,\BR)$.

Next we consider a connected symmetric subgroup $G_1\subset G$.
Without loss of generality we may assume $K_1:=G_1\cap K$ is a maximal compact subgroup of $G_1$.
Then $D_1:=G_1.0\simeq G_1/K_1$ is either a complex submanifold or a totally real submanifold of $D\simeq G/K$.
$(G,G_1)$ is called a symmetric pair of \textit{holomorphic type} for the former case,
and of \textit{anti-holomorphic type} for the latter case (see \cite[Section 3.4]{Kmf1}).
From now on we assume $(G,G_1)$ is a symmetric pair of holomorphic type.
We~take a complex subspace $\fp^+_1\subset \fp^+$ such that $D_1=D\cap\fp^+_1$ holds,
and let $\fp^+_2\subset\fp^+$ be the orthogonal complement of $\fp^+_1$ with respect to a suitable inner product of $\fp^+$.
Then for any holomorphic discrete series representation $\cH_\tau(D,V)$ of $\widetilde{G}$,
the restriction to the subgroup $\widetilde{G}_1$ decomposes into a Hilbert direct sum of irreducible representations of $\widetilde{G}_1$,
and each subrepresentation is generated by a $\widetilde{K}_1^\BC$-submodule of $\cP(\fp^+_2)\otimes \big(V|_{\widetilde{K}_1^\BC}\big)$,
where $\cP(\fp^+_2)$ denotes the space of polynomials on $\fp^+_2$.
That is, if $\cP(\fp^+_2)\otimes \big(V|_{\widetilde{K}_1^\BC}\big)$ is decomposed under $\widetilde{K}_1^\BC$ as
\[
\cP(\fp^+_2)\otimes \big(V|_{\widetilde{K}_1^\BC}\big)\simeq \bigoplus_j m(\tau,\rho_j)(\rho_j,W_j), \qquad m(\tau,\rho_j)\in\BZ_{\ge 0},
\]
then $\cH_\tau(D,V)$ is decomposed under $\widetilde{G}_1$ abstractly as
\[
\cH_\tau(D,V)|_{\widetilde{G}_1}\simeq \hsum_j m(\tau,\rho_j)\cH_{\rho_j}(D_1,W_j)
\]
(see Kobayashi \cite[Lemma 8.8]{Kmf1}, \cite[Section 8]{Kmf0-1}. For earlier results, see also \cite{JV,Ma}).
Especially, if $(\tau,V)=\big(\chi^{-\lambda},\BC\big)$ is 1-dimensional,
then since $\cP(\fp^+_2)$ decomposes multiplicity-freely (i.e., $m(\tau,\rho_j)\le 1$ holds),
$\cH_\lambda(D)|_{\widetilde{G}_1}$ also decomposes multiplicity-freely.
For example, if $\fp^+_2$ is simple, we write the decomposition of $\cP(\fp^+_2)$ under $K_1$ as
\[
\cP(\fp^+_2)=\bigoplus_{\bk\in\BZ_{++}^{r_2}}\cP_\bk(\fp^+_2),
\]
where $\BZ_{++}^{r_2}:=\{\bk=(k_1,\dots,k_{r_2})\in\BZ^{r_2}\mid k_1\ge\cdots\ge k_{r_2}\ge 0\}$.
We~fix characters $\chi$ and $\chi_1$ of~$\widetilde{K}^\BC$ and $\widetilde{K}^\BC_1$ respectively,
and let $\chi|_{\widetilde{K}^\BC_1}=\chi_1^{\varepsilon_1}$, where $\varepsilon_1\in\{1,2\}$.
Then the holomorphic discrete series representation $\cH_\lambda(D)$ of $\widetilde{G}$
is decomposed multiplicity-freely under $\widetilde{G}_1$ as
\begin{equation}\label{intro_decomp}
\cH_\lambda(D)|_{\widetilde{G}_1}\simeq\hsum_{\bk\in\BZ_{++}^{r_2}}\cH_{\varepsilon_1\lambda}(D_1,\cP_\bk(\fp^+_2))
\end{equation}
(see Kobayashi \cite[Theorem 8.3]{Kmf1}).
We~note that if $\cH_\lambda(D)$ is not holomorphic discrete and $\cH_\lambda(D)_{\widetilde{K}}\subsetneq \cP(\fp^+)$ holds,
then (\ref{intro_decomp}) does not hold in general.

In the following we consider holomorphic discrete series representations of scalar type $\cH_\lambda(D)$ of $\widetilde{G}$.
In order to understand the above decomposition (\ref{intro_decomp}) concretely, we want to compute explicitly the inner product
$\langle f,g\rangle_{\hat{\tau}}=\langle f,g\rangle_\lambda$ of $f\in\cP_\bk(\fp^+_2)$ and $g\in\cP(\fp^+)$.
To do this, since ${\rm e}^{\left(x\middle|\frac{\partial}{\partial z}\right)_{\fp^+}}g(z)\bigr|_{z=0}=g(x)$ holds
for any polynomial $g\in\cP(\fp^+)$, it is enough to compute when $g(x)={\rm e}^{(x|\overline{z})_{\fp^+}}$.
Therefore, in this article we aim to compute the inner product
\begin{equation}\label{intro_aim}
\big\langle f(x_2),{\rm e}^{(x|\overline{z})_{\fp^+}}\big\rangle_{\lambda,x}, \qquad f\in\cP_\bk(\fp^+_2),\quad x=(x_1,x_2),\quad z\in\fp^+.
\end{equation}
This inner product is given by an explicit integral for sufficiently large $\lambda$,
and once we compute this explicitly, then we can consider the meromorphic continuation for all $\lambda\in\BC$.
By observing the poles and their orders of the inner product~(\ref{intro_aim}),
we can determine which $\big(\fg,\widetilde{K}\big)$-submodule of $\cO_\lambda(D)_{\widetilde{K}}$ contains the $\big(\fg_1,\widetilde{K}_1\big)$-submodule
generated by $\cP_\bk(\fp^+_2)$.

Next we consider $\widetilde{G}_1$-intertwining operators
\begin{align*}
&\cF^\downarrow_{\lambda,W}\colon\quad \cH_\lambda(D)|_{\widetilde{G}_1}\longrightarrow \cH_{\varepsilon_1\lambda}(D_1,W)
\qquad\text{or}\qquad\cO_\lambda(D)|_{\widetilde{G}_1}\longrightarrow \cO_{\varepsilon_1\lambda}(D_1,W), \\
&\cF^\uparrow_{\lambda,W}\colon\quad \cH_{\varepsilon_1\lambda}(D_1,W)\longrightarrow \cH_\lambda(D)|_{\widetilde{G}_1}
\qquad\text{or}\qquad\cO_{\varepsilon_1\lambda}(D_1,W)\longrightarrow \cO_\lambda(D)|_{\widetilde{G}_1}.
\end{align*}
$\cF^\downarrow_{\lambda,W}$ is called a \textit{symmetry breaking operator},
and $\cF^\uparrow_{\lambda,W}$ is called a \textit{holographic operator}, according to the terminology introduced in \cite{KP1, KP2} and \cite{KP3} respectively.
By the multiplicity-freeness \cite[Theorem A]{Kmf1}, such intertwining operators $\cF^\downarrow_{\lambda,W}$,
$\cF^\uparrow_{\lambda,W}$ between unitary highest weight representations of scalar type are unique up to scalar multiple.
Hence we want to construct these operators explicitly.
Such problem is proposed by T.~Kobayashi from the viewpoint of the representation theory (see~\cite{K1}),
and studied from various viewpoints. For example, symmetry breaking operators given by differential operators are studied
in the context of automorphic forms, conformal geometry and representation theory by, e.g., \cite{BCK, Cl, C, IKO, Ju, Kfmeth, KKP, KOSS, KP1, KP2, OR, P, PZ, R}.
Especially, in \cite{Kfmeth, KP1} it is proved that the symbols of differential symmetry breaking operators are characterized as
polynomial solutions of certain systems of partial differential equations in general settings (F-method),
and symmetry breaking operators between the spaces of holomorphic sections for homogeneous vector bundles on $D$ and $D_1$ are always given by differential operators.
Also, holographic operators for holomorphic discrete series representations are stu\-died by, e.g., \cite{KP3, N2},
and integral intertwining operators for principal series or complementary series representations are studied by, e.g., \cite{KS1, KS2, MO}.
In this article we construct differential symmetry breaking operators $\cF^\downarrow_{\lambda,W}$
for holomorphic discrete series representations by using the author's previous result \cite[Theorem 3.10]{N2}, which claims that
the differential operator
\[
\cF^\downarrow_{\lambda,W}\colon\quad \cH_\lambda(D)|_{\widetilde{G}_1}\longrightarrow \cH_{\varepsilon_1\lambda}(D_1,W), \qquad
f(x)=f(x_1,x_2)\mapsto F^\downarrow_{\lambda,W}\bigg(\frac{\partial}{\partial x}\bigg)f(x)\bigg|_{x_2=0}
\]
defined by using the polynomial $F^\downarrow_{\lambda,W}(z)\in\cP(\fp^-)\otimes W$,
\begin{equation}\label{intro_SBO}
F^\downarrow_{\lambda,W}(z):=\big\langle {\rm e}^{(x|z)_{\fp^+}},\rK(x_2)\big\rangle_{\lambda,x}, \qquad
\rK(x_2)\in \big(\cP(\fp^+_2)\otimes\overline{W}\big)^{K_1}
\end{equation}
becomes a symmetry breaking operator. Moreover, by considering the meromorphic continuation with respect to $\lambda$,
this gives a symmetry breaking operator also for non-discrete highest weight representations.
That is, the construction of differential symmetry breaking operators is reduced to the
computation of the inner products (\ref{intro_aim}), and conversely,
we can use differential equations (F-method) to compute the inner products (\ref{intro_aim}) up to constant multiple.
However, for non-unitary case we do not know a priori whether the symmetry breaking operators are unique or not,
and for uniqueness we need further study as in \cite{KP1, KP2}.
In fact, for tensor product case the space of symmetry breaking operators sometimes becomes 2-dimensional, even for scalar case (see \cite[Section 9]{KP2}).

When $(G,G_1)$ is of the form $(G_1\times G_1,\Delta G_1)$, the computation of the inner product (\ref{intro_aim})
and the construction of symmetry breaking operators are studied in \cite[Section 5.2]{N2}, \cite{PZ} (see also, e.g., \cite{BCK, Cl, KP2, KP3, OR, P}).
In the following we consider the case $G$ is simple, and for a while we assume $\fp^+$, $\fp^+_2$ are of tube type
(i.e., the corresponding Lie groups $G$ and $G_2:=(G^{\sigma\vartheta})_0$ are of tube type, where $\sigma$, $\vartheta$ are suitable involutions on $G$
satisfying $G_1=(G^\sigma)_0$, $K=G^\vartheta$),
so that $\fp^+=\fn^{+\BC}$, $\fp^+_2=\fn^{+\BC}_2$ have complex Jordan algebra structures.
Let $\det_{\fn^+_2}(x_2)$ be the determinant polynomial on $\fp^+_2=\fn^{+\BC}_2$,
so that $\BC\det_{\fn^+_2}(x_2)^k\subset\cP(\fp^+_2)$ gives a 1-dimensional $K_1$-submodule.
In this article, when $\fp^+_2=\fp^+_{11}\oplus\fp^+_{22}$ is a direct sum of
two simple Jordan subalgebras, that is, when $(G,G_1)$ is one of
\begin{gather*}
({\rm SO}_0(2,d+2),{\rm SO}_0(2,d)\times {\rm SO}(2)), \qquad
(\operatorname{Sp}(r,\BR),{\rm U}(r',r'')),
\\
({\rm U}(r,r),{\rm U}(r',r'')\times {\rm U}(r'',r')), \qquad
({\rm SO}^*(4r),{\rm U}(2r',2r'')), \qquad
(E_{7(-25)},{\rm U}(1)\times E_{6(-14)}),
\end{gather*}
for $f(x_2)=\det_{\fn^+_{11}}(x_{11})^kf_2(x_{22})\in\cP_{(k,\dots,k)}(\fp^+_{11})\boxtimes\cP_\bl(\fp^+_{22})$,
we compute the inner product
\[
\big\langle \det_{\fn^+_{11}}(x_{11})^kf_2(x_{22}),{\rm e}^{(x|\overline{z})_{\fp^+}}\big\rangle_{\lambda,x},
\]
and compute the poles and their orders with respect to the parameter $\lambda$ (Theorem~\ref{thm_nonsimple} and~Corollary~\ref{cor_pole_nonsimple}).
Especially if $f(x_2)=\det_{\fn^+_{11}}(x_{11})^k\det_{\fn^+_{22}}(x_{22})^l$,
then the result is given by using a~multivariate generalization of Gauss' hypergeometric polynomials ${}_2F_1$,
which coincides with a~special case of Heckman--Opdam's hypergeometric polynomials of type $BC$
(Corollary~\ref{cor_scalar_nonsimple} and~Theorem~\ref{explicit_nonsimple}).
Similarly, when $\fp^+_2$ is a simple Jordan algebra, that is, when $(G,G_1)$ is one of
\begin{alignat*}{3}
&({\rm SO}_0(2,n),{\rm SO}_0(2,n')\times {\rm SO}(n'')), \qquad &&(\operatorname{Sp}(2r',\BR),\operatorname{Sp}(r',\BR)\times \operatorname{Sp}(r',\BR)),&
\\
&({\rm SO}^*(4s'),{\rm SO}^*(2s')\times {\rm SO}^*(2s')), \qquad&&
({\rm SU}(r,r),{\rm SO}^*(2r)),&
\\
&({\rm SU}(2s,2s),\operatorname{Sp}(2s,\BR)), \qquad&& (E_{7(-25)},{\rm SU}(2,6)),&
\end{alignat*}
$n''\ne 2$, for $f(x_2)\in\cP_{(k+l,k,\dots,k)}(\fp^+_2)$ (for the 1st, 3rd, 4th and 6th cases) or
for $f(x_2)\in\cP_{(k+1,\dots,k+1,k,\dots,k)}(\fp^+_2)$ (for the 2nd and 5th cases), we compute the inner product
\[
\big\langle f(x_2),{\rm e}^{(x|\overline{z})_{\fp^+}}\big\rangle_{\lambda,x},
\]
and compute the poles and their orders (Theorem \ref{thm_simple} and~Corollary~\ref{cor_pole_simple}).
Especially if $l=0$ and $f(x_2)=\det_{\fn^+_2}(x_2)^k\in\cP_{(k,\dots,k)}(\fp^+_2)$,
then the result is again given by multivariate ${}_2F_1$, and if $\rank_\BR G=2$ with general $l$, then the result is given by ordinary ${}_2F_1$
(Corollary~\ref{cor_scalar_simple} and~Theo\-rem \ref{explicit_simple}).
The above two lists exhaust all symmetric pairs of holomorphic type such that~$G$ is simple and $\fp^+$, $\fp^+_2$ are of tube type
(for the classification without tube type assumption see \cite[Table 3.4.1]{Kmf1}).
Moreover, by the $\widetilde{K}_1$-equivariance, the results on poles are generalized for~$\fp^+$, $\fp^+_2$ of non-tube type,
and for $f(x_2)\in\cP_{(k,\dots,k,0,\dots,0)}(\fp^+_{11})\boxtimes\cP_\bl(\fp^+_{22})$, $\cP_{(k+l,k,\dots,k,0,\dots,0)}(\fp^+_2)$
or $\cP_{(k+1,\dots,k+1,k,\dots,k,0,\dots,0)}(\fp^+_2)$ (Corollaries \ref{cor_pole_nonsimple}, \ref{cor_pole_simple}).
These results are applied to construct symmetry breaking operators explicitly.
When $(G,G_1)=({\rm SO}_0(2,n),{\rm SO}_0(2,n-1))$, the resulting operators coincide with the holomorphic version of Juhl's operators
(see \cite{Ju}, \cite[Section 6]{KP2}),
and when $(G,G_1)=(\operatorname{Sp}(r,\BR),\operatorname{Sp}(r',\BR)\times \operatorname{Sp}(r'',\BR))$ $(r=r'+r'')$, the resulting operators for scalar type representations
coincide with the ones in \cite[Section 7]{KP2} (for $r'=1$ case) and \cite{IKO} (for $r'=r''$ case) (see also Remark \ref{rem_previous}).

This paper is organized as follows. In Section~\ref{section_prelim1},
we review Jordan triple systems, Jordan algebras and holomorphic discrete series representations.
In Section~\ref{section_prelim2}, we fix the explicit rea\-li\-za\-tion of Jordan triple systems.
In Section~\ref{section_key}, we rewrite the inner product (\ref{intro_aim}) in a differential expression
by using the inverse Laplace transform on Jordan algebras (Theorem \ref{key_identity}). This formula is regarded as an analogue of
the Rodrigues formulas for Jacobi polynomials (for such formulas for symmetry breaking operators in tensor product case, see also \cite{Cl}).
In Section~\ref{section_nonsimple}, we compute (\ref{intro_aim}) when $\fp^+_2=\fp^+_{11}\oplus\fp^+_{22}$ is a direct sum of two
Jordan triple systems, and in Section~\ref{section_simple}, we compute (\ref{intro_aim}) when $\fp^+_2$ is a simple Jordan triple system.
In Section~\ref{section_submod}, we apply the results on poles to determining which $(\fg,\widetilde{K})$-submodule of
$\cO_\lambda(D)_{\widetilde{K}}$ contains the $(\fg_1,\widetilde{K}_1)$-submodule generated by $\cP_\bk(\fp^+_2)$ for singular $\lambda$.
In Section~\ref{section_SBO}, we construct the symmetry breaking operators explicitly by using (\ref{intro_SBO}),
which are unique up to scalar multiple, and for convenience we also write the author's previous results \cite{N2} on holographic operators.\vspace{-1mm}

\section{Preliminaries 1: General theory}\label{section_prelim1}

In this section we review Jordan triple systems, Jordan algebras and holomorphic discrete series representations.
For detail see, e.g., \cite[Parts III and V]{FKKLR}, \cite{FK,L,Sat}.\vspace{-1mm}

\subsection{Hermitian positive Jordan triple systems}

Let $(\fp^+,\fp^-,\{\cdot,\cdot,\cdot\},\overline{\cdot})$ be a \textit{Hermitian positive Jordan triple system}, that is,
$\fp^\pm$ be finite-dimen\-sional vector spaces over $\BC$, with a non-degenerate pairing
$(\cdot|\cdot)_{\fp^\pm}\colon\fp^\pm\times\fp^\mp\to\BC$,
$\{\cdot,\cdot,\cdot\}\colon\fp^\pm\times\fp^\mp\times\fp^\pm\to\fp^\pm$ be a $\BC$-trilinear map satisfying\vspace{-1mm}
\begin{gather*}
\{x,y,z\}=\{z,y,x\},
\\
\{u,v,\{x,y,z\}\}=\{\{u,v,x\},y,z\}-\{x,\{v,u,y\},z\}+\{x,y,\{u,v,z\}\},
\\
(\{u,v,x\}|y)_{\fp^\pm}=(x|\{v,u,y\})_{\fp^\pm}
\end{gather*}
for any $u,x,z\in\fp^\pm$, $v,y\in\fp^\mp$, and $\overline{\cdot}\colon\fp^\pm\to\fp^\mp$ be a $\BC$-antilinear involutive isomorphism such that
$(x|\overline{x})_{\fp^+}\ge 0$ holds for any $x\in\fp^+$. We~define $D\colon\fp^\pm\times\fp^\mp\to\End_\BC(\fp^\pm)$,
$Q\colon\fp^\pm\times\fp^\pm\to\Hom_\BC(\fp^\mp,\fp^\pm)$ by
\[
D(x,y)z=Q(x,z)y=\{x,y,z\},
\]
and we write $Q(x):=\frac{1}{2}Q(x,x)$. Also let $B\colon\fp^\pm\times\fp^\mp\to\End_\BC(\fp^\pm)$ be the \textit{Bergman operator} given by
\begin{equation*}%\label{Bergman_op}
B(x,y):=I_{\fp^\pm}-D(x,y)+Q(x)Q(y),
\end{equation*}
let $h=h_{\fp^\pm}\colon \fp^\pm\times\fp^\mp\to\BC$ be the \textit{generic norm}, which is a polynomial on $\fp^\pm\times\fp^\mp$,
and write $B(x)\!:=B(x,\overline{x})$, $h(x)\!:=h(x,\overline{x})$. Then the \textit{bounded symmetric domain} $D=D_{\fp^+}\!\subset\!\fp^+$ is given~by
\begin{align}
D&=(\text{connected component of }\{x\in \fp^+\mid B(x)\text{ is positive definite}\}\text{ which contains }0) \notag \\
&=(\text{connected component of }\{x\in \fp^+\mid h(x)>0\}\text{ which contains }0), \label{BSD}
\end{align}
and its \textit{Bergman--Shilov boundary} $\Sigma=\Sigma_{\fp^+}\subset\fp^+$ is given by
\[
\Sigma=\{x\in \fp^+\mid B(x)=0\}.
\]
For $(x,y)\in\fp^\pm\times\fp^\mp$, we say $(x,y)$ is \textit{quasi-invertible}
if $B(x,y)\in\End_\BC(\fp^\pm)$ is invertible, and then the \textit{quasi-inverse} $x^y\in\fp^\pm$ is given by
\[
x^y:=B(x,y)^{-1}(x-Q(x)y)\in\fp^\pm.
\]

An element $e\in\fp^+$ is called a \textit{tripotent} if it satisfies $\{e,\overline{e},e\}=2e$.
Then $D(e,\overline{e})\in\End_\BC(\fp^+)$ has eigenvalues 0, 1, 2. For a tripotent $e\in\fp^+$ we write
\begin{gather}
\fp^+(e)_j=\fp^+(\overline{e})_j:=\{x\in\fp^+\mid D(e,\overline{e})x=jx\}, \notag
\\
\fp^-(\overline{e})_j=\fp^-(e)_j:=\{x\in\fp^-\mid D(\overline{e},e)x=jx\}, \qquad j\in\{0,1,2\}, \label{Peirce}
\end{gather}
so that $\fp^\pm=\fp^\pm(e)_2\oplus\fp^\pm(e)_1\oplus\fp^\pm(e)_0$ holds (\textit{Peirce decomposition}). Then we have
\begin{gather*}
\{\fp^+(e)_j,\fp^-(e)_k,\fp^+(e)_l\}\subset \fp^+(e)_{j-k+l}, \qquad j,k,l\in\{0,1,2\},
\\
\{\fp^+(e)_2,\fp^-(e)_0,\fp^+\}=\{\fp^+(e)_0,\fp^-(e)_2,\fp^+\}=0,
\end{gather*}
where for the first formula we regard the right hand side as zero if $j-k+l\notin\{0,1,2\}$.
Especially~$\fp^+(e)_j$ $(j=0,1,2)$ are Jordan triple subsystems of $\fp^+$.
A non-zero tripotent $e\in\fp^+$ is called \textit{primitive} when $\fp^+(e)_2=\BC e$, and is called \textit{maximal} when $\fp^+(e)_0=\{0\}$.
Then the set of all maximal tripotents in $\fp^+$ coincides with the Bergman--Shilov boundary $\Sigma$.
We~call $\fp^+$ is of \textit{tube type} if $\fp^+(e)_2=\fp^+$ holds for some (or equivalently any) maximal tripotent $e$.
Especially~$\fp^+(e)_2$ is always a Jordan triple system of tube type.
Throughout the paper we assume that the pairing $(\cdot|\cdot)_{\fp^\pm}\colon\fp^\pm\times\fp^\mp\to\BC$ is normalized such that
$(e|\overline{e})_{\fp^+}=(\overline{e}|e)_{\fp^-}=1$ holds for any primitive tripotent $e\in\fp^+$.

A maximal set of primitive tripotents $\{e_1,\dots,e_r\}\subset\fp^+$ satisfying $D(e_i,\overline{e}_j)=0$ $(i\ne j)$ is called
a \textit{Jordan frame} of $\fp^+$. Then $r$ is called the \textit{rank} of the Jordan triple system $\fp^+$, and $e=\sum_{j=1}^re_j$ becomes
a maximal tripotent. When $\fp^+$ is simple, we define integers $d$, $b$, $p$, $n$ by
\begin{gather}
d:=\dim(\fp^+(e_i)_1\cap\fp^+(e_j)_1), \quad i\ne j,\qquad b:=\dim(\fp^+(e_j)_1\cap\fp^+(e)_1), \notag
\\
p:=2+d(r-1)+b,\qquad n:=\dim\fp^+=r+\frac{d}{2}r(r-1)+br. \label{str_const}
\end{gather}
We~note that if $r=1$, then $d$ is not determined uniquely, and any number is allowed.
Then $h(x,y)$ and $B(x,y)$, $(x|y)_{\fp^+}$ and $D(x,y)$ are related as
\[
h(x,y)^p=\Det_{\fp^+}B(x,y),\qquad
p(x|y)_{\fp^+}=\Tr_{\fp^+}D(x,y), \qquad
x\in\fp^+,\quad y\in\fp^-.
\]

\subsection{Jordan algebras}

In this section we fix a tripotent $e\in\fp^+$, and consider $\fp^\pm(e)_2\subset\fp^\pm$ as in (\ref{Peirce}). Then
\[
Q(\overline{e})\colon\ \fp^+(e)_2\longrightarrow\fp^-(e)_2,\qquad Q(e)\colon\ \fp^-(e)_2\longrightarrow\fp^+(e)_2
\]
are bijective and mutually inverse, and $\fp^+(e)_2$ has a Jordan algebra structure with the product
\begin{equation*}%\label{Jordan_prod}
x\cdot y:=\frac{1}{2}\{x,\overline{e},y\},
\end{equation*}
that is,
\[
x\cdot y=y\cdot x,\qquad x^{\mathit{2}}\cdot(x\cdot y)=x\cdot\big(x^{\mathit{2}}\cdot y\big)
\]
holds for any $x,y\in\fp^+(e)_2$, where $x^{\mathit{2}}:=x\cdot x$. The unit element is $e$, and $\fp^+(e)_2$ has the Euclidean real form
\begin{equation*}%\label{Euclidean}
\fn^+:=\{x\in\fp^+(e)_2\mid Q(e)\overline{x}=x\}\subset \fp^+(e)_2.
\end{equation*}
Let $(\cdot|\cdot)_{\fn^+}$ be the symmetric bilinear form on $\fp^+(e)_2=\fn^{+\BC}$ given by
\[
(x|y)_{\fn^+}:=(x|Q(\overline{e})y)_{\fp^+}=(y|Q(\overline{e})x)_{\fp^+}, \qquad x,y\in\fp^+(e)_2.
\]
Then this is positive definite on $\fn^+$. Also let $\tr_{\fn^+}(x):=(x|e)_{\fn^+}$,
let $\det_{\fn^+}(x)$ be the \textit{determinant polynomial} on $\fp^+(e)_2=\fn^{+\BC}$,
and let $\det_{\fn^-}(y)=\det_{\fn^+}(Q(e)y)$ for $y\in\fp^-(e)_2$.

Next, let $L\colon\fp^+(e)_2\to\End_\BC(\fp^+(e)_2)$ be the multiplication operator, that is,
\[
L(x)y:=x\cdot y=\frac{1}{2}D(x,\overline{e})y,
\]
and let $P\colon\fp^+(e)_2\to\End_\BC(\fp^+(e)_2)$, $P, D_{\fn^+}, B_{\fn^+}\colon\fp^+(e)_2\times\fp^+(e)_2\to\End_\BC(\fp^+(e)_2)$
and $h_{\fn^+}\colon\allowbreak\fp^+(e)_2\times\fp^+(e)_2\to\BC$ be the maps given by
\begin{gather*}
P(x):=2L(x)^2-L\big(x^{\mathit{2}}\big)=Q(x)Q(\overline{e}),
\\
P(x,y):=2(L(x)L(y)+L(y)L(x)-L(x\cdot y))=Q(x,y)Q(\overline{e}),
\\
D_{\fn^+}(x,y):=2(L(x\cdot y)+L(x)L(y)-L(y)L(x))=D(x,Q(\overline{e})y),
\\
B_{\fn^+}(x,y):=I_{\fp^+(e)_2}-D_{\fn^+}(x,y)+P(x)P(y)=B(x,Q(\overline{e})y),
\\
h_{\fn^+}(x,y):= h(x,Q(\overline{e})y).
\end{gather*}
Then for any $x,y,z,w\in\fp^+(e)_2$ we have
\begin{gather*}
P(P(x)y)=P(x)P(y)P(x), \qquad
P(B_{\fn^+}(x,y)z)=B_{\fn^+}(x,y)P(z)B_{\fn^+}(y,x),
\\
(P(x,z)y|w)_{\fn^+}=(P(y,w)x|z)_{\fn^+},
\end{gather*}
and if $\fp^+(e)_2$ is simple, then the determinant polynomial $\det_{\fn^+}(x)$ and $P(x)$ are related as
\[
\det_{\fn^+}(x)^{\frac{2n'}{r'}}=\det_{\fn^+}(x)^{2+d(r'-1)}=\Det_{\fp^+(e)_2}(P(x)), \qquad x\in\fp^+(e)_2,
\]
where $r':=\rank\fp^+(e)_2$, $n':=\dim\fp^+(e)_2=r'+\frac{d}{2}r'(r'-1)$.
The following subset $\Omega\subset\fn^+$,
\begin{align*}
\Omega&=(\text{connected component of }\{x\in \fn^+\mid P(x)\text{ is positive definite}\}\text{ which contains }e) \\
&=(\text{connected component of }\{x\in \fn^+\mid \det_{\fn^+}(x)>0\}\text{ which contains }e)
\end{align*}
is called the \textit{symmetric cone}, and $(x|y)_{\fn^+}>0$ holds for any $x,y\in\Omega$.

We~say $x\in\fp^+(e)_2$ is \textit{invertible} if $P(x)$ is invertible, and then the \textit{inverse} $x^\itinv\in\fp^+(e)_2$ is given by
\[
x^\itinv:=P(x)^{-1}x\in\fp^+(e)_2.
\]
Then $P(x)^{-1}=P(x^\itinv)$ holds. Also, for $x\in\fp^\pm(e)_2$, if $Q(x)|_{\fp^\mp(e)_2}$ is invertible, then we set
\[
{}^t\hspace{-1pt}x^\itinv:=\big(Q(x)|_{\fp^\mp(e)_2}\big)^{-1}x\in\fp^\mp(e)_2.
\]
If $\fp^+$ is of tube type and if $e$ is a maximal tripotent so that $\fp^+=\fp^+(e)_2=\fn^{+\BC}$ holds, then
${}^t\hspace{-1pt}x^\itinv\in\fp^\mp$ does not depend on the choice of $e$,
and ${}^t\hspace{-1pt}x^\itinv=Q(e)x^\itinv$, $Q(x)^{-1}=Q\big({}^t\hspace{-1pt}x^\itinv\big)$ hold for $x\in\fp^+$.
Also, for $x,y\in\fp^+=\fn^{+\BC}$, if $y$ is invertible then we have
{\samepage\begin{gather}
P(x+y)=B_{\fn^+}\big({-}x,y^\itinv\big)P(y)=B\big({-}x,{}^t\hspace{-1pt}y^\itinv\big)P(y), \notag
\\
\det_{\fn^+}(x+y)=h_{\fn^+}\big({-}x,y^\itinv\big)\det_{\fn^+}(y)=h\big({-}x,{}^t\hspace{-1pt}y^\itinv\big)\det_{\fn^+}(y). \label{det_h}
\end{gather}}

Next we assume $\fp^+$ is of tube type, take two tripotents $e',e''\in\fp^+$ such that $D(e',\overline{e''})=0$ and $e=e'+e''$ is maximal,
and regard $\fp^+$ as a Jordan algebra by using $e$. Let
\[
\fp^+_{11}:=\fp^+(e')_2=\fp^+(e'')_0, \qquad
\fp^+_{12}:=\fp^+(e')_1=\fp^+(e'')_1, \qquad
\fp^+_{22}:=\fp^+(e')_0=\fp^+(e'')_2,
\]
so that $\fp^+_{11}$, $\fp^+_{22}$ are Jordan subalgebras. We~use the same notation $x^\itinv$ to express the inverses in $\fp^+_{jj}$.
Then for $x_{ij},y_{ij}\in\fp^+_{ij}$ we have
\begin{alignat*}{3}
& P(x_{11}+x_{22})y_{11}=P(x_{11})y_{11}\in\fp^+_{11}, \qquad && P(x_{12})y_{11}\in\fp^+_{22},&
\\
& P(x_{11}+x_{22})y_{12}=P(x_{11},x_{22})y_{12}\in\fp^+_{12}, \qquad && P(x_{12})y_{12}\in\fp^+_{12},&
\\
& P(x_{11}+x_{22})y_{22}=P(x_{22})y_{22}\in\fp^+_{22}, \qquad && P(x_{12})y_{22}\in\fp^+_{11}.&
\end{alignat*}
Moreover we have the following.

\begin{Lemma}\label{lemma_Jordansub}
For $x_{jj},y_{jj}\in\fp^+_{jj}$, $z_{12}\in\fp^+_{12}$, we have
\begin{enumerate}\itemsep=0pt
\item[$1.$] $P(x_{11}+x_{22})z_{12}=P(x_{11},x_{22})z_{12}=4L(x_{11})L(x_{22})z_{12}=4L(x_{22})L(x_{11})z_{12}$.
\item[$2.$] $D_{\fn^+}(x_{11},y_{11})z_{12}=4L(x_{11})L(y_{11})z_{12}$, $D_{\fn^+}(x_{22},y_{22})z_{12}=4L(x_{22})L(y_{22})z_{12}$.
\item[$3.$] $4L(x_{11})L\big(x_{11}^\itinv\big)z_{12}=z_{12}$, $4L(x_{22})L\big(x_{22}^\itinv\big)z_{12}=z_{12}$.
\end{enumerate}
\end{Lemma}

\begin{proof}
(1) By the definition of $P$, $L$ and \cite[Part~V, Proposition~I.2.1(J1.2)]{FKKLR} we have
\begin{gather*}
P(x_{11},x_{22})z_{12} =Q(x_{11},x_{22})Q(\overline{e})z_{12}\!=Q(x_{11},x_{22})Q(\overline{e'},\overline{e''})z_{12}\\
\qquad{} =-D(\{x_{11},\overline{e'},x_{22}\},\overline{e''})z_{12}
 +D(x_{11},\overline{e'})D(x_{22},\overline{e''})z_{12}
+D(x_{22},\overline{e'})D(x_{11},\overline{e''})z_{12} \\
\qquad{} = D(x_{11},\overline{e'})D(x_{22},\overline{e''})z_{12}=D(x_{11},\overline{e})D(x_{22},\overline{e})z_{12}=4L(x_{11})L(x_{22})z_{12}.
\end{gather*}
The last equality is also proved similarly.

(2) Similarly, by \cite[Part V, Proposition I.2.1(J2.1$'$)]{FKKLR} we have
\begin{align*}
D_{\fn^+}(x_{11},y_{11})z_{12}&=D(x_{11},Q(\overline{e})y_{11})z_{12}=D(x_{11},Q(\overline{e'})y_{11})z_{12} \\
&=D(x_{11},\overline{e'})D(y_{11},\overline{e'})z_{12}-Q(x_{11},y_{11})Q(\overline{e'})z_{12} \\
&=D(x_{11},\overline{e})D(y_{11},\overline{e})z_{12}=4L(x_{11})L(y_{11})z_{12}.
\end{align*}
The 2nd formula is also proved similarly.

(3) By (1) we have
\begin{align*}
4L(x_{11})L\big(x_{11}^\itinv\big)z_{12}&=16L(x_{11})L(e'')L\big(x_{11}^\itinv\big)L(e'')z_{12}=P(x_{11}+e'')P\big(x_{11}^\itinv+e''\big)z_{12} \\
&=P(x_{11}+e'')P\big((x_{11}+e'')^\itinv\big)z_{12}=z_{12}.
\end{align*}
The 2nd formula is also proved similarly.
\end{proof}

\subsection{Symmetric subalgebras of Jordan algebras}

In this section we consider a $\BC$-linear involution $\sigma$ on a simple Hermitian positive Jordan triple system $\fp^\pm$, i.e., a Jordan triple system automorphism $\sigma\colon\fp^\pm\to\fp^\pm$ of order~2
which commutes with the $\BC$-antilinear map $\bar{\cdot}\colon\fp^\pm\to\fp^\mp$. We~write
\begin{align*}
\fp^\pm_1&:=\big(\fp^\pm\big)^\sigma=\big\{x\in\fp^\pm\mid \sigma(x)=x\big\}, \\
\fp^\pm_2&:=\big(\fp^\pm\big)^{-\sigma}=\big\{x\in\fp^\pm\mid \sigma(x)=-x\big\}.
\end{align*}
First we prove the following.
\begin{Proposition}
Assume $\fp^+$ is simple, and $\fp^+$, $\fp^+_2$ are of tube type.
\begin{enumerate}\itemsep=0pt
\item[$1.$] A maximal tripotent $e$ of $\fp^+_2$ is also a maximal tripotent of $\fp^+$.
\item[$2.$] If $\fp^+_2$ is not simple, then there exists a tripotent $e'\in\fp^+$ such that $\fp^+_2=\fp^+(e')_2\oplus\fp^+(e')_0$,
$\fp^+_1=\fp^+(e')_1$ hold.
\end{enumerate}
\end{Proposition}
\begin{proof}
(1) Let $e\in\fp^+_2=(\fp^+)^{-\sigma}$ be a maximal tripotent so that $\fp^+_2=(\fp^+(e)_2)^{-\sigma}$ holds,
and take a tripotent $e'\in\fp^+(e)_0=(\fp^+(e)_0)^\sigma$ such that $\tilde{e}=e+e'$ is a maximal tripotent of $\fp^+$.
We~regard $\fp^+$ as a Jordan algebra by $\tilde{e}$. Then for $x\in\fp^+(e)_1=(\fp^+(e)_1)^\sigma$ we have
\[
x^{\mathit{2}}=\frac{1}{2}\{x,\overline{e}+\overline{e'},x\}\in (\fp^+(e)_0)^{-\sigma}\oplus(\fp^+(e)_2)^\sigma
=(\fp^+(e)_2)^\sigma\subset\fp^+(e)_2.
\]
On the other hand, let $\{e_1,\dots,e_{r'}\}\subset\fp^+(e)_2$, $\{e_1',\dots,e_{r''}'\}\subset\fp^+(e)_0$ be Jordan frames.
Then we have
\[
\fp^+(e)_1=\bigoplus_{\substack{1\le i\le r'\\ 1\le j\le r''}}\fp^+(e_i)_1\cap\fp^+(e_j')_1,
\]
and for $x\in\fp^+(e_i)_1\cap\fp^+(e_j')_1\subset\fp^+(e)_1$, by \cite[Proposition IV.1.4]{FK} we have
\[
x^{\mathit{2}}=\frac{1}{2}(x|x)_{\fn^+}(e_i+e_j')\in\fp^+(e)_2\oplus\fp^+(e)_0.
\]
Therefore we must have $\fp^+(e)_1=\{0\}$, so that $\fp^+=\fp^+(e)_2\oplus\fp^+(e)_0$ holds.
Hence $\fp^+(e)_0\subset\fp^+$ is an ideal, and since $\fp^+$ is assumed to be simple, we have $\fp^+(e)_0=\{0\}$,
and therefore $e\in\fp^+$ is a maximal tripotent.

(2) Let $\fp^+_{11}$, $\fp^+_{22}\subset\fp^+_2=(\fp^+)^{-\sigma}$ be non-trivial ideals such that $\fp^+_2=\fp^+_{11}\oplus\fp^+_{22}$ holds,
and let $e'\in\fp^+_{11}$, $e''\in\fp^+_{22}$ be maximal tripotents of $\fp^+_{11}$, $\fp^+_{22}$ respectively.
Then $e=e'+e''$ is a maximal tripotent of $\fp^+_2$ and $\fp^+$. We~regard $\fp^+$ as a Jordan algebra by $e$.
Since $\fp^+_{11}$, $\fp^+_{22}$ are also of tube type, we have
\begin{gather*}
\fp^+_{11}=(\fp^+(e')_2)^{-\sigma}=(\fp^+(e'')_0)^{-\sigma}, \\
\fp^+_{22}=(\fp^+(e')_0)^{-\sigma}=(\fp^+(e'')_2)^{-\sigma}, \\
(\fp^+(e')_1)^{-\sigma}=(\fp^+(e'')_1)^{-\sigma}=\{0\}.
\end{gather*}
Let $x_2\in(\fp^+(e')_2)^\sigma$. Then for $y_1\in\fp^+(e')_1=(\fp^+(e')_1)^\sigma$ we have
\[
L(x_2)y_1=\frac{1}{2}\{x_2,\overline{e},y_1\}\in (\fp^+(e')_1)^{-\sigma}=\{0\}.
\]
Especially, for any $y_1,w_1\in\fp^+(e')_1$, $z_0\in\fp^+(e')_0$, by Lemma \ref{lemma_Jordansub}(1) we have
\[
(P(y_1,w_1)z_0|x_2)_{\fn^+}=(P(z_0,x_2)y_1|w_1)_{\fn^+}=4(L(z_0)L(x_2)y_1|w_1)_{\fn^+}=0,
\]
and since $P(\fp^+(e')_1,\fp^+(e')_1)\fp^+(e')_0=\fp^+(e')_2$ holds, we get $x_2=0$, and $(\fp^+(e')_2)^\sigma=\{0\}$.
Similarly we have $(\fp^+(e')_0)^\sigma=\{0\}$. Therefore we get $\fp^+_{11}=\fp^+(e')_2$, $\fp^+_{22}=\fp^+(e')_0$ and
$\fp^+_1=(\fp^+)^\sigma=\fp^+(e')_1$.
\end{proof}

\begin{Corollary}
For a simple Hermitian positive Jordan triple system $\fp^\pm$ and an involution $\sigma$ on $\fp^\pm$,
$\fp^+_2=(\fp^+)^{-\sigma}$ is a direct sum of at most two simple Jordan triple subsystems.
\end{Corollary}
\begin{proof}
If $\fp^+$, $\fp^+_2$ are of tube type, then since $\fp^+(e')_2$ and $\fp^+(e')_0$ are always simple for any tripotent $e'\in\fp^+$,
this follows from the previous proposition. If $\fp^+$, $\fp^+_2$ are general,
let $\fp^+_2=\bigoplus_{j=1}^k\fp^+_{2(j)}$ be the simple decomposition, and take maximal tripotents $e_j\in\fp^+_{2(j)}$.
Then $e=\sum_{j=1}^k e_j$ is a maximal tripotent of $\fp^+_2$, and we have
$(\fp^+(e)_2)^{-\sigma}=\fp^+_2(e)_2=\bigoplus_{j=1}^k \fp^+_{2(j)}(e)_2$. Therefore by the result for tube type case, we have $k\le 2$.
\end{proof}

In the rest of this section, we assume $\fp^+$, $\fp^+_2$ are of tube type, fix a maximal tripotent $e\in\fp^+_2\subset\fp^+$,
and regard $\fp^+$ as a Jordan algebra by $e$. If $\fp^+_2$ is not simple, we also write $\fp^+_2=\fp^+_{11}\oplus\fp^+_{22}$,
$\fp^+_1=\fp^+_{12}$, and take maximal tripotents $e'\in\fp^+_{11}$, $e''\in\fp^+_{22}$ such that $e=e'+e''$.
Let $\det_{\fn^+}(x)$, $\det_{\fn^+_{jj}}(x_{jj})$ be the determinant polynomials on the Jordan algebras $\fp^+$, $\fp^+_{jj}$,
and express the inverse on these by the same symbol $x^\itinv$, $x_{jj}^\itinv$.

\begin{Proposition}\label{prop_Jordansub}\quad
\begin{enumerate}\itemsep=0pt
\item[$1.$] For $x=x_1+x_2\in\fp^+=\fp^+_1\oplus\fp^+_2$, assume $x_2$ is invertible.
Then $x$ is invertible if and only if $x_2-P(x_1)x_2^\itinv$ is invertible, and we have
\begin{gather*}
x^\itinv=-P\big(x_2-P(x_1)x_2^\itinv\big)^{-1}(x_1-P(x_1)P(x_2)^{-1}x_1)+\big(x_2-P(x_1)x_2^\itinv\big)^\itinv,
\\
\det_{\fn^+}(x)^2=\det_{\fn^+}\big(x_2-P(x_1)x_2^\itinv\big)\det_{\fn^+}(x_2)
=h_{\fp^+}\big(Q(x_1){}^t\hspace{-1pt}x_2^\itinv,{}^t\hspace{-1pt}x_2^\itinv\big)\det_{\fn^+}(x_2)^2
\\ \hphantom{\det_{\fn^+}(x)^2}
{}=h_{\fp^+}\big(x_1,Q(x_2)^{-1}x_1\big)\det_{\fn^+}(x_2)^2.
\end{gather*}

\item[$2.$] For $x=x_{11}+x_{12}+x_{22}\in\fp^+=\fp^+_{11}\oplus\fp^+_{12}\oplus\fp^+_{22}$, assume $x_{11}$ and~$x_{22}$ are invertible
in $\fp^+_{11}$ and~$\fp^+_{22}$ respectively. Then $x$ is invertible if and only if
$x_{11}-P(x_{12})x_{22}^\itinv$ and $x_{22}-P(x_{12})x_{11}^\itinv$ are invertible in $\fp^+_{11}$ and $\fp^+_{22}$ respectively, and we have
\begin{gather*}
x^\itinv\!=\!\big(x_{11}\!-\!P(x_{12})x_{22}^\itinv\big)^\itinv\!
\!-\!P\big(\big(x_{11}\!-\!P(x_{12})x_{22}^\itinv\big)^\itinv\!\!,x_{22}^\itinv\big)x_{12}
\!+\!\big(x_{22}\!-\!P(x_{12})x_{11}^\itinv\big)^\itinv
\\ \hphantom{x^\itinv}
{}\!=\!\big(x_{11}\!-\!P(x_{12})x_{22}^\itinv\big)^\itinv\!
\!-\!P\big(\big(x_{22}\!-\!P(x_{12})x_{11}^\itinv\big)^\itinv\!\!,x_{11}^\itinv\big)x_{12} \!+\!\big(x_{22}\!-\!P(x_{12})x_{11}^\itinv\big)^\itinv\!,
\\
\det_{\fn^+}(x)=\det_{\fn^+_{11}}\big(x_{11}-P(x_{12})x_{22}^\itinv\big)\det_{\fn^+_{22}}(x_{22})
\\ \hphantom{\det_{\fn^+}(x)}
{}=h_{\fp^+_{11}}\big(Q(x_{12}){}^t\hspace{-1pt}x_{22}^\itinv,{}^t\hspace{-1pt}x_{11}^\itinv\big)
\det_{\fn^+_{11}}(x_{11})\det_{\fn^+_{22}}(x_{22})
\\ \hphantom{\det_{\fn^+}(x)}
{}=\det_{\fn^+_{22}}\big(x_{22}-P(x_{12})x_{11}^\itinv\big)\det_{\fn^+_{11}}(x_{11})
\\ \hphantom{\det_{\fn^+}(x)
}=h_{\fp^+_{22}}\big(Q(x_{12}){}^t\hspace{-1pt}x_{11}^\itinv,{}^t\hspace{-1pt}x_{22}^\itinv\big)
\det_{\fn^+_{11}}(x_{11})\det_{\fn^+_{22}}(x_{22}).
\end{gather*}
\end{enumerate}
\end{Proposition}

\begin{proof}
(1) For the 1st formula, first we assume $x_2=e$. Then since $\BC[x_1]\subset\fp^+$ is commutative and associative,
$e-P(x_1)e^\itinv=e-x_1^{\mathit{2}}=(e-x_1)\cdot(e+x_1)=-(x_1+e)\cdot\sigma(x_1+e)$ holds,
and this is invertible if and only if $x_1+e$ is invertible. Easily we have
\begin{align*}
(\text{RHS})&=-P\big(e-P(x_1)e^\itinv\big)^{-1}(x_1-P(x_1)x_1)+\big(e-P(x_1)e^\itinv\big)^\itinv
\\
&=-\big(e-x_1^\mathit{2}\big)^{\mathit{-2}}\cdot\big(x_1-x_1^{\mathit{3}}\big)+\big(e-x_1^{\mathit{2}}\big)^\itinv=(x_1+e)^\itinv
=(\text{LHS}).
\end{align*}
Next we consider general invertible $x_2$. Then there exists $\sqrt{x_2}\in\fp^+_2$ satisfying $(\sqrt{x_2})^{\mathit{2}}=x_2$ (not unique).
Let $z_1:=P\big(\sqrt{x_2}\big)^{-1}x_1$. Then we have
\begin{align*}
x_2-P(x_1)x_2^\itinv&=P\big(\sqrt{x_2}\big)e-P\big(P(\sqrt{x_2})z_1\big)P\big(\sqrt{x_2}\big)^{-1}e
=P\big(\sqrt{x_2}\big)e-P\big(\sqrt{x_2}\big)P(z_1)e
\\
&=P\big(\sqrt{x_2}\big)\big(e-z_1^{\mathit{2}}\big)=P\big(\sqrt{x_2}\big)\big((e-z_1)\cdot(e+z_1)\big),
\end{align*}
and this is invertible if and only if $z_1+e=P\big(\sqrt{x_2}\big)^{-1}(x_1+x_2)$ is invertible, or equivalently $x_1+x_2=x$ is invertible. Moreover we have
\begin{align*}
(\text{RHS}) ={}&-P\big(P\big(\sqrt{x_2}\big)e-P\big(P\big(\sqrt{x_2}\big)z_1\big)P\big(\sqrt{x_2}\big)^{-1}e\big)^{-1}
\\
&\times\big(P\big(\sqrt{x_2}\big)z_1-P\big(P\big(\sqrt{x_2}\big)z_1\big)P\big(\sqrt{x_2}\big)^{-1} z_1\big)
\\
&+\big(P\big(\sqrt{x_2}\big)e-P\big(P\big(\sqrt{x_2}\big)z_1\big)P\big(\sqrt{x_2}\big)^{-1}e\big)^\itinv
\\
={}&-P\big(P\big(\sqrt{x_2}\big)e-P\big(\sqrt{x_2}\big)P(z_1)e\big)^{-1}\big(P\big(\sqrt{x_2}\big)z_1-P\big(\sqrt{x_2}\big)P(z_1)z_1\big) \\
& +\big(P\big(\sqrt{x_2}\big)e-P\big(\sqrt{x_2}\big)P(z_1)e\big)^\itinv
\\
={}&-P\big(\sqrt{x_2}\big)^{-1}P(e-P(z_1)e)^{-1}(z_1-P(z_1)z_1)+P\big(\sqrt{x_2}\big)^{-1}(e-P(z_1)e)^\itinv
\\
={}&P\big(\sqrt{x_2}\big)^{-1}(z_1+e)^\itinv=\big(P\big(\sqrt{x_2}\big)z_1+P\big(\sqrt{x_2}\big)e\big)^\itinv=
(x_1+x_2)^\itinv=(\text{LHS}).
\end{align*}
This proves the 1st formula.
Next we consider the 2nd formula. Since $-\sigma$ fixes $e\in\fp^+_2$, $-\sigma$~is an automorphism of the Jordan algebra,
and therefore preserves the determinant, that is, $\det_{\fn^+}(x_1+x_2)=\det_{\fn^+}(-\sigma(x_1+x_2))=\det_{\fn^+}(-x_1+x_2)$ holds.
Then by (\ref{det_h}) and \cite[Part V, Proposition I.5.1(J4.2)]{FKKLR}, we have
\begin{align*}
\det_{\fn^+}(x_1+x_2)^2&=\det_{\fn^+}(x_1+x_2)\det_{\fn^+}(-x_1+x_2)
\\
&=h_{\fp^+}\big({-}x_1,{}^t\hspace{-1pt}x_2^\itinv\big)h_{\fp^+}\big(x_1,{}^t\hspace{-1pt}x_2^\itinv\big)\det_{\fn^+}(x_2)^2
\\
&=h_{\fp^+}\big(Q(x_1){}^t\hspace{-1pt}x_2^\itinv,{}^t\hspace{-1pt}x_2^\itinv\big)\det_{\fn^+}(x_2)^2
\\
&=\det_{\fn^+}\big(x_2-Q(x_1){}^t\hspace{-1pt}x_2^\itinv\big)\det_{\fn^+}(x_2)
\\
&=\det_{\fn^+}\big(x_2-P(x_1)x_2^\itinv\big)\det_{\fn^+}(x_2),
\end{align*}
and by \cite[Part V, Proposition I.5.1(J4.2), (J4.2$'$)]{FKKLR}, we have
\begin{align*}
h_{\fp^+}\big(Q(x_1){}^t\hspace{-1pt}x_2^\itinv,{}^t\hspace{-1pt}x_2^\itinv\big)
=h_{\fp^+}\big(x_1,Q\big({}^t\hspace{-1pt}x_2^\itinv\big)x_1\big)
=h_{\fp^+}\big(x_1,Q(x_2)^{-1}x_1\big).
\end{align*}

(2) For the 1st formulas, since
\[
x_{11}-P(x_{12})x_{22}^\itinv\in\fp^+_{11}, \qquad
x_{22}-P(x_{12})x_{11}^\itinv\in\fp^+_{22},
\]
by (1) with $x_{12}=x_1$, $x_{11}+x_{22}=x_2$ and by Lemma \ref{lemma_Jordansub}(1) we have
\begin{align*}
x^\itinv={}&-P\big(x_{11}+x_{22}-P(x_{12})\big(x_{11}^\itinv+x_{22}^\itinv\big)\big)^{-1}
\big(x_{12}-P(x_{12})P(x_{11}+x_{22})^{-1} x_{12}\big)
\\
& +\big(x_{11}+x_{22}-P(x_{12})\big(x_{11}^\itinv+x_{22}^\itinv\big)\big)^\itinv
\\
={}&-4L\big((x_{11}\!-P(x_{12})x_{22}^\itinv)^\itinv\big)
L\big((x_{22}\!-P(x_{12})x_{11}^\itinv)^\itinv\big)
\\
& {} \times \big(x_{12}\!-P(x_{12})P\big(x_{11}^\itinv,x_{22}^\itinv\big)x_{12}\big)+\big(x_{11}-P(x_{12})x_{22}^\itinv\big)^\itinv+\big(x_{22}-P(x_{12})x_{11}^\itinv\big)^\itinv.
\end{align*}
Then by \cite[Part V, (J1*) after Remark I.2.1]{FKKLR} and Lemma \ref{lemma_Jordansub}(2), (3) we have
\begin{align*}
x_{12}-P(x_{12})P\big(x_{11}^\itinv,x_{22}^\itinv\big)x_{12}
&=x_{12}-D_{\fn^+}\big(P(x_{12})x_{11}^\itinv,x_{22}^\itinv\big)x_{12}
\\
&=4L(x_{22})L\big(x_{22}^\itinv\big)x_{12}-4L\big(P(x_{12})x_{11}^\itinv\big)L\big(x_{22}^\itinv\big)x_{12}
\\
&=4L(x_{22}-P(x_{12})x_{11}^\itinv)L\big(x_{22}^\itinv\big)x_{12},
\end{align*}
and hence
\begin{gather*}
4L\big(\big(x_{11}-P(x_{12})x_{22}^\itinv\big)^\itinv\big)L\big(\big(x_{22}-P(x_{12})x_{11}^\itinv\big)^\itinv\big)
\big(x_{12}-P(x_{12})P\big(x_{11}^\itinv,x_{22}^\itinv\big)x_{12}\big)
\\ \qquad
=4L\big(\big(x_{11}-P(x_{12})x_{22}^\itinv\big)^\itinv\big)L\big(x_{22}^\itinv\big)x_{12}
=P\big(\big(x_{11}-P(x_{12})x_{22}^\itinv\big)^\itinv,x_{22}^\itinv\big)x_{12}.
\end{gather*}
This proves the 1st equality. The 2nd equality is also proved similarly.
Next we prove the 2nd formulas. To do this, we consider $B_{\fn^+}\big(x_{12},x_{11}^\itinv\big)\in\End_\BC(\fp^+)$. Then this satisfies
\begin{alignat*}{3}
&\big(B_{\fn^+}(x_{12},x_{11}^\itinv)-I\big)y_{11}=-D_{\fn^+}\big(x_{12},x_{11}^\itinv\big)y_{11}
+P(x_{12})P\big(x_{11}^\itinv\big)y_{11}\!\in\!\fp^+_{12}\oplus\fp^+_{22},\quad && y_{11}\in\fp^+_{11},&
\\
&\big(B_{\fn^+}\big(x_{12},x_{11}^\itinv\big)-I\big)y_{12}=-D_{\fn^+}\big(x_{12},x_{11}^\itinv\big)y_{12}\in\fp^+_{22} ,\qquad&& y_{12}\in\fp^+_{12},&
\\
&\big(B_{\fn^+}\big(x_{12},x_{11}^\itinv\big)-I\big)y_{22}=0,\qquad&& y_{22}\in\fp^+_{22}.&
\end{alignat*}
Thus $B_{\fn^+}\big(x_{12},x_{11}^\itinv\big)$ is strictly lower-triangular if we choose a basis suitably,
and hence $\Det_{\fp^+}\allowbreak\big(B_{\fn^+}\big(x_{12},x_{11}^\itinv\big)\big)=1$ holds. Similarly, $B_{\fn^+}(x_{11}^\itinv,x_{12})$ is strictly upper-triangular,
and hence $\Det_{\fp^+}\big(B_{\fn^+}\big(x_{11}^\itinv,x_{12}\big)\big)=1$ holds.
Now since we have
\begin{gather*}
B_{\fn^+}\big(x_{12},x_{11}^\itinv\big)(x_{11}+x_{12}+x_{22})
\\ \qquad
{}=\big(I-D_{\fn^+}\big(x_{12},x_{11}^\itinv\big)+P(x_{12})P\big(x_{11}^\itinv\big)\big)(x_{11}+x_{12}+x_{22})
\\ \qquad
{}=x_{11}+x_{12}+x_{22}-D_{\fn^+}\big(x_{12},x_{11}^\itinv\big)x_{11}-D_{\fn^+}\big(x_{12},x_{11}^\itinv\big)x_{12}
+P(x_{12})P\big(x_{11}^\itinv\big)x_{11}
\\ \qquad
{}=x_{11}+x_{12}+x_{22}-D_{\fn^+}\big(x_{11},x_{11}^\itinv\big)x_{12}-2P(x_{12})x_{11}^\itinv+P(x_{12})x_{11}^\itinv \\ \qquad
{}=x_{11}+x_{12}+x_{22}-x_{12}-P(x_{12})x_{11}^\itinv=x_{11}+x_{22}-P(x_{12})x_{11}^\itinv,
\end{gather*}
we get
\begin{align*}
\det_{\fn^+}(x)^{\frac{2n}{r}}&=\Det_{\fp^+}(P(x))
=\Det_{\fp^+}\big(P\big(B_{\fn^+}\big(x_{12},x_{11}^\itinv\big)^{-1} \big(x_{11}+x_{22}-P(x_{12})x_{11}^\itinv\big)\big)\big)
\\
&=\Det_{\fp^+}\big(B_{\fn^+}\big(x_{12},x_{11}^\itinv\big)^{-1} P\big(x_{11}+x_{22}-P(x_{12})x_{11}^\itinv\big)B_{\fn^+}\big(x_{11}^\itinv,x_{12}\big)^{-1}\big)
\\
&=\Det_{\fp^+}\big(P\big(x_{11}+x_{22}-P(x_{12})x_{11}^\itinv\big)\big)
\\
&=\det_{\fn^+}\big(x_{11}+x_{22}-P(x_{12})x_{11}^\itinv\big)^{\frac{2n}{r}},
\end{align*}
and since $x_{11}\in\fp^+_{11}$, $x_{22}-P(x_{12})x_{11}^\itinv\in\fp^+_{22}$ hold, by (\ref{det_h}),
\begin{align*}
\det_{\fn^+}(x)&=\det_{\fn^+}\big(x_{11}+x_{22}-P(x_{12})x_{11}^\itinv\big)=\det_{\fn^+_{11}}(x_{11})\det_{\fn^+_{22}}\big(x_{22}-P(x_{12})x_{11}^\itinv\big) \\
&=\det_{\fn^+_{11}}(x_{11})\det_{\fn^+_{22}}(x_{22})h_{\fp^+_{22}}\big(P(x_{12}) x_{11}^\itinv,{}^t\hspace{-1pt}x_{22}^\itinv\big)
\\
&=\det_{\fn^+_{11}}(x_{11})\det_{\fn^+_{22}}(x_{22})h_{\fp^+_{22}}\big(Q(x_{12}){}^t\hspace{-1pt} x_{11}^\itinv,{}^t\hspace{-1pt}x_{22}^\itinv\big)
\end{align*}
holds. The other equalities are proved similarly by interchanging $\fp^+_{11}$ and $\fp^+_{22}$.
\end{proof}

\subsection{Structure groups and the Kantor--Koecher--Tits construction}\label{section_KKT}

Let $\fp^\pm$ be a Hermitian positive Jordan triple system. For $l\in\End_\BC(\fp^+)$,
let $\overline{l},{}^t\hspace{-1pt}l\in\End_\BC(\fp^-)$, $l^*\in\End_\BC(\fp^+)$ be the elements
satisfying $\overline{l}\overline{x}=\overline{lx}$,
$(lx|\overline{y})_{\fp^+}=(x|{}^t\hspace{-1pt}l\overline{y})_{\fp^+}=(x|\overline{l^*y})_{\fp^+}$
for any $x,y\in\fp^+$.
Then the \textit{structure group} $\operatorname{Str}(\fp^+)$ is defined as
\[
\operatorname{Str}(\fp^+):=\big\{l\in GL_\BC(\fp^+)\mid \big\{lx,{}^t\hspace{-1pt}l^{-1}y,lz\big\}=l\{x,y,z\}, \ x,z\in\fp^+,\, y\in\fp^-\big\}.
\]
Then for $x\in\fp^+$, $y\in\fp^-$, $l\in\operatorname{Str}(\fp^+)$ we have
\begin{gather*}
D(lx,y)=lD\big(x,{}^t\hspace{-1pt}ly\big)l^{-1},\qquad Q(lx)=lQ(x){}^t\hspace{-1pt}l,
\\
B(lx,y)=lB\big(x,{}^t\hspace{-1pt}ly\big)l^{-1},\qquad h(lx,y)=h\big(x,{}^t\hspace{-1pt}ly\big).
\end{gather*}
Also, for any $x\in\fp^+$, $y\in\fp^-$, $B(x,y)\in\operatorname{Str}(\fp^+)$ holds if $B(x,y)$ is invertible.
Let $\mathfrak{str}(\fp^+)=\fk^\BC$ denote the Lie algebra of $\operatorname{Str}(\fp^+)$.
Then $D(\fp^+,\fp^-)=\mathfrak{str}(\fp^+)$ and $I_{\fp^+}\in\mathfrak{str}(\fp^+)$ hold.

Next we construct a Lie algebra $\fg$ from $\fp^+$ via the \textit{Kantor--Koecher--Tits construction}. As a~vec\-tor space let
\[
\fg^\BC:=\fp^+\oplus\fk^\BC\oplus\fp^-,
\]
and give the Lie algebra structure by
\[
[(x,k,y),(z,l,w)]:=\big(kz-lx,[k,l]+D(x,w)-D(z,y),-{}^t\hspace{-1pt}kw+{}^t\hspace{-1pt}ly\big).
\]
Then $I_{\fp^+}\in\fk^\BC=\mathfrak{str}(\fp^+)$ acts by $+1$, $0$ and $-1$ on each summand respectively.
Let $(\cdot|\cdot)_{\fg^\BC}\colon\fg^\BC\times\fg^\BC\to\BC$ be the $\fg^\BC$-invariant bilinear form normalized such that
$(x|y)_{\fg^\BC}=(x|y)_{\fp^+}$ holds for any $x\in\fp^+$, $y\in\fp^-$.
Next let $\hat{\vartheta}$ be the involution on $\fg^\BC$ given by
\[
\hat{\vartheta}(x,k,y):=(\overline{y},-k^*,\overline{x}),
\]
and let $\fg:=\big(\fg^\BC\big)^{\hat{\vartheta}}$, $\fk:=\big(\fk^\BC\big)^{\hat{\vartheta}}$.
We~fix a connected complex Lie group $G^\BC$ with the Lie algebra~$\fg^\BC$,
and let $G,K^\BC,K,P^+,P^-\subset G^\BC$ be the connected closed subgroups corresponding to the Lie subalgebras
$\fg,\fk^\BC,\fk,\fp^+,\fp^-\subset\fg^\BC$ respectively.
Then we have a covering map ${\rm Ad}|_{\fp^+}\colon K^\BC\to \operatorname{Str}(\fp^+)_0\subset\End_\BC(\fp^+)$.
For $l\in K^\BC$, $x\in\fp^+$, we abbreviate ${\rm Ad}(l)x$ as $lx$.

Next we fix a maximal tripotent $e\in\fp^+$, let $\fp^\pm(e)_2\subset\fp^\pm$ be as in (\ref{Peirce}),
take the corresponding Euclidean real form $\fn^+\subset\fp^+(e)_2$, and let $\fn^-:=\overline{\fn^+}=Q(\overline{e})\fn^+\subset\fp^-(e)_2$.
Also, let
\begin{gather*}
\fk^\BC(e)_2:=D(\fp^+(e)_2,\fp^-(e)_2)=[\fp^+(e)_2,\fp^-(e)_2]\subset\fk^\BC,
\\
\fl:=D(\fn^+,\fn^-)=[\fn^+,\fn^-]=\big\{l\in\fk^\BC(e)_2\mid l=Q(e)\overline{l}Q(\overline{e})\big\}\subset\fk^\BC.
\end{gather*}
Then these become subalgebras, and
\[
{}^c\hspace{-1pt}\fg:=\fn^+\oplus\fl\oplus\fn^-\subset\fg^\BC
\]
also becomes a subalgebra. Let $K^\BC(e)_2$, ${}^c\hspace{-1pt}G\subset G^\BC$ be the connected closed subgroups corresponding to
$\fk^\BC(e)_2$, ${}^c\hspace{-1pt}\fg$ respectively,
and let $L:=K^\BC\cap{}^c\hspace{-1pt}G$, $K_L:=L\cap K$, $\fk_\fl:=\fl\cap \fk$.
Then $L$ acts transitively on the symmetric cone $\Omega\subset\fn^+$, and $K_L$ acts on $\fn^+$ as Jordan algebra automorphisms.
For $l\in\End_\BC(\fp^+(e)_2)$, let $l^\top:=Q(e){}^t\hspace{-1pt}lQ(\overline{e})\in\End_\BC(\fp^+(e)_2)$
so that $(lx|y)_{\fn^+}=(x|l^\top y)_{\fn^+}$ holds for $x,y\in\fp^+(e)_2$.
Then for $x,y\in\fp^+(e)_2$, $l\in K^\BC(e)_2$ we have
\begin{alignat*}{3}
&D_{\fn^+}(lx,y)=lD_{\fn^+}\big(x,l^\top y\big)l^{-1},\qquad&& P(lx)=lP(x)l^\top,&
\\
&B_{\fn^+}(lx,y)=lB_{\fn^+}\big(x,l^\top y\big)l^{-1},\qquad && \det_{\fn^+}(lx)=\det_{\fn^+}(le)\det_{\fn^+}(x).&
\end{alignat*}

\subsection{Root space decomposition}\label{section_root}

In this section we assume $\fp^\pm$ is simple, or equivalently,
the corresponding Lie algebra $\fg$ via the Kantor--Kocher--Tits construction is simple.
Let $r:=\rank\fp^+=\rank_\BR\fg$, and fix a~frame $\{e_1,\dots,e_r\}\subset\fp^+$.
Then $e:=e_1+\cdots+e_r$ is a maximal tripotent, and we consider the corresponding subalgebras $\fn^\pm\subset\fp^\pm(e)_2\subset\fp^\pm$,
$\fl\subset\fk^\BC(e)_2\subset\fk^\BC$ as in the previous section.
Next let $h_j:=D(e_j,\overline{e_j})=[e_j,\overline{e_j}]\in\fl\subset\fk^\BC$. Then
\[
\fa_\fl:=\bigoplus_{j=1}^r\BR h_j\subset\fl\subset\fk^\BC
\]
is a maximal split abelian subalgebra of $\fl$. We~take a Cartan subalgebra $\fh\subset\fk$ containing $\sqrt{-1}\fa_\fl$.
Then simultaneously $\fh^\BC$ becomes a Cartan subalgebra of $\fg^\BC$.
We~define the linear forms $\gamma_i\in\big(\fh^\BC\big)^\vee$ by $\gamma_i(h_j)=2\delta_{ij}$ $(i=1,\dots,r)$ and $\gamma_i|_{(\fa_\fl^\BC)^\bot}=0$,
and we take a positive root system $\Delta_{\fg^\BC}^+=\Delta^+\big(\fg^\BC,\fh^\BC\big)\subset\Delta_{\fg^\BC}=\Delta\big(\fg^\BC,\fh^\BC\big)$ such that
the restriction of $\alpha\in\Delta_{\fg^\BC}^+$ to $\fa_\fl^\BC$ sits in
\begin{gather*}
\big\{\gamma_i|_{\fa_\fl^\BC} \mid 1\le i\le r\big\}
\cup\bigg\{\frac{\gamma_i\pm\gamma_j}{2}\bigg|_{\fa_\fl^\BC}\, \bigg|\, 1\le i<j\le r\bigg\}
\cup\bigg\{\frac{\gamma_i}{2}\bigg|_{\fa_\fl^\BC}\, \bigg|\, 1\le i\le r\bigg\}\cup\{0\},
\end{gather*}
and set $\Delta_{\fk^\BC}=\Delta\big(\fk^\BC,\fh^\BC\big)$,
$\Delta_{\fk^\BC}^+=\Delta^+\big(\fk^\BC,\fh^\BC\big)=\Delta_{\fg^\BC}^+\cap\Delta_{\fk^\BC}$.
Next, let
\begin{alignat*}{3}
& \fp^\pm_{ij}:=\big\{x\in\fp^\pm \mid [h_l,x]=\pm(\delta_{il}+\delta_{jl})x,\, l=1,\dots,r\big\},\qquad&&
1\le i\le j\le r,&
\\
& \fp^\pm_{0j}:=\big\{x\in\fp^\pm \mid [h_l,x]=\pm\delta_{jl}x,\, l=1,\dots,r\big\} ,\qquad && 1\le j\le r,&
\\
& \fk^\BC_{ij}:=\big\{x\in\fk^\BC \mid [h_l,x]=(\delta_{il}-\delta_{jl})x,\, l=1,\dots,r\big\}, \qquad &&
1\le i,j\le r,\; i\ne j,&
\\
& \fk^\BC_{i0}:=\big\{x\in\fk^\BC \mid [h_l,x]=\delta_{il}x,\, l=1,\dots,r\big\},\qquad&& 1\le i\le r,&
\\
& \fk^\BC_{0j}:=\big\{x\in\fk^\BC \mid [h_l,x]=-\delta_{jl}x,\, l=1,\dots,r\big\},\qquad && 1\le j\le r,&
\\
& \fm^\BC:=\big\{x\in\fk^\BC \mid [h_l,x]=0, \; (h_l|x)_{\fg^\BC}=0,\, l=1,\dots,r\big\},\qquad &&&
\\
& \fl_{ij}:=\fk^\BC_{ij}\cap \fl,\qquad && 1\le i,j\le r,\quad i\ne j,&
\\
& \fm_\fl:=\fm^\BC\cap \fl,&&&
\end{alignat*}
so that
\begin{gather*}
\fp^\pm=\bigoplus_{\substack{0\le i\le j\le r\\ (i,j)\ne (0,0)}}\fp^\pm_{ij}, \qquad
\fk^\BC=\fa_\fl^\BC\oplus\fm^\BC\oplus\bigoplus_{\substack{0\le i,j\le r\\ i\ne j}}\fk^\BC_{ij}, \qquad
\fl=\fa_\fl\oplus\fm_\fl\oplus\bigoplus_{\substack{1\le i,j\le r\\ i\ne j}}\fl_{ij}
\end{gather*}
hold. We~set
\[
\fn_\fl:=\bigoplus_{1\le i<j\le r} \fl_{ij}, \qquad \fn_\fl^\top:=\bigoplus_{1\le i<j\le r} \fl_{ji},
\]
let $A_L,N_L,N_L^\top\subset L$ be the connected closed subgroups corresponding to the Lie subalgebras $\fa_\fl$, $\fn_\fl$, $\fn_\fl^\top$
respectively, and let
\[
M_L:=\{k\in K_L\mid {\rm Ad}(k)h_l=0, \, l=1,\dots,r\},
\]
so that $L=K_LA_LN_L=K_LA_LN_L^\top$ holds and $M_LA_LN_L, M_LA_LN_L^\top\subset L$ are minimal parabolic subgroups.
Then $A_LN_L, A_LN_L^\top\subset L$ acts on $\Omega$ simply transitively.

\subsection{Polynomials and hypergeometric series on Jordan triple systems}

For a Hermitian positive Jordan triple system $\fp^\pm$, let $\cP(\fp^\pm)$ be the space of holomorphic polynomials on $\cP(\fp^\pm)$.
Then $K^\BC$ acts on $\cP(\fp^\pm)$ by
\begin{alignat*}{5}
& ({\rm Ad}|_{\fp^+}(l))^\vee f(x)=f\big(l^{-1}x\big),\qquad && l\in K^\BC,\qquad && f\in \cP(\fp^+),\qquad && x\in\fp^+,&\\
& ({\rm Ad}|_{\fp^-}(l))^\vee f(y)=f\big({}^t\hspace{-1pt}ly\big),\qquad && l\in K^\BC,\qquad && f\in \cP(\fp^-),\qquad && y\in\fp^-.&
\end{alignat*}
We~assume $\fp^+$ is simple, fix a frame $\{e_1,\dots,e_r\}\subset\fp^+$,
and consider tripotents $e^k:=\sum_{i=1}^k e_i$ for $k=1,2,\dots,r$. Then the subalgebra
\[ \fp^+\big(e^k\big)_2=\bigoplus_{1\le i\le j\le k}\fp^+_{ij} \]
is of tube type and has a Jordan algebra structure of rank $k$.
Let $\Delta_k(x)$ be the determinant polynomial on $\fp^+(e^k)_2$, and extend to a polynomial on $\fp^+$.
Using this, for $\bm\in\BC^r$ we define the function $\Delta_\bm(x)$ on the symmetric cone $\Omega\subset\fn^+$ by
\[ \Delta_\bm(x):=\Delta_1(x)^{m_1-m_2}\Delta_2(x)^{m_2-m_3}\cdots\Delta_{r-1}(x)^{m_{r-1}-m_r}\Delta_r(x)^{m_r}, \qquad x\in\Omega.
\]
Then for $m\in M_L$, $a={\rm e}^{t_1h_1+\cdots+t_rh_r}\in A_L$, $n\in N_L$ and for $x\in\Omega$, we have
\[
\Delta_\bm\big((man)^\top x\big)=\Delta_\bm\big((man)^\top e\big)\Delta_\bm(x)={\rm e}^{2t_1m_1+\cdots+2t_rm_r}\Delta_\bm(x). \]
Especially, if
\[
\bm\in\BZ_{++}^r:=\big\{\bm=(m_1,\dots,m_r)\in\BZ^r\mid m_1\ge\cdots\ge m_r\ge 0\big\},
\]
then $\Delta_\bm(x)$ is extended to a polynomial on $\fp^+$. Let
\begin{gather*}
\cP_\bm(\fp^+):=\operatorname{span}_\BC\big\{\Delta_\bm\big(l^{-1}x\big)\mid l\in K^\BC\big\}\subset\cP(\fp^+),
\\
\cP_\bm(\fp^-):=\big\{\overline{f(\overline{y})}\mid f(x)\in\cP_\bm(\fp^+)\big\}\subset\cP(\fp^-).
\end{gather*}
Then we have the following.

\begin{Theorem}[Hua--Kostant--Schmid {\cite[Part III, Theorem V.2.1]{FKKLR}}]%\label{HKS}
Under the $K^\BC$-action, $\cP(\fp^\pm)$ is decomposed as
\[
\cP\big(\fp^\pm\big)=\bigoplus_{\bm\in\BZ^{r}_{++}}\cP_\bm\big(\fp^\pm\big).
\]
Each $\cP_\bm(\fp^\pm)$ is irreducible, $\cP_\bm(\fp^+)$ has the lowest weight $-(m_1\gamma_1+\cdots+m_r\gamma_r)$,
and $\cP_\bm(\fp^-)$ has the highest weight $m_1\gamma_1+\cdots+m_r\gamma_r$.
\end{Theorem}
In addition, if $\fp^+$ is of tube type, then for
\[
\bm\in\BZ_+^r:=\big\{\bm=(m_1,\dots,m_r)\in\BZ^r\mid m_1\ge\cdots\ge m_r\big\},
\]
let
\[ \cP_\bm\big(\fp^\pm\big):=\cP_{(m_1-m_r,m_2-m_r,\dots,m_{r-1}-m_r,0)}\big(\fp^\pm\big)\det_{\fn^\pm}(x)^{m_r}\subset\cP\big(\fp^\pm\big)\big[\det_{\fn^\pm}(x)^{-1}\big]. \]
Then for $\bm\in\BZ_+^r$ we have
\[
\big\{ f\big({}^t\hspace{-1pt}x^\itinv\big)\mid f\in\cP_\bm\big(\fp^\pm\big)\big\}=\cP_{-\bm^\vee}\big(\fp^\mp\big),
\]
where $\bm^\vee:=(m_r,\dots,m_1)$.

Next, for $\lambda\in\BC$, $\bs\in\BC^r$, $\bm\in(\BZ_{\ge 0})^r$, $d\in\BC$, let
\begin{gather}
(\lambda+\bs)_{\bm,d}=\prod_{j=1}^r\bigg(\lambda+s_j-\frac{d}{2}(j-1)\bigg)_{m_j},
\qquad
(\lambda)_{\bm,d}=\prod_{j=1}^r\bigg(\lambda-\frac{d}{2}(j-1)\bigg)_{m_j}, \label{Pochhammer}
\end{gather}
where $(\lambda)_m=\lambda(\lambda+1)\cdots(\lambda+m-1)$, and let $\dim\cP_\bm(\fp^+)=:d_\bm^{\fp^+}$.
Using these, for $\bm\in\BZ_{++}^r$, let $\Phi_\bm^{\fn^+}(x), \tilde{\Phi}_\bm^{\fn^+}(x)\in\cP_\bm(\fp^+)^{K_L}$ be the polynomials given by
\begin{equation}\label{Phi^n}
\Phi_\bm^{\fn^+}(x):=\int_{K_L}\Delta_\bm(kx)\,{\rm d}k, \qquad
\tilde{\Phi}_\bm^{\fn^+}(x):=\frac{d_\bm^{\fp^+}}{\left(\frac{n}{r}\right)_{\bm,d}}\Phi_\bm^{\fn^+}(x),
\end{equation}
and for $(t_1,\dots,t_r)\in\BC^r$, define the symmetric polynomials $\tilde{\Phi}^d_\bm(t_1,\dots,t_r)$ by
\begin{equation}
\tilde{\Phi}^d_\bm(t_1,\dots,t_r):=\tilde{\Phi}^{\fn^+}_\bm(t_1e_1+\cdots+t_re_r). \label{Phi^d}
\end{equation}
This does not depend on the choice of Jordan frames $\{e_j\}\subset\fp^+$.
Moreover, this does not depend on $r$, that is, if we construct $\tilde{\Phi}^d_\bm(t_1,\dots,t_r)$, $\tilde{\Phi}^d_\bm(t_1,\dots,t_{r-1})$
from two simple Jordan triple systems of rank $r$ and $r-1$ with common $d$, then we have
\[ \tilde{\Phi}^d_\bm(t_1,\dots,t_{r-1},0)=\tilde{\Phi}^d_\bm(t_1,\dots,t_{r-1}). \]
Next we define the polynomials $\tilde{\Phi}_\bm^{\fp^\pm}(x,y)\in\cP_\bm\big(\fp^\pm\big)\otimes \cP_\bm\big(\fp^\mp\big)$ by
\begin{equation}\label{Phi^p}
\tilde{\Phi}_\bm^{\fp^\pm}(x,y):=\tilde{\Phi}^d_\bm(t_1,\dots,t_r),
\end{equation}
by using the roots $t_1,\dots,t_r$ of $t^rh_{\fp^\pm}\big(t^{-1}x,y\big)\in\BC[t]$. Then these satisfy
\begin{alignat*}{4}
& \tilde{\Phi}_\bm^{\fp^+}(x,y)=\tilde{\Phi}_\bm^{\fp^-}(y,x),\qquad && x\in\fp^+,\qquad && y\in\fp^-,&
\\
& \tilde{\Phi}_\bm^{\fp^+}(lx,y)=\tilde{\Phi}_\bm^{\fp^+}\big(x,{}^t\hspace{-1pt}ly\big),\qquad&&
x\in\fp^+,\quad && y\in\fp^-,\quad l\in K^\BC,&
\\
& \tilde{\Phi}_\bm^{\fp^+}(x,\overline{e})=\tilde{\Phi}_\bm^{\fn^+}(x),\qquad && x\in\fp^+. && &
\end{alignat*}
For these polynomials we have
\begin{alignat}{3}
& {\rm e}^{\tr_{\fn^+}(x)}=\sum_{\bm\in\BZ_{++}^r}\tilde{\Phi}_\bm^{\fn^+}(x),\qquad&&
x\in\fp^+, & \label{expand_exptr}
\\
& {\rm e}^{(x|y)_{\fp^+}}=\sum_{\bm\in\BZ_{++}^r}\tilde{\Phi}_\bm^{\fp^+}(x,y),\qquad &&
x\in\fp^+,\quad y\in\fp^-, &\notag
\\
& {\rm e}^{t_1+\cdots+t_r}=\sum_{\bm\in\BZ_{++}^r}\tilde{\Phi}_\bm^d(t_1,\dots,t_r),\qquad&&
t_j\in\BC, &\label{expand_expsum}
\\
& \det_{\fn^+}(e-x)^{-\lambda}=\sum_{\bm\in\BZ_{++}^r}(\lambda)_{\bm,d}\tilde{\Phi}_\bm^{\fn^+}(x),\qquad &&
x\in D\subset\fp^+, & \notag
\\
& h_{\fp^+}(x,y)^{-\lambda}=\sum_{\bm\in\BZ_{++}^r}(\lambda)_{\bm,d}\tilde{\Phi}_\bm^{\fp^+}(x,y),\qquad &&
x\in D\subset\fp^+,\quad y\in\overline{D}\subset\fp^-,& \label{expand_h}
\\
& \prod_{j=1}^r(1-t_j)^{-\lambda} =\sum_{\bm\in\BZ_{++}^r}(\lambda)_{\bm,d}\tilde{\Phi}_\bm^d(t_1,\dots,t_r),\qquad && t_j\in\BC,\quad |t_j|<1&
\label{expand_binom}
\end{alignat}
(see \cite[Proposition XII.1.3]{FK}, \cite[Part III, Theorem~V.3.10]{FKKLR}). Using these, for $\alpha,\beta,\gamma\in\BC$
with $\gamma\notin\bigcup_{j=1}^r\left(\frac{d}{2}(j-1)-\BZ_{\ge 0}\right)$, let
\begin{gather}
{}_0F_1^{\fn^+}(\gamma;x):=\sum_{\bm\in\BZ_{++}^r}\frac{1}{(\gamma)_{\bm,d}}\tilde{\Phi}^{\fn^+}_\bm(x), \notag
\\
{}_0F_1^{\fp^\pm}(\gamma;x,y):=\sum_{\bm\in\BZ_{++}^r}
\frac{1}{(\gamma)_{\bm,d}}\tilde{\Phi}^{\fp^\pm}_\bm(x,y), \label{0F1^p}
\\
{}_0F_1^d(\gamma;t_1,\dots,t_r):=\sum_{\bm\in\BZ_{++}^r}
\frac{1}{(\gamma)_{\bm,d}}\tilde{\Phi}^d_\bm(t_1,\dots,t_r), \notag
\\
{}_2F_1^{\fn^+}\left(\begin{matrix}\alpha,\beta\\ \gamma \end{matrix};x\right)
:=\sum_{\bm\in\BZ_{++}^r}\frac{(\alpha)_{\bm,d}(\beta)_{\bm,d}}{(\gamma)_{\bm,d}}\tilde{\Phi}^{\fn^+}_\bm(x), \notag
\\
{}_2F_1^{\fp^\pm}\left(\begin{matrix}\alpha,\beta\\ \gamma \end{matrix};x,y\right)
:=\sum_{\bm\in\BZ_{++}^r}\frac{(\alpha)_{\bm,d}(\beta)_{\bm,d}}{(\gamma)_{\bm,d}}\tilde{\Phi}^{\fp^\pm}_\bm(x,y), \label{2F1^p}
\\
{}_2F_1^d\left(\begin{matrix}\alpha,\beta\\ \gamma \end{matrix};t_1,\dots,t_r\right)
:=\sum_{\bm\in\BZ_{++}^r}\frac{(\alpha)_{\bm,d}(\beta)_{\bm,d}}{(\gamma)_{\bm,d}}\tilde{\Phi}^d_\bm(t_1,\dots,t_r). \notag
\end{gather}
${}_0F_1$ converges for all $x\in\fp^+$, $y\in\fp^-$ or $t_j\in\BC$,
and ${}_2F_1$ converges if $x\in D\subset\fp^+$, $y\in \overline{D}\subset\fp^-$ or $|t_j|<1$.
Especially if $\alpha=-k\in-\BZ_{\ge 0}$,
then ${}_2F_1^d\left( \begin{smallmatrix}-k,\beta\\\gamma \end{smallmatrix};t_1,\dots,t_r\right)$
becomes a polynomial on $\BC^r$ of degree $kr$, and is well-defined for
$\gamma\notin\bigcup_{j=1}^r\big(\frac{d}{2}(j-1)-\{0,1,\dots,k\}\big)$.
In this paper we use the same notation ${}_2F_1^d\left( \begin{smallmatrix}-k,\beta\\\gamma \end{smallmatrix};t_1,\dots,t_r\right)$
for this polynomial even if $\gamma$ is a negative (half) integer except for the above set,
and similar for ${}_2F_1^{\fn^+}$, ${}_2F_1^{\fp^\pm}$.
By \cite[Proposition XV.3.4]{FK}, for $t_j\in\BC$, $|t_j|<1$ we have
\begin{align}
{}_2F_1^d\left(\begin{matrix}\alpha,\beta\\\gamma \end{matrix};t_1,\dots,t_r\right)
&=\prod_{j=1}^r(1-t_j)^{-\alpha}{}_2F_1^d\left( \begin{matrix}\alpha,\gamma-\beta\\\gamma \end{matrix};
\frac{t_1}{t_1-1},\dots,\frac{t_r}{t_r-1}\right)
\label{Kummer1} \\
&=\prod_{j=1}^r(1-t_j)^{\gamma-\alpha-\beta}{}_2F_1^d\left( \begin{matrix}\gamma-\alpha,\gamma-\beta\\\gamma \end{matrix};t_1,\dots,t_r\right)\!.
\label{Kummer2}
\end{align}
By this formula we can show that
\[
{}_2F_1^{\fp^+}\left( \begin{matrix}\alpha,\beta\\\gamma \end{matrix};x,y\right)
=h(x,y)^{\gamma-\alpha-\beta}{}_2F_1^{\fp^+}\left( \begin{matrix}\gamma-\alpha,\gamma-\beta\\\gamma \end{matrix};x,y\right)
\]
holds for $x\in D\subset\fp^+$, $y\in\overline{D}\subset\fp^-$.
Next, for $k\in\BZ_{\ge 0}$, let
\[
\underline{k}_r:=(\underbrace{k,\dots,k}_r)\in\BZ_{++}^r.
\]
Then the following formulas hold.

\begin{Proposition}\label{prop_hypergeom}
We~assume $\fp^+=\fn^{+\BC}$ is of tube type.
\begin{enumerate}\itemsep=0pt
\item[$1.$] For $\alpha=-k\in -\BZ_{\ge 0}$ we have
\begin{align*}
{}_2F_1^{\fn^+}\left(\begin{matrix}-k,\beta \\ \gamma\end{matrix};x\right)
&=\frac{(\beta)_{\underline{k}_r,d}}{(\gamma)_{\underline{k}_r,d}}\det_{\fn^+}(-x)^k
{}_2F_1^{\fn^+}\left(\begin{matrix}-k,-k-\gamma+\frac{n}{r} \\ -k-\beta+\frac{n}{r}\end{matrix};x^\itinv\right) \\
&=\frac{(\beta)_{\underline{k}_r,d}}{(\gamma)_{\underline{k}_r,d}}\det_{\fn^+}(e-x)^k
{}_2F_1^{\fn^+}\left(\begin{matrix}-k,\gamma-\beta \\ -k-\beta+\frac{n}{r}\end{matrix};(e-x)^\itinv\right) \\
&=\frac{(\gamma-\beta)_{\underline{k}_r,d}}{(\gamma)_{\underline{k}_r,d}}
{}_2F_1^{\fn^+}\left(\begin{matrix}-k,\beta \\ -k+\beta-\gamma+\frac{n}{r}\end{matrix};e-x\right)\!.
\end{align*}

\item[$2.$] Let $k\in\BZ_{>0}$. Then we have
\[
\lim_{\gamma\to \frac{n}{r}-k}(\gamma)_{\underline{k}_r,d}\,{}_2F_1^{\fn^+}
\left(\begin{matrix}\alpha,\beta\\\gamma\end{matrix};x\right)
=\frac{(\alpha)_{\underline{k}_r,d}(\beta)_{\underline{k}_r,d}}{\left(\frac{n}{r}\right)_{\underline{k}_r,d}}
\det_{\fn^+}(x)^k{}_2F_1^{\fn^+}\left(\begin{matrix}\alpha+k,\beta+k\\\frac{n}{r}+k\end{matrix};x\right)\!. \]
\end{enumerate}
\end{Proposition}

\begin{proof}
(1) Since $(-k)_{\bm,d}=0$ holds for $m_j>k$, we have
\[ {}_2F_1^{\fn^+}\left(\begin{matrix}-k,\beta \\ \gamma\end{matrix};x\right)
=\sum_{\substack{\bm\in\BZ_{++}^r\\ m_j\le k}}
\frac{(-k)_{\underline{k}_r-\bm^\vee,d}(\beta)_{\underline{k}_r-\bm^\vee,d}}{(\gamma)_{\underline{k}_r-\bm^\vee,d}}
\frac{d_{\underline{k}_r-\bm^\vee}^{\fp^+}}{\left(\frac{n}{r}\right)_{\underline{k}_r-\bm^\vee,d}}\Phi_{\underline{k}_r-\bm^\vee}^{\fn^+}(x), \]
and by
\begin{gather*}
\Phi_{\underline{k}_r-\bm^\vee}^{\fn^+}(x)=\det_{\fn^+}(x)^k\Phi_\bm^{\fn^+}\big(x^\itinv\big), \qquad
d_{\underline{k}_r-\bm^\vee}^{\fp^+}=d_\bm^{\fp^+},
\\
(\alpha)_{\underline{k}_r-\bm^\vee,d}=\frac{(\alpha)_{\underline{k}_r,d}}{(\alpha+k-\bm^\vee)_{\bm^\vee,d}}
=\frac{(\alpha)_{\underline{k}_r,d}}{(-1)^{|\bm|}\big({-}\alpha-k+\frac{n}{r}\big)_{\bm,d}},
\\
(-k)_{\underline{k}_r,d}=(-1)^{kr}\bigg(\frac{n}{r}\bigg)_{\underline{k}_r,d},
\end{gather*}
we get
\begin{gather*}
{}_2F_1^{\fn^+}\left(\begin{matrix}-k,\beta \\ \gamma\end{matrix};x\right)
\\ \qquad
{}=\!\sum_{\substack{\bm\in\BZ_{++}^r\\ m_j\le k}}\!
\frac{(-k)_{\underline{k}_r,d}}{\left(\frac{n}{r}\right)_{\bm,d}}
\frac{(\beta)_{\underline{k}_r,d}}{\left(-\beta-k+\frac{n}{r}\right)_{\bm,d}}
\frac{\left(-\gamma-k+\frac{n}{r}\right)_{\bm,d}}{(\gamma)_{\underline{k}_r,d}}
\frac{(-k)_{\bm,d}}{\left(\frac{n}{r}\right)_{\underline{k}_r,d}}
d_\bm^{\fp^+}\det_{\fn^+}(x)^k\Phi_\bm^{\fn^+}\big(x^\itinv\big)
\\ \qquad
{}=\frac{(\beta)_{\underline{k}_r,d}}{(\gamma)_{\underline{k}_r,d}}(-1)^{kr}
\det_{\fn^+}(x)^k\sum_{\substack{\bm\in\BZ_{++}^r\\ m_j\le k}}
\frac{(-k)_{\bm,d}\left(-\gamma-k+\frac{n}{r}\right)_{\bm,d}}{\left(-\beta-k+\frac{n}{r}\right)_{\bm,d}}
\frac{d_\bm^{\fp^+}}{\left(\frac{n}{r}\right)_{\bm,d}}\Phi_\bm^{\fn^+}\big(x^\itinv\big)
\\ \qquad
{}=\frac{(\beta)_{\underline{k}_r,d}}{(\gamma)_{\underline{k}_r,d}}\det_{\fn^+}(-x)^k
{}_2F_1^{\fn^+}\left(\begin{matrix}-k,-k-\gamma+\frac{n}{r} \\ -k-\beta+\frac{n}{r}\end{matrix};x^\itinv\right)\!.
\end{gather*}
This proves the 1st equality. The 2nd equality follows from (\ref{Kummer1})
with $t_j\mapsto \frac{1}{t_j}$, $\frac{t_j}{t_j-1}\mapsto \frac{1}{1-t_j}$.
The 3rd equality follows from the 1st equality and $\left(-k-\beta+\frac{n}{r}\right)_{\underline{k}_r,d}=(-1)^{kr}(\beta)_{\underline{k}_r,d}$.

(2) We~have
\begin{gather*}
(\gamma)_{\underline{k}_r,d}\,{}_2F_1^{\fn^+}\left(\begin{matrix}\alpha,\beta\\\gamma\end{matrix};x\right) \\ \qquad
{}=(\gamma)_{\underline{k}_r,d}\sum_{\substack{\bm\in\BZ_{++}^r \\ m_r\le k-1}}
\frac{(\alpha)_{\bm,d}(\beta)_{\bm,d}}{(\gamma)_{\bm,d}}
\frac{d_\bm^{\fp^+}}{\left(\frac{n}{r}\right)_{\bm,d}}\Phi_\bm^{\fn^+}(x)
+\sum_{\substack{\bm\in\BZ_{++}^r \\ m_r\ge k}}\frac{(\alpha)_{\bm,d}(\beta)_{\bm,d}}{(\gamma+k)_{\bm-\underline{k}_r,d}}
\frac{d_\bm^{\fp^+}}{\left(\frac{n}{r}\right)_{\bm,d}}\Phi_\bm^{\fn^+}(x),
\end{gather*}
and when we take the limit $\gamma\to\frac{n}{r}-k$, the 1st term vanishes.
Therefore by changing $\bm$ to $\bm+\underline{k}_r$, we get
\begin{gather*}
\lim_{\gamma\to \frac{n}{r}-k}(\gamma)_{\underline{k}_r,d}\,{}_2F_1^{\fn^+}\left(\begin{matrix}\alpha,\beta\\\gamma\end{matrix};x\right)
=\sum_{\bm\in\BZ_{++}^r}\frac{(\alpha)_{\bm+\underline{k}_r,d}(\beta)_{\bm+\underline{k}_r,d}\,d_{\bm+\underline{k}_r}^{\fp^+}}
{\left(\frac{n}{r}\right)_{\bm,d}\left(\frac{n}{r}\right)_{\bm+\underline{k}_r,d}}\Phi_{\bm+\underline{k}_r}^{\fn^+}(x) \\ \qquad
{}=\frac{(\alpha)_{\underline{k}_r,d}(\beta)_{\underline{k}_r,d}}{\left(\frac{n}{r}\right)_{\underline{k}_r,d}}
\sum_{\bm\in\BZ_{++}^r}\frac{(\alpha+k)_{\bm,d}(\beta+k)_{\bm,d}\,d_\bm^{\fp^+}}
{\left(\frac{n}{r}\right)_{\bm,d}\left(\frac{n}{r}+k\right)_{\bm,d}}\Phi_{\bm}^{\fn^+}(x)\det_{\fn^+}(x)^k \\ \qquad
{}=\frac{(\alpha)_{\underline{k}_r,d}(\beta)_{\underline{k}_r,d}}{\left(\frac{n}{r}\right)_{\underline{k}_r,d}}
\det_{\fn^+}(x)^k{}_2F_1^{\fn^+}\left(\begin{matrix}\alpha+k,\beta+k\\\frac{n}{r}+k\end{matrix};x\right)\!. \tag*{\qed}
\end{gather*}
\renewcommand{\qed}{}
\end{proof}

\begin{Remark}
When $r=1$, (1) follows from the more general formula \cite[equations~2.9(33), (34)]{E},
\begin{align*}
{}_2F_1\left(\begin{matrix}\alpha,\beta\\\gamma\end{matrix};t\right)
={}&\frac{\Gamma(\beta-\alpha)\Gamma(\gamma)}{\Gamma(\beta)\Gamma(\gamma-\alpha)}(-t)^{-\alpha}
{}_2F_1\left(\begin{matrix}\alpha,\alpha-\gamma+1\\\alpha-\beta+1\end{matrix};\frac{1}{t}\right)
\\
&+\frac{\Gamma(\alpha-\beta)\Gamma(\gamma)}{\Gamma(\alpha)\Gamma(\gamma-\beta)}(-t)^{-\beta}
{}_2F_1\left(\begin{matrix}\beta,\beta-\gamma+1\\\beta-\alpha+1\end{matrix};\frac{1}{t}\right)
\\
={}&\frac{\Gamma(\gamma)\Gamma(\gamma-\alpha-\beta)}{\Gamma(\gamma-\alpha)\Gamma(\gamma-\beta)}
{}_2F_1\left(\begin{matrix}\alpha,\beta\\\alpha+\beta-\gamma+1\end{matrix};1-t\right)
\\
&+\frac{\Gamma(\gamma)\Gamma(\alpha+\beta-\gamma)} {\Gamma(\alpha)\Gamma(\beta)}(1-t)^{\gamma-\alpha-\beta}
{}_2F_1\left(\begin{matrix}\gamma-\alpha,\gamma-\beta\\\gamma-\alpha-\beta+1\end{matrix};1-t\right)\!.
\end{align*}
\end{Remark}
\begin{Remark}\label{rem_Heckman-Opdam}
In this paper, $\tilde{\Phi}_\bm^d(t_1,\dots,t_r)$, ${}_0F_1^d(\gamma;t_1,\dots,t_r)$ and
${}_2F_1^d\left(\begin{smallmatrix}\alpha,\beta\\\gamma\end{smallmatrix};t_1,\dots,t_r\right)$ are defined only when
\[
\begin{cases}
d=1,2,4,& r\in\BZ_{>0},
\\
d=8,&r=1,2,3,
\\
d\in\BZ_{>0},\quad d\ne 1,2,4,8, &r=1,2,
\end{cases}
\]
but these can be defined for general $d\in\BC$, $r\in\BZ_{>0}$ by using the Jack polynomials,
and (\ref{expand_expsum}), (\ref{expand_binom}), (\ref{Kummer1}) and (\ref{Kummer2}) hold for general $(d,r)$ (see Yan \cite{Ya}).
This ${}_2F_1^d$ coincides with a special case of Heckman--Opdam's multivariate hypergeometric functions of type $BC_r$
under some change of variables, that is, for $\alpha\in\BC$, $\bk=(k_s,k_l,k_m)\in\BC^3$ and for $s_1,\dots,s_r\in\BC/2\pi\sqrt{-1}\BZ$, we have
\begin{gather}
{}_2F_1^{2k_m}\left( \begin{matrix}\alpha,-\alpha+k_s+2k_l+k_m(r-1)\\k_s+k_l+k_m(r-1)+\frac{1}{2}\end{matrix};
-\sinh^2\frac{s_1}{2},\dots,-\sinh^2\frac{s_r}{2}\right) \notag
\\ \qquad
{}=F_{BC_r}(\underline{\alpha}_r-\rho(\bk),\bk;s_1,\dots,s_r), \label{Heckman-Opdam}
\end{gather}
where
\begin{align*}
\rho(\bk)&=\frac{1}{2}\Bigg(k_s\sum_{j=1}^r\epsilon_j+k_l\sum_{j=1}^r2\epsilon_j
+k_m\sum_{1\le i<j\le r}((\epsilon_i-\epsilon_j)+(\epsilon_i+\epsilon_j))\Bigg)
\\
&=\underline{\frac{k_s}{2}+k_l}_r+k_m(r-1,r-2,\dots,1,0)\in\BC^r
\end{align*}
(see Beerends--Opdam \cite{BO}). Here $\{\epsilon_j\}\subset\BC^r$ denotes the standard basis.
\end{Remark}

\subsection{Holomorphic discrete series representations}

From a Hermitian positive Jordan triple system $\fp^+$, we construct a Lie group $G$ by the Kantor--Koecher--Tits construction,
and let $K\subset G$ be the maximal compact subgroup, as in Section~\ref{section_KKT}.
Then the quotient space $G/K$ is diffeomorphic to the bounded symmetric domain $D$ in (\ref{BSD}) via the Borel embedding (Harish-Chandra realization):
\[
\xymatrix{ G/K \ar[r] \ar@{-->}[d]^{\mbox{\rotatebox{90}{$\sim$}}} & G^\BC/K^\BC P^- \\
 D \ar@{^{(}->}[r] & \fp^+. \ar[u]_{\exp} }
 \]
For $g\in G^\BC$, $x\in \fp^+$, if $g\exp(x)\in P^+K^\BC P^-$ holds, then we write
\[ g\exp(x)=\exp(\pi^+(g,x))\kappa(g,x)\exp(\pi^-(g,x)), \]
where $\pi^+(g,x)\in\fp^+$, $\kappa(g,x)\in K^\BC$, and $\pi^-(g,x)\in\fp^-$.
Then since $\pi^+(g,x)\in D$ holds for $(g,x)\in G\times D$, $\pi^+$ gives the action of $G$ on $D$. From now on we abbreviate $\pi^+(g,x)=:gx$.
Also, since the Bergman operator $B(x,\overline{y})$ is invertible for $x,y\in D$, $B(x,\overline{y})\in\operatorname{Str}(\fp^+)$ holds
on $D\times\overline{D}$.
Moreover, since $D$ is simply connected, we can lift $B\colon D\times\overline{D}\to\operatorname{Str}(\fp^+)$ to the map
$B\colon D\times\overline{D}\to K^\BC$ via the local isomorphism ${\rm Ad}|_{\fp^+}\colon K^\BC\to\operatorname{Str}(\fp^+)$. Then
\[
B(gx,\overline{gy})=\kappa(g,x)B(x,\overline{y})\kappa(g,y)^*,\qquad g\in G,\quad x,y\in D
\]
holds.

{\sloppy
Now let $(\tau,V)$ be an irreducible holomorphic representation of the universal covering group~$\widetilde{K}^\BC$ of $K^\BC$.
Then the universal covering group $\widetilde{G}$ of $G$ acts on the space $\cO(D,V)=\cO_\tau(D,V)$ of
$V$-valued holomorphic functions on $D$ by
\[
(\hat{\tau}(g)f)(x):=\tau(\kappa(g^{-1},x))^{-1}f(g^{-1}x), \qquad
g\in\widetilde{G},\quad x\in D,\quad f\in\cO(D,V),
\]}
and its differential action is given by
\[
({\rm d}\hat{\tau}(z,k,w)f)(x)={\rm d}\tau(k-D(x,w))f(x)+\frac{\rm d}{{\rm d}t}\biggr|_{t=0}f(x-t(z+kx-Q(x)w))
\]
for $z\in\fp^+$, $k\in\fk^\BC$, $w\in\fp^-$.
Next we lift the Bergman operator to $B\colon D\times\overline{D}\to\widetilde{K}^\BC$,
and we consider the function $\tau(B(x,\overline{y}))\in\cO(D\times\overline{D},\End_\BC(V))$.
Then this is invariant under the diagonal action of $\widetilde{G}$.
Hence if $\tau(B(x,\overline{y}))$ is positive definite, that is,
$\sum_{j,k=1}^N(\tau(B(x_j,\overline{x_k}))v_j,v_k)_\tau\ge 0$ holds for any $\{x_j\}_{j=1}^N\subset D$ and $\{v_j\}_{j=1}^N\subset V$,
then there exists a unique Hilbert subspace $\cH_\tau(D,V)\subset\cO_\tau(D,V)$ with the reproducing kernel $\tau(B(x,\overline{y}))$,
on which $\widetilde{G}$ acts unitarily via~$\hat{\tau}$.
The representation $(\hat{\tau},\cH_\tau(D,V))$ is called a \textit{unitary highest weight representation}.
The representations $(\tau,V)$ of $\widetilde{K}^\BC$ such that the unitary highest weight representations
$\cH_\tau(D,V)\subset\cO(D,V)$ exist are classified by Enright--Howe--Wallach \cite{EHW} and Jakobsen \cite{J}.
Especially, if its inner product is given by the converging integral
\[ \langle f,g\rangle_{\hat{\tau}}:=C_\tau\int_D \big(\tau\big(B(x)^{-1}\big)f(x),g(x)\big)_{\tau}\Det_{\fp^+}(B(x))^{-1}{\rm d}x \]
(a \textit{weighted Bergman inner product}), then $(\hat{\tau},\cH_\tau(D,V))$ is called a \textit{holomorphic discrete series representation}.
Here we determine the constant $C_\tau$ such that $\Vert v\Vert_{\hat{\tau}}=|v|_\tau$ holds for all constant functions $v$.

Next, let $\chi\colon \widetilde{K}^\BC\to\BC^\times$ be the character of $\widetilde{K}^\BC$ normalized such that
\begin{equation}\label{char}
{\rm d}\chi([x,y])=(x|y)_{\fp^+}, \qquad x\in\fp^+,\quad y\in\fp^-,
\end{equation}
so that $h(x,y)=\chi(B(x,y))$ holds.
If $(\tau,V)=\big(\chi^{-\lambda}\otimes \tau_0,V\big)$ for some representation $(\tau_0,V)$ of~$K^\BC$,
then we write $\cH_\tau(D,V)=\cH_\lambda(D,V)\subset\cO(D,V)=\cO_\lambda(D,V)$. Then the reproducing kernel is given by
$h(x,\overline{y})^{-\lambda}\tau_0(B(x,\overline{y}))$, and the inner product is given by
\[ \langle f,g\rangle_{\lambda,\tau_0}:=C_{\lambda,\tau_0}\int_D \big(\tau\big(B(x)^{-1}\big)f(x),g(x)\big)_{\tau_0}h(x)^\lambda
\Det_{\fp^+}(B(x))^{-1}{\rm d}x \]
if it converges. Especially if $(\tau_0,V)$ is trivial, we also write $\cH_\tau(D,V)=\cH_\lambda(D)\subset\cO(D)=\cO_\lambda(D)$.

From now on we assume $\fp^+$ is simple and $(\tau_0,V)$ is trivial. Then $\cH_\lambda(D)$ is holomorphic discrete if $\lambda>p-1$,
and the inner product is given by
\begin{gather*}
\langle f,g\rangle_\lambda=\langle f,g\rangle_{\lambda,\fp^+}
:=C_\lambda\int_D f(x)\overline{g(x)}h(x)^{\lambda-p}{\rm d}x, \\
C_\lambda:=\frac{\prod_{j=1}^r\Gamma\left(\lambda-\frac{d}{2}(j-1)\right)}
{\pi^n\prod_{j=1}^r\Gamma\left(\lambda-\frac{n}{r}-\frac{d}{2}(j-1)\right)}.
\end{gather*}
Next we consider another inner product on $\cP(\fp^+)$, called the \textit{Fischer inner product}
(see, e.g., \cite[Section XI.1]{FK}), defined by
\[ \langle f,g\rangle_F=\langle f,g\rangle_{F,\fp^+}:=\frac{1}{\pi^n}\int_{\fp^+}f(x)\overline{g(x)}{\rm e}^{-(x|\overline{x})_{\fp^+}}{\rm d}x
=\overline{g\left(\overline{\frac{\partial}{\partial x}}\right)}f(x)\biggr|_{x=0}. \]
Then we have the following.
\begin{Theorem}[Faraut--Kor\'anyi \cite{FK0}, {\cite[Part III, Corollary V.3.9]{FKKLR}}]\label{FKThm}
For $f\in\cP_\bm(\fp^+)$, $g\in\cP(\fp^+)$, we have
\[ \langle f,g\rangle_\lambda=\frac{1}{(\lambda)_{\bm,d}}\langle f,g\rangle_F. \]
\end{Theorem}
Here $(\lambda)_{\bm,d}$ is as in (\ref{Pochhammer}).
Since the reproducing kernel on $\cP_\bm(\fp^+)$ with respect to $\langle\cdot,\cdot\rangle_F$ is given by ${\rm e}^{(x|\overline{y})_{\fp^+}}$,
the following holds.
\begin{Corollary}\label{FKCor}
For $f\in\cP_\bm(\fp^+)$, $y\in\fp^+$, we have
\[
\big\langle f,{\rm e}^{(\cdot|\overline{y})_{\fp^+}}\big\rangle_\lambda=\frac{1}{(\lambda)_{\bm,d}}f(y).
\]
\end{Corollary}
Now let $e\in\fp^+$ be a tripotent, and let $\rank(\fp^+(e)_2)=:r'$.
For $x\in\fp^+$, let $x'\in\fp^+(e)_2$ be the orthogonal projection,
and regard $\cP(\fp^+(e)_2)\subset\cP(\fp^+)$ via the projection.
Then for $\bm=(m_1,\dots,m_{r'})\in\BZ_{++}^{r'}$ we have $\cP_\bm(\fp^+(e)_2)\subset\cP_{(m_1,\dots,m_{r'},0,\dots,0)}(\fp^+)$ since
they have the common lowest weight vector. Since $\fp^+$ and $\fp^+(e)_2$ have the common number $d$,
by the above formula we get the following.
\begin{Corollary}\label{inner_sub}
Let $e\in\fp^+$ be a tripotent. For $f\in\cP(\fp^+(e)_2)$, $y\in\fp^+$, we have
\[
\big\langle f(x),{\rm e}^{(x|\overline{y})_{\fp^+}}\big\rangle_{\lambda,x,\fp^+}
=\big\langle f(x'),{\rm e}^{(x'|\overline{y'})_{\fp^+(e)_2}}\big\rangle_{\lambda,x',\fp^+(e)_2}.
\]
\end{Corollary}

Also, by Theorem \ref{FKThm} or (\ref{expand_h}),
the reproducing kernel $h(x,\overline{y})^{-\lambda}$ is positive definite if and only if
\[
\lambda\in\bigg\{0, \frac{d}{2},d,\dots,\frac{d}{2}(r-1)\bigg\}\cup\bigg(\frac{d}{2}(r-1),\infty\bigg),
\]
and for these $\lambda$ there exists a unitary subrepresentation $\cH_\lambda(D)\subset\cO_\lambda(D)$.
Especially, if $\lambda>\frac{d}{2}(r-1)$, then the space of $\widetilde{K}$-finite vectors coincides with the space of all polynomials.
\[
\cH_\lambda(D)_{\widetilde{K}}=\cP(\fp^+) \qquad \text{if}\quad \lambda>\frac{d}{2}(r-1).
\]

\subsection{Restriction to symmetric subalgebras}

Let $G\subset G^\BC$ be connected simple Lie groups constructed from a simple Hermitian positive Jordan triple system $\fp^\pm$
via the Kantor--Koecher--Tits construction (Section~\ref{section_KKT}), so that $\fg^\BC=\fp^+\oplus\fk^\BC\oplus\fp^-$ holds.
We~consider a $\BC$-linear involution $\sigma$ on $\fp^\pm$, and extend to an involution on
$\fk^\BC=\mathfrak{str}(\fp^+)\subset\End(\fp^+)$ by $\sigma(l):=\sigma l\sigma$. Using this, we set
\begin{gather*}
\fp^\pm_1:=\big(\fp^\pm\big)^\sigma=\big\{x\in\fp^\pm\mid \sigma(x)=x\big\}, \\
\fp^\pm_2:=\big(\fp^\pm\big)^{-\sigma}=\big\{x\in\fp^\pm\mid \sigma(x)=-x\big\}, \\
\fk^\BC_1:=\big(\fk^\BC\big)^\sigma=\big\{l\in\fk^\BC\mid \sigma(l)=l\big\}, \\
\fk_1:=\fk^\sigma=\fk^\BC_1\cap\fk, \\
\fg^\BC_1:=\big(\fg^\BC\big)^\sigma=\fp^+_1\oplus\fk^\BC_1\oplus\fp^-_1, \\
\fg_1:=\fg^\sigma=\fg^\BC_1\cap\fg,
\end{gather*}
and let $G_1, G_1^\BC, K_1, K^\BC_1\subset G^\BC$ be the connected closed subgroups corresponding to
$\fg_1, \fg_1^\BC, \allowbreak \fk_1, \fk^\BC_1\subset \fg^\BC$ respectively.
Such $(G,G_1)$ is called a symmetric pair of holomorphic type (see \cite[Section~3.4]{Kmf1}).
Also let $\widetilde{G}_1\subset\widetilde{G}$ and $\widetilde{K}_1^\BC\subset\widetilde{K}^\BC$ be the connected closed subgroups
of the universal covering groups of~$G$ and~$K^\BC$ corresponding to $\fg_1\subset\fg$ and $\fk_1^\BC\subset\fk^\BC$ respectively,
and let $D\subset \fp^+$, $D_1\subset \fp^+_1$ be the corresponding bounded symmetric domains, so that
$D\simeq G/K\simeq \widetilde{G}/\widetilde{K}$, $D_1\simeq G_1/K_1\simeq \widetilde{G}_1/\widetilde{K}_1$ hold.

Next, we consider a finite-dimensional irreducible representations $(\tau,V)$ of $\widetilde{K}^\BC$,
and let $\cH_\tau(D,V)\subset\cO_\tau(D,V)$ be the corresponding unitary highest weight representations of $\widetilde{G}$ if it exists.
Then the restriction $\cH_\tau(D,V)|_{\widetilde{G}_1}$ decomposes discretely and the multiplicity is uniformly bounded
if $(G,G_1)$ is a symmetric pair of holomorphic type (see \cite[Theorem B]{Kmf1}).
Moreover, if~$\cH_\tau(D,V)_{\widetilde{K}}=\cP(\fp^+,V)$ holds,
then since $\fp^+_1\subset\fg^\BC$ acts on $\cH_\tau(D,V)_{\widetilde{K}}$ by differential operators of constant coefficients
along $\fp^+_1\subset\fp^+$, the set of $\fp^+_1$-null vectors is
\[
\big(\cH_\tau(D,V)_{\widetilde{K}}\big)^{\fp^+_1}=\cP\big(\fp^+_2,V\big),
\]
and every irreducible $\widetilde{G}_1$-submodule of $\cH_\tau(D,V)|_{\widetilde{G}_1}$ is generated by
some irreducible $\widetilde{K}_1^\BC$-submodule of $\cP(\fp^+_2,V)$. That is, if $\cP(\fp^+_2,V)$ is decomposed under $\widetilde{K}_1^\BC$ as
\[
\cP(\fp^+_2,V|_{\widetilde{K}_1^\BC})\simeq \bigoplus_j m(\tau,\rho_j)(\rho_j,W_j), \qquad m(\tau,\rho_j)\in\BZ_{\ge 0},
\]
then $\cH_\tau(D,V)$ is decomposed under $\widetilde{G}_1$ abstractly as
\begin{equation}\label{HKKS}
\cH_\tau(D,V)|_{\widetilde{G}_1}\simeq \hsum_j m(\tau,\rho_j)\cH_{\rho_j}(D_1,W_j)
\end{equation}
(see Kobayashi \cite[Lemma 8.8]{Kmf1}, \cite[Section 8]{Kmf0-1}. For earlier results, see also \cite{JV,Ma}).
By this decomposition, there exist $\widetilde{G}_1$-intertwining operators
\begin{gather*}
\cF_{\tau\rho_j}^\downarrow\colon\ \cH_\tau(D,V)|_{\widetilde{G}_1}\longrightarrow\cH_{\rho_j}(D_1,W_j), \\
\cF_{\tau\rho_j}^\uparrow\colon\ \cH_{\rho_j}(D_1,W_j)\longrightarrow\cH_\tau(D,V)|_{\widetilde{G}_1}.
\end{gather*}
$\cF_{\tau\rho_j}^\downarrow$ is called a symmetry breaking operator, and $\cF_{\tau\rho_j}^\uparrow$ is called a holographic operator,
according to the terminology introduced in \cite{KP1, KP2} and \cite{KP3} respectively.
By the Riesz representation theorem, these operators are always given by taking the inner products with some kernel functions,
and by \cite[Section 3]{N2} these are rewritten in differential expressions as follows.
\begin{Theorem}[{\cite[Theorem 3.10]{N2}}]\label{thm_intertwining}
Let $(\tau,V)$ be a $\widetilde{K}^\BC$-module, and $(\rho,W)$ be a $\widetilde{K}_1^\BC$-module
which appears in $\cP(\fp^+_2)\otimes V|_{\widetilde{K}_1^\BC}$.
Let $\rK(x_2)\in \cP(\fp^-_2,\Hom_\BC(V,W))$ be an operator-valued polynomial satisfying
\[
\rK\big({}^t\hspace{-1pt}kx_2\big)=\rho(k)^{-1}\rK(x_2)\tau(k), \qquad
x_2\in\fp^-_2,\quad k\in\widetilde{K}_1^\BC.
\]
\begin{enumerate}\itemsep=0pt
\item[$1.$] Assume $\cH_\tau(D,V)_{\widetilde{K}}=\cP(\fp^+,V)$.
We~define the operator-valued polynomial $F_{\tau\rho}^\downarrow(z)=F_{\tau\rho}^\downarrow(z_1,z_2)\in\cP(\fp^-,\Hom_\BC(V,W))$ by
\[
F_{\tau\rho}^\downarrow(z)=F_{\tau\rho}^\downarrow(z_1,z_2)
:=\big\langle {\rm e}^{(x|z)_{\fp^+}}I_V,\rK(\overline{x_2})^*\big\rangle_{\hat{\tau},\fp^+,x}
=\big\langle \rK(x_2),{\rm e}^{(x|\overline{z})_{\fp^-}}I_{\overline{V}}\big\rangle_{\hat{\overline{\tau}},\fp^-,x}.
\]
Then the linear map
\begin{gather*}
\cF_{\tau\rho}^\downarrow\colon\quad \cH_\tau(D,V)\longrightarrow\cH_\rho(D_1,W), \qquad
(\cF_{\tau\rho}^\downarrow f)(x_1)
=F_{\tau\rho}^\downarrow\bigg(\frac{\partial}{\partial x}\bigg)\bigg|_{x_2=0}f(x)
\end{gather*}
intertwines the $\widetilde{G}_1$-action.
\item[$2.$] Assume $(G,G_1)$ is symmetric, and also assume ``\,$\cH_\tau(D,V)_{\widetilde{K}}=\cP(\fp^+,V)$ holds'' or \linebreak ``\,$\cH_\rho(D_1,W)$ is
a holomorphic discrete series representation''. We~define the operator-valued function
$F_{\tau\rho}^\uparrow(x_2;w_1)\in\cO((D\cap\fp^+_2)\times\fp^-_1,\Hom(W,V))$ by
\[
F_{\tau\rho}^\uparrow(x_2;w_1):=\big\langle {\rm e}^{(y_1|w_1)_{\fp^+_1}}I_W,
\rK\big((\overline{x_2})^{Q(y_1)\overline{x_2}}\big)
\tau(B(y_1,\overline{x_2}))\big\rangle_{\hat{\rho},\fp^+_1,y_1}.
\]
Then the linear map
\begin{gather*}
\cF_{\tau\rho}^\uparrow\colon\quad \cH_\rho(D_1,W)_{\widetilde{K}_1}\longrightarrow \cO_\tau(D,V)_{\widetilde{K}}, \qquad
(\cF_{\tau\rho}^\uparrow f)(x)=F_{\tau\rho}^\uparrow\bigg(x_2;\frac{\partial}{\partial x_1}\bigg)f(x_1)
\end{gather*}
intertwines the $(\fg_1,\widetilde{K}_1)$-action.
\end{enumerate}
\end{Theorem}

Here we identify $\Hom_\BC(V,W)\simeq\overline{V}\otimes W\simeq \Hom_\BC\big(\overline{W},\overline{V}\big)$ via
the $\widetilde{K}$- and $\widetilde{K}_1$-invariant inner products on $V$ and $W$ respectively,
and we normalize $\frac{\partial}{\partial x}$ and $\frac{\partial}{\partial x_1}$ with respect to the bilinear forms
$(\cdot|\cdot)_{\fp^+}\colon\fp^+\times\fp^-\to\BC$ and $(\cdot|\cdot)_{\fp^+_1}\colon\fp^+_1\times\fp^-_1\to\BC$ respectively.
We~note that since $F^\downarrow_{\tau\rho}(z)$ is always a polynomial, $\cF^\downarrow_{\tau\rho}$ is a finite-order differential operator,
and hence extends to the continuous map $\cF_{\tau\rho}^\downarrow\colon \cO_\tau(D,V)\to\cO_\rho(D_1,W)$.
On the other hand, since $F^\uparrow_{\tau\rho}(x_2;w_1)$ is not a~poly\-no\-mial
and $\cF^\uparrow_{\tau\rho}$ is an infinite-order differential operator in general,
this is originally defined on the space of $\widetilde{K}_1$-finite vectors.
However, if $\cH_\rho(D_1,W)$ is holomorphic discrete, then $\big(\cF^\uparrow_{\tau\rho}f\big)(x)$ converges uniformly on every compact set
in some subset $D'\subset D$, and analytically continued for all $x\in D$,
and therefore $\cF^\uparrow_{\tau\rho}$ extends to the continuous map $\cF_{\tau\rho}^\uparrow\colon \cO_\rho(D_1,W)\to \cO_\tau(D,V)$
(see \cite[Theorems 3.12]{N2}).

The symmetry breaking operators are also characterized in terms of polynomial solutions of certain differential equations.
That is, any operator in $\Hom_{\widetilde{G}_1}(\cO_\tau(D,V),\cO_\rho(D_1,W))$ is given by a differential operator
(localness theorem, \cite[Theorem 5.3]{KP1}), and its symbol satisfies some differential equation.
For $w\in\fp^\pm$, let $\cB_\tau^{\fp^\pm}(w)$ be the differential operators given by
\begin{gather}
\cB_\tau^{\fp^+}(w)\colon\ \cP(\fp^+,V)\longrightarrow \cP(\fp^+,V), \notag
\\
\bigl(\cB_\tau^{\fp^+}(w)f\bigr)(z):=\sum_{\alpha,\beta}\frac{1}{2}(z|Q(e_\alpha^-,e_\beta^-)w)_{\fp^+}
\frac{\partial^2f}{\partial z_\alpha^+\partial z_\beta^+}(z)
-{\rm d}\tau(D(w,e_\alpha^-))\sum_{\alpha}\frac{\partial f}{\partial z_\alpha^+}(z),
\label{Bessel_op}
\\
\cB_\tau^{\fp^-}(w)\colon\ \cP(\fp^-,\Hom_\BC(V,\BC))\longrightarrow \cP(\fp^-,\Hom_\BC(V,\BC)), \notag
\\
\bigl(\cB_\tau^{\fp^-}(w)f\bigr)(z):=\sum_{\alpha,\beta}\frac{1}{2}(Q(e_\alpha^+,e_\beta^+)w|z)_{\fp^+}
\frac{\partial^2f}{\partial z_\alpha^-\partial z_\beta^-}(z)
-\sum_{\alpha}\frac{\partial f}{\partial z_\alpha^-}(z){\rm d}\tau(D(e_\alpha^+,w)), \notag
\end{gather}
where $\{e_\alpha^\mp\}\subset\fp^\mp$ is a basis of $\fp^\mp$, and $\{z_\alpha^\pm\}$ is the coordinate of $\fp^\pm$ for the dual basis
$\{e_\alpha^\pm\}\subset\fp^\pm$ with respect to $(\cdot|\cdot)_{\fp^\mp}\colon \fp^\mp\times\fp^\pm\to\BC$, so that
\begin{gather*}
(\cB_\tau^{\fp^-}(w))_z {\rm e}^{(x|z)_{\fp^+}}I_V=\bigl((Q(x)w|z)_{\fp^+}I_V-{\rm d}\tau(D(x,w))\bigr){\rm e}^{(x|z)_{\fp^+}}
=({\rm d}\hat{\tau}(0,0,w))_x {\rm e}^{(x|z)_{\fp^+}}I_V,
\\
\bigl(\cB_\tau^{\fp^-}(w)(f(\overline{\cdot})^*)\bigr)(z)
=\bigl((\cB_\tau^{\fp^+}(\overline{w})f)(\overline{z})\bigr)^*,\qquad
w,z\in\fp^-,\quad x\in\fp^+,\quad f\in\cP(\fp^+,V)
\end{gather*}
hold, and $\cB_\tau^{\fp^-}(w)$ coincides with the algebraic Fourier transform of the $L^2$-dual operator of ${\rm d}\hat{\tau}(0,0,w)$.
These are generalizations of the Bessel operator $\cB_\nu$ in \cite{D} or \cite[Section XV.2]{FK}.
\begin{Theorem}[{F-method, \cite[Theorem 3.1]{KP2}}]\label{thm_Fmethod}
Let $(\tau,V)$ and $(\rho,W)$ be $\widetilde{K}^\BC$- and $\widetilde{K}_1^\BC$-modules respectively.
Then we have the isomorphisms
\begin{gather*}
\Hom_{\widetilde{G}_1}(\cO_\tau(D,V),\cO_\rho(D_1,W))
\simeq \Hom_{\fk^\BC_1\oplus\fp^-_1}\big(W^\vee,\operatorname{ind}_{\fk^\BC\oplus\fp^-}^{\fg^\BC}\big(V^\vee\big)\big) \\ \qquad
{}\simeq \left\{ F(z)\in\cP(\fp^-,\Hom_\BC(V,W))\ \middle| \begin{array}{cl}
F({}^t\hspace{-1pt}kz)=\rho(k)^{-1}F(z)\tau(k) & \big(z\in\fp^-,\; k\in\widetilde{K}_1^\BC\big), \\
\bigl(\cB_\tau^{\fp^-}(w_1)\otimes I_W\bigr)F(z)=0 & (w_1\in \fp^-_1) \end{array} \right\}\!.
\end{gather*}
\end{Theorem}

Here, the isomorphism between the left and right hand sides is given by taking the symbol
of the differential operator. Moreover, if $\cH_\tau(D,V)\ne\{0\}$, then it induces the operator in
$\Hom_{\widetilde{G}_1}(\cH_\tau(D,V),\cH_\rho(D_1,W))$ (see \cite[Theorem 5.13]{KP1}).
Now, as in \cite[Remark 3.11]{N2} we have
\begin{align}
\bigl(\bigl(\cB_\tau^{\fp^-}(w_1)\otimes I_W\bigr)F_{\tau\rho}^\downarrow\bigr)(z)
&=\bigl(\big(\cB_\tau^{\fp^-}(w_1)\big)_z\otimes I_W\bigr)\big\langle {\rm e}^{(x|z)_{\fp^+}}I_V,\rK(\overline{x_2})^*\big\rangle_{\hat{\tau},\fp^+,x}\notag
\\
&=\big\langle ({\rm d}\hat{\tau}(0,0,w_1))_x {\rm e}^{(x|z)_{\fp^+}}I_V,\rK(\overline{x_2})^*\big\rangle_{\hat{\tau},\fp^+,x} \notag
\\
&=-\big\langle {\rm e}^{(x|z)_{\fp^+}}I_V,({\rm d}\hat{\tau}(\overline{w_1},0,0))_x\rK(\overline{x_2})^*\big\rangle_{\hat{\tau},\fp^+,x}
=0, \label{diff_eq_Bessel}
\end{align}
and since
\begin{align*}
\dim\Hom_{\widetilde{G}_1}(\cO_\tau(D,V),\cO_\rho(D_1,W))&=\dim\Hom_{\widetilde{G}_1}(\cH_\tau(D,V),\cH_\rho(D_1,W)) \\
&=\dim\Hom_{\widetilde{K}_1^\BC}\big(\cP(\fp^+_2)\otimes V,W\big)
\\
&=\dim\,\bigl(\cP(\fp^-_2)\otimes\Hom_\BC(V,W)\bigr)^{\widetilde{K}_1^\BC}
\end{align*}
holds if $\cH_\tau(D,V)_{\widetilde{K}_1}=\cP(\fp^+,V)$ by (\ref{HKKS}), these two approaches coincide for holomorphic discrete range.
On the other hand, we do not know a priori whether the analytic continuation of $\cF_{\tau\rho}^\downarrow$ exhausts all of symmetry breaking
operators or not, and we need further study on the differential equation (\ref{diff_eq_Bessel}) for non-holomorphic discrete range.

Now we assume $\fp^+$ is simple and $(\tau,V)=\big(\chi^{-\lambda},\BC\big)$ with $\lambda>\frac{d}{2}(r-1)$.
Then for any involution $\sigma$ on $\fp^+$, $\fp^+_2=(\fp^+)^{-\sigma}$ is a direct sum of at most two simple Jordan triple subsystems.
Let $\varepsilon_1\in\{1,2\}$ be the number such that
\[
(x_1|y_1)_{\fp^+}=\varepsilon_1(x_1|y_1)_{\fp^+_1}, \qquad
x_1\in\fp^+_1,\quad y_1\in\fp^-_1.
\]
If $\fp^+_2$ is simple of rank $r_2$, then according to the decomposition
\[ \cP(\fp^+_2)=\bigoplus_{\bk\in\BZ_{++}^{r_2}}\cP_\bk\big(\fp^+_2\big) \]
under $K_1^\BC$, $\cH_\lambda(D)$ is abstractly decomposed under $\widetilde{G}_1$ as
\[ \cH_\lambda(D)|_{\widetilde{G}_1}\simeq\hsum_{\bk\in\BZ_{++}^{r_2}}\cH_{\varepsilon_1\lambda}\big(D_1,\cP_\bk\big(\fp^+_2\big)\big). \]
Similarly, if $\fp^+_2=\fp^+_{11}\oplus\fp^+_{22}$ is a direct sum of two simple subsystems,
let $\rank\fp^+_{11}=r'$, $\rank\fp^+_{22}\allowbreak=r''$, and write $\fp^+_1=\fp^+_{12}$, $D_1=D_{12}$.
Then according to the decomposition
\[ \cP(\fp^+_2)=\bigoplus_{\bk\in\BZ_{++}^{r'}}\bigoplus_{\bl\in\BZ_{++}^{r''}}\cP_\bk\big(\fp^+_{11}\big)\boxtimes\cP_\bl\big(\fp^+_{22}\big) \]
under $K_1^\BC$, $\cH_\lambda(D)$ is abstractly decomposed under $\widetilde{G}_1$ as
\[
\cH_\lambda(D)|_{\widetilde{G}_1}\simeq\hsum_{\bk\in\BZ_{++}^{r'}}\hsum_{\bl\in\BZ_{++}^{r''}}
\cH_{\varepsilon_1\lambda}\big(D_{12},\cP_\bk\big(\fp^+_{11}\big)\boxtimes\cP_\bl\big(\fp^+_{22}\big)\big)
\]
(see Kobayashi \cite[Theorem 8.3]{Kmf1}). Since these decompose multiplicity-freely,
the $\widetilde{G}_1$-inter\-twi\-ning operators are unique up to scalar multiple for $\lambda>\frac{d}{2}(r-1)$.

The object of this article is to compute the inner product $\langle f,g\rangle_\lambda$ for $f\in\cP_\bk\big(\fp^+_2\big)$
or $f\in\cP_\bk\big(\fp^+_{11}\big)\boxtimes\cP_\bl\big(\fp^+_{22}\big)$ and $g\in\cP\big(\fp^+\big)$. Then since
\begin{equation}\label{inner_Fubini}
\langle f(x),g(x)\rangle_{\lambda,x}
=\big\langle f(x),\big\langle g(z),{\rm e}^{(z|\overline{x})_{\fp^+}}\big\rangle_{F,z}\big\rangle_{\lambda,x}
=\big\langle\big\langle f(x),{\rm e}^{(x|\overline{z})_{\fp^+}}\big\rangle_{\lambda,x},g(z)\big\rangle_{F,z}
\end{equation}
holds, it suffices to compute $\big\langle f,{\rm e}^{(\cdot|\overline{z})_{\fp^+}}\big\rangle_\lambda$.
As in (\ref{diff_eq_Bessel}), this becomes a solution of $\cB_{\chi^{-\lambda}}^{\fp^+}(w_1)$ for $w_1\in\fp^+_1$.
This result is applied for determining which $\big(\fg,\widetilde{K}\big)$-submodule of $\cO_\lambda(D)_{\widetilde{K}}$ contains the
$\big(\fg_1,\widetilde{K}_1\big)$-submodule generated by $\cP_\bk(\fp^+_2)$ for singular $\lambda$,
and for explicit construction of the intertwining operator (symmetry breaking operator)
$\cH_\lambda(D)\to\cH_{\varepsilon_1\lambda}\big(D_1,\cP_\bk\big(\fp^+_2\big)\big)$.
However, it seems difficult to compute this for general $\bk\in\BZ_{++}^{r_2}$ or $(\bk,\bl)\in\BZ_{++}^{r'}\times\BZ_{++}^{r''}$.
In this paper, we assume $\fp^+$, $\fp^+_2$ are of tube type,
and compute $\big\langle f,{\rm e}^{(\cdot|\overline{z})_{\fp^+}}\big\rangle_\lambda$ for
$f\in\cP_{(k+l,\underline{k}_{r_2-1})}\big(\fp^+_2\big)$ or $f\in\cP_{(\underline{k+1}_l,\underline{k}_{r_2-l})}\big(\fp^+_2\big)$ when $\fp^+_2$ is simple,
and for $f\in\cP_{\underline{k}_{r'}}\big(\fp^+_{11}\big)\boxtimes\cP_\bl\big(\fp^+_{22}\big)$ when $\fp^+_2=\fp^+_{11}\oplus\fp^+_{22}$.
Also, without the tube type assumption, we compute the poles of the meromorphic continuation of
$\big\langle f,{\rm e}^{(\cdot|\overline{z})_{\fp^+}}\big\rangle_\lambda$ with respect to $\lambda\in\BC$ for wider cases.

\section{Preliminaries 2: Explicit realization}\label{section_prelim2}

A simple Hermitian positive Jordan triple system $\fp^\pm$ is isomorphic to one of the following:
\begin{alignat*}{3}
\fp^\pm&=\BC^n, \qquad&&\Sym(r,\BC), \qquad&&M(q,s;\BC), \\ &\eqspace\Skew(s,\BC), \qquad&&\Herm(3,\BO)^\BC, \qquad&&M(1,2;\BO)^\BC.
\end{alignat*}
Here, $\Sym(r,\BC)$ and $\Skew(s,\BC)$ denote the spaces of symmetric and skew-symmetric matrices over $\BC$ respectively,
and $\Herm(3,\BO)$ denotes the space of $3\times 3$ Hermitian matrices over the octonions $\BO$.
In this section we fix the Jordan triple system structures on $\fp^\pm$ and the realization of the representation spaces $\cH_\lambda(D)$.

\subsection[The case: $\fp^\pm=\Sym(r,\BC)$, $M(q,s;\BC)$, $\Skew(s,\BC)$]
{The case: $\boldsymbol{\fp^\pm=\Sym(r,\BC)}$, $\boldsymbol{M(q,s;\BC)}$, $\boldsymbol{\Skew(s,\BC)}$}

In this section we consider $\fp^\pm=\Sym(r,\BC)$, $M(q,s;\BC)$, $\Skew(s,\BC)$, with the Jordan triple system structure
\[
\{\cdot,\cdot,\cdot\}\colon\quad \fp^\pm\times\fp^\mp\times\fp^\pm\longrightarrow \fp^\pm, \qquad
\{x,y,z\}:=x{}^t\hspace{-1pt}yz+z{}^t\hspace{-1pt}yx,
\]
and the antilinear map $\overline{\cdot}\colon \fp^\pm\to\fp^\mp$ taking the complex conjugate on each component.
Especially, for $x\in\fp^\pm$, the map $Q(x)$ is given by
\[
Q(x)\colon\quad \fp^\mp\longrightarrow \fp^\pm, \qquad Q(x)y=x{}^t\hspace{-1pt}yx.
\]
The numbers $(n,r,d,b,p)$ are given by
\[
(n,r,d,b,p)=\begin{cases}
\left(\frac{1}{2}r(r+1), r, 1, 0, r+1\right)\!, & \fp^\pm=\Sym(r,\BC),
\\
(qs, \min\{q,s\}, 2, |q-s|, q+s), & \fp^\pm=M(q,s;\BC),
\\
\left(\frac{1}{2}s(s-1), \frac{s}{2}, 4, 0, 2(s-1)\right)\!,
& \fp^\pm=\Skew(s,\BC),\quad s\colon \text{even},
\\
\left(\frac{1}{2}s(s-1), \left\lfloor\frac{s}{2}\right\rfloor, 4, 2, 2(s-1)\right)\!, & \fp^\pm=\Skew(s,\BC),\quad s\colon \text{odd}, \end{cases}
\]
the generic norm $h\colon \fp^+\times\fp^-\to \BC$ is given by
\[
h(x,y)=\begin{cases}
\det\big(I-x{}^t\hspace{-1pt}y\big), & \fp^\pm=\Sym(r,\BC),\ M(q,s;\BC),
\\
\det\big(I-x{}^t\hspace{-1pt}y\big)^{1/2}, & \fp^\pm=\Skew(s,\BC),
\end{cases}
\]
the bilinear form $(\cdot|\cdot)_{\fp^+}\colon \fp^+\times\fp^-\to \BC$ is given by
\[
(x|y)_{\fp^+}=\begin{cases}
\tr\big(x{}^t\hspace{-1pt}y\big), & \fp^\pm=\Sym(r,\BC),\ M(q,s;\BC),
\\
\frac{1}{2}\tr\big(x{}^t\hspace{-1pt}y\big), & \fp^\pm=\Skew(s,\BC),
\end{cases}
\]
and the bounded symmetric domain $D\subset\fp^+$ is given by
\[
D=\{x\in\fp^+\mid I-xx^*\text{ is positive definite}\}.
\]
If $\fp^\pm=\Sym(r,\BC)$, $M(r,\BC)$ or $\Skew(2r,\BC)$, then $\fp^\pm$ is of tube type.
In this case, $e\in\fp^+$ is a~maximal tripotent if and only if $ee^*=I$, and for a~fixed maximal tripotent $e\in\fp^+$,
$\fp^+$ has a~Jordan algebra structure
\[
x\cdot y=\frac{1}{2}(xe^*y+ye^*x), \qquad x,y\in\fp^+,
\]
with the Euclidean real form
\[
\fn^+=\{x\in\fp^+\mid ex^*e=x\},
\]
and the determinant polynomials
\[
\det_{\fn^+}(x)=\begin{cases}
\det(x)\overline{\det(e)},
\\
\Pf(x)\overline{\Pf(e)}, \end{cases} \qquad
\det_{\fn^-}(y)=\begin{cases} \det(y)\det(e), & \fp^\pm=\Sym(r,\BC),\ M(r,\BC),
\\
\Pf(y)\Pf(e), & \fp^\pm=\Skew(2r,\BC).
\end{cases}
\]
For $x\in\fp^+$, the inverse $x^\itinv=P(x)^{-1}x\in\fp^+$ and ${}^t\hspace{-1pt}x^\itinv=Q(x)^{-1}x\in\fp^-$ are given by
\[
x^\itinv=ex^{-1}e, \qquad {}^t\hspace{-1pt}x^\itinv={}^t\hspace{-1pt}x^{-1}.
\]

Next we consider the Kantor--Koecher--Tits Lie group $G$ and the subgroup $K^\BC\subset G^\BC$,
\[
(G,K^\BC)= \begin{cases}
(\operatorname{Sp}(r,\BR),{\rm GL}(r,\BC)), & \fp^\pm=\Sym(r,\BC),
\\
({\rm SU}(q,s),{\rm S}({\rm GL}(q,\BC)\times {\rm GL}(s,\BC))), & \fp^\pm=M(q,s;\BC),
\\
({\rm SO}^*(2s),{\rm GL}(s,\BC)), & \fp^\pm=\Skew(s,\BC),
\end{cases}
\]
where the action ${\rm Ad}|_{\fp^+}$ of $K^\BC$ on $\fp^+$ is given by
\begin{alignat*}{3}
& k.x=kx{}^t\hspace{-1pt}k,\qquad && k\in K^\BC={\rm GL}(r,\BC), \quad x\in\fp^+=\Sym(r,\BC),&
\\
& (k,l).x=kxl^{-1}, \qquad && (k,l)\in K^\BC={\rm S}({\rm GL}(q,\BC)\times {\rm GL}(s,\BC)), \quad x\in\fp^+=M(q,s;\BC),&
\\
& k.x=kx{}^t\hspace{-1pt}k,\qquad && k\in K^\BC={\rm GL}(s,\BC), \quad x\in\fp^+=\Skew(s,\BC).&
\end{alignat*}
Then the Bergman operator $B\colon \fp^+\times\fp^-\supset D\times \overline{D}\to K^\BC$ is given by
\begin{alignat*}{3}
& B(x,y)=I-x{}^t\hspace{-1pt}y, \qquad &&
\fp^\pm=\Sym(r,\BC),\ \Skew(s,\BC),&
\\
& B(x,y)=\big(I-x{}^t\hspace{-1pt}y, \big(I-{}^t\hspace{-1pt}yx\big)^{-1}\big), \qquad &&
\fp^\pm=M(q,s;\BC),&
\end{alignat*}
and for $\lambda\in\BC$, the character $\chi^{-\lambda}$ of the universal covering group
$\widetilde{K}^\BC$ normalized as in (\ref{char}) is given by
\begin{alignat*}{3}
& \chi^{-\lambda}(k)=\det(k)^{-\lambda},\qquad &&
\fp^\pm=\Sym(r,\BC),&
\\
& \chi^{-\lambda}(k,l)=\det(k)^{-\lambda}=\det(l)^\lambda,\qquad &&
\fp^\pm=M(q,s;\BC),&
\\
& \chi^{-\lambda}(k)=\det(k)^{-\lambda/2},\qquad &&
\fp^\pm=\Skew(s,\BC).&
\end{alignat*}
When $\fp^\pm=M(q,s;\BC)$ we also consider
\[
\big(G,K^\BC\big)=({\rm U}(q,s),{\rm GL}(q,\BC)\times {\rm GL}(s,\BC))
\]
instead of $({\rm SU}(q,s),{\rm S}({\rm GL}(q,\BC)\times {\rm GL}(s,\BC)))$, and for $\lambda_1,\lambda_2\in\BC$,
let $\chi^{-\lambda_1,-\lambda_2}$ be the character of $\widetilde{K}^\BC$ given by
\begin{equation}\label{char_U(q,s)}
\chi^{-\lambda_1,-\lambda_2}(k,l):=\det(k)^{-\lambda_1}\det(l)^{\lambda_2}, \qquad
k\in {\rm GL}(q,\BC),\qquad l\in {\rm GL}(s,\BC).
\end{equation}
Then $\chi^{-\lambda_1,-\lambda_2}|_{{\rm S}({\rm GL}(q,\BC)\times {\rm GL}(s,\BC))^\sim}=\chi^{-\lambda_1-\lambda_2}$ holds.

Next, let $(\tau,V)$ be an irreducible representation of $\widetilde{K}^\BC$,
and consider the Hilbert subspace $\cH_\tau(D,V)\subset\cO(D,V)=\cO_\tau(D,V)$ with the reproducing kernel $\tau(B(x,\overline{y}))$
if it exists.
Then $\widetilde{G}$ acts unitarily on $\cH_\tau(D,V)$. If it is a holomorphic discrete series representation, then the inner product is given by
\begin{gather*}
\langle f,g\rangle_{\hat{\tau}} = \begin{cases}
\displaystyle C_\tau\!\int_D \!\big(\tau\big((I-xx^*)^{-1}\big)f(x),g(x)\big)_\tau \det(I-xx^*)^{-(r+1)}{\rm d}x, \quad\ \fp^\pm=\Sym(r,\BC),\!\!
\vspace{1mm}\\
\displaystyle C_\tau\!\int_D \!\big(\tau\big((I-xx^*)^{-1},I-x^*x\big)f(x),g(x)\big)_\tau \det(I-xx^*)^{-(q+s)}{\rm d}x,
\\
\hspace{99.7mm} \fp^\pm=M(q,s;\BC),\!\!
\\
\displaystyle C_\tau\!\int_D \!\big(\tau\big((I-xx^*)^{-1}\big)f(x),g(x)\big)_\tau \det(I-xx^*)^{-(s-1)}{\rm d}x,\hspace{5.5mm} \fp^\pm=\Skew(s,\BC).\!\!
\end{cases}\!\!\!\!\!\!\!\!\!
\end{gather*}
When $(\tau,V)$ is of the form $(\tau,V)=\big(\chi^{-\lambda}\otimes\tau_0,V\big)$ or $\big(\chi^{-\lambda_1,-\lambda_2}\otimes\tau_0,V\big)$
for a fixed representation $(\tau_0,V)$ of $K^\BC$, we write $\cO_\tau(D,V)=\cO_\lambda(D,V)$ or $\cO_{\lambda_1+\lambda_2}(D,V)$,
$\cH_\tau(D,V)=\cH_\lambda(D,V)$ or~$\cH_{\lambda_1+\lambda_2}(D,V)$.
Similarly, when $(\tau,V)=\big(\chi^{-\lambda},\BC\big)$ or $\big(\chi^{-\lambda_1,-\lambda_2},\BC\big)$, we write
$(\hat{\tau},\cO_\tau(D,V))=(\tau_\lambda,\cO_\lambda(D))$ or $(\tau_{\lambda_1,\lambda_2},\cO_{\lambda_1+\lambda_2}(D))$,
$(\hat{\tau},\cH_\tau(D,V))=(\tau_\lambda,\cH_\lambda(D))$ or $(\tau_{\lambda_1,\lambda_2},\cH_{\lambda_1+\lambda_2}(D))$,
and write $\langle \cdot,\cdot\rangle_{\hat{\tau}}=\langle \cdot,\cdot\rangle_{\lambda}$.
Then $\cH_\lambda(D)$ is holomorphic discrete if $\lambda>p-1$.

Next we fix notations of irreducible representations of $\widetilde{\rm GL}(s,\BC)$.
We~take a basis $\{\epsilon_j\}_{j=1}^s\subset\fh^{\BC\vee}$ of the dual space of a Cartan subalgebra
$\fh^\BC\subset\mathfrak{gl}(s,\BC)$ such that the positive root system of $\mathfrak{gl}(s,\BC)$ is given by
$\{ \epsilon_i-\epsilon_j\mid 1\le i<j\le s\}$. For $\bm=(m_1,\dots,m_s)\in\BC^s$ such that $m_j-m_{j+1}\in\BZ_{\ge 0}$,
let $V_\bm^{(s)}$ be the irreducible representation of $\widetilde{\rm GL}(s,\BC)$ with the highest weight $m_1\epsilon_1+\cdots+m_s\epsilon_s$, and
let $V_\bm^{(s)\vee}$ be the irreducible representation of $\widetilde{\rm GL}(s,\BC)$ with the lowest weight $-m_1\epsilon_1-\cdots-m_s\epsilon_s$.
Especially, if $\bm\in\BZ_{+}^s$ then $V_\bm^{(s)}$ and $V_\bm^{(s)\vee}$ are reduced to the representations of ${\rm GL}(s,\BC)$.
Using these notations, each $K^\BC$-subrepresentation of $\cP(\fp^+)=\bigoplus_{\bm\in\BZ_{++}^r}\cP_\bm(\fp^+)$ is given by
\begin{alignat*}{3}
& \cP_\bm({\rm Sym}(r,\BC))\simeq V_{(2m_1,2m_2,\dots,2m_r)}^{(r)\vee}=:V_{2\bm}^{(r)\vee}, \qquad && \bm\in\BZ_{++}^r,&
\\
& \cP_\bm(M(q,s;\BC))\simeq V_{\bm}^{(q)\vee}\boxtimes V_{(m_1,\dots,m_q,0,\dots,0)}^{(s)}=:V_{\bm}^{(q)\vee}\boxtimes V_{\bm}^{(s)},\qquad&&
q\le s,\quad \bm\in\BZ_{++}^q,&
 \\
& \cP_\bm(M(q,s;\BC))\simeq V_{(m_1,\dots,m_s,0,\dots,0)}^{(q)\vee}\boxtimes V_{\bm}^{(s)}=:V_{\bm}^{(q)\vee}\boxtimes V_{\bm}^{(s)},\qquad&&
q\ge s,\quad \bm\in\BZ_{++}^s,&
 \\
&\cP_\bm({\rm Skew}(s,\BC))\simeq V_{(m_1,m_1,m_2,m_2,\dots,m_{\lfloor s/2\rfloor},m_{\lfloor s/2\rfloor}(,0))}^{(s)\vee}=:V_{\bm^2}^{(s)\vee},\qquad&&
\bm\in\BZ_{++}^{\lfloor s/2\rfloor},&
\end{alignat*}
and the characters $\chi^{-\lambda}$, $\chi^{-\lambda_1,-\lambda_2}$ of $\widetilde{K}^\BC$ are given by
\begin{alignat*}{3}
& \chi^{-\lambda}=V_{(\lambda,\dots,\lambda)}^{(r)\vee}=V_{\underline{\lambda}_r}^{(r)\vee},\qquad && G=\operatorname{Sp}(r,\BR),&
\\
& \chi^{-\lambda_1,-\lambda_2}=V_{(\lambda_1,\dots,\lambda_1)}^{(q)\vee}\boxtimes V_{(\lambda_2,\dots,\lambda_2)}^{(s)}
=V_{\underline{\lambda_1}_q}^{(q)\vee}\boxtimes V_{\underline{\lambda_2}_s}^{(s)},\qquad&&
G={\rm U}(q,s),
\\
&\chi^{-\lambda}=V_{\left(\frac{\lambda}{2},\dots,\frac{\lambda}{2}\right)}^{(s)\vee}
=V_{\underline{\lambda/2}_s}^{(s)\vee},\qquad && G={\rm SO}^*(2s).&
\end{alignat*}

Next, for $d=1,2,4$, $\bm\in\BZ_{++}^r$, we define a polynomial $\tilde{\Phi}_\bm^{(d)}(x)$ on $M(r,\BC)$ by
\begin{equation}\label{Phi^(d)}
\tilde{\Phi}_\bm^{(d)}(x):=\tilde{\Phi}_\bm^d(t_1,\dots,t_r),
\end{equation}
where $t_1,\dots,t_r$ are the eigenvalues of $x$, and $\tilde{\Phi}_\bm^d(t_1,\dots,t_r)$ is as in (\ref{Phi^d}).
Then for $x,y\in M(q,s;\BC)$, since the non-zero eigenvalues of $x{}^t\hspace{-1pt}y$, $y{}^t\hspace{-1pt}x$, ${}^t\hspace{-1pt}xy$ and
${}^t\hspace{-1pt}yx$ coincide, we have
\[
\tilde{\Phi}_\bm^{(d)}\big(x{}^t\hspace{-1pt}y\big)=\tilde{\Phi}_\bm^{(d)}\big(y{}^t\hspace{-1pt}x\big)
=\tilde{\Phi}_\bm^{(d)}\big({}^t\hspace{-1pt}xy\big)=\tilde{\Phi}_\bm^{(d)}\big({}^t\hspace{-1pt}yx\big).
\]
Also, by (\ref{expand_expsum}), (\ref{expand_binom}), we have
\begin{alignat*}{3}
& {\rm e}^{\tr(x)}=\sum_{\bm\in\BZ_{++}^r}\tilde{\Phi}_\bm^{(d)}(x),\qquad && x\in M(r,\BC),&
\\
& \det(I-x)^{-\lambda}=\sum_{\bm\in\BZ_{++}^r}(\lambda)_{\bm,d}\tilde{\Phi}_\bm^{(d)}(x),\qquad&&
x\in M(r,\BC),\quad |x|_{\op}<1,&
\end{alignat*}
where $|x|_{\op}$ is the operator norm on $M(r,\BC)$. Similarly, let
\[
M'(s,\BC):=\{x\in M(s,\BC)\mid \text{all non-zero eigenvalues of }x\text{ have even multiplicities}\},
\]
and for $d=1,2,4$, $\bm\in\BZ_{++}^{\lfloor s/2\rfloor}$, we define a function $\tilde{\Phi}_\bm^{[d]}(x)$ on $M'(s,\BC)$ by
\begin{equation}\label{Phi^[d]}
\tilde{\Phi}_\bm^{[d]}(x):=\tilde{\Phi}_\bm^d(t_1,\dots,t_{\lfloor s/2\rfloor}),
\end{equation}
where $t_1$, $t_1$, $t_2$, $t_2$, $\dots$, $t_{\lfloor s/2\rfloor}$, $t_{\lfloor s/2\rfloor}(,0)$ are the eigenvalues of $x$.
Then we have
\begin{alignat*}{3}
& {\rm e}^{\frac{1}{2}\tr(x)}=\sum_{\bm\in\BZ_{++}^{\lfloor s/2\rfloor}}\tilde{\Phi}_\bm^{[d]}(x),\qquad &&
x\in M'(s,\BC),&
\\
& \det(I-x)^{-\lambda/2}=\sum_{\bm\in\BZ_{++}^{\lfloor s/2\rfloor}}(\lambda)_{\bm,d}\tilde{\Phi}_\bm^{[d]}(x),\qquad&&
x\in M'(s,\BC),\quad |x|_{\op}<1.&
\end{alignat*}
Using these, $\tilde{\Phi}_\bm^{\fp^+}(x,y)\in\cP_\bm(\fp^+)\otimes\cP_\bm(\fp^-)$ is given by
\[
\tilde{\Phi}_\bm^{\fp^+}(x,y)=
\begin{cases}
\tilde{\Phi}_\bm^{(1)}\big(x{}^t\hspace{-1pt}y\big),& \fp^\pm=\Sym(r,\BC), \\
\tilde{\Phi}_\bm^{(2)}\big(x{}^t\hspace{-1pt}y\big),& \fp^\pm=M(q,s;\BC), \\
\tilde{\Phi}_\bm^{[4]}\big(x{}^t\hspace{-1pt}y\big),& \fp^\pm=\Skew(s,\BC).
\end{cases}
\]
Also, for $\alpha,\beta,\gamma\in\BC$ with $\gamma\notin\bigcup_{j=1}^r\left(\frac{d}{2}(j-1)-\BZ_{\ge 0}\right)$ we define
\begin{alignat}{3}
& {}_0F_1^{(d)}(\gamma;x):=\sum_{\bm\in\BZ_{++}^r}\frac{1}{(\gamma)_{\bm,d}}\tilde{\Phi}_\bm^{(d)}(x),\qquad && x\in M(r,\BC), &\label{0F1^(d)}
\\
& {}_0F_1^{[d]}(\gamma;x):=\sum_{\bm\in\BZ_{++}^{\lfloor s/2\rfloor}}\frac{1}{(\gamma)_{\bm,d}}\tilde{\Phi}_\bm^{[d]}(x),\qquad && x\in M'(s,\BC),&
\label{0F1^[d]}
\\
& {}_2F_1^{(d)}\left( \begin{matrix}\alpha,\beta\\\gamma \end{matrix};x\right)
:=\sum_{\bm\in\BZ_{++}^r}\frac{(\alpha)_{\bm,d}(\beta)_{\bm,d}}{(\gamma)_{\bm,d}} \tilde{\Phi}_\bm^{(d)}(x),\qquad && x\in M(r,\BC),\quad |x|_\op<1,&
\label{2F1^(d)}
\\
&{}_2F_1^{[d]}\left(\begin{matrix}\alpha,\beta\\\gamma \end{matrix};x\right)
:=\sum_{\bm\in\BZ_{++}^{\lfloor s/2\rfloor}}\frac{(\alpha)_{\bm,d}(\beta)_{\bm,d}}{(\gamma)_{\bm,d}}\tilde{\Phi}_\bm^{[d]}(x),\qquad&&
x\in M'(s,\BC),\quad |x|_\op<1. & \label{2F1^[d]}
\end{alignat}

\subsection[The case: $\fp^\pm=\BC^n$]{The case: $\boldsymbol{\fp^\pm=\BC^n}$}

In this section we consider $\fp^\pm=\BC^n$ with $n\ge 3$.
For $x={}^t\hspace{-1pt}(x_1,\dots,x_n)$, $y={}^t\hspace{-1pt}(y_1,\dots,y_n)\in\BC^n$,
let $q(x):=x_1^2+\cdots+x_n^2$, $q(x,y):=x_1y_1+\cdots+x_ny_n$. We~consider the Jordan triple system structure
\[
Q(x)\colon\quad \fp^\mp\longrightarrow\fp^\pm, \qquad Q(x)y:=2q(x,y)x-q(x)y,
\]
and the antilinear map $\overline{\cdot}\colon\fp^\pm\to\fp^\mp$ taking the complex conjugate on each component.
Then we have $(n,r,d,b,p)=(n,2,n-2,0,n)$. The generic norm $h\colon \fp^+\times\fp^-\to\BC$ is given by
\[ h(x,y)=1-2q(x,y)+q(x)q(y), \]
the bilinear form $(\cdot|\cdot)_{\fp^+}\colon\fp^+\times\fp^-\to\BC$ is given by
\[ (x|y)_{\fp^+}=2q(x,y), \]
the Bergman operator $B\colon \fp^+\times\fp^-\to\End_\BC(\fp^+)\simeq M(n,\BC)$ is given by
\[ B(x,y)=h(x,y)I-2(1-q(x,y))\big(x{}^t\hspace{-1pt}y-y{}^t\hspace{-1pt}x\big)+\big(x{}^t\hspace{-1pt}y-y{}^t\hspace{-1pt}x\big)^2, \]
and the bounded symmetric domain $D\subset\fp^+$ is given by
\[
D=\big\{x\in\fp^+\mid 1-2q(x,\overline{x})+|q(x)|^2>0,\, |q(x)|<1\big\}.
\]
$\fp^+=\BC^n$ is of tube type, and $e\in\fp^+$ is a maximal tripotent if and only if $e=tu$ holds for some $t\in\BC$, $|t|=1$ and
$u\in \BR^n\subset\BC^n$, $q(u)=1$. For a fixed maximal tripotent $e=tu\in\fp^+$, a Jordan algebra structure is defined on $\fp^+$
such that $e$ becomes a unit element, with the Euclidean real form $\fn^+\subset\fp^+$ given by
\[
\fn^+=t\big(\BR u+\sqrt{-1}(\BR u)^\bot\big)\subset\BC^n,
\]
where $(\BR u)^\bot\subset\BR^n$ is the orthgonal complement of $\BR u$ in $\BR^n$ with respect to $q(\cdot,\cdot)$.
The corresponding determinant polynomials $\det_{\fn^\pm}(x)$ on $\fp^\pm$ are given by
\[
\det_{\fn^+}(x)=q(x)\overline{q(e)}, \qquad \det_{\fn^-}(x)=q(x)q(e),
\]
and for $x\in\fp^+$, the inverse $x^\itinv=P(x)^{-1}x\in\fp^+$ and ${}^t\hspace{-1pt}x^\itinv=Q(x)^{-1}x\in\fp^-$ are given by
\[
x^\itinv=\frac{q(e)}{q(x)}\bigg(2\frac{q(e,x)}{q(e)}e-x\bigg),\qquad {}^t\hspace{-1pt}x^\itinv=\frac{1}{q(x)}x.
\]

Next we consider the Kantor--Koecher--Tits Lie group $G$ and the subgroup $K^\BC\subset G^\BC$,
\[
(G,K^\BC)=({\rm SO}_0(2,n),{\rm SO}(2,\BC)\times {\rm SO}(n,\BC)),
\]
where $K^\BC$ acts on $\fp^+=\BC^n$ via the map ${\rm Ad}|_{\fp^+}\colon K^\BC\to \operatorname{Str}(\fp^+)\subset \End_\BC(\fp^+)\simeq M(n,\BC)$,
\[
\left( \begin{pmatrix} \cos\theta & \sin\theta \\ -\sin\theta & \cos\theta \end{pmatrix}\!,l\right)\mapsto {\rm e}^{\sqrt{-1}\theta}l.
\]
We~lift the Bergman operator $B\colon \fp^+\times\fp^-\supset D\times\overline{D}\to\operatorname{Str}(\fp^+)\subset\End_\BC(\fp^+)$ to the map
$B\colon D\times\overline{D}\to K^\BC$ or $B\colon D\times\overline{D}\to \widetilde{K}^\BC$ through the above covering map.
Also, for $\lambda\in\BC$, the character $\chi^{-\lambda}$ of $\widetilde{K}^\BC$ normalized as in (\ref{char}) is given by
\[
\chi^{-\lambda} \left( \begin{pmatrix} \cos\theta & \sin\theta \\ -\sin\theta & \cos\theta \end{pmatrix}\!,l\right)
={\rm e}^{-\sqrt{-1}\lambda\theta}.
\]

Next, let $(\tau,V)$ be an irreducible representation of $\widetilde{K}^\BC$, and consider the Hilbert subspace
$\cH_\tau(D,V)\subset\cO(D,V)=\cO_\tau(D,V)$ with the reproducing kernel $\tau(B(x,\overline{y}))$ if it exists.
Then $\widetilde{G}$ acts unitarily on $\cH_\tau(D,V)$. If it is a holomorphic discrete series representation,
then the inner product is given by
\[ \langle f,g\rangle_{\hat{\tau}}=C_\tau\int_D \big(\tau\big(B(x,\overline{x})^{-1}\big)f(x),g(x)\big)_\tau
(1-2q(x,\overline{x})+|q(x)|^2)^{-n}{\rm d}x, \]
where the Lebesgue measure ${\rm d}x$ on $\fp^+=\BC^n$ is normalized with respect to $(x|\overline{y})_{\fp^+}=2q(x,\overline{y})$.
When $(\tau,V)$ is of the form $(\tau,V)=\big(\chi^{-\lambda}\otimes\tau_0,V\big)$ for a fixed representation $(\tau_0,V)$ of $K^\BC$,
we write $\cO_\tau(D,V)=\cO_\lambda(D,V)$, $\cH_\tau(D,V)=\cH_\lambda(D,V)$. Similarly, when $(\tau,V)=\big(\chi^{-\lambda},\BC\big)$,
we write $(\hat{\tau},\cO_\tau(D,V))=(\tau_\lambda,\cO_\lambda(D))$, $(\hat{\tau},\cH_\tau(D,V))=(\tau_\lambda,\cH_\lambda(D))$,
and write $\langle\cdot,\cdot\rangle_{\hat{\tau}}=\langle\cdot,\cdot\rangle_\lambda$.
Then $\cH_\lambda(D)$ is holomorphic discrete if $\lambda>n-1$, and then the inner product is given by
\[ \langle f,g\rangle_\lambda=C_\lambda\int_D f(x)\overline{g(x)} \big(1-2q(x,\overline{x})+|q(x)|^2\big)^{\lambda-n}{\rm d}x. \]

If $n=1,2$, then $\BC$ is of rank 1, $\BC^2$ is not simple, and we have local isomorphisms
\begin{gather*}
{\rm SO}_0(2,1)\simeq {\rm SL}(2,\BR)=\operatorname{Sp}(1,\BR),
\\
{\rm SO}_0(2,2)\simeq {\rm SL}(2,\BR)\times {\rm SL}(2,\BR)=\operatorname{Sp}(1,\BR)\times \operatorname{Sp}(1,\BR).
\end{gather*}
For these cases, let $\cH_\lambda(D_{{\rm SO}_0(2,1)})\subset\cO(D_{{\rm SO}_0(2,1)})$ be the Hilbert space with the reproducing kernel
\[
\big(1-2q(x,\overline{y})+q(x)\overline{q(y)}\big)^{-\lambda}=(1-x\overline{y})^{-2\lambda},
\]
and let $\cH_\lambda(D_{{\rm SO}_0(2,2)})\subset\cO(D_{{\rm SO}_0(2,2)})$ be the Hilbert space with the reproducing kernel
\begin{gather*}
\big(1-2q(x,\overline{y})+q(x)\overline{q(y)}\big)^{-\lambda}
\\ \qquad
{}
=\bigl(1-\big(x_1+\sqrt{-1}x_2\big)\big(\overline{y_1+\sqrt{-1}y_2}\big)\bigr)^{-\lambda}
\bigl(1-\big(x_1-\sqrt{-1}x_2\big)\big(\overline{y_1-\sqrt{-1}y_2}\big)\bigr)^{-\lambda},
\end{gather*}
so that we have the isomorphisms
\begin{gather*}
\cH_\lambda(D_{{\rm SO}_0(2,1)})\simeq \cH_{2\lambda}(D_{{\rm SL}(2,\BR)}),
\\
\cH_\lambda(D_{{\rm SO}_0(2,2)})\simeq \cH_\lambda(D_{{\rm SL}(2,\BR)})\hboxtimes\cH_\lambda(D_{{\rm SL}(2,\BR)}).
\end{gather*}

Next we fix notations of irreducible representations of $\operatorname{Spin}(n,\BC)$ for $n\ge 3$.
We~take a basis $\{\epsilon_j\}_{j=1}^{\lfloor n/2\rfloor}\subset\fh^{\BC\vee}$ of the dual space of a Cartan subalgebra
$\fh^\BC\subset\mathfrak{so}(n,\BC)$ such that the positive root system of $\mathfrak{so}(n,\BC)$ is given by
\begin{alignat*}{3}
& \{\epsilon_i\pm\epsilon_j\mid 1\le i<j\le n/2\},\qquad&& n\colon \text{even},&
\\
& \{\epsilon_i\pm\epsilon_j\mid 1\le i<j\le \lfloor n/2\rfloor\}\cup\{\epsilon_j\mid 1\le j\le \lfloor n/2\rfloor\},\qquad && n\colon \text{odd}.&
\end{alignat*}
For $\bm\in\BZ^{\lfloor n/2\rfloor}\cup\left(\BZ+\frac{1}{2}\right)^{\lfloor n/2\rfloor}$ such that
$m_1\ge\cdots\ge m_{n/2-1}\ge |m_{n/2}|$ for even $n$,
$m_1\ge\cdots\ge m_{\lfloor n/2\rfloor}\ge 0$ for odd $n$, let $V_\bm^{[n]}$ be the irreducible representation of $\operatorname{Spin}(n,\BC)$
with the highest weight $m_1\epsilon_1+\cdots+m_{\lfloor n/2\rfloor}\epsilon_{\lfloor n/2\rfloor}$,
and let $V_\bm^{[n]\vee}$ be the irreducible representation of $\operatorname{Spin}(n,\BC)$
with the lowest weight $-m_1\epsilon_1-\cdots-m_{\lfloor n/2\rfloor}\epsilon_{\lfloor n/2\rfloor}$.
Especially, if $\bm\in\BZ^{\lfloor n/2\rfloor}$ then $V_\bm^{[n]}$, $V_\bm^{[n]\vee}$ are reduced to the representations of ${\rm SO}(n,\BC)$.
Using these notations, each $K^\BC$-subrepresentation of $\cP(\BC^n)=\bigoplus_{\bm\in\BZ_{++}^2}\cP_\bm(\BC^n)$ for $n\ge 3$ is given by
\[
\cP_\bm(\BC^n)\simeq \chi^{-m_1-m_2}\boxtimes V_{(m_1-m_2,0,\dots,0)}^{[n]\vee}.
\]
When $n=1$ we write
\[
\cP(\BC)=\bigoplus_{m=0}^\infty \BC x^m
=:\bigoplus_{m=0}^\infty \cP_{\left(\lceil\frac{m}{2}\rceil,\lfloor\frac{m}{2}\rfloor\right)}(\BC),
\]
and write $V^{[1]}_*=V^{[1]\vee}_*=\BC$. When $n=2$ we write
\[
\cP(\BC^2)=\bigoplus_{m_1,m_2=0}^\infty \BC \big(x_1+\sqrt{-1}x_2\big)^{m_1}\big(x_1-\sqrt{-1}x_2\big)^{m_2}
=:\bigoplus_{m_1,m_2=0}^\infty \cP_{(m_1,m_2)}\big(\BC^2\big),
\]
and for $m\in\BZ$, let $V^{[2]}_m=V^{[2]\vee}_{-m}=\BC_m$ be the representation of ${\rm SO}(2,\BC)$ such that
$\cP_\bm\big(\BC^2\big)\simeq\chi^{-m_1-m_2}\boxtimes V^{[2]\vee}_{m_1-m_2}$ holds.

\subsection[The case: $\fp^\pm=\Herm(3,\BO)^\BC$, $M(1,2;\BO)^\BC$]
{The case: $\boldsymbol{\fp^\pm=\Herm(3,\BO)^\BC}$, $\boldsymbol {M(1,2;\BO)^\BC}$}

In this section we consider the exceptional Jordan triple systems. For detail see also \cite[Sec\-ti\-ons~4.3 and 4.4]{N2}.
First we consider $\fp^\pm=\Herm(3,\BO)^\BC$, the complexification of the space of $3\times 3$ Hermitian matrices
over the octonions~$\BO$. Let $(\cdot|\cdot)_{\fp^+}\colon \fp^+\times\fp^-\to\BC$ be the bilinear form
\[
(x|y)_{\fp^+}:=\Re_\BO\tr(xy)=\Re_\BO\tr(yx),
\]
where $\Re_\BO\colon \BO^\BC\to\BC$ is the $\BC$-linear extension of the octonionic real part.
For $x\in\Herm(3,\BO)^\BC$, let $x^\sharp$ be the adjugate matrix, that is,
\[
\begin{pmatrix}\xi_1&x_3&\hat{x}_2\\\hat{x}_3&\xi_2&x_1\\x_2&\hat{x}_1&\xi_3\end{pmatrix}^\sharp
:=\begin{pmatrix}\xi_2\xi_3-x_1\hat{x}_1&\hat{x}_2\hat{x}_1-\xi_3x_3&x_3x_1-\xi_2\hat{x}_2\\
x_1x_2-\xi_3\hat{x}_3&\xi_3\xi_1-x_2\hat{x}_2&\hat{x}_3\hat{x}_2-\xi_1x_1\\
\hat{x}_1\hat{x}_3-\xi_2x_2&x_2x_3-\xi_1\hat{x}_1&\xi_1\xi_2-x_3\hat{x}_3\end{pmatrix}\!, \qquad \xi_i\in\BC,\quad x_i\in\BO^\BC,
\]
where $x_i\mapsto \hat{x}_i$ is the $\BC$-linear extension of the octonionic conjugate, let
\[
\det(x):=\frac{1}{3}\tr\big(xx^\sharp\big)=\frac{1}{3}\tr\big(x^\sharp x\big),
\]
so that $\frac{1}{2}(xx^\sharp+x^\sharp x)=\det(x)I$ holds,
and for $x,y\in\Herm(3,\BO)^\BC$, let $x\times y$ be the \textit{Freudenthal product} given by
\[
x\times y:=(x+y)^\sharp-x^\sharp-y^\sharp,
\]
so that $x\times x=2x^\sharp$ holds.
Now we consider the Jordan triple system structure
\[
Q(x)\colon\ \fp^\mp\longrightarrow \fp^\pm, \qquad Q(x)y:=(x|y)_{\fp^+}x-x^\sharp\times y,
\]
and let $\overline{\cdot}\colon\fp^\pm\to\fp^\mp$ be the complex conjugate with respect to the real form $\Herm(3,\BO)\subset\Herm(3,\BO)^\BC$.
Then we have $(n,r,d,b,p)=(27,3,8,0,18)$. The generic norm $h_{\fp^+}\colon \fp^+\times\fp^-\to\BC$ is given by
\[
h_{\fp^+}(x,y)=1-(x|y)_{\fp^+}+\big(x^\sharp|y^\sharp\big)_{\fp^+}-(\det x)(\det y),
\]
and the bounded symmetric domain $D\subset\fp^+$ is the connected component of $\{x\in\fp^+\mid h(x,\overline{x})\allowbreak>0\}$ which contains the origin.
$\fp^\pm=\Herm(3,\BO)^\BC$ is of tube type, and for $x\in\fp^\pm$, ${}^t\hspace{-1pt}x^\itinv=Q(x)^{-1}x\in\fp^\mp$ is given by
\[
{}^t\hspace{-1pt}x^\itinv=x^{-1}=\frac{1}{\det(x)}x^\sharp.
\]
The Kantor--Koecher--Tits Lie group is given by $G=E_{7(-25)}$, and the maximal compact subgroup is given by $K={\rm U}(1)\times E_6$.

Next we take a primitive tripotent $e'\in\fp^+$, and consider the Peirce decomposition
$\fp^\pm=\fp^\pm(e')_2\oplus\fp^\pm(e')_1\oplus\fp^\pm(e')_0$.
Then since $K$ acts transitively on the set of all primitive tripotents, $\fp^\pm(e')_j$ are isomorphic for all primitive tripotents $e'$,
and we have
\[
\fp^\pm(e')_2\simeq \BC, \qquad \fp^\pm(e')_1\simeq M(1,2;\BO)^\BC, \qquad \fp^\pm(e')_0\simeq \Herm(2,\BO)^\BC\simeq \BC^{10},
\]
where the isomorphism $\Herm(2,\BO)^\BC\simeq \BC^{10}$ is taken such that the quadratic forms
\[
\det \begin{pmatrix}\xi_2&x_1\\\hat{x}_1&\xi_3 \end{pmatrix}=\xi_2\xi_3-x_1\hat{x}_1, \qquad
q(x)=\sum_{j=1}^{10}x^2_j\in\cP\big(\BC^{10}\big)
\]
coincide, and the bilinear forms on $\fp^+(e')_0\times\fp^-(e')_0$,
\[
(x|y)_{\Herm(2,\BO)^\BC}=\Re_\BO\tr(xy), \qquad (x|y)_{\BC^{10}}=2q(x,y)=2\sum_{j=1}^{10}x_jy_j
\]
also coincide. Especially if $e'=E_{11}=\left(\begin{smallmatrix}1&0&0\\0&0&0\\0&0&0\end{smallmatrix}\right)$,
then for $x_{12}\in M(1,2;\BO)^\BC\simeq\fp^+(E_{11})_1$, $y_{11}\in\BC\simeq\fp^-(E_{11})_2$,
$y_{22}=\left(\begin{smallmatrix}\eta_2&y_1\\\hat{y}_1&\eta_3\end{smallmatrix}\right)
\in\Herm(2,\BO)^\BC\simeq\fp^-(E_{11})_0$ we have
\begin{gather*}
Q\left( \begin{pmatrix}0&x_{12}\\{}^t\hspace{-1pt}\hat{x}_{12}&0\end{pmatrix}\right) \begin{pmatrix}y_{11}&0\\0&y_{22}\end{pmatrix}
=\begin{pmatrix} \Re_\BO\big(x_{12}\big(y_{22}{}^t\hspace{-1pt}\hat{x}_{12}\big)\big)&0\\0&y_{11}\big({}^t\hspace{-1pt}\hat{x}_{12}x_{12}\big) \end{pmatrix}\!, \\
{\vphantom{\biggl(}}^t\!\! \begin{pmatrix}y_{11}&0\\0&y_{22}\end{pmatrix}^\itinv =\begin{pmatrix}y_{11}^{-1}&0\\0&y_{22}^{-1}\end{pmatrix}
=\begin{pmatrix} y_{11}^{-1}&0&0\\0&\eta_3/\det(y_{22})&-y_1/\det(y_{22})
\\0&-\hat{y}_1/\det(y_{22})&\eta_2/\det(y_{22})\end{pmatrix}\!.
\end{gather*}
Next we consider the involutions $\sigma_{e'}^\pm$ on $\fp^+$ given by
\begin{gather*}
\sigma_{e'}^+:=\exp\big(\sqrt{-1}\pi D(e',\overline{e'})\big) =I_{\fp^+(e')_2}-I_{\fp^+(e')_1}+I_{\fp^+(e')_0},
\\
\sigma_{e'}^-:=\exp\big(\sqrt{-1}\pi (I_{\fp^+}-D(e',\overline{e'}))\big) =-I_{\fp^+(e')_2}+I_{\fp^+(e')_1}-I_{\fp^+(e')_0},
\end{gather*}
and extend to the involutions on $\fg^\BC=\mathfrak{e}_7^\BC$ such that they act on
$\fk^\BC=\mathfrak{str}(\fp^+)\subset\End_\BC(\fp^+)$ by~$\sigma_{e'}^\pm(l):=\sigma_{e'}^\pm l\sigma_{e'}^\pm $,
and act on $\fp^-$ by transpose. Then we have
\begin{gather*}
\fk^{\sigma_{e'}^\pm}\simeq \mathfrak{u}(1)\oplus\mathfrak{u}(1)\oplus\mathfrak{so}(10), \qquad
\fg^{\sigma_{e'}^+}\simeq \mathfrak{sl}(2,\BR)\oplus\mathfrak{so}(2,10), \qquad
\fg^{\sigma_{e'}^-}\simeq \mathfrak{u}(1)\oplus\mathfrak{e}_{6(-14)}.
\end{gather*}
Next let ${\rm d}\chi_{E_{7(-25)}}$, ${\rm d}\chi_{{\rm SL}(2,\BR)}$, ${\rm d}\chi_{E_{6(-14)}}$, ${\rm d}\chi_{\operatorname{Spin}_0(2,10)}$ be the characters of
$\fk^\BC\simeq \mathfrak{gl}(1,\BC)\oplus\mathfrak{e}_6^\BC$ and $[\fp^+(e')_j,\fp^-(e')_j]\subset\fk^\BC$ respectively
normalized as in (\ref{char}),
and let ${\rm d}\chi_{{\rm U}(1)}$ be a suitable character of the centralizer of $\mathfrak{e}_{6(-14)}\otimes_\BR\BC$ in $\fg^\BC$,
so that the restriction of ${\rm d}\chi_{E_{7(-25)}}$ to $(\fk^{\BC})^{\sigma_{e'}^\pm}$ is given by
\begin{gather*}
{\rm d}\chi_{E_{7(-25)}}\bigr|_{(\fk^{\BC})^{\sigma_{e'}^\pm}}={\rm d}\chi_{{\rm SL}(2,\BR)}+{\rm d}\chi_{\operatorname{Spin}_0(2,10)}={\rm d}\chi_{E_{6(-14)}}+{\rm d}\chi_{{\rm U}(1)}.
\end{gather*}
Then as $((K^\BC)^{\sigma_{e'}^\pm})_0$-modules, $\cP_\bm(\fp^+(e')_j)\subset\cP(\fp^+(e')_j)$ are given by
\begin{gather*}
\cP_m(\BC)\simeq \chi_{{\rm SL}(2,\BR)}^{-2m}=\chi_{E_{6(-14)}}^{-m}\boxtimes\chi_{{\rm U}(1)}^{2m} ,\qquad m\in\BZ_{\ge 0},
\\
\cP_\bm({\rm Herm}(2,\BO)^\BC)\simeq \chi_{\operatorname{Spin}_0(2,10)}^{-|\bm|}\boxtimes V_{(m_1-m_2,0,0,0,0)}^{[10]\vee}
\\ \hphantom{\cP_\bm({\rm Herm}(2,\BO)^\BC)}
{}=\chi_{E_{6(-14)}}^{-|\bm|/2}\boxtimes\chi_{{\rm U}(1)}^{-2|\bm|}\boxtimes V_{(m_1-m_2,0,0,0,0)}^{[10]\vee},\qquad \bm\in\BZ_{++}^2,
\\
\cP_\bm(M(1,2;\BO)^\BC)\simeq \chi_{{\rm SL}(2,\BR)}^{-|\bm|}\boxtimes\chi_{\operatorname{Spin}_0(2,10)}^{-|\bm|/2}
\boxtimes V_{\left(\frac{m_1+m_2}{2},\frac{m_1-m_2}{2},\frac{m_1-m_2}{2},\frac{m_1-m_2}{2}, \frac{m_1-m_2}{2}\right)}^{[10]\vee} \hspace{-50pt}
\\ \hphantom{\cP_\bm(M(1,2;\BO)^\BC)}
{}=\chi_{E_{6(-14)}}^{-\frac{3}{4}|\bm|}
\boxtimes V_{\left(\frac{m_1+m_2}{2},\frac{m_1-m_2}{2},\frac{m_1-m_2}{2},\frac{m_1-m_2}{2}, \frac{m_1-m_2}{2}\right)}^{[10]\vee},\qquad
 \bm\in\BZ_{++}^2.
\end{gather*}

Next, we fix an imaginary unit $\rj\in\BO$ of the Cayley--Dickson extension $\BO=\BH\oplus\BH \rj$,
and as in \cite{Y1, Y2} we consider the map
\begin{gather*}
\tilde{\iota}_\rj\colon\ \BH^3\oplus\Herm(3,\BH)\overset{\sim}{\longrightarrow}\Herm(3,\BO),
\\
\left( (a_1,a_2,a_3),\begin{pmatrix} \xi_1&x_3&\hat{x}_2\\\hat{x}_3&\xi_2&x_1\\x_2&\hat{x}_1&\xi_3 \end{pmatrix}\right)
\mapsto \begin{pmatrix} \xi_1&x_3+a_3\rj&\hat{x}_2-a_2\rj\\\hat{x}_3-a_3\rj&\xi_2&x_1+a_1\rj\\x_2+a_2\rj&\hat{x}_1-a_1\rj&\xi_3 \end{pmatrix}\!.
\end{gather*}
We~fix two isomorphisms
\[
\varphi^\pm\colon\ \BH\overset{\sim}{\longrightarrow}\{a\in M(2,\BC)\mid aJ_2=J_2\overline{a}\}
\]
such that $\varphi^-(a)=\overline{\varphi^+(a)}$ holds,
where $J_2:=\left(\begin{smallmatrix}\hphantom{-}0&1\\-1&0\end{smallmatrix}\right)$, and identify
\begin{gather*}
\BH^3\overset{\sim}{\underset{\varphi^\pm}{\longrightarrow}} \{a\in M(2,6;\BC)\mid aJ_6=J_2\overline{a}\}\subset M(2,6;\BC),
\\
\Herm(3,\BH)\overset{\sim}{\underset{\varphi^\pm}{\longrightarrow}} \{x\in\Herm(6,\BC)\mid xJ_6=J_6\overline{x}\}
\\ \hphantom{\Herm(3,\BH)}
\overset{\sim}{\underset{\cdot J_6}{\longrightarrow}} \{x\in\Skew(6,\BC)\mid xJ_6=J_6\overline{x}\}\subset\Skew(6,\BC),
\end{gather*}
where $J_6:=\left(\begin{smallmatrix}J_2&0&0\\0&J_2&0\\0&0&J_2\end{smallmatrix}\right)$,
and $\cdot J_6$ denotes the right multiplication by $J_6$.
Then by complexifying this, we have the linear isomorphisms
\begin{equation}\label{iotaj}
\iota_\rj^\pm\colon\ M(2,6;\BC)\oplus\Skew(6,\BC)\overset{\sim}{\longrightarrow}\fp^\pm=\Herm(3,\BO)^\BC,
\end{equation}
and the Jordan triple system structure on $M(2,6;\BC)\oplus\Skew(6,\BC)$ is given by
\begin{gather*}
Q((a,x))(b,y)
\\ \qquad
{}=\biggl(\!a{}^t\hspace{-1pt}ba-ax{}^t\hspace{-1pt}y+J_2bx^\#+\frac{1}{2}\tr\big(x{}^t\hspace{-1pt}y\big)a,\,
 x{}^t\hspace{-1pt}yx+\big({}^t\hspace{-1pt}aJ_2a\big)\mathbin{\dot{\times}}{}^t\hspace{-1pt}y
+\tr\big(a{}^t\hspace{-1pt}b\big)x-x{}^t\hspace{-1pt}ab-{}^t\hspace{-1pt}bax\!\biggr),
\end{gather*}
where for $x\in\Skew(6,\BC)$, $x^\#\in\Skew(6,\BC)$ is defined by
\begin{equation}\label{adj6}
(x^\#)_{kl}:=(-1)^{k+l}\Pf\big((x_{ij})_{i,j\in\{1,\dots,6\}\setminus\{k,l\}}\big), \qquad
1\le k<l\le 6,
\end{equation}
so that $xx^\#=x^\#x=\Pf(x)I$ holds, and for $x,y\in\Skew(6,\BC)$, $x\mathbin{\dot{\times}}y\in\Skew(6,\BC)$ is defined by
\[
x\mathbin{\dot{\times}}y:=(x+y)^\#-x^\#-y^\#.
\]
Also, the bilinear form $(\cdot|\cdot)_{\fp^+}\colon\fp^+\times\fp^-\to\BC$ becomes
\[
((a,x)|(b,y))_{\fp^+}=\tr\big(a{}^t\hspace{-1pt}b\big)+\frac{1}{2}\tr\big(x{}^t\hspace{-1pt}y\big).
\]
Let $\sigma_\rj^\pm$ be the involutions on $\fp^+$ given by
\[
\sigma_\rj^+:=I_{M(2,6;\BC)}-I_{\Skew(6,\BC)},\qquad \sigma_\rj^-:=-I_{M(2,6;\BC)}+I_{\Skew(6,\BC)},
\]
and extend to the involutions on $\fg^\BC=\mathfrak{e}_7^\BC$. Then we have
\begin{gather*}
\fk^{\sigma_\rj^\pm}\simeq \mathfrak{s}(\mathfrak{u}(2)\oplus\mathfrak{u}(6))\simeq \mathfrak{su}(2)\oplus\mathfrak{u}(6), \qquad
\fg^{\sigma_\rj^+}\simeq \mathfrak{su}(2,6), \qquad \fg^{\sigma_\rj^-}\simeq \mathfrak{su}(2)\oplus\mathfrak{so}^*(12),
\end{gather*}
where $\big(K^{\sigma_\rj^\pm}\big)_0\simeq {\rm S}({\rm U}(2)\times {\rm U}(6))\simeq {\rm SU}(2)\times {\rm U}(6)$ (up to covering)
acts on $M(2,6;\BC)\allowbreak\oplus\allowbreak\Skew(6,\BC)$ by
\begin{alignat*}{3}
& (k,l)(a,x)=\big(kal^{-1},\det(l)^{-1}lx{}^t\hspace{-1pt}l\big),\qquad&&
(k,l)\in {\rm S}({\rm U}(2)\times {\rm U}(6)),&
\\
& (k,l)(a,x)=\big(\det(l)^{1/2}kal^{-1},lx{}^t\hspace{-1pt}l\big), \qquad &&
(k,l)\in {\rm SU}(2)\times {\rm U}(6).&
\end{alignat*}
The restriction of the character $\chi_{E_{7(-25)}}$ of $K^\BC$ to $((K^\BC)^{\sigma_\rj^\pm})_0$ is given by
\begin{alignat*}{3}
& \chi_{E_{7(-25)}}\bigr|_{((K^\BC)^{\sigma_\rj^+})_0}\simeq V_{\left(\frac{3}{4},\frac{3}{4}\right)}^{(2)}\boxtimes
V_{\left(-\frac{1}{4},-\frac{1}{4},-\frac{1}{4},-\frac{1}{4},-\frac{1}{4},-\frac{1}{4}\right)}^{(6)}\qquad&&
\text{as an }{\rm S}({\rm U}(2)\times {\rm U}(6))\text{-module},&
\\
& \chi_{E_{7(-25)}}\bigr|_{((K^\BC)^{\sigma_\rj^-})_0}\simeq V_{(0,0)}^{(2)}\boxtimes
V_{\left(\frac{1}{2},\frac{1}{2},\frac{1}{2},\frac{1}{2},\frac{1}{2},\frac{1}{2}\right)}^{(6)}\qquad &&
\text{as an }{\rm SU}(2)\times {\rm U}(6)\text{-module},&
\end{alignat*}
$\cP_\bm(M(2,6;\BC))\subset\cP(M(2,6;\BC))$ $(\bm\in\BZ_{++}^2)$ is given by
\begin{gather*}
\cP_\bm(M(2,6;\BC))\simeq V_{(m_1,m_2)}^{(2)\vee}\boxtimes V_{(0,0,0,0,-m_2,-m_1)}^{(6)\vee} \qquad
\text{as an }{\rm S}({\rm U}(2)\times {\rm U}(6))\text{-module},
\\
\cP_\bm(M(2,6;\BC))\simeq V_{\left(\frac{m_1-m_2}{2},\frac{-m_1+m_2}{2}\right)}^{(2)\vee}\boxtimes
V_{\left(\frac{m_1+m_2}{2},\frac{m_1+m_2}{2},\frac{m_1+m_2}{2}, \frac{m_1+m_2}{2},\frac{m_1-m_2}{2},\frac{-m_1+m_2}{2}\right)}^{(6)\vee}
\\ \hphantom{\cP_\bm(M(2,6;\BC))}
{}\simeq V_{(m_1-m_2,0)}^{(2)\vee}\boxtimes
V_{\left(\frac{m_1+m_2}{2},\frac{m_1+m_2}{2},\frac{m_1+m_2}{2},\frac{m_1+m_2}{2}, \frac{m_1-m_2}{2},\frac{-m_1+m_2}{2}\right)}^{(6)\vee}
\\ \hphantom{\cP_\bm(M(2,6;\BC))=}{}
 \hspace*{55mm} \text{as an }{\rm SU}(2)\times {\rm U}(6)\text{-module},
\end{gather*}
and $\cP_\bm({\rm Skew}(6,\BC))\subset\cP({\rm Skew}(6,\BC))$ $(\bm\in\BZ_{++}^3)$ is given by
\begin{gather*}
\cP_\bm({\rm Skew}(6,\BC))\simeq V_{\left(\frac{m_1+m_2+m_3}{2},\frac{m_1+m_2+m_3}{2}\right)}^{(2)\vee}\boxtimes
V_{\substack{\left(\frac{m_1-m_2-m_3}{2},\frac{m_1-m_2-m_3}{2},\frac{-m_1+m_2-m_3}{2},\hspace{13pt}\right.
\\
\left.\hspace{4pt}\frac{-m_1+m_2-m_3}{2},\frac{-m_1-m_2+m_3}{2},\frac{-m_1-m_2+m_3}{2}\right)}}^{(6)\vee}\hspace{-150pt} \\ \hphantom{\cP_\bm({\rm Skew}(6,\BC))}
{}\simeq V_{(0,0)}^{(2)\vee}\boxtimes V_{\substack{(-m_2-m_3,-m_2-m_3,-m_1-m_3, \\ \ -m_1-m_3,-m_1-m_2,-m_1-m_2)}}^{(6)\vee}\qquad
\text{as an }{\rm S}({\rm U}(2)\times {\rm U}(6))\text{-module},
 \\
\cP_\bm({\rm Skew}(6,\BC))\simeq V_{(0,0)}^{(2)\vee}\boxtimes V_{(m_1,m_1,m_2,m_2,m_3,m_3)}^{(6)\vee}\hspace{19mm}
\text{as an }{\rm SU}(2)\times {\rm U}(6)\text{-module}.
\end{gather*}

Next we consider Jordan triple subsystems of $\fp^{\prime\pm}:=M(1,2;\BO)^\BC\simeq \fp^\pm(e')_1$.
Then we have $(n,r,d,b,p)=(16,2,6,4,12)$, and the generic norm on $\fp^{\prime+}$ is given by the restriction of that of
$\fp^+=\Herm(3,\BO)^\BC$. This $\fp^{\prime\pm}$ is not of tube type. We~take a maximal tripotent $e\in\fp^{\prime+}$ and
consider the Peirce decomposition $\fp^{\prime+}=\fp^{\prime+}(e)_2\oplus\fp^{\prime+}(e)_1$. Then we have
\[
\fp^{\prime\pm}(e)_2\simeq\BO^\BC\simeq\BC^8, \qquad
\fp^{\prime\pm}(e)_1\simeq\BO^\BC\simeq\BC^8.
\]
Let $\sigma_e$ be the involution on $\fp^{\prime+}$ given by
\[ \sigma_e=\exp\left(\pi\sqrt{-1}D(e,\overline{e})\right)=I_{\fp^{\prime+}(e)_2}-I_{\fp^{\prime+}(e)_1}, \]
and extend to the involution on $\fg^{\prime\BC}=\mathfrak{e}_6^\BC$. Then we have
\begin{gather*}
(\fk')^{\sigma_e}\simeq \mathfrak{u}(1)\oplus\mathfrak{u}(1)\oplus\mathfrak{so}(8), \qquad
(\fg')^{\sigma_e}\simeq \mathfrak{u}(1)\oplus\mathfrak{so}(2,8).
\end{gather*}
Let ${\rm d}\chi_{E_{6(-14)}}$ and ${\rm d}\chi_{\operatorname{Spin}_0(2,8)}$ be the characters on
$\fk^{\prime\BC}\simeq \mathfrak{gl}(1,\BC)\oplus\mathfrak{so}(10,\BC)$ and
$[\fp^{\prime+}(e)_2,\fp^{\prime-}(e)_2]\allowbreak\simeq\mathfrak{gl}(1,\BC) \oplus\mathfrak{so}(8,\BC)\subset\fk^{\prime\BC}$ respectively
normalized as in (\ref{char}),
and let ${\rm d}\chi_{{\rm U}(1)}$ be a suitable character of the centralizer of $\mathfrak{so}(2,8)\otimes_\BR\BC$ in $\fg^{\prime\BC}$,
so that the restriction of ${\rm d}\chi_{E_{6(-14)}}$ to $(\fk^{\prime\BC})^{\sigma_e}$ is given by
\[
{\rm d}\chi_{E_{6(-14)}}\bigr|_{(\fk^{\prime\BC})^{\sigma_e}}={\rm d}\chi_{\operatorname{Spin}_0(2,8)}+{\rm d}\chi_{{\rm U}(1)}.
\]
Then as $((K^{\prime\BC})^{\sigma_e})_0$-modules, $\cP_\bm(\fp^{\prime+}(e)_j)\subset\cP(\fp^{\prime+}(e)_j)$ $\big(\bm\in\BZ_{++}^2\big)$ are given by
\begin{gather*}
\cP_\bm(\fp^{\prime+}(e)_2)\simeq \chi_{\operatorname{Spin}_0(2,8)}^{-|\bm|}\boxtimes
V_{\left(\frac{m_1+m_2}{2},\frac{m_1-m_2}{2},\frac{m_1-m_2}{2},\frac{m_1-m_2}{2}\right)}^{[8]\vee},
\\
\cP_\bm(\fp^{\prime+}(e)_1)\simeq \chi_{\operatorname{Spin}_0(2,8)}^{-|\bm|/2}\boxtimes\chi_{{\rm U}(1)}^{-\frac{3}{2}|\bm|}\boxtimes
V_{\left(\frac{m_1+m_2}{2},\frac{m_1-m_2}{2},\frac{m_1-m_2}{2},-\frac{m_1-m_2}{2}\right)}^{[8]\vee}.
\end{gather*}

Next we take the primitive tripotent $e'=E_{11}\in\fp^+$, and realize $\fp^{\prime\pm}$ as $\fp^{\prime\pm}:=\fp^\pm(E_{11})_1$.
Then under the identification $\iota_\rj^\pm\colon M(2,6;\BC)\oplus\Skew(6,\BC)\xrightarrow{\sim}\fp^\pm$, we have
\begin{align*}
\fp^{\prime\pm}&\simeq\left\{\begin{pmatrix}0_{2\times 2}&a\end{pmatrix} \mid a\in M(2,4;\BC)\right\}
\oplus\left\{\begin{pmatrix}0_{2\times 2}&{}^t\hspace{-1pt}x\\x&0_{4\times 4}\end{pmatrix}\ \middle|\ x\in M(4,2;\BC)\right\} \\
&\simeq M(2,4;\BC)\oplus M(4,2;\BC).
\end{align*}
The Jordan triple system structure on $M(2,4;\BC)\oplus M(4,2;\BC)$ is given by
\begin{gather*}
Q((a,x))(b,y)
\\ \qquad
{}=\big(a{}^t\hspace{-1pt}ba-J_2b\big(xJ_2{}^t\hspace{-1pt}x\big)^\#+\tr\big(x{}^t\hspace{-1pt}y\big) a-ax{}^t\hspace{-1pt}y,\, x{}^t\hspace{-1pt}yx-\big({}^t\hspace{-1pt}aJ_2a\big)^\# yJ_2+\tr\big(a{}^t\hspace{-1pt}b\big)x-{}^t\hspace{-1pt}bax\big),
\end{gather*}
where for $x\in\Skew(4,\BC)$, $x^\#\in\Skew(4,\BC)$ is defined by
\begin{gather*}
\begin{pmatrix}0&a&b&c\\-a&0&d&e\\-b&-d&0&f\\-c&-e&-f&0\end{pmatrix}^\#
:=\begin{pmatrix}0&-f&e&-d\\f&0&-c&b\\-e&c&0&-a\\d&-b&a&0\end{pmatrix}\!.
\end{gather*}
We~consider the restriction of the involution $\sigma_\rj^+$ on $\fg^\BC=\mathfrak{e}_7^\BC$ to the subalgebra
$\fg^{\prime\BC}=\mathfrak{e}_6^\BC$. Then we have
\begin{gather*}
(\fk')^{\sigma_\rj^+}\simeq \mathfrak{s}(\mathfrak{u}(2)\oplus\mathfrak{u}(4))\oplus\mathfrak{su}(2), \qquad
(\fg')^{\sigma_\rj^+}\simeq \mathfrak{su}(2,4)\oplus\mathfrak{su}(2),
\end{gather*}
where $((K')^{\sigma_\rj^+})_0\simeq {\rm S}({\rm U}(2)\times {\rm U}(4))\times {\rm SU}(2)$ (up to covering) acts on $M(2,4;\BC)\oplus M(4,2;\BC)$ by
\[
(k_1,k_2,k_3)(a,x)=\big(k_1ak_2^{-1},\det(k_2)^{-1}k_2xk_3^{-1}\big).
\]
The restriction of the character $\chi_{E_{6(-14)}}$ of $K^{\prime\BC}$ to $\big((K^{\prime\BC})^{\sigma_\rj^+}\big)_0$ is given by
\[
\chi_{E_{6(-14)}}\bigr|_{\big((K^{\prime\BC})^{\sigma_\rj^+}\big)_0}\simeq
V_{\big(\frac{2}{3},\frac{2}{3}\big)}^{(2)}\boxtimes V_{\big({-}\frac{1}{3},-\frac{1}{3},-\frac{1}{3},-\frac{1}{3}\big)}^{(4)}
\boxtimes V_{(0,0)}^{(2)},
\]
and $\cP_\bm(M(2,4;\BC))$, $\cP_\bm(M(4,2;\BC))$ $\big(\bm\in\BZ_{++}^2\big)$ are given by
\begin{gather*}
\cP_\bm(M(2,4;\BC))\simeq V_{(m_1,m_2)}^{(2)\vee}\boxtimes V_{(0,0,-m_2,-m_1)}^{(4)\vee}\boxtimes V_{(0,0)}^{(2)\vee},
\\
\cP_\bm(M(4,2;\BC))
\\ \qquad
{}\simeq V_{\big(\frac{m_1+m_2}{2},\frac{m_1+m_2}{2}\big)}^{(2)\vee}
\boxtimes V_{\big(\frac{m_1-m_2}{2},\frac{-m_1+m_2}{2},\frac{-m_1-m_2}{2},\frac{-m_1-m_2}{2}\big)}^{(4)\vee}
\boxtimes V_{\big(\frac{m_1-m_2}{2},\frac{-m_1+m_2}{2}\big)}^{(2)\vee}
\\ \qquad
{}\simeq V_{(0,0)}^{(2)\vee}\boxtimes V_{(-m_2,-m_1,-m_1-m_2,-m_1-m_2)}^{(4)\vee}\boxtimes V_{(m_1-m_2,0)}^{(2)\vee}.
\end{gather*}

Next we take the primitive tripotent
$e'=\iota_\rj^+\left(\left(\begin{smallmatrix}0_{1\times 5}&1\\0_{1\times 5}&0\end{smallmatrix}\right)\!,0_{6\times 6}\right)\in\fp^+$,
and realize $\fp^{\prime\pm}$ as $\fp^{\prime\pm}:=\fp^\pm(e')_1$.
Then under the identification $\iota_\rj^\pm\colon M(2,6;\BC)\oplus\Skew(6,\BC)\xrightarrow{\sim}\fp^\pm$, we have
\begin{align*}
\fp^{\prime\pm}&\simeq\left\{\begin{pmatrix}a&0\\0_{1\times 5}&\alpha \end{pmatrix}\ \middle|\ a\in M(1,5;\BC),\; \alpha\in\BC\right\}
\oplus\left\{\begin{pmatrix}x&0_{5\times 1}\\0_{1\times 5}&0\end{pmatrix}\ \middle|\ x\in \Skew(5,\BC)\right\}
\\
&\simeq \BC\oplus\Skew(5,\BC)\oplus M(1,5;\BC).
\end{align*}
The Jordan triple system structure on $\BC\oplus\Skew(5,\BC)\oplus M(1,5;\BC)$ is given by
\begin{gather*}
Q((\alpha,x,a))(\beta,y,b)=\biggl(\alpha^2\beta+\frac{1}{2}\tr\big(x{}^t\hspace{-1pt}y\big)\alpha +\mathbf{Pf}(x){}^t\hspace{-1pt}b,\, x{}^t\hspace{-1pt}yx +\alpha\Proj\bigg(\begin{pmatrix}{}^t\hspace{-1pt}y&-{}^t\hspace{-1pt}a\\a&0\end{pmatrix}^\#\bigg)
\\ \hphantom{Q((\alpha,x,a))(\beta,y,b)=}
{}+\big(\alpha\beta+a{}^t\hspace{-1pt}b\big)x-x{}^t\hspace{-1pt}ab-b{}^t\hspace{-1pt}ax, \, a{}^t\hspace{-1pt}ba+\frac{1}{2}\tr\big(x{}^t\hspace{-1pt}y\big)a-ax{}^t\hspace{-1pt}y+\beta\mathbf{Pf}(x)\biggr),
\end{gather*}
where $\Proj\colon \Skew(6,\BC)\to\Skew(5,\BC)$ is defined by $\left(\begin{smallmatrix}x&-{}^t\hspace{-1pt}a\\a&0\end{smallmatrix}\right)
\mapsto x$, and for $x\in\Skew(5,\BC)$, $\mathbf{Pf}(x)\in M(1,5;\BC)$ is defined by
\begin{equation}\label{boldPf}
(\mathbf{Pf}(x))_k:=(-1)^k\Pf\big((x_{ij})_{i,j\in\{1,\dots,5\}\setminus\{k\}}\big), \qquad
1\le k\le 5.
\end{equation}
We~consider the restriction of the involutions $\sigma_\rj^\pm$ on $\fg^\BC=\mathfrak{e}_7^\BC$ to the subalgebra
$\fg^{\prime\BC}=\mathfrak{e}_6^\BC$. Then we have
\begin{gather*}
(\fk')^{\sigma_\rj^\pm}\simeq \mathfrak{u}(1)\oplus\mathfrak{s}(\mathfrak{u}(1)\oplus\mathfrak{u}(5))
\simeq \mathfrak{u}(1)\oplus\mathfrak{u}(5),
\\
(\fg')^{\sigma_\rj^+}\simeq \mathfrak{sl}(2,\BR)\oplus\mathfrak{su}(1,5), \qquad
(\fg')^{\sigma_\rj^-}\simeq \mathfrak{u}(1)\oplus\mathfrak{so}^*(10),
\end{gather*}
where $((K')^{\sigma_\rj^\pm})_0\simeq {\rm U}(1)\times {\rm S}({\rm U}(1)\times {\rm U}(5))\simeq {\rm U}(1)\times {\rm U}(5)$ (up to covering) acts on
$\BC\oplus\Skew(5,\BC)\allowbreak\oplus M(1,5;\BC)$ by
\begin{gather*}
\big(k,\det(l)^{-1},l\big)(\alpha,x,a)
=\big(k^2\alpha,k\det(l)^{-1}lx{}^t\hspace{-1pt}l,\det(l)^{-1}al^{-1}\big),
\\ \hphantom{\big(k,\det(l)^{-1},l\big)(\alpha,x,a)={}}
\big(k,\det(l)^{-1},l\big)\in {\rm U}(1)\times {\rm S}({\rm U}(1)\times {\rm U}(5)),
\\
(k,l)(\alpha,x,a)=\big(k^{-3}\det(l)^{1/2}\alpha,lx{}^t\hspace{-1pt}l,k^3\det(l)^{1/2}al^{-1}\big), \qquad (k,l)\in {\rm U}(1)\times {\rm U}(5).
\end{gather*}
Let ${\rm d}\chi_{{\rm SL}(2,\BR)}$ be the character on the subalgebra $\mathfrak{gl}(1,\BC)\subset\fk^{\prime\BC}$
corresponding to $\BC\subset\fp^{\prime\pm}$,
and let ${\rm d}\chi_{{\rm U}(1)}$ be a suitable character of the centralizer of $\mathfrak{so}^*(10)\otimes_\BR\BC$ in $\fg^{\prime\BC}$,
so that the restriction of the character $\chi_{E_{6(-14)}}$ of $K^{\prime\BC}$ to $\big((K^{\prime\BC})^{\sigma_\rj^\pm}\big)_0$ is given by
\begin{gather*}
\chi_{E_{6(-14)}}\bigr|_{\big((K^{\prime\BC})^{\sigma_\rj^+}\big)_0}\simeq
\chi_{{\rm SL}(2,\BR)}\boxtimes V_{\frac{5}{6}}^{(1)}\boxtimes V_{\big({-}\frac{1}{6},-\frac{1}{6},-\frac{1}{6},-\frac{1}{6},-\frac{1}{6}\big)}^{(5)}
\\ \hphantom{\chi_{E_{6(-14)}}\bigr|_{((K^{\prime\BC})^{\sigma_\rj^+})_0}\simeq}
\hspace*{43mm} \text{as a }{\rm U}(1)\times {\rm S}({\rm U}(1)\times {\rm U}(5))\text{-module},
\\
\chi_{E_{6(-14)}}\bigr|_{\big((K^{\prime\BC})^{\sigma_\rj^-}\big)_0}\simeq
\chi_{{\rm U}(1)}\boxtimes V_{\left(\frac{1}{2},\frac{1}{2},\frac{1}{2},\frac{1}{2},\frac{1}{2}\right)}^{(5)}
\qquad \text{as a }{\rm U}(1)\times {\rm U}(5)\text{-module},
\end{gather*}
$\cP_m(\BC)$ $(m\in\BZ_{\ge 0})$ is given by
\begin{alignat*}{3}
& \cP_m(\BC)\simeq \chi_{{\rm SL}(2,\BR)}^{-2m}\boxtimes V_{0}^{(1)\vee}\boxtimes V_{(0,0,0,0,0)}^{(5)\vee}\qquad&&
\text{as a }{\rm U}(1)\times {\rm S}({\rm U}(1)\times {\rm U}(5))\text{-module},&
\\
& \cP_m(\BC)\simeq \chi_{{\rm U}(1)}^{3m}\boxtimes V_{\left(\frac{m}{2},\frac{m}{2},\frac{m}{2},\frac{m}{2},\frac{m}{2}\right)}^{(5)\vee}\qquad&&
\text{as a }{\rm U}(1)\times {\rm U}(5)\text{-module},&
\end{alignat*}
$\cP_m(M(1,5;\BC))$ $(m\in\BZ_{\ge 0})$ is given by
\begin{alignat*}{3}
& \cP_m(M(1,5;\BC))\simeq \chi_{{\rm SL}(2,\BR)}^0\boxtimes V_{m}^{(1)\vee}\boxtimes V_{(0,0,0,0,-m)}^{(5)\vee}\qquad&&
\text{as a }{\rm U}(1)\times {\rm S}({\rm U}(1)\times {\rm U}(5))\text{-module},&
\\
& \cP_m(M(1,5;\BC))\simeq \chi_{{\rm U}(1)}^{-3m}\boxtimes V_{\left(\frac{m}{2},\frac{m}{2},\frac{m}{2},\frac{m}{2},-\frac{m}{2}\right)}^{(5)\vee}\qquad&&
\text{as a }{\rm U}(1)\times {\rm U}(5)\text{-module},&
\end{alignat*}
and $\cP_\bm({\rm Skew}(5,\BC))$ $\big(\bm\in\BZ_{++}^2\big)$ is given by
\begin{gather*}
\cP_\bm({\rm Skew}(5,\BC))\simeq \chi_{{\rm SL}(2,\BR)}^{-|\bm|}\boxtimes V_{\frac{m_1+m_2}{2}}^{(1)\vee}
\boxtimes V_{\left(\frac{m_1-m_2}{2},\frac{m_1-m_2}{2},\frac{-m_1+m_2}{2},\frac{-m_1+m_2}{2}, \frac{-m_1-m_2}{2}\right)}^{(5)\vee}
\\ \hphantom{\cP_\bm({\rm Skew}(5,\BC))}
{}\simeq \chi_{{\rm SL}(2,\BR)}^{-|\bm|}\boxtimes V_0^{(1)\vee}
\boxtimes V_{(-m_2,-m_2,-m_1,-m_1,-m_1-m_2)}^{(5)\vee}
\\ \hphantom{\cP_\bm({\rm Skew}(5,\BC)\simeq)}
\hspace*{48mm} \text{as a }{\rm U}(1)\times {\rm S}({\rm U}(1)\times {\rm U}(5))\text{-module},
\\
\cP_\bm({\rm Skew}(5,\BC))\simeq \chi_{{\rm U}(1)}^0\boxtimes V_{(m_1,m_1,m_2,m_2,0)}^{(5)\vee}\qquad
\text{as a }{\rm U}(1)\times {\rm U}(5)\text{-module}.
\end{gather*}

\section{Key identities}\label{section_key}

Let $\fp^+$ be a simple Hermitian positive Jordan triple system.
In this section we assume that $\fp^+$ is of tube type, fix a maximal tripotent $e\in\fp^+$,
and consider the corresponding Jordan algebra structure on $\fp^+$.
Let $\fn^+\subset\fp^+$ be the Euclidean real form, and let $\Omega\subset\fn^+$ be the symmetric cone.
We~set $\rank\fp^+=:r$, $\dim\fp^+=:n=r+\frac{d}{2}r(r-1)$.
For $k\in\BZ_{\ge 0}$ we write $\underline{k}_r:=\smash{(\underbrace{k,\dots,k}_r)}$.
First we want to prove the following.
\begin{Theorem}\label{product}
Let $f_1,f_2\in\cP(\fp^+)$. We~take $k\in\BZ_{\ge 0}$ such that
\[
f_1^{\sharp(k)}(x):=\det_{\fn^+}(x)^kf_1\big(x^\itinv\big)
\]
is a polynomial. Then for $\Re\lambda>\frac{2n}{r}-1$, $z\in\Omega\subset\fn^+\subset\fp^+$, we have
\begin{gather*}
\big\langle f_1(x)f_2(x),{\rm e}^{(x|\overline{z})_{\fp^+}}\big\rangle_{\lambda,x}
\\ \qquad
{}=\frac{1}{(\lambda)_{\underline{k}_r,d}}\det_{\fn^+}(z)^{-\lambda+\frac{n}{r}}f_1^{\sharp(k)} \bigg(\frac{\partial}{\partial z}\bigg) \det_{\fn^+}(z)^{\lambda+k-\frac{n}{r}}
\big\langle f_2(x),{\rm e}^{(x|\overline{z})_{\fp^+}}\big\rangle_{\lambda+k,x}.
\end{gather*}
\end{Theorem}
Here, we normalize $\frac{\partial}{\partial z}$ with respect to the bilinear form
$(\cdot|\cdot)_{\fn^+}=(\cdot|Q(\overline{e})\cdot)_{\fp^+}\colon\fp^+\times\fp^+\to\BC$.
In order to prove the theorem, we recall the \textit{Laplace transform} on the symmetric cone $\Omega\subset\fn^+$.
For a measurable function $F$ on $\Omega$, its Laplace transform is defined by
\[
F(z)\mapsto \hat{F}(w):=\int_\Omega {\rm e}^{-(z|w)_{\fn^+}}F(z)\,{\rm d}z,\qquad w\in\Omega+\sqrt{-1}\fn^+.
\]
If the integral converges, then $\hat{F}(w)$ is a holomorphic function on $\Omega+\sqrt{-1}\fn^+$.
Moreover, for $z,a\in\Omega$,
\[
\frac{1}{(2\pi\sqrt{-1})^n}\int_{a+\sqrt{-1}\fn^+}{\rm e}^{(z|w)_{\fn^+}}\hat{F}(w)\,{\rm d}w=F(z)
\]
holds if the integral converges, and does not depend on the choice of $a\in\Omega$ (see \cite[Section~IX.3]{FK}).
Especially, if $F(z)\in\cP_\bm(\fp^+)\det_{\fn^+}(z)^{\lambda-\frac{n}{r}}$,
then its Laplace transform is known. For $\lambda\in\BC$, $\bm\in\BC^r$ let
\begin{gather}
\Gamma_r^d(\lambda+\bm):=(2\pi)^{dr(r-1)/4}\prod_{j=1}^r\Gamma\bigg(\lambda+m_j-\frac{d}{2}(j-1)\bigg), \notag
\\
\Gamma_r^d(\lambda):=(2\pi)^{dr(r-1)/4}\prod_{j=1}^r\Gamma\bigg(\lambda-\frac{d}{2}(j-1)\bigg). \label{Gamma}
\end{gather}

\begin{Theorem}[Gindikin {\cite[Lemma XI.2.3]{FK}}]\label{Laplace}
Let $\bm\in\BZ_{++}^r$, $f\in\cP_\bm(\fp^+)$.
\begin{enumerate}\itemsep=0pt
\item[$1.$] For $\Re\lambda>\frac{n}{r}-1$, $w\in\Omega+\sqrt{-1}\fn^+$, we have
\[
\int_\Omega {\rm e}^{-(z|w)_{\fn^+}}f(z)\det_{\fn^+}(z)^{\lambda-\frac{n}{r}}{\rm d}z=\Gamma_r^d(\lambda+\bm)f\big(w^\itinv\big)\det_{\fn^+}(w)^{-\lambda}.
\]
\item[$2.$] For $\Re\lambda>\frac{2n}{r}-1$, $z,a\in\Omega$, we have
\[
\frac{\Gamma_r^d(\lambda+\bm)}{(2\pi\sqrt{-1})^n}\int_{a+\sqrt{-1}\fn^+}{\rm e}^{(z|w)_{\fn^+}}f\big(w^\itinv\big)\det_{\fn^+}(w)^{-\lambda}{\rm d}w
=f(z)\det_{\fn^+}(z)^{\lambda-\frac{n}{r}},
\]
and this does not depend on the choice of $a\in\Omega$.
\end{enumerate}
\end{Theorem}

Especially, for $f(x)\in\cP_\bm(\fp^+)$ we have
\[
\det_{\fn^+}(z)^{-\lambda+\frac{n}{r}}\frac{\Gamma_r^d(\lambda)}{(2\pi\sqrt{-1})^n}
\int_{a+\sqrt{-1}\fn^+}{\rm e}^{(z|w)_{\fn^+}}f\big(w^\itinv\big)\det_{\fn^+}(w)^{-\lambda}{\rm d}w=\frac{1}{(\lambda)_{\bm,d}}f(z),
\]
and comparing this with Corollary~\ref{FKCor}
\[
\big\langle f,{\rm e}^{(\cdot|\overline{z})_{\fp^+}}\big\rangle_\lambda =\frac{1}{(\lambda)_{\bm,d}}f(z),
\]
we get the following.

\begin{Corollary}\label{key_formula}
For $\Re\lambda>\frac{2n}{r}-1$, $f\in\cP(\fp^+)$, $z,a\in\Omega$ we have
\[
\big\langle f,{\rm e}^{(\cdot|\overline{z})_{\fp^+}}\big\rangle_\lambda
=\det_{\fn^+}(z)^{-\lambda+\frac{n}{r}}\frac{\Gamma_r^d(\lambda)}{(2\pi\sqrt{-1})^n}
\int_{a+\sqrt{-1}\fn^+}{\rm e}^{(z|w)_{\fn^+}}f\big(w^\itinv\big)\det_{\fn^+}(w)^{-\lambda}{\rm d}w.
\]
\end{Corollary}

\begin{proof}[Proof of Theorem \ref{product}]
Let $f_1,f_2\in\cP(\fp^+)$. Then by using Corollary~\ref{key_formula} twice, for $\Re\lambda>\frac{2n}{r}-1$, $z\in\Omega$ we get
\begin{gather*}
\big\langle f_1(x)f_2(x),{\rm e}^{(x|\overline{z})_{\fp^+}}\big\rangle_{\lambda,x}
\\ \qquad
{}=\det_{\fn^+}(z)^{-\lambda+\frac{n}{r}}\frac{\Gamma_r^d(\lambda)}{(2\pi\sqrt{-1})^n}
\int_{a+\sqrt{-1}\fn^+}{\rm e}^{(z|w)_{\fn^+}}f_1\big(w^\itinv\big)f_2\big(w^\itinv\big)\det_{\fn^+}(w)^{-\lambda}{\rm d}w
\\ \qquad
{}=\det_{\fn^+}(z)^{-\lambda+\frac{n}{r}}\frac{\Gamma_r^d(\lambda)}{(2\pi\sqrt{-1})^n}
\int_{a+\sqrt{-1}\fn^+}{\rm e}^{(z|w)_{\fn^+}}f_1^{\sharp(k)}(w)f_2\big(w^\itinv\big)\det_{\fn^+}(w)^{-\lambda-k}{\rm d}w
\\ \qquad
{}=\det_{\fn^+}(z)^{-\lambda+\frac{n}{r}}f_1^{\sharp(k)}\bigg(\frac{\partial}{\partial z}\bigg)\frac{\Gamma_r^d(\lambda)}{(2\pi\sqrt{-1})^n}
\int_{a+\sqrt{-1}\fn^+}{\rm e}^{(z|w)_{\fn^+}}f_2\big(w^\itinv\big)\det_{\fn^+}(w)^{-\lambda-k}{\rm d}w
\\ \qquad
{}=\det_{\fn^+}(z)^{-\lambda+\frac{n}{r}}f_1^{\sharp(k)}\bigg(\frac{\partial}{\partial z}\bigg)
\frac{\Gamma_r^d(\lambda)}{\Gamma_r^d(\lambda+k)}\det_{\fn^+}(z)^{\lambda+k-\frac{n}{r}}
\big\langle f_2(x),{\rm e}^{(x|\overline{z})_{\fp^+}}\big\rangle_{\lambda+k,x}
\\ \qquad
{}=\frac{1}{(\lambda)_{\underline{k}_r,d}}\det_{\fn^+}(z)^{-\lambda+\frac{n}{r}}f_1^{\sharp(k)} \bigg(\frac{\partial}{\partial z}\bigg)
\det_{\fn^+}(z)^{\lambda+k-\frac{n}{r}}\big\langle f_2(x),{\rm e}^{(x|\overline{z})_{\fp^+}}\big\rangle_{\lambda+k,x}. \qedhere
 \tag*{\qed}
\end{gather*}
\renewcommand{\qed}{}
\end{proof}

Next we consider an involution $\sigma$ on $\fp^+$, and let $\fp^+_1:=(\fp^+)^\sigma$, $\fp^+_2:=(\fp^+)^{-\sigma}$.
We~take $\varepsilon_2 \in\{1,2\}$ such that
\begin{equation}\label{epsilon}
(x_2|\overline{y_2})_{\fp^+}=\varepsilon_2(x_2|\overline{y_2})_{\fp^+_2}, \qquad
x_2,y_2\in\fp^+_2
\end{equation}
holds. In this section we assume that $\fp^+_2$ is also of tube type.
We~take a maximal tripotent $e\in\fp^+_2$. Then $e$ also becomes a maximal tripotent of $\fp^+$.
We~consider the corresponding Jordan algebra structure on $\fp^+$,
let $\fn^+\subset\fp^+$, $\fn^+_2\subset\fp^+_2$ be the Euclidean real forms,
and let $\Omega\subset\fn^+$, $\Omega_2\subset\fn^+_2$ be the symmetric cones. Then we have
\[
\det_{\fn^+}(x_2)=\det_{\fn^+_2}(x_2)^{\varepsilon_2}, \qquad
x_2\in\fp^+_2.
\]
In addition, if $\fp^+_2$ is a direct sum of two Jordan triple subsystems,
then we write $\fp^+_2=:\fp^+_{11}\oplus\fp^+_{22}$, $\fp^+_1=:\fp^+_{12}$, $\rank\fp^+_{11}=:r'$, $\rank\fp^+_{22}=:r''$.
In this case always $\varepsilon_2=1$ holds, and $\fp^+_{11}=\fn^{+\BC}_{11}$ and $\fp^+_{22}=\fn^{+\BC}_{22}$ are also of tube type.
Now we want to compute
\[
\big\langle F(x_2),{\rm e}^{(x|\overline{z})_{\fp^+}}\big\rangle_{\lambda,x}, \qquad F(x_2)\in\cP_\bk\big(\fp^+_2\big)
\]
for $\fp^+=\fp^+_1\oplus\fp^+_2$ case, and
\[
\big\langle f_1(x_{11})f_2(x_{22}),{\rm e}^{(x|\overline{z})_{\fp^+}}\big\rangle_{\lambda,x}, \qquad
f_1(x_{11})\in\cP_\bk(\fp^+_{11}),\quad
f_2(x_{22})\in\cP_\bl\big(\fp^+_{22}\big)
\]
for $\fp^+=\fp^+_{11}\oplus\fp^+_{12}\oplus\fp^+_{22}$ case.
However, it seems difficult to compute these for general case.
From now on we assume $F(x_2)=\det_{\fn^+_2}(x_2)^kf(x_2)$, $f(x_2)\in\cP(\fp^+_2)\cap\cP_\bl(\fp^+)$ for the former case,
and $f_1(x_{11})=\det_{\fn^+_{11}}(x_{11})^k\in\cP_{\underline{k}_{r'}}\big(\fp^+_{11}\big)$ for the latter case.

\begin{Theorem}\label{key_identity}\quad
\begin{enumerate}\itemsep=0pt
\item[$1.$] For $\fp^+=\fp^+_1\oplus\fp^+_2$ case, let $\Re\lambda>\frac{2n}{r}-1$, $k\in\BZ_{\ge 0}$, $\bl\in\BZ_{++}^r$,
and $f\in\cP\big(\fp^+_2\big)\cap\cP_\bl(\fp^+)$. Then for $z=z_1+z_2\in\Omega\subset\fn^+\subset\fp^+$ we have
\begin{gather*}
\big\langle \det_{\fn^+_2}(x_2)^kf(x_2),{\rm e}^{(x|\overline{z})_{\fp^+}}\big\rangle_{\lambda,x}
\\ \qquad
{}=\frac{1}{(\lambda)_{\underline{2k/\varepsilon_2}_r+\bl,d}}\det_{\fn^+}(z)^{-\lambda+\frac{n}{r}}
\det_{\fn^+_2}\bigg(\frac{1}{\varepsilon_2}\frac{\partial}{\partial z_2}\bigg)^k
\det_{\fn^+}(z)^{\lambda+\frac{2k}{\varepsilon_2}-\frac{n}{r}}f(z_2).
\end{gather*}
\item[$2.$] For $\fp^+=\fp^+_{11}\oplus\fp^+_{12}\oplus\fp^+_{22}$ case, let $\Re\lambda>\frac{2n}{r}-1$, $k\in\BZ_{\ge 0}$, $\bl\in\BZ_{++}^{r''}$,
and $f\in\cP_\bl\big(\fp^+_{22}\big)$. Then for $z=z_{11}+z_{12}+z_{22}\in\Omega\subset\fn^+\subset\fp^+$ we have
\begin{gather*}
\big\langle \det_{\fn^+_{11}}(x_{11})^kf(x_{22}),{\rm e}^{(x|\overline{z})_{\fp^+}}\big\rangle_{\lambda,x}
\\ \qquad
{}=\frac{1}{(\lambda)_{\underline{k}_r+\bl,d}}\det_{\fn^+}(z)^{-\lambda+\frac{n}{r}}
\det_{\fn^+_{22}}\bigg(\frac{\partial}{\partial z_{22}}\bigg)^k\det_{\fn^+}(z)^{\lambda+k-\frac{n}{r}}f(z_{22}).
\end{gather*}
\end{enumerate}
\end{Theorem}
Here we normalize $\frac{\partial}{\partial z_2}$, $\frac{\partial}{\partial z_{22}}$ with respect to the bilinear forms
$(\cdot|\cdot)_{\fn^+_2}\colon\fp^+_2\times\fp^+_2\to\BC$, $(\cdot|\cdot)_{\fn^+_{22}}\colon\fp^+_{22}\times\fp^+_{22}\to\BC$,
and identify $\bl\in\BZ_{++}^{r''}$ and $(l_1,\dots,l_{r''},0,\dots,0)\in\BZ_{++}^r$.
This theorem is regarded as an analogue of the Rodrigues formulas for Jacobi polynomials.
For such formulas for symmetry breaking operators in tensor product case, see also \cite{Cl}.
\begin{proof}
(1) Let
\[
\Proj_2\colon \quad \fp^+=\fp^+_1\oplus\fp^+_2\longrightarrow\fp^+_2, \qquad
(z_1,z_2)\longmapsto z_2
\]
be the orthogonal projection, and for $k\in\BZ_{\ge 0}$, $\bl\in\BZ_{++}^r$, $f\in\cP(\fp^+_2)\cap\cP_\bl(\fp^+)$, let
\[
f_1(x):=\det_{\fn^+_2}(x_2)^k=\det_{\fn^+_2}(\Proj_2(x))^k, \qquad
f_2(x):=f(x_2)=f(\Proj_2(x)).
\]
Then by Proposition \ref{prop_Jordansub}(1), for $x=x_1+x_2\in\fp^+=\fp^+_1\oplus\fp^+_2$ we have
\begin{gather*}\begin{split}
& \Proj_2\big(x^\itinv\big)=\big(x_2-P(x_1)x_2^\itinv\big)^\itinv,
\\
& \det_{\fn^+}(x)^2=\det_{\fn^+}\big(x_2-P(x_1)x_2^\itinv\big)\det_{\fn^+}(x_2),
\end{split}
\end{gather*}
and hence
\begin{align*}
f_1^{\sharp(2k/\varepsilon_2)}(x)&=\det_{\fn^+}(x)^{2k/\varepsilon_2}\det_{\fn^+_2}\big(\Proj_2\big(x^\itinv\big)\big)^k \\
&=\det_{\fn^+}(x_2)^{k/\varepsilon_2}\det_{\fn^+}\big(x_2-P(x_1)x_2^\itinv\big)^{k/\varepsilon_2} \det_{\fn^+_2}\big(x_2-P(x_1)x_2^\itinv\big)^{-k}
\\
&=\det_{\fn^+_2}(x_2)^k
\end{align*}
holds. Also, by Corollary~\ref{FKCor} we have
\[
\big\langle f_2(x),{\rm e}^{(x|\overline{z})_{\fp^+}}\big\rangle_{\lambda+2k/\varepsilon_2,x}
=\big\langle f(x_2),{\rm e}^{(x|\overline{z})_{\fp^+}}\big\rangle_{\lambda+2k/\varepsilon_2,x}
=\frac{1}{\big(\lambda+\frac{2k}{\varepsilon_2}\big)_{\bl,d}}f(z_2).
\]
Therefore by Theorem \ref{product} we get
\begin{gather*}
\big\langle \det_{\fn^+_2}(x_2)^kf(x_2),{\rm e}^{(x|\overline{z})_{\fp^+}}\big\rangle_{\lambda,x}
\\ \qquad
{}=\frac{1}{(\lambda)_{\underline{2k/\varepsilon_2}_r,d}}\det_{\fn^+}(z)^{-\lambda+\frac{n}{r}}
f_1^{\sharp(2k/\varepsilon_2)}\bigg(\frac{\partial}{\partial z}\bigg)\det_{\fn^+}(z)^{\lambda+\frac{2k}{\varepsilon_2}-\frac{n}{r}}
\big\langle f_2(x),{\rm e}^{(x|\overline{z})_{\fp^+}}\big\rangle_{\lambda+2k/\varepsilon_2,x}
\\ \qquad
{}=\frac{1}{(\lambda)_{\underline{2k/\varepsilon_2}_r,d}}\det_{\fn^+}(z)^{-\lambda+\frac{n}{r}}
\det_{\fn^+_2}\bigg(\Proj_2\bigg(\frac{\partial}{\partial z}\bigg)\bigg)^k
\det_{\fn^+}(z)^{\lambda+\frac{2k}{\varepsilon_2} -\frac{n}{r}}\frac{1}{\big(\lambda+\frac{2k}{\varepsilon_2}\big)_{\bl,d}}f(z_2)
\\ \qquad
{}=\frac{1}{(\lambda)_{\underline{2k/\varepsilon_2}_r+\bl,d}}\det_{\fn^+}(z)^{-\lambda+\frac{n}{r}}
\det_{\fn^+_2}\bigg(\frac{1}{\varepsilon_2}\frac{\partial}{\partial z_2}\bigg)^k\det_{\fn^+}(z)^{\lambda+\frac{2k}{\varepsilon_2}-\frac{n}{r}}f(z_2).
\end{gather*}

(2) Let
\[
\Proj_{ij}\colon\quad \fp^+=\fp^+_{11}\oplus\fp^+_{12}\oplus\fp^+_{22}\longrightarrow\fp^+_{ij}, \qquad
(z_{11},z_{12},z_{22})\longmapsto z_{ij}
\]
be the orthogonal projection, and for $k\in\BZ_{\ge 0}$, $\bl\in\BZ_{++}^{r''}$, $f\in\cP_\bl(\fp^+_{22})$, let
\[
f_1(x):=\det_{\fn^+_{11}}(x_{11})^k=\det_{\fn^+_{11}}(\Proj_{11}(x))^k, \qquad f_2(x):=f(x_{22})=f(\Proj_{22}(x)).
\]
Then by Proposition \ref{prop_Jordansub}(2), for $x=x_{11}+x_{12}+x_{22}\in\fp^+_{11}\oplus\fp^+_{12}\oplus\fp^+_{22}$ we have
\begin{gather*}
\Proj_{11}\big(x^\itinv\big)=\big(x_{11}-P(x_{12})x_{22}^\itinv\big)^\itinv,
\\
\det_{\fn^+}(x)=\det_{\fn^+_{11}}\big(x_{11}-P(x_{12})x_{22}^\itinv\big)\det_{\fn^+_{22}}(x_{22}),
\end{gather*}
and hence
\begin{align}
f_1^{\sharp(k)}(x)&=\det_{\fn^+}(x)^k\det_{\fn^+_{11}}\big(\Proj_{11}\big(x^\itinv\big)\big)^k \notag
\\
&=\det_{\fn^+_{22}}(x_{22})^k\det_{\fn^+_{11}}\big(x_{11}-P(x_{12})x_{22}^\itinv\big)^k \det_{\fn^+_{11}}\big(x_{11}-P(x_{12})x_{22}^\itinv\big)^{-k} \notag
\\
&=\det_{\fn^+_{22}}(x_{22})^k \label{det_proj}
\end{align}
holds. Also, since $f_2(x)=f(x_{22})\in\cP_\bl(\fp^+_{22})\subset\cP_\bl(\fp^+)$ holds, by Corollary~\ref{FKCor} we have
\[
\big\langle f_2(x),{\rm e}^{(x|\overline{z})_{\fp^+}}\big\rangle_{\lambda+k,x}
=\big\langle f(x_{22}),{\rm e}^{(x|\overline{z})_{\fp^+}}\big\rangle_{\lambda+k,x} =\frac{1}{(\lambda+k)_{\bl,d}}f(z_{22}).
\]
Therefore by Theorem \ref{product} we get
\begin{gather*}
\big\langle \det_{\fn^+_{11}}(x_{11})^kf(x_{22}),{\rm e}^{(x|\overline{z})_{\fp^+}}\big\rangle_{\lambda,x}
\\
{}=\frac{1}{(\lambda)_{\underline{k}_r,d}}\det_{\fn^+}(z)^{-\lambda+\frac{n}{r}}
f_1^{\sharp(k)}\bigg(\frac{\partial}{\partial z}\bigg)\det_{\fn^+}(z)^{\lambda+k-\frac{n}{r}}
\big\langle f_2(x),{\rm e}^{(x|\overline{z})_{\fp^+}}\big\rangle_{\lambda+k,x}
\\
{}=\frac{1}{(\lambda)_{\underline{k}_r,d}}\det_{\fn^+}(z)^{-\lambda+\frac{n}{r}}
\det_{\fn^+_{22}}\bigg(\Proj_{22}\bigg(\frac{\partial}{\partial z}\bigg)\bigg)^k
\det_{\fn^+}(z)^{\lambda+k-\frac{n}{r}}\frac{1}{(\lambda+k)_{\bl,d}}f(z_{22})
\\
{}=\frac{1}{(\lambda)_{\underline{k}_r+\bl,d}}\det_{\fn^+}(z)^{-\lambda+\frac{n}{r}}
\det_{\fn^+_{22}}\bigg(\frac{\partial}{\partial z_{22}}\bigg)^k \det_{\fn^+}(z)^{\lambda+k-\frac{n}{r}}f(z_{22}). \tag*{\qed}
\end{gather*}
\renewcommand{\qed}{}
\end{proof}

\section[Computation for non-simple $\fp^+_2$]
{Computation for non-simple $\boldsymbol{\fp^+_2}$}\label{section_nonsimple}

In this section we treat $\fp^+=\fp^+_1\oplus\fp^+_2$ such that $\fp^+_2$ is non-simple,
and write $\fp^+_2=:\fp^+_{11}\oplus\fp^+_{22}$, $\fp^+_1=:\fp^+_{12}$, so that
\begin{gather*}
\big(\fp^+,\fp^+_{11},\fp^+_{12},\fp^+_{22}\big)
\\ \qquad
{}= \begin{cases}
\big(\BC^{d+2},\BC,\BC^{d},\BC\big) & (\textit{Case }1), \\
\big({\rm Sym}(r,\BC),\Sym(r',\BC),M(r',r'';\BC),\Sym(r'',\BC)\big) & (\textit{Case }2), \\
\big(M(q,s;\BC),M(q',s';\BC),M(q',s'';\BC)\oplus M(q'',s';\BC),M(q'',s'';\BC)\big) & (\textit{Case }3), \\
\big({\rm Skew}(s,\BC),\Skew(s',\BC),M(s',s'';\BC),\Skew(s'',\BC)\big) & (\textit{Case }4), \\
\big({\rm Herm}(3,\BO)^\BC,\BC,M(1,2;\BO)^\BC,\Herm(2,\BO)^\BC\big) & (\textit{Case }5), \\
\big(M(1,2;\BO)^\BC,\BC,\Skew(5,\BC),M(1,5;\BC)\big) & (\textit{Case }6) \end{cases}
\end{gather*}
with $r=r'+r''$, $q=q'+q''$, $s=s'+s''$.
Then the corresponding symmetric pairs are
\[ (G,G_1)= \begin{cases}
\big({\rm SO}_0(2,d+2),{\rm SO}_0(2,d)\times {\rm SO}(2)\big) & (\textit{Case }1), \\
\big({\rm Sp}(r,\BR),{\rm U}(r',r'')\big) & (\textit{Case }2), \\
\big({\rm SU}(q,s),{\rm S}({\rm U}(q',s'')\times {\rm U}(q'',s'))\big) & (\textit{Case }3), \\
\big({\rm SO}^*(2s),{\rm U}(s',s'')\big) & (\textit{Case }4), \\
\big(E_{7(-25)},{\rm U}(1)\times E_{6(-14)}\big) & (\textit{Case }5), \\
\big(E_{6(-14)},{\rm U}(1)\times {\rm SO}^*(10)\big) & (\textit{Case }6) \end{cases} \]
(up to covering).
Let $\dim\fp^+=:n$, $\dim\fp^+_{11}=:n'$, $\dim\fp^+_{22}=:n''$, $\rank\fp^+=:r$, $\rank\fp^+_{11}=:r'$, $\rank\fp^+_{22}=:r''$,
and let $d$, $d'$, $d''$ be the numbers defined in (\ref{str_const}) for $\fp^+$, $\fp^+_{11}$, $\fp^+_{22}$ respectively.
Then the numbers $(r,r',r'',d)$ are given by
\[ (r,r',r'',d)=\begin{cases}
(2,1,1,d) & (\textit{Case }1), \\
(r,r',r'',1) & (\textit{Case }2), \\
(\min\{q,s\},\min\{q',s'\},\min\{q'',s''\},2) & (\textit{Case }3), \\
(\lfloor s/2\rfloor, \lfloor s'/2\rfloor, \lfloor s''/2\rfloor,4) & (\textit{Case }4), \\
(3,1,2,8) & (\textit{Case }5), \\
(2,1,1,6) & (\textit{Case }6), \end{cases} \]
and we have $d=d'=d''$ if $r',r''\ne 1$. Even when $r'$ or $r''=1$, since $d'$, $d''$ are not determined uniquely and any numbers are allowed,
we may assume $d=d'=d''$.

\subsection{Main theorem}

In this section we assume $\fp^+_{11}$ is of tube type, that is, we assume $q'=s'=r'$ for \textit{Case} 3, and $s'=2r'$ for \textit{Case} 4.
We~fix a maximal tripotent $e'\in\fp^+_{11}$, regard $\fp^+_{11}$ as a Jordan algebra,
and let $\fn^+_{11}$ be the corresponding Euclidean real form of $\fp^+_{11}$.

\begin{Theorem}\label{thm_nonsimple}
Assume $\fp^+_{11}$ is of tube type.
Let $\Re\lambda>p-1$, $k\in\BZ_{\ge 0}$, $\bl\in\BZ_{++}^{r''}$, and $f\in\cP_\bl(\fp^+_{22})$.
For $\bm,\bn\in\BZ_{++}^{r''}$, we define the polynomials $\tilde{f}_{\bm,\bn}\in\cP_\bm(\fp^+_{22})\otimes\cP_\bn(\fp^+_{22})$ by
\begin{equation}\label{fmn}
\sum_{\bm,\bn\in\BZ_{++}^{r''}}\tilde{f}_{\bm,\bn}(x_{22},y_{22})=f(x_{22}+y_{22}).
\end{equation}
We~write $z=z_{11}+z_{12}+z_{22}\in\fp^+=\fp^+_{11}\oplus\fp^+_{12}\oplus\fp^+_{22}$.
Then we have
\begin{gather*}
\big\langle \det_{\fn^+_{11}}(x_{11})^kf(x_{22}),{\rm e}^{(x|\overline{z})_{\fp^+}}\big\rangle_{\lambda,x}\\
=\frac{\big(\lambda-\frac{d}{2}r'\big)_{\underline{k}_{r''},d}}{(\lambda)_{\underline{k}_{r}+\bl,d}} \det_{\fn^+_{11}}(z_{11})^k
\sum_{\bm}\frac{(-k)_{\bm,d}}{\big(\lambda-\frac{d}{2}r'\big)_{\bm,d}}
 \sum_{\bn}
\tilde{f}_{\bm,\bn}\big(Q(z_{12}){}^t\hspace{-1pt}z_{11}^\itinv-z_{22},z_{22}\big)
\\
= \frac{\big(\lambda-\frac{d}{2}r'\big)_{\underline{k}_{r''},d}}{(\lambda)_{\underline{k}_{r}+\bl,d}} \det_{\fn^+_{11}}(z_{11})^k
\sum_{\bn}\frac{\big(\lambda+k-\frac{d}{2}r'\big)_{\bn,d}} {\big(\lambda-\frac{d}{2}r'\big)_{\bn,d}}
 \sum_{\bm}
\tilde{f}_{\bm,\bn}\big(Q(z_{12}){}^t\hspace{-1pt}z_{11}^\itinv,z_{22}-Q(z_{12}){}^t\hspace{-1pt} z_{11}^\itinv\big)
\\
= \frac{(\lambda-\frac{d}{2}r')_{\underline{k}_{r''},d}} {(\lambda)_{\underline{k}_{r}+\bl,d}\big(\lambda-\frac{d}{2}r'\big)_{\bl,d}}
\det_{\fn^+_{11}}(z_{11})^k
 \sum_{\bm,\bn}(-k)_{\bm,d}\bigg(\lambda+k-\frac{d}{2}r'\bigg)_{\bn,d}
\tilde{f}_{\bm,\bn}\big(Q(z_{12}){}^t\hspace{-1pt}z_{11}^\itinv,z_{22}\big).
\end{gather*}
\end{Theorem}
Here, since $f(x_{22}+y_{22})$ is a homogeneous polynomial of degree $|\bl|$, $\tilde{f}_{\bm,\bn}$ is non-zero only if $|\bm|+|\bn|=|\bl|$,
and these are finite sums. Moreover, if $r'<r''$, then since the rank of $Q(z_{12}){}^t\hspace{-1pt}z_{11}^\itinv$ is at most $r'$,
$\tilde{f}_{\bm,\bn}\big(Q(z_{12}){}^t\hspace{-1pt}z_{11}^\itinv,*\big)$ is non-zero only if $\bm\in\BZ_{++}^{r'}$.

To prove the theorem, first we assume $\fp^+$, $\fp^+_{22}$ are also of tube type, fix a maximal tripotent
$e=e'+e''\in\fp^+_{11}\oplus\fp^+_{22}\subset\fp^+$, regard $\fp^+$, $\fp^+_{11}$, $\fp^+_{22}$ as Jordan algebras,
and let $\Omega\subset\fn^+\subset\fp^+$, $\Omega_{jj}\subset\fn^+_{jj}\subset\fp^+_{jj}$ be the corresponding Euclidean real forms
and the symmetric cones.
Then by Theorem \ref{key_identity}(2) and Proposition \ref{prop_Jordansub}(2), for $z\in\Omega\subset\fn^+\subset\fp^+$ we have
\begin{gather}
\big\langle \det_{\fn^+_{11}}(x_{11})^kf(x_{22}),{\rm e}^{(x|\overline{z})_{\fp^+}}\big\rangle_{\lambda,x} \notag
\\ \qquad
{}=\frac{1}{(\lambda)_{\underline{k}_r+\bl,d}}\det_{\fn^+}(z)^{-\lambda+\frac{n}{r}}
\det_{\fn^+_{22}}\bigg(\frac{\partial}{\partial z_{22}}\bigg)^k\det_{\fn^+}(z)^{\lambda+k-\frac{n}{r}}f(z_{22}) \notag
\\ \qquad
{}=\frac{1}{(\lambda)_{\underline{k}_r+\bl,d}}
\big(\det_{\fn^+_{11}}(z_{11}) \det_{\fn^+_{22}}(z_{22}-P(z_{12})z_{11}^\itinv)\big)^{-\lambda+\frac{n}{r}} \det_{\fn^+_{22}}\bigg(\frac{\partial}{\partial z_{22}}\bigg)^k\notag
\\ \qquad\hphantom{=}
{} \ \times
\big(\det_{\fn^+_{11}}(z_{11})\det_{\fn^+_{22}} \big(z_{22}-P(z_{12})z_{11}^\itinv\big)\big)^{\lambda+k-\frac{n}{r}}f(z_{22}) \notag
\\ \qquad
{}=\frac{1}{(\lambda)_{\underline{k}_r+\bl,d}}\det_{\fn^+_{11}}(z_{11})^k
\det_{\fn^+_{22}}\big(z_{22}-Q(z_{12}){}^t\hspace{-1pt}z_{11}^\itinv\big)^{-\lambda+\frac{n}{r}} \det_{\fn^+_{22}}\bigg(\frac{\partial}{\partial z_{22}}\bigg)^k\notag
\\ \qquad\hphantom{=}
{} \ \times
\det_{\fn^+_{22}}\big(z_{22}-Q(z_{12}){}^t\hspace{-1pt}z_{11}^\itinv\big)^{\lambda+k-\frac{n}{r}}f(z_{22}). \label{key_nonsimple}
\end{gather}
We~note that since for $z=z_{11}+z_{12}+z_{22}\in\Omega$ we have
\[ \det_{\fn^+}(z)=\det_{\fn^+_{11}}(z_{11})\det_{\fn^+_{22}}\big(z_{22}-Q(z_{12}){}^t\hspace{-1pt}z_{11}^\itinv\big)>0, \]
and since $\Omega$ is connected, $z_{11}\in\Omega_{11}$, $z_{22}-Q(z_{12}){}^t\hspace{-1pt}z_{11}^\itinv\in\Omega_{22}$ hold.

Next we recall the \textit{Riesz integral} from \cite[Theorem VII.2.2]{FK}.
For a closed set $S\subset\fn^+_{22}$ and for $k\in\BZ_{\ge 0}$ let
\[
C^k_S(\fn^+_{22}):=\big\{F\in C^k(\fn^+_{22})\mid \operatorname{supp} F\subset S\big\}.
\]
Let $\Omega_{22}^{cl}\subset\fn^+_{22}$ be the closure of the symmetric cone $\Omega_{22}$.
When $S\cap \Omega^{cl}_{22}$ is compact, for $F\in C^0_S(\fn^+_{22})$ the Riesz integral $R_\lambda$ is defined by
\[
R_\lambda(F):=\frac{1}{\Gamma_{r''}^d(\lambda)}\int_{\Omega_{22}} F(w_{22})\det_{\fn^+_{22}}(w_{22})^{\lambda-\frac{n''}{r''}}{\rm d}w_{22},\qquad \Re\lambda>\frac{n''}{r''}-1,
\]
where $\Gamma_{r''}^d(\lambda)$ is as (\ref{Gamma}). Then if $F\in C^{lr''}_S(\fn^+_{22})$, $k\le l$, we have
\[
R_\lambda(F)=R_{\lambda+k}\bigg(\det_{\fn^+_{22}}\bigg({-}\frac{\partial}{\partial w_{22}}\bigg)^kF\bigg),
\]
and this is analytically continued for $\Re\lambda>\frac{n''}{r''}-l-1$.
Moreover, since $R_0=\delta_0$, for $k\in\BZ_{\ge 0}$ with $k<-\frac{n''}{r''}+l+1$ we have
\[
R_{-k}(F)=\det_{\fn^+_{22}}\bigg({-}\frac{\partial}{\partial w_{22}}\bigg)^kF(w_{22})\biggr|_{w_{22}=0}
=\det_{\fn^+_{22}}\bigg(\frac{\partial}{\partial w_{22}}\bigg)^kF(-w_{22})\biggr|_{w_{22}=0}.
\]
In the following we set $Q(z_{12}){}^t\hspace{-1pt}z_{11}^\itinv=:a_{22}$.
Then for $z_{22},a_{22}\in\fn^+_{22}$ with $z_{22}-a_{22}\in\Omega_{22}$, by putting
\[ F(w_{22}):=\det_{\fn^+_{22}}(z_{22}-a_{22}-w_{22})^{\lambda+k-\frac{d}{2}r'-\frac{n''}{r''}}f(z_{22}-w_{22})
1_{z_{22}-a_{22}-\Omega_{22}}(w_{22})
\]
in the above formula,
\begin{gather*}
\det_{\fn^+_{22}}(z_{22}-a_{22})^{-\lambda+\frac{n}{r}}\det_{\fn^+_{22}}\bigg(\frac{\partial}{\partial z_{22}}\bigg)^k
\det_{\fn^+_{22}}(z_{22}-a_{22})^{\lambda+k-\frac{n}{r}}f(z_{22})
\\
{}=\det_{\fn^+_{22}}(z_{22} - a_{22})^{-\lambda+\frac{n}{r}}\det_{\fn^+_{22}}
\bigg(\frac{\partial}{\partial w_{22}}\bigg)^k
\det_{\fn^+_{22}}(z_{22} + w_{22} - a_{22})^{\lambda + k - \frac{n}{r}}f(z_{22} + w_{22})\biggr|_{w_{22}=0}
\end{gather*}
is given by the analytic continuation of
\begin{gather*}
\frac{\det_{\fn^+_{22}}(z_{22}-a_{22})^{-\lambda+\frac{n}{r}}}{\Gamma_{r''}^d(-k)}
\int_{\Omega_{22}\cap(z_{22}-a_{22}-\Omega_{22})}\det_{\fn^+_{22}}(w_{22})^{-k-\frac{n''}{r''}}
\\ \qquad \hphantom{=}
{} \ \times\det_{\fn^+_{22}}(z_{22}-a_{22}-w_{22})^{\lambda+k-\frac{n}{r}}f(z_{22}-w_{22})\,{\rm d}w_{22} \\ \qquad
{}=\frac{\det_{\fn^+_{22}}(z_{22}-a_{22})^{-\lambda+\frac{d}{2}r'+\frac{n''}{r''}}}{\Gamma_{r''}^d(-k)}
\int_{\Omega_{22}\cap(z_{22}-a_{22}-\Omega_{22})}\det_{\fn^+_{22}}(w_{22})^{-k-\frac{n''}{r''}}
\\ \qquad \hphantom{=}
{} \ \times\det_{\fn^+_{22}}(z_{22}-a_{22}-w_{22})^{\lambda+k-\frac{d}{2}r'-\frac{n''}{r''}} f(z_{22}-w_{22})\,{\rm d}w_{22},
\end{gather*}
which originally converges when $-\Re k>\frac{n''}{r''}-1$, $\Re\left(\lambda+k-\frac{d}{2}r'\right)>\frac{n''}{r''}-1$.
Now for a while we omit the subscript 22, and prove the following.

\begin{Proposition}\label{Euler_int}
Let $\Re\mu>\frac{n}{r}-1$, $\Re\nu>\frac{n}{r}-1$, $\bl\in\BZ_{++}^r$ and $f\in\cP_\bl(\fp^+)$.
We~define $\tilde{f}_{\bm,\bn}\in\cP_\bm(\fp^+)\otimes\cP_\bn(\fp^+)$ as \eqref{fmn}.
Then for $z,a\in\fn^+$ with $z-a\in\Omega$, we have
\begin{gather*}
\det_{\fn^+}(z-a)^{-\mu-\nu+\frac{n}{r}}\int_{(a+\Omega)\cap(z-\Omega)}\det_{\fn^+}(w-a)^{\mu-\frac{n}{r}}
\det_{\fn^+}(z-w)^{\nu-\frac{n}{r}}f(z+a-w)\,{\rm d}w
\\ \qquad
{}=\det_{\fn^+}(z-a)^{-\mu-\nu+\frac{n}{r}}\int_{\Omega\cap(z-a-\Omega)}
\det_{\fn^+}(w)^{\mu-\frac{n}{r}}\det_{\fn^+}(z-a-w)^{\nu-\frac{n}{r}}f(z-w)\,{\rm d}w
\\ \qquad
{}=\frac{\Gamma_r^d(\mu)\Gamma_r^d(\nu)}{\Gamma_r^d(\mu+\nu)}\sum_{\bm}
\frac{(\mu)_{\bm,d}}{(\mu+\nu)_{\bm,d}}\sum_{\bn}\tilde{f}_{\bm,\bn}(a-z,z)
\\ \qquad
{}=\frac{\Gamma_r^d(\mu)\Gamma_r^d(\nu)}{\Gamma_r^d(\mu+\nu)}\sum_{\bn}
\frac{(\nu)_{\bn,d}}{(\mu+\nu)_{\bn,d}}\sum_{\bm} \tilde{f}_{\bm,\bn}(a,z-a)
\\ \qquad
{}=\frac{\Gamma_r^d(\mu)\Gamma_r^d(\nu)}{\Gamma_r^d(\mu+\nu+\bl)}\sum_{\bm,\bn}(\mu)_{\bm,d}(\nu)_{\bn,d} \tilde{f}_{\bm,\bn}(a,z).
\end{gather*}
\end{Proposition}

Then Theorem \ref{thm_nonsimple} for tube type $\fp^+_{22}$ is proved from Proposition \ref{Euler_int}
by dividing each side by~$\Gamma_{r''}^d(\mu)$ and putting $\mu=-k$, $\nu=\lambda+k-\frac{d}{2}r'$, so that
\begin{gather*}
\frac{\Gamma_{r''}^d(\nu)}{\Gamma_{r''}^d(\mu+\nu)}
=\frac{\Gamma_{r''}^d\big(\lambda+k-\frac{d}{2}r'\big)}{\Gamma_{r''}^d\big(\lambda-\frac{d}{2}r'\big)}
=\bigg(\lambda-\frac{d}{2}r'\bigg)_{\underline{k}_{r''},d},
\\
\frac{\Gamma_{r''}^d(\nu)}{\Gamma_{r''}^d(\mu+\nu+\bl)}
=\frac{\Gamma_{r''}^d\big(\lambda+k-\frac{d}{2}r'\big)}{\Gamma_{r''}^d\big(\lambda-\frac{d}{2}r'\big)}
\frac{\Gamma_{r''}^d\big(\lambda-\frac{d}{2}r'\big)}{\Gamma_{r''}^d\big(\lambda-\frac{d}{2}r'+\bl\big)}
=\frac{\big(\lambda-\frac{d}{2}r'\big)_{\underline{k}_{r''},d}}{\big(\lambda-\frac{d}{2}r'\big)_{\bl,d}},
\end{gather*}
and since each side of Theorem \ref{thm_nonsimple} is a polynomial in $z$, the equalities hold not only for $z\in\Omega$
but also for all $z\in\fp^+$ by analytic continuation.

For the proof of Proposition \ref{Euler_int}, we prepare some notations.
For $\bm,\bn\in\BZ_{++}^r$, we decompose $\cP_\bm(\fp^+)\otimes\cP_\bn(\fp^+)\subset\cP(\fp^+\oplus\fp^+)$ as
\[ \cP_\bm(\fp^+)\otimes\cP_\bn(\fp^+)=\bigoplus_{\bl\in\frac{1}{2}\BZ_{++}^r}\cP_{\bm,\bn}^\bl(\fp^+\oplus\fp^+), \]
where $\cP_{\bm,\bn}^\bl(\fp^+\oplus\fp^+)$ is the sum of $L$-submodules in $\cP_\bm(\fp^+)\otimes\cP_\bn(\fp^+)$
with the restricted lowest weight $-l_1\gamma_1-\cdots-l_r\gamma_r$ under the diagonal action of $L$,
where $L$ is the real form of $K^\BC$ given in Section~\ref{section_KKT},
and $\{\gamma_j\}\subset\fa_\fl^\vee$ is as in Section~\ref{section_root}.
Then we have the following formula, which is a generalization of the \textit{beta integral} \cite[Theorem~VII.1.7]{FK}.

\begin{Proposition}%\label{beta_int}
Let $\Re\mu>\frac{n}{r}-1$, $\Re\nu>\frac{n}{r}-1$, $\bm,\bn\in\BZ_{++}^r$, $\bl\in\frac{1}{2}\BZ_{++}^r$.
Then for any $\tilde{f}(x,y)\in \cP_{\bm,\bn}^\bl(\fp^+\oplus\fp^+)$, $z\in\Omega$, we have
\begin{gather*}
\det_{\fn^+}(z)^{-\mu-\nu+\frac{n}{r}}\int_{\Omega\cap(z-\Omega)}\tilde{f}(w,z-w)
\det_{\fn^+}(w)^{\mu-\frac{n}{r}}\det_{\fn^+}(z-w)^{\nu-\frac{n}{r}}{\rm d}w
\\ \qquad
{}=\frac{\Gamma_r^d(\mu+\bm)\Gamma_r^d(\nu+\bn)}{\Gamma_r^d(\mu+\nu+\bl)}\tilde{f}(z,z).
\end{gather*}
\end{Proposition}
Especially, if $\bl\notin\BZ_{++}^r$, then $\tilde{f}(z,z)=0\in\cP(\fp^+)$ and the above integral vanishes.
\begin{proof}
For $\tilde{f}(x,y)\in\cP(\fp^+)\otimes\cP(\fp^+)$ we define a function $\tilde{f}^\wedge(z)\in C(\Omega)$ by
\[
\tilde{f}^\wedge(z):=\det_{\fn^+}(z)^{-\mu-\nu+\frac{n}{r}}
\int_{\Omega\cap(z-\Omega)}\tilde{f}(w,z-w)\det_{\fn^+}(w)^{\mu-\frac{n}{r}} \det_{\fn^+}(z-w)^{\nu-\frac{n}{r}}{\rm d}w.
\]
Then for $l\in L$ we have
\begin{align*}
(\tilde{f}(l\cdot,l\cdot))^\wedge(z)
={}&\det_{\fn^+}(z)^{-\mu-\nu+\frac{n}{r}}\int_{\Omega\cap(z-\Omega)}\!\!\!\!\!\!
\tilde{f}(lw,l(z\!-w))\det_{\fn^+}(w)^{\mu\!-\frac{n}{r}}\det_{\fn^+}(z\!-w)^{\nu-\frac{n}{r}}{\rm d}w \\
={}&\det_{\fn^+}(z)^{-\mu-\nu+\frac{n}{r}}\int_{\Omega\cap(lz-\Omega)}
\tilde{f}(w,lz-w)\det_{\fn^+}\big(l^{-1}w\big)^{\mu-\frac{n}{r}}
\\
&\times\det_{\fn^+}\big(z-l^{-1}w\big)^{\nu-\frac{n}{r}}\det_{\fn^+}\big(l^{-1}e\big)^{\frac{n}{r}}{\rm d}w
\\
={}&\det_{\fn^+}(lz)^{-\mu-\nu+\frac{n}{r}}\int_{\Omega\cap(lz-\Omega)}\!\!\!\!\!
\tilde{f}(w,lz-w)\det_{\fn^+}(w)^{\mu-\frac{n}{r}}\det_{\fn^+}(lz-w)^{\nu-\frac{n}{r}}{\rm d}w
\\
={}&\tilde{f}^\wedge(lz).
\end{align*}
Therefore $\tilde{f}(x,y)\mapsto \tilde{f}^\wedge(z)$ is $L$-equivariant.
Now we compute $\tilde{f}^\wedge(z)$ when $\tilde{f}(x,y)\in \cP_{\bm,\bn}^\bl(\fp^+\allowbreak\oplus\fp^+)$ is a lowest weight vector.
Then $\tilde{f}(x,y)$ satisfies $\tilde{f}(lx,ly)=\Delta_\bl(le)\tilde{f}(x,y)$ for $l\in A_LN_L^\top$.
In~this case, when $z=le$ with $l\in A_LN_L^\top$, we have
\begin{align}\label{tildeKz}
\tilde{f}^\wedge(z)=\tilde{f}^\wedge(le)=\big(\tilde{f}(l\cdot,l\cdot)\big)^\wedge(e)=\big(\Delta_\bl(le)\tilde{f}(\cdot,\cdot)\big)^\wedge(e)
=\Delta_\bl(z)\tilde{f}^\wedge(e),
\end{align}
that is,
\[ \tilde{f}^\wedge(e)\Delta_\bl(z)\det_{\fn^+}(z)^{\mu+\nu-\frac{n}{r}}
=\int_{\Omega\cap(z-\Omega)}\tilde{f}(w,z-w)\det_{\fn^+}(w)^{\mu-\frac{n}{r}}\det_{\fn^+}(z-w)^{\nu-\frac{n}{r}}{\rm d}w \]
holds. Then by multiplying ${\rm e}^{-\tr_{\fn^+}(z)}$ and integrating on $\Omega$ with respect to $z$,
the left hand side is computed by using Theorem \ref{Laplace}(1) as
\[ \int_\Omega {\rm e}^{-\tr_{\fn^+}(z)}\tilde{f}^\wedge(e)\Delta_\bl(z)\det_{\fn^+}(z)^{\mu+\nu-\frac{n}{r}}{\rm d}z
=\Gamma_r^d(\mu+\nu+\bl)\tilde{f}^\wedge(e), \]
and the right hand side is computed by using the variable change $z=x+y$, $w=x$ and again by Theorem \ref{Laplace}(1) as
\begin{gather*}
\iint_{\substack{z\in\Omega, w\in\Omega\cap(z-\Omega)}}{\rm e}^{-\tr_{\fn^+}(z)}\tilde{f}(w,z-w)\det_{\fn^+}(w)^{\mu-\frac{n}{r}}
\det_{\fn^+}(z-w)^{\nu-\frac{n}{r}}{\rm d}w{\rm d}z
\\ \qquad
{}=\iint_{\substack{x\in\Omega, y\in\Omega}}{\rm e}^{-\tr_{\fn^+}(x+y)}\tilde{f}(x,y)\det_{\fn^+}(x)^{\mu-\frac{n}{r}}
\det_{\fn^+}(y)^{\nu-\frac{n}{r}}{\rm d}x{\rm d}y
\\ \qquad
{}=\Gamma_r^d(\mu+\bm)\Gamma_r^d(\nu+\bn)\tilde{f}(e,e).
\end{gather*}
Therefore we have
\[ \tilde{f}^\wedge(e)=\frac{\Gamma_r^d(\mu+\bm)\Gamma_r^d(\nu+\bn)}{\Gamma_r^d(\mu+\nu+\bl)}\tilde{f}(e,e), \]
and by (\ref{tildeKz}), we get
\begin{align*}
\tilde{f}^\wedge(z)=\Delta_\bl(z)\tilde{f}^\wedge(e)
=\Delta_\bl(z)\frac{\Gamma_r^d(\mu+\bm)\Gamma_r^d(\nu+\bn)}{\Gamma_r^d(\mu+\nu+\bl)}\tilde{f}(e,e)
=\frac{\Gamma_r^d(\mu+\bm)\Gamma_r^d(\nu+\bn)}{\Gamma_r^d(\mu+\nu+\bl)}\tilde{f}(z,z).
\end{align*}
By the $L$-equivariance, this holds for all $\tilde{f}(x,y)\in \cP_{\bm,\bn}^\bl(\fp^+\oplus\fp^+)$.
\end{proof}

\begin{proof}[Proof of Proposition \ref{Euler_int}]
The 1st equality is clear. Next we prove the 2nd equality. Since as a function of $w$, $\tilde{f}_{\bm,\bn}(-w,z)\in
\cP_\bm(\fp^+)=\cP_{\bm,\underline{0}_r}^\bm(\fp^+\oplus\fp^+)$ holds, we have
\begin{gather*}
\det_{\fn^+}(z-a)^{-\mu-\nu+\frac{n}{r}}\int_{\Omega\cap(z-a-\Omega)}
\det_{\fn^+}(w)^{\mu-\frac{n}{r}}\det_{\fn^+}(z-a-w)^{\nu-\frac{n}{r}}f(z-w)\,{\rm d}w
\\
{}=\det_{\fn^+}(z-a)^{-\mu-\nu+\frac{n}{r}}\sum_{\bm,\bn}\int_{\Omega\cap(z-a-\Omega)}
\det_{\fn^+}(w)^{\mu-\frac{n}{r}}\det_{\fn^+}(z-a-w)^{\nu-\frac{n}{r}}
\tilde{f}_{\bm,\bn}(-w,z)\,{\rm d}w\\
{} =\sum_{\bm,\bn}\frac{\Gamma_r^d(\mu+\bm)\Gamma_r^d(\nu)} {\Gamma_r^d(\mu+\nu+\bm)}\tilde{f}_{\bm,\bn}(-(z-a),z)
=\frac{\Gamma_r^d(\mu)\Gamma_r^d(\nu)}{\Gamma_r^d(\mu+\nu)}\sum_{\bm}\frac{(\mu)_{\bm,d}}{(\mu+\nu)_{\bm,d}}
\sum_{\bn}\tilde{f}_{\bm,\bn}(a-z,z).
\end{gather*}
The 3rd equality is also proved similarly
\begin{gather*}
\det_{\fn^+}(z-a)^{-\mu-\nu+\frac{n}{r}}\int_{\Omega\cap(z-a-\Omega)}
\det_{\fn^+}(w)^{\mu-\frac{n}{r}}\det_{\fn^+}(z-a-w)^{\nu-\frac{n}{r}}f(z-w)\,{\rm d}w
\\
{}=\det_{\fn^+}(z-a)^{-\mu-\nu+\frac{n}{r}}
\int_{\Omega\cap(z-a-\Omega)} \det_{\fn^+}(w)^{\mu-\frac{n}{r}} \det_{\fn^+}(z-a-w)^{\nu-\frac{n}{r}}
 f(a+z-a-w)\,{\rm d}w \\
{}=\det_{\fn^+}(z-a)^{-\mu-\nu+\frac{n}{r}}
\sum_{\bm,\bn}\int_{\Omega\cap(z-a-\Omega)}\hspace{-9mm}\det_{\fn^+}(w)^{\mu-\frac{n}{r}} \det_{\fn^+}(z-a-w)^{\nu-\frac{n}{r}}\tilde{f}_{\bm,\bn}(a,z-a-w)\,{\rm d}w
\\
{}=\sum_{\bm,\bn}\frac{\Gamma_r^d(\mu)\Gamma_r^d(\nu+\bn)}{\Gamma_r^d(\mu+\nu+\bn)}\tilde{f}_{\bm,\bn}(a,z-a)
\\
{}=\frac{\Gamma_r^d(\mu)\Gamma_r^d(\nu)}{\Gamma_r^d(\mu+\nu)}\sum_{\bn}\frac{(\nu)_{\bn,d}}{(\mu+\nu)_{\bn,d}}
\sum_{\bm}\tilde{f}_{\bm,\bn}(a,z-a).
\end{gather*}
Finally we prove the 4th equality. Since the map $\cP_\bl(\fp^+)\ni f\mapsto \tilde{f}_{\bm,\bn}\in
\cP_\bm(\fp^+)\otimes\cP_\bn(\fp^+)$ is $L\simeq \Delta L$-equivariant,
we may assume $f=\tilde{\Phi}_\bl^{\fn^+}\in\cP(\fp^+)^{K_L}$ (see (\ref{Phi^n})).
We~sum up the left hand side with respect to $\bl\in\BZ_{++}^r$. Then by (\ref{expand_exptr}) we have
\begin{gather*}
\sum_{\bl\in\BZ_{++}^r}\det_{\fn^+}(z-a)^{-\mu-\nu+\frac{n}{r}}
\int_{\Omega\cap(z-a-\Omega)}\det_{\fn^+}(w)^{\mu-\frac{n}{r}}\det_{\fn^+}(z-a-w)^{\nu-\frac{n}{r}} \tilde{\Phi}^{\fn^+}_\bl(z-w)\,{\rm d}w
\\ \qquad
{}=\det_{\fn^+}(z-a)^{-\mu-\nu+\frac{n}{r}}
\int_{\Omega\cap(z-a-\Omega)}\det_{\fn^+}(w)^{\mu-\frac{n}{r}}\det_{\fn^+}(z-a-w)^{\nu-\frac{n}{r}}{\rm e}^{\tr_{\fn^+}(z-w)}{\rm d}w
\\ \qquad
{}=\det_{\fn^+}(z-a)^{-\mu-\nu+\frac{n}{r}}\int_{\Omega\cap(e-\Omega)}\det_{\fn^+} \big(P\big((z-a)^{\mathit{1/2}}\big)w\big)^{\mu-\frac{n}{r}}
\\ \qquad\hphantom{=}
{} \ \times \det_{\fn^+}\big(z-a-P\big((z-a)^{\mathit{1/2}}\big)w\big)^{\nu-\frac{n}{r}}
{\rm e}^{\tr_{\fn^+}(z-P((z-a)^{\mathit{1/2}})w)}\det_{\fn^+}(z-a)^{\frac{n}{r}}{\rm d}w
\\ \qquad
{}=\int_{\Omega\cap(e-\Omega)}\det_{\fn^+}(w)^{\mu-\frac{n}{r}}\det_{\fn^+}(e-w)^{\nu-\frac{n}{r}}{\rm e}^{(z|e)_{\fn^+}-(z-a|w)_{\fn^+}}{\rm d}w
\\ \qquad
{}=\int_{\Omega\cap(e-\Omega)}\det_{\fn^+}(w)^{\mu-\frac{n}{r}}\det_{\fn^+}(e-w)^{\nu-\frac{n}{r}}{\rm e}^{(a|w)_{\fn^+}+(z|e-w)_{\fn^+}}{\rm d}w.
\end{gather*}
For $\bm,\bn\in\BZ_{++}^r$, $\bl\in\frac{1}{2}\BZ_{++}^r$,
we define $K_{\bm,\bn}^\bl(a,b;c,d)\in
\cP_{\bm,\bn}^\bl(\fp^+\oplus\fp^+)_{a,b}\otimes\cP_{\bm,\bn}^\bl(\fp^+\oplus\fp^+)_{c,d}$ by
\[
{\rm e}^{(a|c)_{\fn^+}+(b|d)_{\fn^+}}=\sum_{\bm\in\BZ_{++}^r}\sum_{\bn\in\BZ_{++}^r} \sum_{\bl\in\frac{1}{2}\BZ_{++}^r}K_{\bm,\bn}^\bl(a,b;c,d).
\]
Then we have
\begin{gather*}
\int_{\Omega\cap(e-\Omega)}\det_{\fn^+}(w)^{\mu-\frac{n}{r}}\det_{\fn^+}(e-w)^{\nu-\frac{n}{r}}{\rm e}^{(a|w)_{\fn^+}+(z|e-w)_{\fn^+}}{\rm d}w
\\ \qquad
{}=\sum_{\bm\in\BZ_{++}^r}\sum_{\bn\in\BZ_{++}^r}\sum_{\bl\in\frac{1}{2}\BZ_{++}^r}
\int_{\Omega\cap(e-\Omega)}\!\!\!\!\!\det_{\fn^+}(w)^{\mu-\frac{n}{r}} \det_{\fn^+}(e-w)^{\nu-\frac{n}{r}}K_{\bm,\bn}^\bl (a,z;w,e-w)\,{\rm d}w
\\ \qquad
{}=\sum_{\bm\in\BZ_{++}^r}\sum_{\bn\in\BZ_{++}^r}\sum_{\bl\in\BZ_{++}^r}
\frac{\Gamma_r^d(\mu+\bm)\Gamma_r^d(\nu+\bn)}{\Gamma_r^d(\mu+\nu+\bl)}K_{\bm,\bn}^\bl (a,z;e,e)
\\ \qquad
{}=\sum_{\bl\in\BZ_{++}^r}\frac{\Gamma_r^d(\mu)\Gamma_r^d(\nu)}{\Gamma_r^d(\mu+\nu+\bl)} \sum_{\bm\in\BZ_{++}^r}\sum_{\bn\in\BZ_{++}^r}
(\mu)_{\bm,d}(\nu)_{\bn,d} K_{\bm,\bn}^\bl (a,z;e,e).
\end{gather*}
Now since $K_{\bm,\bn}^\bl (a,z;e,e)=(\tilde{\Phi}_\bl^{\fn^+})^\sim_{\bm,\bn}(a,z)$ holds, we get
\begin{gather*}
\sum_{\bl\in\BZ_{++}^r}\det_{\fn^+}(z-a)^{-\mu-\nu+\frac{n}{r}}
\int_{\Omega\cap(z-a-\Omega)}\det_{\fn^+}(w)^{\mu-\frac{n}{r}} \det_{\fn^+}(z-a-w)^{\nu-\frac{n}{r}}\tilde{\Phi}_\bl^{\fn^+}(z-w)\,{\rm d}w
\\ \qquad
{}=\sum_{\bl\in\BZ_{++}^r}\frac{\Gamma_r^d(\mu)\Gamma_r^d(\nu)}{\Gamma_r^d(\mu+\nu+\bl)} \sum_{\bm\in\BZ_{++}^r}\sum_{\bn\in\BZ_{++}^r}
(\mu)_{\bm,d}(\nu)_{\bn,d} \big(\tilde{\Phi}_\bl^{\fn^+}\big)^\sim_{\bm,\bn}(a,z).
\end{gather*}
Then by projecting both sides to $\bigoplus_{\bm,\bn}\cP_{\bm,\bn}^\bl(\fp^+\oplus\fp^+)$, we get the 4th equality
for $f=\tilde{\Phi}_\bl^{\fn^+}$, and by the $L$-equivariace, this holds for all $f\in\cP_\bl(\fp^+)$.
\end{proof}

This completes the proof of Theorem \ref{thm_nonsimple} for tube type $\fp^+_{22}$.
When $\fp^+_{22}$ is of non-tube type, we take maximal tripotents $e'\in\fp^+_{11}$, $e''\in\fp^+_{22}$,
and let $\fp^{+\prime}:=\fp^+(e'+e'')_2$, $\fp^{+\prime}_{22}:=\fp^+_{22}(e'')_2$, $\fp^{+\prime}_{12}:=\fp^{+\prime}\cap\fp^+_{12}$,
so that $\fp^{+\prime}=\fp^+_{11}\oplus\fp^{+\prime}_{12}\oplus\fp^{+\prime}_{22}$ holds.
For $x\in\fp^+$, let $x'\in\fp^{+\prime}$ be the orthogonal projection of $x$ onto $\fp^{+\prime}$.
Now we take a polynomial $f(x'_{22})\in\cP_\bl(\fp^{+\prime}_{22})$, and extend to a polynomial $f(x_{22})\in\cP_\bl(\fp^+_{22})$
on $\fp^+_{22}$. Then we have
\[
\tilde{f}_{\bm,\bn}(x_{22},y_{22})=\tilde{f}_{\bm,\bn}(x_{22}',y_{22}')
\in\cP_\bm(\fp^{+\prime}_{22})\otimes\cP_\bn(\fp^{+\prime}_{22}) \subset\cP_\bm(\fp^+_{22})\otimes\cP_\bn(\fp^+_{22}),
\]
and since $\big(Q(z_{12}){}^t\hspace{-1pt}z_{11}^\itinv\big)'=Q(z_{12}'){}^t\hspace{-1pt}z_{11}^\itinv$ holds,
by Corollary~\ref{inner_sub} we get
\begin{gather*}
\big\langle \det_{\fn^+_{11}}(x_{11})^kf(x_{22}),{\rm e}^{(x|\overline{z})_{\fp^+}}\big\rangle_{\lambda,x,\fp^+}
=\big\langle \det_{\fn^+_{11}}(x_{11})^kf(x_{22}'),{\rm e}^{(x'|\overline{z'})_{\fp^{+\prime}}}\big\rangle_{\lambda,x',\fp^{+\prime}}
\\ \qquad
{}=\frac{\big(\lambda-\frac{d}{2}r'\big)_{\underline{k}_{r''},d}}{(\lambda)_{\underline{k}_{r}+\bl,d} \big(\lambda\!-\!\frac{d}{2}r'\big)_{\bl,d}}
\det_{\fn^+_{11}}(z_{11})^k\sum_{\bm,\bn}(-k)_{\bm,d}\bigg(\lambda\!+\!k\!-\!\frac{d}{2}r'\bigg)_{\bn,d}
\tilde{f}_{\bm,\bn}\big(Q(z_{12}'){}^t\hspace{-1pt}z_{11}^\itinv,z_{22}'\big)
\\ \qquad
{}=\frac{\big(\lambda-\frac{d}{2}r'\big)_{\underline{k}_{r''},d}}{(\lambda)_{\underline{k}_{r}+\bl,d} \big(\lambda\!-\!\frac{d}{2}r'\big)_{\bl,d}}
\det_{\fn^+_{11}}(z_{11})^k\sum_{\bm,\bn}(-k)_{\bm,d}\bigg(\lambda\!+\!k\!-\!\frac{d}{2}r'\bigg)_{\bn,d}
\tilde{f}_{\bm,\bn}\big(Q(z_{12}){}^t\hspace{-1pt}z_{11}^\itinv,z_{22}\big).
\end{gather*}
Now let $\fk^\BC_{22}:=[\fp^+_{22},\fp^-_{22}]\subset\fk^\BC$, and let $K_{22}^\BC\subset K^\BC$ be the corresponding connected subgroup.
Then since both
\[
\cP(\fp^+_{22})\longrightarrow \cP(\fp^+), \qquad f(x_{22})\mapsto
\big\langle \det_{\fn^+_{11}}(x_{11})^kf(x_{22}),{\rm e}^{(x|\overline{z})_{\fp^+}}\big\rangle_{\lambda,x}
\]
and
\[
\cP(\fp^+_{22})\longrightarrow \cP_\bm(\fp^+_{22})\otimes\cP_\bn(\fp^+_{22}), \qquad
f(x_{22})\mapsto \tilde{f}_{\bm,\bn}(x_{22},y_{22})
\]
are $K_{22}^\BC$-equivariant, and since $\cP_\bl(\fp^+_{22})$ is generated by $\cP_\bl(\fp^{+\prime}_{22})$ as a $K^\BC_{22}$-module,
the above formula holds for all $f\in\cP_\bl(\fp^+_{22})$. This completes the proof of Theorem~\ref{thm_nonsimple} for non-tube type~$\fp^+_{22}$.

Especially if both $\fp^+_{11}$, $\fp^+_{22}$ are of tube type (i.e., {\it Cases} 1, 2, 5, {\it Case} 3 with $q'=s'=r'$, $q''=s''=r''$,
and {\it Case} 4 with $s'=2r'$, $s''=2r''$), and if $\bl=\underline{l}_{r''}$, $f(x_{22})=\det_{\fn^+_{22}}(x_{22})^l$,
then since by (\ref{det_h}) and (\ref{expand_h}),
\begin{align*}
\det_{\fn^+_{22}}(x_{22}+y_{22})^l&=\det_{\fn^+_{22}}(y_{22})^l h_{\fp^+_{22}}\big({-}x_{22},{}^t\hspace{-1pt}y_{22}^\itinv\big)^l \\
&=\det_{\fn^+_{22}}(y_{22})^l\sum_{\bm}(-1)^{|\bm|}(-l)_{\bm,d}\tilde{\Phi}_\bm^{\fp^+_{22}}\big(x_{22},{}^t\hspace{-1pt}y_{22}^\itinv\big)
\end{align*}
holds, where $\tilde{\Phi}_\bm^{\fp^+_{22}}$ is as in (\ref{Phi^p}), we have
\[
\tilde{f}_{\bm,\bn}(x_{22},y_{22})=\begin{cases}
\ds \det_{\fn^+_{22}}(y_{22})^l(-1)^{|\bm|}(-l)_{\bm,d}
\tilde{\Phi}_\bm^{\fp^+_{22}}\big(x_{22},{}^t\hspace{-1pt}y_{22}^\itinv\big), & \bn=\underline{l}_{r''}-\bm^\vee,
\\
0, & \bn\ne\underline{l}_{r''}-\bm^\vee, \end{cases}
\]
where $\bm^\vee=(m_{r''},m_{r''-1},\dots,m_1)$. Therefore we get
\begin{gather*}
\big\langle \det_{\fn^+_{11}}(x_{11})^k\det_{\fn^+_{22}}(x_{22})^l,{\rm e}^{(x|\overline{z})_{\fp^+}}\big\rangle_{\lambda,x}
\\ \qquad
{}=\frac{\big(\lambda-\frac{d}{2}r'\big)_{\underline{k}_{r''},d}}
{(\lambda)_{\underline{k}_{r}+\underline{l}_{r''},d} \big(\lambda-\frac{d}{2}r'\big)_{\underline{l}_{r''},d}}
\det_{\fn^+_{11}}(z_{11})^k\det_{\fn^+_{22}}(z_{22})^l
\\ \qquad\hphantom{=}
{}\ \times\sum_{\bm}(-1)^{|\bm|}(-k)_{\bm,d}(-l)_{\bm,d}
\bigg(\lambda+k-\frac{d}{2}r'\bigg)_{\underline{l}_{r''}-\bm^\vee,d}
\tilde{\Phi}_\bm^{\fp^+_{22}}\big(Q(z_{12}){}^t\hspace{-1pt}z_{11}^\itinv,{}^t\hspace{-1pt}z_{22}^\itinv\big) \\ \qquad
{}=\frac{\big(\lambda-\frac{d}{2}r'\big)_{\underline{k}_{r''},d} \big(\lambda+k-\frac{d}{2}r'\big)_{\underline{l}_{r''},d}}
{(\lambda)_{\underline{k}_{r}+\underline{l}_{r''},d} \big(\lambda-\frac{d}{2}r'\big)_{\underline{l}_{r''},d}}
\det_{\fn^+_{11}}(z_{11})^k\det_{\fn^+_{22}}(z_{22})^l
\\ \qquad\hphantom{=}
{}\ \times\sum_{\bm}\frac{(-k)_{\bm,d}(-l)_{\bm,d}}{\left(-\lambda-k-l+\frac{n}{r}\right)_{\bm,d}}
\tilde{\Phi}_\bm^{\fp^+_{22}}\big(Q(z_{12}){}^t\hspace{-1pt}z_{11}^\itinv,{}^t\hspace{-1pt}z_{22}^\itinv\big),
\end{gather*}
and the constant is computed as
\begin{gather}
\frac{\big(\lambda-\frac{d}{2}r'\big)_{\underline{k}_{r''},d} \big(\lambda+k-\frac{d}{2}r'\big)_{\underline{l}_{r''},d}}
{(\lambda)_{\underline{k}_{r}+\underline{l}_{r''},d} \big(\lambda-\frac{d}{2}r'\big)_{\underline{l}_{r''},d}} \notag
\\ \qquad
{}=\frac{\big(\lambda-\frac{d}{2}r'\big)_{\underline{k+l}_{r''},d}}
{(\lambda)_{\underline{k+l}_{r''},d}\big(\lambda-\frac{d}{2}r''\big)_{\underline{k}_{r'},d}
\big(\lambda-\frac{d}{2}r'\big)_{\underline{l}_{r''},d}}
=\frac{\big(\lambda-\frac{d}{2}r''\big)_{\underline{k+l}_{r'},d}}
{(\lambda)_{\underline{k+l}_{r'},d}\big(\lambda-\frac{d}{2}r''\big)_{\underline{k}_{r'},d}
\big(\lambda-\frac{d}{2}r'\big)_{\underline{l}_{r''},d}} \notag
\\ \qquad
{}=\frac{\big(\lambda+l-\frac{d}{2}r'\big)_{\underline{k}_{r''},d}}
{(\lambda)_{\underline{k+l}_{r''},d}\big(\lambda-\frac{d}{2}r''\big)_{\underline{k}_{r'},d}}
=\frac{\big(\lambda+k-\frac{d}{2}r''\big)_{\underline{l}_{r'},d}}
{(\lambda)_{\underline{k+l}_{r'},d}\big(\lambda-\frac{d}{2}r'\big)_{\underline{l}_{r''},d}} \notag
\\ \qquad
{}=\begin{cases} \ds \frac{\big(\lambda+\max\{k,l\}-\frac{d}{2}r'\big)_{\underline{\min\{k,l\}}_{r''},d}}
{(\lambda)_{\underline{k+l}_{r''},d}\big(\lambda-\frac{d}{2}r''\big)_{\underline{k}_{r'-r''},d}
\big(\lambda-\frac{d}{2}r'\big)_{\underline{\min\{k,l\}}_{r''},d}}, &r'\ge r'',
\\
\ds \frac{\big(\lambda+\max\{k,l\}-\frac{d}{2}r''\big)_{\underline{\min\{k,l\}}_{r'},d}}
{(\lambda)_{\underline{k+l}_{r'},d}\big(\lambda-\frac{d}{2}r'\big)_{\underline{l}_{r''-r'},d}
\big(\lambda-\frac{d}{2}r''\big)_{\underline{\min\{k,l\}}_{r'},d}}, &r'\le r''. \end{cases}\label{constant_nonsimple}
\end{gather}
Also, by Proposition \ref{prop_Jordansub}(2) we have
\[ t^rh_{\fp^+_{22}}\big(t^{-1}Q(z_{12}){}^t\hspace{-1pt}z_{11}^\itinv,{}^t\hspace{-1pt}z_{22}^\itinv\big)
=t^rh_{\fp^+_{11}}\big(t^{-1}Q(z_{12}){}^t\hspace{-1pt}z_{22}^\itinv,{}^t\hspace{-1pt}z_{11}^\itinv\big)\in\BC[t], \]
and since $\tilde{\Phi}_\bm^{\fp^+_{22}}\big(Q(z_{12}){}^t\hspace{-1pt}z_{11}^\itinv,{}^t\hspace{-1pt}z_{22}^\itinv\big)$ depends only on the roots of
this polynomial, we have
\[ \tilde{\Phi}_\bm^{\fp^+_{22}}\big(Q(z_{12}){}^t\hspace{-1pt}z_{11}^\itinv,{}^t\hspace{-1pt}z_{22}^\itinv\big)
=\tilde{\Phi}_\bm^{\fp^+_{11}}\big(Q(z_{12}){}^t\hspace{-1pt}z_{22}^\itinv,{}^t\hspace{-1pt}z_{11}^\itinv\big). \]
Therefore we have the following.
\begin{Corollary}\label{cor_scalar_nonsimple}
Assume $\fp^+_{11}$, $\fp^+_{22}$ are of tube type, and let
$\Re\lambda>\frac{2n}{r}-1$, $k,l\in\BZ_{\ge 0}$.
Then for $z=z_{11}+z_{12}+z_{22}\in\fp^+$, we have
\begin{gather*}
\big\langle \det_{\fn^+_{11}}(x_{11})^k\det_{\fn^+_{22}}(x_{22})^l,{\rm e}^{(x|\overline{z})_{\fp^+}}\big\rangle_{\lambda,x}
 \\ \qquad
{}=\begin{cases} \ds \frac{\big(\lambda+\max\{k,l\}-\frac{d}{2}r'\big)_{\underline{\min\{k,l\}}_{r''},d}}
{(\lambda)_{(\underline{k+l}_{r''},\underline{k}_{r'-r''},\underline{\min\{k,l\}}_{r''}),d}}
\det_{\fn^+_{11}}(z_{11})^k\det_{\fn^+_{22}}(z_{22})^l
\\ \qquad
\ds {}\times
{}_2F_1^{\fp^+_{22}}\left(\begin{matrix} -k,-l \\ -\lambda-k-l+\frac{n}{r}\end{matrix};
Q(z_{12}){}^t\hspace{-1pt}z_{11}^\itinv,{}^t\hspace{-1pt}z_{22}^\itinv\right)\!, & r'\ge r'',
\\
\ds \frac{\big(\lambda+\max\{k,l\}-\frac{d}{2}r''\big)_{\underline{\min\{k,l\}}_{r'},d}}
{(\lambda)_{(\underline{k+l}_{r'},\underline{l}_{r''-r'},\underline{\min\{k,l\}}_{r'}),d}}
\det_{\fn^+_{11}}(z_{11})^k\det_{\fn^+_{22}}(z_{22})^l
\\ \qquad
\ds{} \times
{}_2F_1^{\fp^+_{11}}
\left(\begin{matrix} -k,-l \\ -\lambda-k-l+\frac{n}{r}\end{matrix};
Q(z_{12}){}^t\hspace{-1pt}z_{22}^\itinv,{}^t\hspace{-1pt}z_{11}^\itinv\right)\!, & r'\le r''.
\end{cases}
\end{gather*}
\end{Corollary}
Here ${}_2F_1^{\fp^+_{jj}}$ is as in (\ref{2F1^p}).
By this corollary and by Remark \ref{rem_Heckman-Opdam}, the inner product
$\big\langle\! \det_{\fn^+_{11}}\hspace{-2pt}(x_{11})^k\allowbreak\times\det_{\fn^+_{22}}(x_{22})^l,{\rm e}^{(x|\overline{z})_{\fp^+}}\big\rangle_{\lambda,x}$ is
given by using a Heckman--Opdam's multivariate hypergeometric polynomial of type $BC_{\min\{r',r''\}}$. %%% Delete $\times$ if the position of line break is changed.
On the other hand, for general $\bl\in\BZ_{++}^{r''}$, $f(x_{22})\in\cP_\bl(\fp^+_{22})$, the inner product
$\big\langle \det_{\fn^+_{11}}(x_{11})^kf(x_{22}),{\rm e}^{(x|\overline{z})_{\fp^+}}\big\rangle_{\lambda,x}$ is not always
given by that, since the generalization of $\underline{l}_{r''}$ to $\bl\in\BZ_{++}^{r''}$ in the above formula
does not fit with (\ref{Heckman-Opdam}).

For later use we prove the following. Suppose $\fp^+=\fn^{+\BC}$ is simple and of tube type.
\begin{Proposition}\label{prop_identity_nonsimple}
Let $\bl\in\BZ_{++}^r$, $j\in\{0,1,\dots,l_r\}$. For $f(x)\in\cP_\bl(\fp^+)$,
we put $g(x):=\det_{\fn^+}(x)^{-j}f(x)\in\cP_{\bl-\underline{j}_r}(\fp^+)$.
\begin{enumerate}\itemsep=0pt
\item[$1.$] Let $\bm,\bn\in\BZ_{++}^r$. Then for $x,y\in\fp^+$ we have
\begin{align*}
\bigg(\frac{n}{r}-j+\bl\bigg)_{\underline{j}_r,d}\tilde{g}_{\bm,\bn}(x,y)
&=\bigg(\frac{n}{r}+\bm\bigg)_{\underline{j}_r,d}\det_{\fn^+}(x)^{-j}\tilde{f}_{\bm+\underline{j}_r,\bn}(x,y) \\
&=\bigg(\frac{n}{r}+\bn\bigg)_{\underline{j}_r,d}\det_{\fn^+}(y)^{-j}\tilde{f}_{\bm,\bn+\underline{j}_r}(x,y).
\end{align*}
\item[$2.$] Let $k\in\BZ_{\ge 0}$, $k\ge j$. Then for $z,a\in\fp^+$ we have
\begin{gather*}
\det_{\fn^+}\bigg(\frac{\partial}{\partial z}\bigg)^k\det_{\fn^+}(z-a)^{k-j}f(z)
\\ \qquad
{}=\bigg(\frac{n}{r}-j+\bl\bigg)_{\underline{j}_r,d}\!\!\!\!
\det_{\fn^+}(z-a)^{-j}\det_{\fn^+}\bigg(\frac{\partial}{\partial z}\bigg)^{k-j}\!\!\!\!\det_{\fn^+}(z-a)^k\det_{\fn^+}(z)^{-j}f(z).
\end{gather*}
\end{enumerate}
\end{Proposition}

\begin{proof}
(1) We~compute $\det_{\fn^+}\big(\frac{\partial}{\partial x}\big)^jf(x+y)$ in two ways. By \cite[Proposition VII.1.6]{FK},
\begin{align*}
\det_{\fn^+}\bigg(\frac{\partial}{\partial x}\bigg)^j\!f(x\!+\!y)
={}&\det_{\fn^+}\bigg(\frac{\partial}{\partial w}\bigg)^jf(w)\biggl|_{w=x+y}
=\bigg(\frac{n}{r}\!-\!j\!+\!\bl\bigg)_{\underline{j}_r,d}\!\!\!\!\det_{\fn^+}(x+y)^{-j}f(x+y)
\\
={}&\bigg(\frac{n}{r}-j+\bl\bigg)_{\underline{j}_r,d}\sum_{\bm,\bn}\tilde{g}_{\bm,\bn}(x,y),
\\
\det_{\fn^+}\bigg(\frac{\partial}{\partial x}\bigg)^j\!f(x\!+\!y)
={}&\det_{\fn^+}\bigg(\frac{\partial}{\partial x}\bigg)^j\sum_{\bm,\bn}\tilde{f}_{\bm,\bn}(x,y)\\
={}& \det_{\fn^+}(x)^{-j}\sum_{\substack{\bm\\ m_r\ge j}}\bigg(\frac{n}{r}-j+\bm\bigg)_{\underline{j}_r,d}
\sum_{\bn}\tilde{f}_{\bm,\bn}(x,y)\\
={}& \det_{\fn^+}(x)^{-j}\sum_{\bm}\bigg(\frac{n}{r}+\bm\bigg)_{\underline{j}_r,d} \sum_{\bn}\tilde{f}_{\bm+\underline{j}_r,\bn}(x,y)
\end{align*}
hold. We~note that $\left(\frac{n}{r}-j+\bm\right)_{\underline{j}_r,d}=0$ holds if $m_r<j$.
Then projecting both onto $\cP_\bm(\fp^+)\otimes\cP_\bn(\fp^+)$, we get the 1st equality. The 2nd equality is also proved similarly.

(2) By dividing the 3rd equality of Proposition \ref{Euler_int} by $\Gamma_r^d(\mu)$ and putting $\mu=-k$,
\begin{gather*}
\det_{\fn^+}(z-a)^{k-\nu+\frac{n}{r}}\det_{\fn^+}\bigg(\frac{\partial}{\partial z}\bigg)^k\det_{\fn^+}(z-a)^{\nu-\frac{n}{r}}f(z)
\\ \qquad
{}=(\nu-k)_{\underline{k}_r,d}\sum_{\bn}\frac{(\nu)_{\bn,d}}{(\nu-k)_{\bn,d}}\sum_{\bm} \tilde{f}_{\bm,\bn}(a,z-a)
=\sum_{\bn}(\nu-k+\bn)_{\underline{k}_r,d}\sum_{\bm} \tilde{f}_{\bm,\bn}(a,z-a)
\end{gather*}
holds, and putting $\nu=k-j+\frac{n}{r}$, we get
\begin{gather*}
\det_{\fn^+}\bigg(\frac{\partial}{\partial z}\bigg)^k\det_{\fn^+}(z-a)^{k-j}f(z)
\\ \qquad
{}=\det_{\fn^+}(z-a)^{-j}\sum_{\substack{\bn\\ n_r\ge j}}\bigg(\frac{n}{r}-j+\bn\bigg)_{\underline{k}_r,d}\sum_{\bm} \tilde{f}_{\bm,\bn}(a,z-a)
\\ \qquad
{}=\det_{\fn^+}(z-a)^{-j}\sum_{\bn}\bigg(\frac{n}{r}+\bn\bigg)_{\underline{k}_r,d}\sum_{\bm} \tilde{f}_{\bm,\bn+\underline{j}_r}(a,z-a)
\\ \qquad
{}=\det_{\fn^+}(z-a)^{-j}\sum_{\bn}\bigg(\frac{n}{r}+j+\bn\bigg)_{\underline{k-j}_r,d}
\bigg(\frac{n}{r}+\bn\bigg)_{\underline{j}_r,d}\sum_{\bm} \tilde{f}_{\bm,\bn+\underline{j}_r}(a,z-a).
\end{gather*}
Similarly, for $g(x):=\det_{\fn^+}(x)^{-j}f(x)$ we have
\begin{gather*}
\bigg(\frac{n}{r}-j+\bl\bigg)_{\underline{j}_r,d}
\det_{\fn^+}(z-a)^{-j}\det_{\fn^+}\bigg(\frac{\partial}{\partial z}\bigg)^{k-j}\det_{\fn^+}(z-a)^k\det_{\fn^+}(z)^{-j}f(z)
\\ \qquad
{}=\bigg(\frac{n}{r}-j+\bl\bigg)_{\underline{j}_r,d}
\sum_{\bn}\bigg(\frac{n}{r}+j+\bn\bigg)_{\underline{k-j}_r,d}\sum_{\bm} \tilde{g}_{\bm,\bn}(a,z-a).
\end{gather*}
Then comparing these and combining with (1), we get the desired formula.
\end{proof}

If $\bl=\underline{l}_r$ for some $l\in\BZ_{\ge 0}$, then Proposition~\ref{prop_identity_nonsimple}(2) also follows from (\ref{Kummer2}).
By Propositions~\ref{prop_identity_nonsimple}(2) and~\ref{prop_Jordansub}(2), easily we get the following.
\begin{Corollary}\label{cor_identity_nonsimple}
Suppose $\fp^+_{11}$, $\fp^+_{22}$ are of tube type. Let $j,k\in\BZ_{\ge 0}$, $\bl\in\BZ_{++}^{r''}$, $j\le \min\{k,l_{r''}\}$
and let $f(x_{22})\in\cP_\bl(\fp^+_{22})$. Then for $z=z_{11}+z_{12}+z_{22}\in\fp^+=\fp^+_{11}\oplus\fp^+_{12}\oplus\fp^+_{22}$ we have
\begin{gather*}
\det_{\fn^+_{22}}\bigg(\frac{\partial}{\partial z_{22}}\bigg)^k\det_{\fn^+}(z)^{k-j}f(z_{22})
\\ \qquad
{}=\bigg(\frac{n''}{r''}-j+\bl\bigg)_{\underline{j}_{r''},d}
\det_{\fn^+}(z)^{-j}\det_{\fn^+_{22}}\bigg(\frac{\partial}{\partial z_{22}}\bigg)^{k-j}\det_{\fn^+}(z)^k\det_{\fn^+_{22}}(z_{22})^{-j}f(z_{22}).
\end{gather*}
\end{Corollary}

\subsection{Computation of poles}

By Theorem \ref{thm_nonsimple}, we can compute the poles of $\big\langle f_1(x_{11})f_2(x_{22}),{\rm e}^{(x|\overline{z})_{\fp^+}}\big\rangle_{\lambda,x}$
for $f_1(x_{11})\in\cP_{\underline{k}_a}(\fp^+_{11})$, $f_2(x_{22})\in\cP_\bl(\fp^+_{22})$
with $k\in\BZ_{\ge 0}$, $1\le a\le r'$, $\bl\in\BZ_{++}^{r''}$, without the assumption that $\fp^+$,
$\fp^+_{11}$ and~$\fp^+_{22}$ are of tube type.
In the following, for $\bk=(k_1,\dots,k_r)$, $\bl=(l_1,\dots,l_r)\in\BZ^r$,
we write $\min\{\bk,\bl\}:=(\min\{k_1,l_1\},\dots,\min\{k_r,l_r\})$, $\max\{\bk,\bl\}:=(\max\{k_1,l_1\},\dots,\max\{k_r,l_r\})\in\BZ^r$,
and write $\bk\le \bl$ if $k_j\le l_j$ for all $1\le j\le r$.
\begin{Corollary}\label{cor_pole_nonsimple}
For $k\in\BZ_{\ge 0}$, $1\le a\le r'$, $\bl\in\BZ_{++}^{r''}$, let
\begin{gather}
\bl':=\begin{cases} (l_1,\dots,l_{r''},0,\dots,0), &a\ge r''
\\
(l_1,\dots,l_{a}), &a\le r'' \end{cases} \quad \in \BZ_{++}^{a}, \notag
\\
\bl'':=\begin{cases} (0,\dots,0), &a\ge r''
\\
(l_{a+1},\dots,l_{r''},0,\dots,0), &a\le r'' \end{cases}\quad \in \BZ_{++}^{r''},
\label{lprime}
\end{gather}
and let $f_1(x_{11})\in \cP_{\underline{k}_a}(\fp^+_{11})$, $f_2(x_{22})\in \cP_\bl(\fp^+_{22})$.
\begin{enumerate}\itemsep=0pt
\item[$1.$] When $z_{12}=0$, we have
\begin{gather*}
\big\langle f_1(x_{11})f_2(x_{22}),{\rm e}^{(x_{11}+x_{22}|\overline{z_{11}+z_{22}})_{\fp^+}}\big\rangle_{\lambda,x}
\\ \qquad
{}=\frac{\big(\lambda-\frac{d}{2}a\big)_{\underline{k}_{r''}+\bl,d}}
{(\lambda)_{\underline{k}_{a+r''}+\bl,d}\big(\lambda-\frac{d}{2}a\big)_{\bl,d}}f_1(z_{11})f_2(z_{22}) \\ \qquad
{}=\frac{\big(\lambda-\frac{d}{2}a+ \max\{\underline{k}_{r''}+\bl'',\bl\}\big)_{\min\{\bl-\bl'',\underline{k}_{r''}\},d}}
{(\lambda)_{\underline{k}_{a}+\bl',d} \big(\lambda-\frac{d}{2}a\big)_{\min\{\underline{k}_{r''}+\bl'',\bl\},d}}f_1(z_{11})f_2(z_{22}).
\end{gather*}
\item[$2.$] As a function of $\lambda$,
\begin{align*}
(\lambda)_{\underline{k}_{a}+\bl',d}\bigg(\lambda-\frac{d}{2}a\bigg)_{\min\{\underline{k}_{r''}+\bl'',\bl\},d}
\big\langle f_1(x_{11})f_2(x_{22}),{\rm e}^{(x|\overline{z})_{\fp^+}}\big\rangle_{\lambda,x}
\\ \qquad
{}=(\lambda)_{(\underline{k}_{a}+\bl',\min\{\underline{k}_{r''}+\bl'',\bl\}),d}
\big\langle f_1(x_{11})f_2(x_{22}),{\rm e}^{(x|\overline{z})_{\fp^+}}\big\rangle_{\lambda,x}
\end{align*}
is holomorphically continued for all $\lambda\in\BC$.
\item[$3.$] Moreover, when $f_1(x_{11})f_2(x_{22})\ne 0$,
if $l_{a+1}=0$ or $k=0$ or $\bl=\underline{l}_{a'}$ $(l\in\BZ_{\ge 0}$, $1\le a'\le r'')$,
then the above formula gives a non-zero polynomial in $z\in\fp^+$ for all $\lambda\in\BC$.
\end{enumerate}
\end{Corollary}
\begin{proof}
(1) First we assume $\fp^+_{11}$ is of tube type, $a=r'$, $f_1(x_{11})=\det_{\fn^+_{11}}(x_{11})^k$,
and write $f_2(x_{22})=f(x_{22})\in\cP_\bl(\fp^+_{22})$. In the 3rd equality of Theorem \ref{thm_nonsimple},
if we substitute $z_{12}=0$, then only $\bm=\underline{0}_{r'}$, $\bn=\bl$ term remains, and we have
\begin{gather*}
\big\langle \det_{\fn^+_{11}}(x_{11})^kf(x_{22}),{\rm e}^{(x_{11}+x_{22}|\overline{z_{11}+z_{22}})_{\fp^+}}\big\rangle_{\lambda,x}
\\ \qquad
{}=\frac{\big(\lambda-\frac{d}{2}r'\big)_{\underline{k}_{r''},d}}{(\lambda)_{\underline{k}_{r}+\bl,d} \big(\lambda-\frac{d}{2}r'\big)_{\bl,d}}
\det_{\fn^+_{11}}(z_{11})^k(-k)_{\underline{0}_{r''},d}\bigg(\lambda+k-\frac{d}{2}r'\bigg)_{\bl,d}
\tilde{f}_{\underline{0}_{r''},\bl}\big(Q(0){}^t\hspace{-1pt}z_{11}^\itinv,z_{22}\big)
\\ \qquad
{}=\frac{\big(\lambda-\frac{d}{2}r'\big)_{\underline{k}_{r''}+\bl,d}}
{(\lambda)_{\underline{k}_{r}+\bl,d}\big(\lambda-\frac{d}{2}r'\big)_{\bl,d}} \det_{\fn^+_{11}}(z_{11})^kf(z_{22}).
\end{gather*}
Next we consider general $\fp^+_{11}$ and $1\le a\le r'$. We~take a tripotent $e'\in\fp^+_{11}$ of rank $a$,
a maximal tripotent $e''\in\fp^+_{22}$, and let $\fp^{+\prime}:=\fp^+(e'+e'')_2$, $\fp^{+\prime}_{ij}:=\fp^+_{ij}\cap\fp^{+\prime}$.
For $z\in\fp^+$, let $z'\in\fp^{+\prime}$ be the projection of $z$ onto $\fp^{+\prime}$.
Then for $f_1(x_{11})=\det_{\fn^{+\prime}_{11}}(x_{11}')^k\in\cP_{\underline{k}_a}(\fp^+_{11})$,
$f_2(x_{22})=f(x_{22}')\in\cP_\bl\big(\fp^{+\prime}_{22}\big)\subset\cP_\bl\big(\fp^+_{22}\big)$, by Corollary~\ref{inner_sub} we have
\begin{gather*}
\big\langle f_1(x_{11})f_2(x_{22}),{\rm e}^{(x_{11}+x_{22}|\overline{z_{11}+z_{22}})_{\fp^+}}\big\rangle_{\lambda,x,\fp^+}
\\ \qquad
{}=\big\langle \det_{\fn^{+\prime}_{11}}(x_{11}')^kf(x_{22}'),
{\rm e}^{(x_{11}'+x_{22}'|\overline{z_{11}'+z_{22}'})_{\fp^{+\prime}}}\big\rangle_{\lambda,x',\fp^{+\prime}} \\ \qquad
{}=\frac{\big(\lambda-\frac{d}{2}a\big)_{\underline{k}_{r''}+\bl,d}}
{(\lambda)_{\underline{k}_{a+r''}+\bl,d}\big(\lambda-\frac{d}{2}a\big)_{\bl,d}} \det_{\fn^{+\prime}_{11}}(z_{11}')^kf(z_{22}')
=\frac{\big(\lambda-\frac{d}{2}a\big)_{\underline{k}_{r''}+\bl,d}}
{(\lambda)_{\underline{k}_{a+r''}+\bl,d}\big(\lambda-\frac{d}{2}a\big)_{\bl,d}}f_1(z_{11})f_2(z_{22}),
\end{gather*}
and since $f_1(x_{11})f_2(x_{22})\mapsto \big\langle f_1(x_{11})f_2(x_{22}),{\rm e}^{(x|\overline{z})_{\fp^+}}\big\rangle_{\lambda,x}$ is
$K_1$-equivariant, this holds for all $f_1(x_{11})\in\cP_{\underline{k}_a}(\fp^+_{11})$, $f_2(x_{22})\in\cP_\bl\big(\fp^+_{22}\big)$.
Finally, the constant is computed as
\begin{align*}
\frac{\big(\lambda-\frac{d}{2}a\big)_{\underline{k}_{r''}+\bl,d}}
{(\lambda)_{\underline{k}_{a+r''}+\bl,d}\big(\lambda-\frac{d}{2}a\big)_{\bl,d}}
&=\frac{\big(\lambda-\frac{d}{2}a\big)_{\underline{k}_{r''}+\bl,d}}{(\lambda)_{\underline{k}_a+\bl',d}
\big(\lambda-\frac{d}{2}a\big)_{\underline{k}_{r''}+\bl'',d}\big(\lambda-\frac{d}{2}a\big)_{\bl,d}} \\
&=\frac{\big(\lambda-\frac{d}{2}a+\max\{\underline{k}_{r''}+\bl'',\bl\}\big)
_{\underline{k}_{r''}+\bl-\max\{\underline{k}_{r''}+\bl'',\bl\},d}}
{(\lambda)_{\underline{k}_a+\bl',d}\big(\lambda-\frac{d}{2}a\big)_{\min\{\underline{k}_{r''}+\bl'',\bl\},d}} \\
&=\frac{\big(\lambda-\frac{d}{2}a+ \max\{\underline{k}_{r''}+\bl'',\bl\}\big)_{\min\{\bl-\bl'',\underline{k}_{r''}\},d}}
{(\lambda)_{\underline{k}_a+\bl',d} \big(\lambda-\frac{d}{2}a\big)_{\min\{\underline{k}_{r''}+\bl'',\bl\},d}}.
\end{align*}

(2) Again we assume $\fp^+_{11}$ is of tube type, $a=r'$, $f_1(x_{11})=\det_{\fn^+_{11}}(x_{11})^k$,
and write $f_2(x_{22})=f(x_{22})\in\cP_\bl(\fp^+_{22})$.
First let $r'\ge r''$. Then by the 1st equality of Theorem \ref{thm_nonsimple} we have
\begin{gather*}
\big\langle \det_{\fn^+_{11}}(x_{11})^kf(x_{22}),{\rm e}^{(x|\overline{z})_{\fp^+}}\big\rangle_{\lambda,x}
\\ \qquad
{}=\frac{\big(\lambda-\frac{d}{2}r'\big)_{\underline{k}_{r''},d}}{(\lambda)_{\underline{k}_{r}+\bl,d}} \det_{\fn^+_{11}}(z_{11})^k
\sum_{\bm}\frac{(-k)_{\bm,d}}{\big(\lambda-\frac{d}{2}r'\big)_{\bm,d}}
\sum_{\bn}\tilde{f}_{\bm,\bn}\big(Q(z_{12}){}^t\hspace{-1pt}z_{11}^\itinv-z_{22},z_{22}\big)
\\ \qquad
{}=\frac{1}{(\lambda)_{\underline{k}_{r'}+\bl,d}}\det_{\fn^+_{11}}(z_{11})^k
\sum_{\bm}\frac{(-k)_{\bm,d}}{\big(\lambda-\frac{d}{2}r'\big)_{\bm,d}}
\sum_{\bn}\tilde{f}_{\bm,\bn}\big(Q(z_{12}){}^t\hspace{-1pt}z_{11}^\itinv-z_{22},z_{22}\big).
\end{gather*}
Then since $(-k)_{\bm,d}$ is non-zero only if $m_j\le k$, and
$\sum_\bn\tilde{f}_{\bm,\bn}\left(Q(z_{12}){}^t\hspace{-1pt}z_{11}^\itinv-z_{22},z_{22}\right)$ is non-zero only if $m_j\le l_j$, we have
\begin{gather}
(\lambda)_{\underline{k}_{r'}+\bl,d}\bigg(\lambda-\frac{d}{2}r'\bigg)_{\min\{\underline{k}_{r''},\bl\},d}
\big\langle \det_{\fn^+_{11}}(x_{11})^kf(x_{22}),{\rm e}^{(x|\overline{z})_{\fp^+}}\big\rangle_{\lambda,x} \notag
\\ \qquad
{}=\det_{\fn^+_{11}}(z_{11})^k \sum_{\bm\le\min\{\underline{k}_{r''},\bl\}}
(-k)_{\bm,d}\bigg(\lambda-\frac{d}{2}r'+\bm\bigg)_{\min\{\underline{k}_{r''},\bl\}-\bm,d} \notag
\\ \qquad\hphantom{=}
{} \ \times\sum_{\bn}\tilde{f}_{\bm,\bn}\big(Q(z_{12}){}^t\hspace{-1pt}z_{11}^\itinv-z_{22},z_{22}\big). \label{r'ger''}
\end{gather}
Therefore the claim follows. Next let $r'<r''$.
Then since $\tilde{f}_{\bm,\bn}\left(Q(z_{12}){}^t\hspace{-1pt}z_{11}^\itinv-z_{22},z_{22}\right)$ are not
linearly independent in general, we use the 2nd equality of the theorem instead. Then we have
\begin{gather*}
\big\langle \det_{\fn^+_{11}}(x_{11})^kf(x_{22}),{\rm e}^{(x|\overline{z})_{\fp^+}}\big\rangle_{\lambda,x} \\ =\frac{\big(\lambda-\frac{d}{2}r'\big)_{\underline{k}_{r''},d}}{(\lambda)_{\underline{k}_{r}+\bl,d}}\det_{\fn^+_{11}}(z_{11})^k
\sum_{\bn}\frac{\big(\lambda+k-\frac{d}{2}r'\big)_{\bn,d}}{\big(\lambda-\frac{d}{2}r'\big)_{\bn,d}}
\sum_{\bm} \tilde{f}_{\bm,\bn}\big(Q(z_{12}){}^t\hspace{-1pt}z_{11}^\itinv,z_{22}-Q(z_{12}){}^t\hspace{-1pt}z_{11}^\itinv\big) \\
=\frac{\det_{\fn^+_{11}}(z_{11})^k}{(\lambda)_{\underline{k}_{r'}+\bl',d}}
\sum_{\bn}\frac{\big(\lambda-\frac{d}{2}r'\big)_{\underline{k}_{r''}+\bn,d}}
{\big(\lambda-\frac{d}{2}r'\big)_{\underline{k}_{r''}+\bl'',d}\big(\lambda-\frac{d}{2}r'\big)_{\bn,d}}
\sum_{\bm} \tilde{f}_{\bm,\bn}\big(Q(z_{12}){}^t\hspace{-1pt}z_{11}^\itinv,z_{22}-Q(z_{12}){}^t \hspace{-1pt}z_{11}^\itinv\big).
\end{gather*}
Then since the rank of $Q(z_{12}){}^t\hspace{-1pt}z_{11}^\itinv$ is at most $r'$,
$\tilde{f}_{\bm,\bn}\left(Q(z_{12}){}^t\hspace{-1pt}z_{11}^\itinv,z_{22}-Q(z_{12}){}^t\hspace{-1pt}z_{11}^\itinv\right)$
is non-zero only if $m_{r'+1}=\cdots=m_{r''}=0$,
and $\sum_{\bm}\tilde{f}_{\bm,\bn}\left(Q(z_{12}){}^t\hspace{-1pt}z_{11}^\itinv,z_{22}-Q(z_{12}){}^t\hspace{-1pt}z_{11}^\itinv\right)$
is non-zero only if $l_{r'+j}\le n_j\le l_j$ for $1\le j\le r''-r'$ and $0\le n_j\le l_j$ for $r''-r'<j\le r''$,
which is proved case-by-case by using the Littlewood--Richardson rule. Therefore we have
\begin{gather*}
(\lambda)_{\underline{k}_{r'}+ \bl',d}\bigg(\lambda-\frac{d}{2}r'\bigg)_{\min\{\underline{k}_{r''}+\bl'',\bl\},d}
\big\langle \det_{\fn^+_{11}}(x_{11})^kf(x_{22}),{\rm e}^{(x|\overline{z})_{\fp^+}}\big\rangle_{\lambda,x}
\\
{}=\bigg(\lambda-\frac{d}{2}r'\bigg)_{\min\{\underline{k}_{r''}+\bl'',\bl\},d}\sum_{\bl''\le\bn\le\bl}
\frac{\big(\lambda-\frac{d}{2}r'+\max\{\underline{k}_{r''}+\bl'',\bn\}\big)_{\underline{k}_{r''}+\bn-\max\{\underline{k}_{r''}+\bl'',\bn\},d}}
{\big(\lambda-\frac{d}{2}r'\big)_{\min\{\underline{k}_{r''}+\bl'',\bn\},d}}
\\ \hphantom{=}
{} \ \times\det_{\fn^+_{11}}(z_{11})^k\sum_{\bm}\tilde{f}_{\bm,\bn}
\big(Q(z_{12}){}^t\hspace{-1pt}z_{11}^\itinv,z_{22}-Q(z_{12}){}^t\hspace{-1pt}z_{11}^\itinv\big)
\\
{}=\det_{\fn^+_{11}}(z_{11})^k\sum_{\bl''\le\bn\le\bl}
\bigg(\lambda-\frac{d}{2}r'+\min\{\underline{k}_{r''}+\bl'',\bn\}\bigg)
_{\min\{\underline{k}_{r''}+\bl'',\bl\}-\min\{\underline{k}_{r''}+\bl'',\bn\},d}
\\ \hphantom{=}
{}\ \times \bigg( \!\lambda - \frac{d}{2}r' + \max\{\underline{k}_{r''} + \bl'',\bn\}\!\bigg)
{\vphantom{\biggr)}}_{\min\{\bn-\bl'',\underline{k}_{r''}\},d}
\sum_{\bm}\tilde{f}_{\bm,\bn} \big(Q(z_{12}){}^t\hspace{-1pt}z_{11}^\itinv, z_{22} - Q(z_{12}){}^t\hspace{-1pt}z_{11}^\itinv\big),
\end{gather*}
and the claim follows.
For general $\fp^+_{11}$ and $1\le a\le r'$, this is proved as (1) by using Corollary~\ref{inner_sub} and the $K_1$-equivariance.

(3) Again we assume $\fp^+_{11}$ is of tube type and let $a=r'\ge r''$. Then in (\ref{r'ger''}),
by Lemma \ref{lem_nonzero_nonsimple} given later,
$\tilde{f}_{\min\{\underline{k}_{r''},\bl\},\max\{\bl-\underline{k}_{r''},\underline{0}_{r''}\}}\ne 0$
holds, and its coefficient does not depend on $\lambda\in\BC$ and is non-zero.
Now since all non-zero $\tilde{f}_{\bm,\bn}$ are linearly independent, (\ref{r'ger''}) is non-zero for all $\lambda\in\BC$.
For general $1\le a\le r'$ and $\bl\in\BZ_{++}^{r''}$ with $l_{a+1}=0$,
we take maximal tripotents $e'\in\fp^+_{11}$ of rank $a$ and $e''\in\fp^+_{22}$ of rank $\min\{a,r''\}$.
Then this is proved by reducing to $\fp^+(e'+e'')_2$ using Corollary~\ref{inner_sub} and the $K_1$-equivariance.
Next, when $k=0$, by $\cP_\bl(\fp^+_{22})\subset\cP_\bl(\fp^+)$ and Corollary~\ref{FKCor}, or the 3rd equality of Theorem~\ref{thm_nonsimple},
we have
\[
(\lambda)_{\bl,d}\big\langle f(x_{22}),{\rm e}^{(x|\overline{z})_{\fp^+}}\big\rangle_{\lambda,x} =f(z_{22})\ne 0.
\]
Finally, when $\bl=\underline{l}_{a'}$ with $a'>a$, by exchanging the roles of $(\fp^+_{11},k,a)$ and $(\fp^+_{22},l,a')$,
this is reduced to the $l_{a+1}=0$ case.
\end{proof}

To complete the proof of Corollary~\ref{cor_pole_nonsimple}(3), for $\bl\in\BZ_{++}^r$ let
\[
\Proj_\bl\colon\ \cP(\fp^+)\longrightarrow \cP_\bl(\fp^+)
\]
be the orthogonal projection.
\begin{Lemma}\label{lem_nonzero_nonsimple}
Let $\bl,\bm,\bn\in\BZ_{++}^r$. Then the linear map
\begin{align*}
\alpha_{\bm,\bn}^\bl\colon\ \cP_\bl(\fp^+)&\longrightarrow\cP_\bm(\fp^+)\otimes\cP_\bn(\fp^+), \\
f(z)&\longmapsto \tilde{f}_{\bm,\bn}(x,y):=(\Proj_{\bm,x}\otimes\Proj_{\bn,y})f(x+y)
\end{align*}
is injective if and only if the linear map
\[
\beta_{\bm,\bn}^\bl\colon\quad \cP_\bm(\fp^+)\otimes\cP_\bn(\fp^+)\longrightarrow\cP_\bl(\fp^+), \qquad
g(x,y)\longmapsto \Proj_{\bl,z} (g(z,z))
\]
is non-zero. Especially, for $k\in\BZ_{\ge 0}$,
\[
\alpha_{\min\{\underline{k}_r,\bl\},\max\{\bl-\underline{k}_r,\underline{0}_r\}}^\bl\colon\
\cP_\bl(\fp^+)\longrightarrow\cP_{\min\{\underline{k}_r,\bl\}}(\fp^+) \otimes\cP_{\max\{\bl-\underline{k}_r,\underline{0}_r\}}(\fp^+)
\]
is injective.
\end{Lemma}

\begin{proof}
For $f\in\cP_\bl(\fp^+)$, $g_1\in\cP_\bm(\fp^+)$, $g_2\in\cP_\bn(\fp^+)$, we have
\begin{gather*}
\big\langle (\alpha_{\bm,\bn}^\bl f)(x,y),g_1(x)g_2(y)\big\rangle_{F,\fp^+\oplus\fp^+}
\\ \qquad
{}
=\langle f(x+y),g_1(x)g_2(y)\rangle_{F,\fp^+\oplus\fp^+}
=\overline{g_1\bigg(\overline{\frac{\partial}{\partial x}}\bigg)}\overline{g_2\bigg(\overline{\frac{\partial}{\partial y}}\bigg)}
f(x+y)\biggr|_{x=y=0}
\\ \qquad
{}=\overline{g_1\bigg(\overline{\frac{\partial}{\partial z}}\bigg)}\overline{g_2\bigg(\overline{\frac{\partial}{\partial z}}\bigg)}
f(z)\biggr|_{z=0} \!\!\!\!
=\big\langle f(z),g_1(z)g_2(z)\big\rangle_{F,\fp^+}
=\big\langle f(z),(\beta_{\bm,\bn}^\bl(g_1\otimes g_2))(z)\big\rangle_{F,\fp^+},
\end{gather*}
and hence $\alpha_{\bm,\bn}^\bl$ and $\beta_{\bm,\bn}^\bl$ are mutually adjoint with respect to the Fischer norms.
Therefore $\alpha_{\bm,\bn}^\bl$ is injective if and only if $\beta_{\bm,\bn}^\bl$ is surjective,
and since $\cP_\bl(\fp^+)$ is irreducible, this holds if and only if $\beta_{\bm,\bn}^\bl$ is non-zero.
Hence the first claim follows. Especially if $\bm=\min\{\underline{k}_r,\bl\}$, $\bn=\max\{\bl-\underline{k}_r,\underline{0}_r\}$,
then for the lowest weight vectors,
\begin{align*}
\big(\beta_{\min\{\underline{k}_r,\bl\},\max\{\bl-\underline{k}_r,\underline{0}_r\}}^\bl
\big(\Delta_{\min\{\underline{k}_r,\bl\}}(x) \otimes\Delta_{\max\{\bl-\underline{k}_r,\underline{0}_r\}}(y)\big)\big)(z)=\Delta_\bl(z)\ne 0
\end{align*}
holds, and the second claim follows.
\end{proof}

Especially, by (\ref{inner_Fubini}) and Corollary~\ref{cor_pole_nonsimple}(1),
for $f_1(x_{11})\in \cP_{\underline{k}_a}(\fp^+_{11})$, $f_2(x_{22})\in \cP_\bl(\fp^+_{22})$ we have
\begin{gather}
\Vert f_1(x_{11})f_2(x_{22})\Vert_{\lambda,x,\fp^+}^2 =\big\langle\big\langle f_1(x_{11})f_2(x_{22}),{\rm e}^{(x_{11}+x_{22}|\overline{z_{11}+z_{22}})_{\fp^+}}\big\rangle_{\lambda,x,\fp^+},
f_1(z_{11})f_2(z_{22})\big\rangle_{F,z,\fp^+} \notag
\\ \qquad
{}=\frac{\big(\lambda-\frac{d}{2}a\big)_{\underline{k}_{r''}+\bl,d}}
{(\lambda)_{\underline{k}_{a+r''}+\bl,d}\big(\lambda-\frac{d}{2}a\big)_{\bl,d}}
\Vert f_1(z_{11})\Vert_{F,z_{11},\fp^+_{11}}^2\Vert f_2(z_{22})\Vert_{F,z_{22},\fp^+_{22}}^2, \label{norm_nonsimple}
\end{gather}
and moreover, if $\fp^+_{11}$, $\fp^+_{22}$ are of tube type, $a=r'$, $\bl=\underline{l}_{r''}$ and $f_1(x_{11})=\det_{\fn^+_{11}}(x_{11})^k$,
$f_2(x_{22})=\det_{\fn^+_{22}}(x_{22})^l$, then since
\[
\big\Vert \det_{\fn^+_{11}}(z_{11})^k\big\Vert^2_{F,z_{11},\fp^+_{11}} =\bigg(\frac{n'}{r'}\bigg)_{\underline{k}_{r'},d}, \qquad
\big\Vert \det_{\fn^+_{22}}(z_{22})^l\big\Vert^2_{F,z_{22},\fp^+_{22}} =\bigg(\frac{n''}{r''}\bigg)_{\underline{l}_{r''},d}
\]
hold (see \cite[Proposition XI.4.1(ii)]{FK}), combining with (\ref{constant_nonsimple}) we get
\begin{align*}
\big\Vert \det_{\fn^+_{11}}(x_{11})^k\det_{\fn^+_{22}}(x_{22})^l\big\Vert^2_{\lambda,x,\fp^+}
&=\frac{\big(\frac{n'}{r'}\big)_{\underline{k}_{r'},d}\big(\frac{n''}{r''}\big)_{\underline{l}_{r''},d}
\big(\lambda+l-\frac{d}{2}r'\big)_{\underline{k}_{r''},d}}
{(\lambda)_{\underline{k+l}_{r''},d}\big(\lambda-\frac{d}{2}r''\big)_{\underline{k}_{r'},d}}
\\
&=\frac{\big(\frac{n'}{r'}\big)_{\underline{k}_{r'},d}\big(\frac{n''}{r''}\big)_{\underline{l}_{r''},d}
\big(\lambda+k-\frac{d}{2}r''\big)_{\underline{l}_{r'},d}}
{(\lambda)_{\underline{k+l}_{r'},d}\big(\lambda-\frac{d}{2}r'\big)_{\underline{l}_{r''},d}}.
\end{align*}

\subsection{Individual cases}

In this section we observe $\big\langle \det_{\fn^+_{11}}(x_{11})^k\det_{\fn^+_{22}}(x_{22})^l,
{\rm e}^{(x|\overline{z})_{\fp^+}}\big\rangle_{\lambda,x}$
(Corollary~\ref{cor_scalar_nonsimple}) under the explicit realization of $\fp^+$.
For {\it Case} 1, if we realize $\fp^+_{11}\simeq\BC$, $\fp^+_{12}\simeq\BC^d$, $\fp^+_{22}\simeq\BC$ as
\begin{gather*}
\fp^+_{11}:=\big\{(0,\dots,0,z,-\sqrt{-1}z) \mid z\in\BC\big\}\subset\fp^+=\BC^{d+2},
\\
\fp^+_{12}:=\big\{(z_1,\dots,z_d,0,0)\mid z_j\in\BC\big\}\subset\fp^+=\BC^{d+2},
\\
\fp^+_{22}:=\big\{(0,\dots,0,z,\sqrt{-1}z) \mid z\in\BC\big\}\subset\fp^+=\BC^{d+2},
\end{gather*}
and take the maximal tripotent $e=(0,\dots,0,1,0)\in\fp^+_{11}\oplus\fp^+_{22}\subset\fp^+$,
then for $z=(z_1,\dots,z_d,\allowbreak z_{d+1},z_{d+2})=(z',z_{d+1},z_{d+2})\in\fp^+=\BC^{d+2}$ we have
\begin{gather*}
\det_{\fn^+_{11}}(z_{11})=z_{d+1}+\sqrt{-1}z_{d+2}, \\ \det_{\fn^+_{22}}(z_{22})=z_{d+1}-\sqrt{-1}z_{d+2},
\\
\tilde{\Phi}_m^{\fp^+_{22}}\big(Q(z_{12}){}^t\hspace{-1pt}z_{11}^\itinv,{}^t\hspace{-1pt}z_{22}^\itinv\big)
=\frac{1}{m!}\big(\big(Q(z_{12}){}^t\hspace{-1pt}z_{11}^\itinv|{}^t\hspace{-1pt}z_{22}^\itinv\big)_{\fp^+_{22}}\big)^m
=\frac{1}{m!}\bigg({-}\frac{q(z')}{z_{d+1}^2+z_{d+2}^2}\bigg)^m,
\end{gather*}
and therefore we get
\begin{gather*}
\big\langle \big(x_{d+1}+\sqrt{-1}x_{d+2}\big)^k\big(x_{d+1}-\sqrt{-1}x_{d+2}\big)^l,{\rm e}^{2q(x,\overline{z})}\big\rangle_{\lambda,x}
\\ \qquad
{}=\frac{\big(\lambda+\max\{k,l\}-\frac{d}{2}\big)_{\min\{k,l\}}} {(\lambda)_{k+l}\big(\lambda-\frac{d}{2}\big)_{\min\{k,l\}}}
\big(z_{d+1}+\sqrt{-1}z_{d+2}\big)^k\big(z_{d+1}-\sqrt{-1}z_{d+2}\big)^l
\\ \qquad\hphantom{=}
{} \ \times{}_2F_1\left(\begin{matrix} -k,-l \\ -\lambda-k-l+\frac{d}{2}+1\end{matrix};-\frac{q(z')}{z_{d+1}^2+z_{d+2}^2}\right)\!.
\end{gather*}

For {\it Case} 3, for $z=\left(\begin{smallmatrix}z_{11}&z_{12}\\z_{21}&z_{22}\end{smallmatrix}\right)\in M(r,\BC)=M(r',\BC)\oplus M(r',r'';\BC)\oplus M(r'',r';\BC)
\oplus M(r'',\BC)$, since
\begin{gather*}
Q(z_{12}){}^t\hspace{-1pt}z_{11}^\itinv=\begin{pmatrix}0&z_{12}\\z_{21}&0\end{pmatrix}
\begin{pmatrix}z_{11}^{-1}&0\\0&0\end{pmatrix}\begin{pmatrix}0&z_{12}\\z_{21}&0\end{pmatrix}
=\begin{pmatrix}0&0\\0&z_{21}z_{11}^{-1}z_{12}\end{pmatrix}\!, \\
\tilde{\Phi}_\bm^{\fp^+_{22}}\big(Q(z_{12}){}^t\hspace{-1pt}z_{11}^\itinv,{}^t\hspace{-1pt}z_{22}^\itinv\big)
=\tilde{\Phi}_\bm^{(2)}\big(z_{21}z_{11}^{-1}z_{12}z_{22}^{-1}\big) =\tilde{\Phi}_\bm^{(2)}\big(z_{11}^{-1}z_{12}z_{22}^{-1}z_{21}\big),
\end{gather*}
where $\tilde{\Phi}_\bm^{(2)}$ is as in (\ref{Phi^(d)}), we have
\begin{gather*}
\big\langle \det(x_{11})^k\det(x_{22})^l,{\rm e}^{\tr(xz^*)}\big\rangle_{\lambda,x}
\\ \qquad
{}=\frac{(\lambda+l-r')_{\underline{k}_{r''},2}}
{(\lambda)_{\underline{k+l}_{r''},2}(\lambda-r'')_{\underline{k}_{r'},2}}\det(z_{11})^k\det(z_{22})^l
{}_2F_1^{(2)}\left(\begin{matrix} -k,-l
\\
-\lambda\!-\!k\!-\!l\!+\!r'\!+\!r''\end{matrix};z_{11}^{-1}z_{12}z_{22}^{-1}z_{21}\right)\!,
\end{gather*}
where ${}_2F_1^{(2)}$ is as in (\ref{2F1^(d)}).
This result is extended to non-tube type case. Let $q=q'+q''$, $s=s'+s''$, and assume $q'\le s'$, $q''\le s''$.
For $z=\left(\begin{smallmatrix}z_{11}&z_{12}\\z_{21}&z_{22}\end{smallmatrix}\right)\in M(q,s;\BC)=M(q',s';\BC)\oplus M(q',s'';\BC)\oplus M(q'',s';\BC)
\oplus M(q'',s'';\BC)$,
we write $z_{11}=(z_{11}',z_{11}'')\in M(q',\BC)\oplus M(q',s'-q';\BC)$,
$z_{12}=(z_{12}',z_{12}'')\in M(q',q'';\BC)\oplus M(q',s''-q'';\BC)$, $z_{21}=(z_{21}',z_{21}'')\in M(q'',q';\BC)\oplus M(q'',s'-q';\BC)$,
$z_{22}=(z_{22}',z_{22}'')\in M(q'',\BC)\oplus M(q'',s''-q'';\BC)$,
and $z'=\left(\begin{smallmatrix} z_{11}'&z_{12}'\\z_{21}'&z_{22}'\end{smallmatrix}\right)\in M(q'+q'',\BC)$.
Then for $w_{11}=(I_{q'},0)\in M(q',s';\BC)$, $w_{22}=(I_{q''},0)\in M(q'',s'';\BC)$, by Corollary~\ref{inner_sub} we have
\begin{gather*}
\big\langle \det\big(x_{11}{}^t\hspace{-1pt}w_{11}\big)^k\det\big(x_{22}{}^t\hspace{-1pt}w_{22}\big)^l,{\rm e}^{\tr(xz^*)}\big\rangle_{\lambda,x,M(q,s;\BC)}
\\ \qquad
{}=\big\langle \det(x_{11}')^k\det(x_{22}')^l,{\rm e}^{\tr(x'(z')^*)}\big\rangle_{\lambda,x',M(q'+q'',\BC)}
\\ \qquad
{}=\frac{(\lambda+l-q')_{\underline{k}_{q''},2}}
{(\lambda)_{\underline{k+l}_{q''},2}(\lambda-q'')_{\underline{k}_{q'},2}}
\det(z_{11}')^k\det(z_{22}')^l
\\ \qquad \hphantom{=}
{} \ \times {}_2F_1^{(2)}\left(\begin{matrix} -k,-l \\ -\lambda-k-l+q'+q''\end{matrix};
(z_{11}')^{-1}z_{12}'(z_{22}')^{-1}z_{21}'\right)
\\ \qquad
{}=\frac{(\lambda+l-q')_{\underline{k}_{q''},2}}
{(\lambda)_{\underline{k+l}_{q''},2}(\lambda-q'')_{\underline{k}_{q'},2}}
\det\big(z_{11}{}^t\hspace{-1pt}w_{11}\big)^k\det\big(z_{22}{}^t\hspace{-1pt}w_{22}\big)^l
\\ \qquad\hphantom{=}
{} \ \times {}_2F_1^{(2)}\left(\begin{matrix} -k,-l \\ -\lambda-k-l+q'+q''\end{matrix};
\big(z_{11}{}^t\hspace{-1pt}w_{11}\big)^{-1}z_{12}{}^t\hspace{-1pt}w_{22}\big(z_{22}{}^t\hspace{-1pt} w_{22}\big)^{-1}z_{21}{}^t\hspace{-1pt}w_{11}\right)\!.
\end{gather*}
By the ${\rm GL}(s',\BC)\times {\rm GL}(s'',\BC)$-equivariance,
this holds for all $w_{11}\in M(q',s';\BC)$ and $w_{22}\in M(q'',s'';\BC)$.
When $q'\le s'$, $q''\ge s''$, a similar formula is also proved by reducing to the inner product on $M(q'+s'',\BC)$.

For {\it Case} 6, since $\fp^+_{22}=M(1,5;\BC)$ is not of tube type, the determinant polynomial on $\fp^+_{22}$ is not defined.
Instead we consider the polynomial $f(x_{22})=\big(x_{22}{}^t\hspace{-1pt}w_{22}\big)^l\in\cP_l(M(1,5;\BC))$ with $w_{22}\in M(1,5;\BC)$.
Then we have
\[ \tilde{f}_{m,l-m}(x_{22},y_{22})=\frac{(-1)^m(-l)_m}{m!}\big(x_{22}{}^t\hspace{-1pt}w_{22}\big)^m\big(y_{22}{}^t\hspace{-1pt}w_{22}\big)^{l-m}, \]
and for $z_{11}\in\BC$, $z_{12}\in\Skew(5,\BC)$, we have $Q(z_{12}){}^t\hspace{-1pt}z_{11}^\itinv=\frac{1}{z_{11}}\mathbf{Pf}(z_{12})$,
where $\mathbf{Pf}(z_{12})$ is defined as in (\ref{boldPf}). Therefore we get
\begin{gather*}
\big\langle x_{11}^k\big(x_{22}{}^t\hspace{-1pt}w_{22}\big)^l,{\rm e}^{(x|\overline{z})_{\fp^+}}\big\rangle_{\lambda,x}
\\ =\frac{(\lambda-3)_k\,z_{11}^k}{(\lambda)_{(k+l,k),6}(\lambda-3)_l}
\sum_{m=0}^l(-k)_m(\lambda+k-3)_{l-m}\frac{(-1)^m(-l)_m}{m!}
\bigg(\frac{\mathbf{Pf}(z_{12}){}^t\hspace{-1pt}w_{22}}{z_{11}}\bigg)^m
\big(z_{22}{}^t\hspace{-1pt}w_{22}\big)^{l-m}
\\
=\frac{(\lambda+k-3)_l}{(\lambda)_{k+l}(\lambda-3)_l}
z_{11}^k\big(z_{22}{}^t\hspace{-1pt}w_{22}\big)^l\sum_{m=0}^l\frac{(-k)_m(-l)_m}{(-\lambda-k-l+4)_mm!}
\bigg(\frac{\mathbf{Pf}(z_{12}){}^t\hspace{-1pt}w_{22}}{z_{11}(z_{22}{}^t\hspace{-1pt}w_{22})}\bigg)^m\!.
\end{gather*}
We~summarize the results for $f(x_{22})=\det_{\fn^+_{22}}(x_{22})^l$.
Here ${}_2F_1^{(d)}$ and ${}_2F_1^{[d]}$ are as in (\ref{2F1^(d)}) and~(\ref{2F1^[d]}).

\begin{Theorem}\label{explicit_nonsimple}\quad
\begin{enumerate}\itemsep=0pt
\item[$1.$] Let $\big(\fp^+,\fp^+_{11},\fp^+_{12},\fp^+_{22}\big)=\big(\BC^{d+2},\BC,\BC^d,\BC\big)$. Then for $k,l\in\BZ_{\ge 0}$ and for
$z=(z',z_{d+1},z_{d+2})\allowbreak\in\fp^+$, we have
\begin{gather*}
\big\langle \big(x_{d+1}+\sqrt{-1}x_{d+2}\big)^k\big(x_{d+1}-\sqrt{-1}x_{d+2}\big)^l,{\rm e}^{2q(x,\overline{z})}\big\rangle_{\lambda,x}
\\ \qquad
{}=\frac{\big(\lambda+\max\{k,l\}-\frac{d}{2}\big)_{\min\{k,l\}}} {(\lambda)_{k+l}\big(\lambda-\frac{d}{2}\big)_{\min\{k,l\}}}
\big(z_{d+1}+\sqrt{-1}z_{d+2}\big)^k\big(z_{d+1}-\sqrt{-1}z_{d+2}\big)^l
\\ \qquad\hphantom{=}
{}\ \times {}_2F_1\left(\begin{matrix} -k,-l \\ -\lambda-k-l+\frac{d}{2}+1\end{matrix};-\frac{q(z')}{z_{d+1}^2+z_{d+2}^2}\right)\!.
\end{gather*}

\item[$2.$] Let $\big(\fp^+,\fp^+_{11},\fp^+_{12},\fp^+_{22}\big)=({\rm Sym}(r,\BC), \Sym(r',\BC), M(r',r'';\BC),
\Sym(r'',\BC))$ with $r=r'+r''$. Then for $k,l\in\BZ_{\ge 0}$ and for $z=\left(\begin{smallmatrix}z_{11}&z_{12}\\{}^t\hspace{-1pt}z_{12}&z_{22}\end{smallmatrix}\right)\in\fp^+$, we have
\begin{gather*}
\big\langle \det(x_{11})^k\det(x_{22})^l,{\rm e}^{\tr(xz^*)}\big\rangle_{\lambda,x}
\\ \qquad
{}=\begin{cases} \ds \frac{\big(\lambda+\max\{k,l\}-\frac{1}{2}r'\big)_{\underline{\min\{k,l\}}_{r''},1}}
{(\lambda)_{\underline{k+l}_{r''},1}\big(\lambda-\frac{1}{2}r''\big)_{\underline{k}_{r'-r''},1}
\big(\lambda-\frac{1}{2}r'\big)_{\underline{\min\{k,l\}}_{r''},1}} & (r'\ge r'' )
\\
\ds \frac{\big(\lambda+\max\{k,l\}-\frac{1}{2}r''\big)_{\underline{\min\{k,l\}}_{r'},1}}
{(\lambda)_{\underline{k+l}_{r'},1}\big(\lambda-\frac{1}{2}r'\big)_{\underline{l}_{r''-r'},1}
\big(\lambda-\frac{1}{2}r''\big)_{\underline{\min\{k,l\}}_{r'},1}} &(r'\le r'')
\end{cases}
\\ \qquad\hphantom{=}
{} \ \times\det(z_{11})^k\det(z_{22})^l
{}_2F_1^{(1)}\left(\begin{matrix} -k,-l \\ -\lambda-k-l+\frac{r+1}{2}\end{matrix};
z_{11}^{-1}z_{12}z_{22}^{-1}{}^t\hspace{-1pt}z_{12}\right)\!.
\end{gather*}

\item[$3.$] Let $\big(\fp^+,\fp^\pm_{11},\fp^+_{12},\fp^\pm_{22}\big)=(M(q,s;\BC), M(q',s';\BC), M(q',s'';\BC)\oplus M(q'',s';\BC), M(q'',s'';\BC))$
with $q=q'+q''$, $s=s'+s''$, $q'\le s'$. Then for $k,l\in\BZ_{\ge 0}$ and
for $z=\left(\begin{smallmatrix}z_{11}&z_{12}\\z_{21}&z_{22}\end{smallmatrix}\right)\in\fp^+$, $w_{11}\in\fp^-_{11}$, $w_{22}\in\fp^-_{22}$,
if $q''\le s''$, then we have
\begin{gather*}
\big\langle \det\big(x_{11}{}^t\hspace{-1pt}w_{11}\big)^k\det\big(x_{22}{}^t\hspace{-1pt}w_{22}\big)^l,{\rm e}^{\tr(xz^*)}\big\rangle_{\lambda,x}
\\ \qquad
{}=\begin{cases}
\ds \frac{\big(\lambda+\max\{k,l\}-q'\big)_{\underline{\min\{k,l\}}_{q''},2}}
{(\lambda)_{\underline{k+l}_{q''},2}(\lambda-q'')_{\underline{k}_{q'-q''},2}
(\lambda-q')_{\underline{\min\{k,l\}}_{q''},2}} &(q'\ge q'')
\\
\ds \frac{(\lambda+\max\{k,l\}-q'')_{\underline{\min\{k,l\}}_{q'},2}}
{(\lambda)_{\underline{k+l}_{q'},2}(\lambda-q')_{\underline{l}_{q''-q'},2}
(\lambda-q'')_{\underline{\min\{k,l\}}_{q'},2}} &(q'\le q'')
\end{cases}
\\ \qquad\hphantom{=}
{} \ \times \det\big(z_{11}{}^t\hspace{-1pt}w_{11}\big)^k\det\big(z_{22}{}^t\hspace{-1pt}w_{22}\big)^l
\\ \qquad\hphantom{=}
{} \ \times {}_2F_1^{(2)}\left(\begin{matrix} -k,-l \\ -\lambda-k-l+q'+q''\end{matrix};
\big(z_{11}{}^t\hspace{-1pt}w_{11}\big)^{-1}z_{12}{}^t\hspace{-1pt }w_{22}\big(z_{22}{}^t\hspace{-1pt}w_{22}\big)^{-1}z_{21}{}^t\hspace{-1pt}w_{11}\right)\!.
\end{gather*}
Similarly, if $q''\ge s''$, then we have
\begin{gather*}
\big\langle \det\big(x_{11}{}^t\hspace{-1pt}w_{11}\big)^k\det\big({}^t\hspace{-1pt}w_{22}x_{22}\big)^l,{\rm e}^{\tr(xz^*)}\big\rangle_{\lambda,x}
\\ \qquad
{}=\begin{cases}
\ds \frac{(\lambda+\max\{k,l\}-q')_{\underline{\min\{k,l\}}_{s''},2}}
{(\lambda)_{\underline{k+l}_{s''},2}(\lambda-s'')_{\underline{k}_{q'-s''},2}
(\lambda-q')_{\underline{\min\{k,l\}}_{s''},2}} & (q'\ge s'')
\\
\ds \frac{(\lambda+\max\{k,l\}-s'')_{\underline{\min\{k,l\}}_{q'},2}}
{(\lambda)_{\underline{k+l}_{q'},2}\left(\lambda-q'\right)_{\underline{l}_{s''-q'},2}
(\lambda-s'')_{\underline{\min\{k,l\}}_{q'},2}} & (q'\le s'')
\end{cases}
\\ \qquad\hphantom{=}
{} \ \times \det\big(z_{11}{}^t\hspace{-1pt}w_{11}\big)^k\det\big({}^t\hspace{-1pt}w_{22}z_{22}\big)^l
\\ \qquad\hphantom{=}
{} \ \times {}_2F_1^{(2)}\left(\begin{matrix} -k,-l \\ -\lambda-k-l+q'+s''\end{matrix};
\big(z_{11}{}^t\hspace{-1pt}w_{11}\big)^{-1}z_{12}\big({}^t\hspace{-1pt}w_{22}z_{22}\big)^{-1}{}^t\hspace{-1pt} w_{22}z_{21}{}^t\hspace{-1pt}w_{11}\right)\!.
\end{gather*}
\item[$4.$] Let $\big(\fp^+,\fp^+_{11},\fp^+_{12},\fp^+_{22}\big)=({\rm Skew}(2r,\BC), \Skew(2r',\BC), M(2r',2r'';\BC), \Skew(2r'',\BC))$
with $r=r'+r''$. Then for $k,l\in\BZ_{\ge 0}$ and
for $z=\left(\begin{smallmatrix}z_{11}&z_{12}\\-{}^t\hspace{-1pt}z_{12}&z_{22}\end{smallmatrix}\right)\in\fp^+$, we have
\begin{gather*}
\big\langle \Pf(x_{11})^k\Pf(x_{22})^l,{\rm e}^{\frac{1}{2}\tr(xz^*)}\big\rangle_{\lambda,x}
\\ \qquad
{}=\begin{cases}
\ds \frac{(\lambda+\max\{k,l\}-2r')_{\underline{\min\{k,l\}}_{r''},4}}
{(\lambda)_{\underline{k+l}_{r''},4}(\lambda-2r'')_{\underline{k}_{r'-r''},4}
(\lambda-2r')_{\underline{\min\{k,l\}}_{r''},4}} & (r'\ge r'')
\\
\ds \frac{(\lambda+\max\{k,l\}-2r'')_{\underline{\min\{k,l\}}_{r'},4}}
{(\lambda)_{\underline{k+l}_{r'},4}(\lambda-2r')_{\underline{l}_{r''-r'},4}
(\lambda-2r'')_{\underline{\min\{k,l\}}_{r'},4}} & (r'\le r'')
\end{cases}
\\ \qquad\hphantom{=}
{} \ \times\Pf(z_{11})^k\Pf(z_{22})^l{}_2F_1^{[4]}\left(\begin{matrix} -k,-l \\ -\lambda-k-l+2r-1\end{matrix};
-z_{11}^{-1}z_{12}z_{22}^{-1}{}^t\hspace{-1pt}z_{12}\right)\!.
\end{gather*}

\item[$5.$]
Let $\big(\fp^+,\fp^+_{11},\fp^+_{12},\fp^+_{22}\big)=\big({\rm Herm}(3,\BO)^\BC, \BC, M(1,2;\BO)^\BC,\Herm(2,\BO)^\BC\big)$. Then for $k,l\in\BZ_{\ge 0}$ and for $z=\left(\begin{smallmatrix}z_{11}&z_{12}\\{}^t\hspace{-1pt}\hat{z}_{12}&z_{22}\end{smallmatrix}\right) \in\fp^+\rule[0pt]{0pt}{11pt}$, we have
\begin{align*}
\big\langle x_{11}^k\det(x_{22})^l,{\rm e}^{(x|\overline{z})_{\fp^+}}\big\rangle_{\lambda,x}
={}&\frac{(\lambda+\max\{k,l\}-8)_{\min\{k,l\}}}{(\lambda)_{k+l}(\lambda-4)_{l}(\lambda-8)_{\min\{k,l\}}} z_{11}^k\det(z_{22})^l
\\
&{} \times {}_2F_1\left(\begin{matrix} -k,-l \\ -\lambda-k-l+9\end{matrix};
\frac{\Re_\BO\left(z_{12}(z_{22}^{-1}{}^t\hspace{-1pt}\hat{z}_{12})\right)}{z_{11}}\right)\!.
\end{align*}

\item[$6.$] Let $\big(\fp^+,\fp^+_{11},\fp^+_{12},\fp^\pm_{22}\big)=\big(M(1,2;\BO)^\BC,\BC,\Skew(5,\BC),M(1,5;\BC)\big)$. Then for $k,l\in\BZ_{\ge 0}$
and for $z=(z_{11},z_{12},z_{22})\in\fp^+$, $w_{22}\in\fp^-_{22}$, we have
\begin{align*}
\big\langle x_{11}^k\big(x_{22}{}^t\hspace{-1pt}w_{22}\big)^l,{\rm e}^{(x|\overline{z})_{\fp^+}}\big\rangle_{\lambda,x}
={}&\frac{(\lambda+\max\{k,l\}-3)_{\min\{k,l\}}}{(\lambda)_{k+l}(\lambda-3)_{\min\{k,l\}}}
z_{11}^k\big(z_{22}{}^t\hspace{-1pt}w_{22}\big)^l
\\
&{} \times {}_2F_1\left(\begin{matrix} -k,-l \\ -\lambda-k-l+4\end{matrix};
\frac{\mathbf{Pf}(z_{12}){}^t\hspace{-1pt}w_{22}}{z_{11}(z_{22}{}^t\hspace{-1pt}w_{22})}\right)\!,
\end{align*}
where $\mathbf{Pf}(z_{12})$ is defined as in \eqref{boldPf}.
\end{enumerate}
\end{Theorem}

\section{Computation for simple $\fp^+_2$}\label{section_simple}

In this section we treat $\big(\fp^+,\fp^+_1,\fp^+_2\big)$ with $\fp^+=\fp^+_1\oplus\fp^+_2$ such that $\fp^+_2$ is simple.
Let $\dim\fp^+=:n$, $\dim\fp^+_2=:n_2$, $\rank\fp^+=:r$, $\rank\fp^+_2=:r_2$, let $d$, $d_2$ be the numbers defined in (\ref{str_const})
for $\fp^+$, $\fp^+_2$ respectively, and let $\varepsilon_2\in\{1,2\}$ be as in (\ref{epsilon}).

\subsection{Preliminaries for main theorem}

First we consider the cases such that $\fp^+_2(e)_2=\fp^+(e)_2$ holds for some (or equivalently any) maximal tripotent $e\in\fp^+_2$, i.e.,
$\big(\fp^+,\fp^+_1,\fp^+_2\big)$ is one of
\[
\big(\fp^+,\fp^+_1,\fp^+_2\big)= \begin{cases}
\big(M(q,s;\BC),M(q,s';\BC),M(q,s'';\BC)\big), \qquad s'+s''=s,
\\
\big({\rm Skew}(s,\BC),\Skew(s-1,\BC),M(1,s-1;\BC)\big),
\\
\big({\rm Skew}(s,\BC),M(1,s-1;\BC),\Skew(s-1,\BC)\big),
\\
\big(\BC^{2s},\big(\big(1,\sqrt{-1}\big)\BC\big)^s,\big(\big(1,-\sqrt{-1}\big)\BC\big)^s\big)\simeq\big(\BC^{2s},M(1,s;\BC),M(1,s;\BC)\big),
\\
\big(M(1,2;\BO)^\BC,\BO^\BC,\BO^\BC\big)\simeq\big(M(1,2;\BO)^\BC,\BC^8,\BC^8\big), \end{cases}
\]
so that the corresponding symmetric pairs are
\[
(G,G_1)= \begin{cases}
\big({\rm SU}(q,s),{\rm S}({\rm U}(q,s')\times {\rm U}(s''))\big),
\\
\big({\rm SO}^*(2s),{\rm SO}(2)\times {\rm SO}^*(2(s-1))\big),
\\
\big({\rm SO}^*(2s),{\rm U}(1,s-1)\big),
\\
\big({\rm SO}_0(2,2s),{\rm U}(1,s)\big),
\\
\big(E_{6(-14)},{\rm U}(1)\times \operatorname{Spin}_0(2,8)\big). \end{cases}
\]
Then for $\bk\in\BZ_{++}^{r_2}$, since $\cP_\bk(\fp^+_2)$ and $\cP_\bk(\fp^+)$ are generated by
$\cP_\bk\big(\fp^+_2(e)_2\big)=\cP_\bk\big(\fp^+(e)_2\big)$ as $K^\BC_1$- and $K^\BC$-modules respectively,
$\cP_\bk\big(\fp^+_2\big)\subset\cP_\bk\big(\fp^+\big)$ holds, and therefore by Corollary~\ref{FKCor}, for $f(x_2)\in\cP_\bk(\fp^+_2)$ we have
\begin{equation}\label{easycase}
\big\langle f(x_2),{\rm e}^{(x|\overline{z})_{\fp^+}}\big\rangle_{\lambda,x} =\frac{1}{(\lambda)_{\bk,d}}f(z_2).
\end{equation}
Thus in the following we assume $\fp^+_2(e)_2\subsetneq \fp^+(e)_2$ holds for some (or equivalently any) maximal tripotent $e\in\fp^+_2$.
That is, we treat the cases
\begin{gather*}
\big(\fp^+,\fp^+_1,\fp^+_2\big)= \begin{cases}
\big(\BC^n,\BC^{n'},\BC^{n''}\big) & (\textit{Case }1), \\
\big({\rm Sym}(r,\BC),\Sym(r',\BC)\oplus\Sym(r'',\BC),M(r',r'';\BC)\big) & (\textit{Case }2), \\
\big({\rm Skew}(s,\BC),\Skew(s',\BC)\oplus\Skew(s'',\BC),M(s',s'';\BC)\big) & (\textit{Case }3), \\
\big(M(r,\BC),\Skew(r,\BC),\Sym(r,\BC)\big) & (\textit{Case }4), \\
\big(M(r,\BC),\Sym(r,\BC),\Skew(r,\BC)\big) & (\textit{Case }5), \\
\big({\rm Herm}(3,\BO)^\BC,M(2,6;\BC),\Skew(6,\BC)\big) & (\textit{Case }6), \\
%\\\end{cases} \\ &\eqspace \begin{cases} %%% Modify if necessary.
\big({\rm Herm}(3,\BO)^\BC,\Skew(6,\BC),M(2,6;\BC)\big) & (\textit{Case }7), \\
\big({\rm Herm}(3,\BO)^\BC,\BC\oplus\Herm(2,\BO)^\BC,M(1,2;\BO)^\BC\big) & (\textit{Case }8), \\
\big(M(1,2;\BO)^\BC,M(2,4;\BC),M(4,2;\BC)\big) & (\textit{Case }9), \\
\big(M(1,2;\BO)^\BC,\BC\oplus M(1,5;\BC),\Skew(5,\BC)\big) & (\textit{Case }10) \end{cases}
\end{gather*}
($n=n'+n''$, $n\ge 3$, $n''\ne 2$ for {\it Case} 1, $r=r'+r''$ for {\it Case} 2, and $s=s'+s''$, $s',s''\ge 2$ for {\it Case} 3).
Then the corresponding symmetric pairs are
\[ (G,G_1)= \begin{cases}
\big({\rm SO}_0(2,n),{\rm SO}_0(2,n')\times {\rm SO}(n'')\big) & (\textit{Case }1), \\
\big({\rm Sp}(r,\BR),\operatorname{Sp}(r',\BR)\times \operatorname{Sp}(r'',\BR)\big) & (\textit{Case }2), \\
\big({\rm SO}^*(2s),{\rm SO}^*(2s')\times {\rm SO}^*(2s'')\big) & (\textit{Case }3), \\
\big({\rm SU}(r,r),{\rm SO}^*(2r)) & (\textit{Case }4\big), \\
\big({\rm SU}(r,r),\operatorname{Sp}(r,\BR)\big) & (\textit{Case }5), \\
\big(E_{7(-25)},{\rm SU}(2,6)\big) & (\textit{Case }6), \\
\big(E_{7(-25)},{\rm SU}(2)\times {\rm SO}^*(12)\big) & (\textit{Case }7), \\
\big(E_{7(-25)},{\rm SL}(2,\BR)\times \operatorname{Spin}_0(2,10)\big) & (\textit{Case }8), \\
\big(E_{6(-14)},{\rm SU}(2,4)\times {\rm SU}(2)\big) & (\textit{Case }9), \\
\big(E_{6(-14)},{\rm SL}(2,\BR)\times {\rm SU}(1,5)\big) & (\textit{Case }10) \end{cases} \]
(up to covering), and the numbers $(r,r_2,d,d_2,\varepsilon_2)$ are given by
\[ (r,r_2,d,d_2,\varepsilon_2)= \begin{cases}
(2,2,n-2,n''-2,1) & (\textit{Case }1,\; n''\ge 3), \\
(2,1,n-2,-,2) & (\textit{Case }1,\; n''=1), \\
(r,\min\{r',r''\},1,2,2) & (\textit{Case }2), \\
(\lfloor s/2\rfloor, \min\{s',s''\}, 4,2,1) & (\textit{Case }3), \\
(r,r,2,1,1) & (\textit{Case }4), \\
(r,\lfloor r/2\rfloor, 2,4,2) & (\textit{Case }5), \\
(3,3,8,4,1) & (\textit{Case }6), \\
(3,2,8,2,1) & (\textit{Case }7), \\
(3,2,8,6,1) & (\textit{Case }8), \\
(2,2,6,2,1) & (\textit{Case }9), \\
(2,2,6,4,1) & (\textit{Case }10). \end{cases} \]
We~note that if $r_2=1$, then $d_2$ is not determined uniquely, and any number is allowed.
We~also write
\[ (\fp^+_{11},\fp^+_{12},\fp^+_{22})= \begin{cases}
\big({\rm Sym}(r',\BC),M(r',r'';\BC),\Sym(r'',\BC)\big) & (\textit{Case }2), \\
\big({\rm Skew}(s',\BC),M(s',s'';\BC),\Skew(s'',\BC)\big) & (\textit{Case }3), \\
\big(\BC,M(1,2;\BO)^\BC,\Herm(2,\BO)^\BC\big) & (\textit{Case }8), \\
\big(\BC,\Skew(5,\BC),M(1,5;\BC)\big) & (\textit{Case }10). \end{cases} \]
First we prove the following.

\begin{Proposition}\label{prop_simple_prelim}\quad
\begin{enumerate}\itemsep=3pt
\item[$1.$] Let $\varepsilon_2=1$. Then for $l\in\BZ_{\ge 0}$, $\cP_{(l,0,\dots,0)}\big(\fp^\pm_2\big)\subset\cP_{(l,0,\dots,0)}(\fp^\pm)$ holds.
\item[$2.$] Let $\varepsilon_2=2$. Then for $0\le l\le r_2$, $\cP_{\underline{1}_l}\big(\fp^\pm_2\big)
=\cP_{(\smash{\overbrace{\scriptstyle{1,\dots,1}}^l},0,\dots,0)}\big(\fp^\pm_2\big) \subset\cP_{\underline{1}_l}\big(\fp^\pm\big)$ holds.
\end{enumerate}
\end{Proposition}

\begin{proof}
(1) For {\it Case} 1, since $\cP_{(l,0)}(\BC^{n''})$, $\cP_{(l,0)}(\BC^n)$ coincide with the spaces of homogeneous harmonic polynomials of degree $l$
of $n'', n$ variables respectively, $\cP_{(l,0)}(\BC^{n''})\subset\cP_{(l,0)}(\BC^n)$ holds.
For other cases, we take a tripotent $e\in\fp^+_2$ of rank 2, and set $\fp^+(e)_2=:\fp^{+\prime}$, $\fp^+_2(e)_2=:\fp^{+\prime}_2$.
Then these are of rank 2 and we have $\fp^{+\prime}\simeq \BC^{d+2}$, $\fp^{+\prime}_2\simeq \BC^{d_2+2}$, and by {\it Case} 1,
$\cP_{(l,0)}(\fp^{+\prime}_2)\subset\cP_{(l,0)}(\fp^{+\prime})$ holds.
Since $\cP_{(l,0,\dots,0)}(\fp^+)$ and $\cP_{(l,0,\dots,0)}\big(\fp^+_2\big)$ are generated by $\cP_{(l,0)}(\fp^{+\prime})$
and $\cP_{(l,0)}\big(\fp^{+\prime}_2\big)$ as $K^\BC$- and $K_1^\BC$-modules respectively, we have
$\cP_{(l,0,\dots,0)}(\fp^+_2)\subset\cP_{(l,0,\dots,0)}(\fp^+)$.

(2) For {\it Case} 1 with $n''=1$, this is proved as in $n''\ge 3$ cases. For {\it Case} 2, for $0\le l\le \min\{r',r''\}$, $\bm\in\BZ_{++}^r$,
we consider the space of abstract homomorphisms
\begin{gather*}
\Hom_{{\rm U}(r')\times {\rm U}(r'')}\big(\cP_{\underline{1}_l}(M(r',r'';\BC)),\cP_\bm({\rm Sym}(r,\BC))\big)
\\ \qquad
{}\simeq \Hom_{{\rm U}(r')\times {\rm U}(r'')}\big(V_{\underline{1}_l}^{(r')\vee}\boxtimes V_{\underline{1}_l}^{(r'')\vee},
V_{2\bm}^{(r)\vee}\bigr|_{{\rm U}(r')\times {\rm U}(r'')}\big)
\\ \qquad
{}\simeq \Hom_{{\rm U}(r)}\big(V_{2\bm}^{(r)\vee}, V_{\underline{1}_l}^{(r)\vee}\otimes V_{\underline{1}_l}^{(r)\vee}\big),
\end{gather*}
where the last isomorphism follows from \cite[Theorem 9.2.3]{GW}. Then since the target space is decomposed as
\[
V_{\underline{1}_l}^{(r)\vee}\otimes V_{\underline{1}_l}^{(r)\vee}
\simeq \bigoplus_{j=0}^l V_{(\underline{2}_j,\underline{1}_{2l-2j})}^{(r)\vee},
\]
the above space is non-zero only when $\bm=\underline{1}_l$. Finally, for {\it Case} 5, for $0\le l\le \lfloor r/2\rfloor$, $\bm\in\BZ_{++}^r$,
\[
\Hom_{{\rm U}(r)}\left(\cP_{\underline{1}_l}({\rm Skew}(r,\BC)), \cP_\bm(M(r,\BC))\right)
\simeq \Hom_{{\rm U}(r)}\big(V_{\underline{1}_{2l}}^{(r)\vee}, V_{\bm}^{(r)\vee}\otimes V_{\bm}^{(r)\vee}\big)\ne\{0\}
\]
implies $2|\bm|=|\underline{1}_{2l}|=2l$ and $m_j\le (\underline{1}_{2l})_j$ for $1\le j\le r$,
that is, we have $\bm=\underline{1}_l$.
\end{proof}

From now on we assume that $\fp^+$ and $\fp^+_2$ are of tube type, that is, we consider {\it Case} 1, {\it Case}~2 with $r'=r''$, {\it Case} 3 with $s'=s''$,
{\it Case} 4, {\it Case} 5 with $r=2s$, and {\it Case} 6.
We~fix a~maximal tripotent $e\in\fp^+_2\subset\fp^+$, regard $\fp^+$, $\fp^+_2$ as Jordan algebras,
and let $\Omega=\Omega^+\subset\fn^+\subset\fp^+$, $\Omega_2=\Omega_2^+\subset\fn^+_2\subset\fp^+_2$ be the corresponding Euclidean real forms
and the symmetric cones. We~also write $\Omega^-:=Q(\overline{e})\Omega\subset\fn^-\subset\fp^-$,
$\Omega^-_2:=Q(\overline{e})\Omega_2\subset\fn^-_2\subset\fp^-_2$.
Then by Theorem \ref{key_identity}(1),
when $\varepsilon_2=1$, for $f(x_2)\in\cP_{(l,0,\dots,0)}(\fp^+_2)$, $z=z_1+z_2\in\Omega$ we have
\begin{gather*}
\big\langle \det_{\fn^+_2}(x_2)^kf(x_2),{\rm e}^{(x|\overline{z})_{\fp^+}}\big\rangle_{\lambda,x}
\\ \qquad
{}=\frac{1}{(\lambda)_{\underline{2k}_r+(l,0,\dots,0),d}}\det_{\fn^+}(z)^{-\lambda+\frac{n}{r}}
\det_{\fn^+_2}\bigg(\frac{\partial}{\partial z_2}\bigg)^k\det_{\fn^+}(z)^{\lambda+2k-\frac{n}{r}}f(z_2) \\ \qquad
{}=\frac{1}{(\lambda)_{(2k+l,\underline{2k}_{r-1}),d}}\det_{\fn^+}(z)^{-\lambda+\frac{n}{r}}
\det_{\fn^+_2}\bigg(\frac{\partial}{\partial z_2}\bigg)^k\det_{\fn^+}(z)^{\lambda+2k-\frac{n}{r}}f(z_2),
\end{gather*}
and when $\varepsilon_2=2$, for $f(x_2)\in\cP_{\underline{1}_l}\big(\fp^+_2\big)$, $z=z_1+z_2\in\Omega$ we have
\begin{gather*}
\big\langle \det_{\fn^+_2}(x_2)^kf(x_2),{\rm e}^{(x|\overline{z})_{\fp^+}}\big\rangle_{\lambda,x}
\\ \qquad
{}=\frac{1}{(\lambda)_{\underline{k}_r+\underline{1}_l,d}}\det_{\fn^+}(z)^{-\lambda+\frac{n}{r}}
\det_{\fn^+_2}\bigg(\frac{1}{2}\frac{\partial}{\partial z_2}\bigg)^k\det_{\fn^+}(z)^{\lambda+k-\frac{n}{r}}f(z_2)
\\ \qquad
{}=\frac{1}{(\lambda)_{(\underline{k+1}_l,\underline{k}_{r-l}),d}}\det_{\fn^+}(z)^{-\lambda+\frac{n}{r}}
\det_{\fn^+_2}\bigg(\frac{1}{2}\frac{\partial}{\partial z_2}\bigg)^k\det_{\fn^+}(z)^{\lambda+k-\frac{n}{r}}f(z_2).
\end{gather*}
Also, by Proposition \ref{prop_Jordansub}(1), for $\mu\in\BC$ we have
\[ \det_{\fn^+}(z)^{-\mu}=\det_{\fn^+}(z_2)^{-\mu}h_{\fp^+}\big(z_1,Q(z_2)^{-1}z_1\big)^{-\mu/2}
=\det_{\fn^+_2}(z_2)^{-\varepsilon_2\mu}h_{\fp^+}\big(z_1,Q\big({}^t\hspace{-1pt}z_2^\itinv\big)z_1\big)^{-\mu/2}. \]
Next we consider the expansion of this.

\begin{Proposition}%\label{prop_Phi_mid}
For $(z_1,w_2)\in\fp^\pm_1\times\fp^\mp_2$, let $t_1,\dots,t_{\lfloor r/2\rfloor}$ be the roots of the polynomial
$t^{\lfloor r/2\rfloor}h_{\fp^\pm}\big(t^{-1}z_1,Q(w_2)z_1\big)^{1/2}\in\BC[t]$,
and for $\bm\in\BZ_{++}^{\lfloor r/2\rfloor}$,
define $\tilde{\Phi}_\bm^{\fp^\pm_1,\fp^\mp_2}(z_1,w_2)\in\cP\big(\fp^\pm_1\big)\otimes\cP\big(\fp^\mp_2\big)$ by
\begin{equation*}%\label{Phi_mid}
\tilde{\Phi}_\bm^{\fp^\pm_1,\fp^\mp_2}(z_1,w_2):=\tilde{\Phi}_\bm^d(t_1,\dots,t_{\lfloor r/2\rfloor}),
\end{equation*}
where $\tilde{\Phi}_\bm^d(t_1,\dots,t_{\lfloor r/2\rfloor})$ is as in \eqref{Phi^d}. Then for $|t_j|<1$,
\[
h_{\fp^\pm}(z_1,Q(w_2)z_1)^{-\mu/2}
=\sum_{\bm\in\BZ_{++}^{\lfloor r/2\rfloor}}(\mu)_{\bm,d}\tilde{\Phi}_\bm^{\fp^\pm_1,\fp^\mp_2}(z_1,w_2) \]
holds. Moreover,
\begin{enumerate}\itemsep=0pt
\item[$1.$] If $\varepsilon_2=1$, then $\tilde{\Phi}_\bm^{\fp^\pm_1,\fp^\mp_2}(z_1,w_2)\in\cP\big(\fp^\pm_1\big)\otimes\cP_{\bm^2}\big(\fp^\mp_2\big)$ holds.
\item[$2.$] If $\varepsilon_2=2$, then $\tilde{\Phi}_\bm^{\fp^\pm_1,\fp^\mp_2}(z_1,w_2)\in\cP\big(\fp^\pm_1\big)\otimes\cP_{2\bm}\big(\fp^\mp_2\big)$ holds.
\end{enumerate}
\end{Proposition}
Here we write $\bm^2=(m_1,m_1,m_2,m_2,\dots,m_{\lfloor r/2\rfloor},m_{\lfloor r/2\rfloor}(,0))$.
We~note that since for $(z_1,w_2)\allowbreak=(z_1,{}^t\hspace{-1pt}z_2^\itinv)$ with $z=z_1+z_2\in\Omega$ we have
\begin{gather*}
\det_{\fn^+}(z)=\det_{\fn^+_2}(z_2)^{\varepsilon_2}h_{\fp^+}\big(z_1,Q\big({}^t\hspace{-1pt}z_2^\itinv\big)z_1\big)^{1/2}
=\det_{\fn^+_2}(z_2)^{\varepsilon_2}\prod_{j=1}^{\lfloor r/2\rfloor}(1-t_j)>0, \\
h_{\fp^+}\big(z_1,Q\big({}^t\hspace{-1pt}z_2^\itinv\big)z_1\big)=h_{\fn^+}\big(z_1,P\big(z_2^\itinv\big)z_1\big)
=h_{\fn^+}\big(P\big(z_2^{-\mathit{1/2}}\big)z_1,P\big(z_2^{-\mathit{1/2}}\big)z_1\big),
\end{gather*}
and since $\Omega$ is connected, we have $0\le t_j<1$, and the above sum converges for $(z_1,w_2)=\big(z_1,{}^t\hspace{-1pt}z_2^\itinv\big)$.
\begin{proof}
First, for {\it Case} 1 with $n''\ge 3$, we have $r/2=1$, and for $z_1\in\fp^+_1=\BC^{n'}$, $w_2\in\fp^-_2=\BC^{n''}$, since
\[ th_{\fp^+}\big(t^{-1}z_1,Q(w_2)z_1\big)^{1/2} =t\bigg(1+\frac{2}{t}q(z_1)q(w_2)+\frac{1}{t^2}q(z_1)^2q(w_2)^2\bigg)^{1/2}
=t+q(z_1)q(w_2),
\]
we have $t=-q(z_1)q(w_2)$, and
\begin{gather*}
\tilde{\Phi}_m^{\fp^+_1,\fp^-_2}(z_1,w_2)=\frac{1}{m!}t^m=\frac{1}{m!}(-q(z_1)q(w_2))^m
\in\cP_{(m,m)}(\BC^{n'})_{z_1}\otimes\cP_{(m,m)}(\BC^{n''})_{w_2},
\\
h_{\fp^+}(z_1,Q(w_2)z_1)^{-\mu/2}=(1+q(z_1)q(w_2))^{-\mu} =\sum_{m=0}^\infty\frac{(\mu)_m}{m!}(-q(z_1)q(w_2))^m
\end{gather*}
hold. For $n''=1$ case this is also proved similarly.
Next we consider {\it Cases} 2, 3. For {\it Case} 2 we have $r=2r'$, $\varepsilon_2=2$, and for {\it Case} 3 we have $r=s'$, $\varepsilon_2=1$.
Then for $z_1=\left(\begin{smallmatrix} z_{11}&0\\ 0&z_{22}\end{smallmatrix}\right)\in\fp^+_1$,
$w_2=\left(\begin{smallmatrix} 0&w_{12}\\ \pm{}^t\hspace{-1pt}w_{12}&0\end{smallmatrix}\right)\in\fp^-_2$ with
\[
(z_{11},w_{12},z_{22})\in
\begin{cases}
\Sym(r',\BC)\oplus M(r',\BC)\oplus\Sym(r',\BC) & (\textit{Case }2),
\\
\Skew(s',\BC)\oplus M(s',\BC)\oplus\Skew(s',\BC) & (\textit{Case }3),
\end{cases}
\]
we have
\begin{gather*}
t^{\lfloor r/2\rfloor}h_{\fp^+}\big(t^{-1}z_1,Q(w_2)z_1\big)^{1/2}
\\ \qquad
{}=t^{\lfloor r/2\rfloor}\det\left(I-t^{-1}\begin{pmatrix}z_{11}& 0\\ 0& z_{22}\end{pmatrix}
\begin{pmatrix}0& \pm w_{12}\\ {}^t\hspace{-1pt}w_{12}& 0\end{pmatrix}\begin{pmatrix}z_{11}& 0\\ 0& z_{22}\end{pmatrix}
\begin{pmatrix}0& \pm w_{12}\\ {}^t\hspace{-1pt}w_{12}& 0\end{pmatrix}\right)^{\varepsilon_2/4}
\\ \qquad
{}=t^{\lfloor r/2\rfloor}\det\left(I-t^{-1}\begin{pmatrix}\pm z_{11}w_{12}z_{22}{}^t\hspace{-1pt}w_{12}& 0
\\
0& \pm z_{22}{}^t\hspace{-1pt}w_{12}z_{11}w_{12}\end{pmatrix}\right)^{\varepsilon_2/4}\\
\qquad
=t^{\lfloor r/2\rfloor}\det\left(I\mp t^{-1}z_{11}w_{12}z_{22}{}^t\hspace{-1pt}w_{12}\right)^{\varepsilon_2/2}
\\ \qquad
{}=\begin{cases}
\det\left(tI-z_{11}w_{12}z_{22}{}^t\hspace{-1pt}w_{12}\right) & (\textit{Case }2),
\\
\det\left(tI+z_{11}w_{12}z_{22}{}^t\hspace{-1pt}w_{12}\right)^{1/2} & (\textit{Case }3,\ s'\colon\text{even}),
\\
\left(t^{-1}\det\left(tI+z_{11}w_{12}z_{22}{}^t\hspace{-1pt}w_{12}\right)\right)^{1/2}& (\textit{Case }3,\ s'\colon\text{odd}). \end{cases}
\end{gather*}
Therefore we have
\begin{equation}\label{Phi_mid_Case23}
\tilde{\Phi}_\bm^{\fp^+_1,\fp^-_2}(z_1,w_2)=\begin{cases}
\tilde{\Phi}_\bm^{(1)}\left(z_{11}w_{12}z_{22}{}^t\hspace{-1pt}w_{12}\right)
=\tilde{\Phi}_\bm^{\Sym(r',\BC)}\left(z_{11},w_{12}z_{22}{}^t\hspace{-1pt}w_{12}\right) & (\textit{Case }2), \\
\tilde{\Phi}_\bm^{[4]}\left(-z_{11}w_{12}z_{22}{}^t\hspace{-1pt}w_{12}\right)
=\tilde{\Phi}_\bm^{\Skew(s',\BC)}\left(z_{11},w_{12}z_{22}{}^t\hspace{-1pt}w_{12}\right) & (\textit{Case }3), \end{cases}
\end{equation}
where $\tilde{\Phi}_\bm^{(1)}$, $\tilde{\Phi}_\bm^{[4]}$ and $\tilde{\Phi}_\bm^{\fp^+_{11}}$ (with $\fp^+_{11}=\Sym(r',\BC)$, $\Skew(s',\BC)$)
are as in (\ref{Phi^(d)}), (\ref{Phi^[d]}) and (\ref{Phi^p}) respectively, and
\begin{gather*}
h_{\fp^+}(z_1,Q(w_2)z_1)^{-\mu/2}=\det\big(I-z_{11}w_{12}z_{22}{}^t\hspace{-1pt}w_{12}\big)^{-\mu}
\\ \hphantom{h_{\fp^+}(z_1,Q(w_2)z_1)^{-\mu/2}}
{}=\sum_{\bm\in\BZ_{++}^{r'}}(\mu)_{\bm,1}\tilde{\Phi}_\bm^{(1)} \big(z_{11}w_{12}z_{22}{}^t\hspace{-1pt}w_{12}\big)
\hspace{13mm} (\textit{Case }2),
\\
h_{\fp^+}(z_1,Q(w_2)z_1)^{-\mu/2}=\det\big(I+z_{11}w_{12}z_{22}{}^t\hspace{-1pt}w_{12}\big)^{-\mu/2}
\\ \hphantom{h_{\fp^+}(z_1,Q(w_2)z_1)^{-\mu/2}}
{}=\sum_{\bm\in\BZ_{++}^{\lfloor s'/2\rfloor}}(\mu)_{\bm,4}\tilde{\Phi}_\bm^{[4]}\big({-}z_{11}w_{12}z_{22}{}^t\hspace{-1pt}w_{12}\big)
\qquad (\textit{Case }3).
\end{gather*}
Since as polynomials of $z_{11}$, $z_{22}$,
\begin{gather*}\begin{split}
&\tilde{\Phi}_\bm^{(1)}\big(z_{11}w_{12}z_{22}{}^t\hspace{-1pt}w_{12}\big)
=\tilde{\Phi}_\bm^{\Sym(r',\BC)}\big(z_{11},w_{12}z_{22}{}^t\hspace{-1pt}w_{12}\big)
\\ &\qquad
{}\in\cP_\bm({\rm Sym}(r',\BC))_{z_{11}}\boxtimes\cP_\bm({\rm Sym}(r',\BC))_{z_{22}}\simeq V_{2\bm}^{(r')\vee}\boxtimes V_{2\bm}^{(r')\vee}\hspace{10mm}
(\textit{Case }2), \\
&\tilde{\Phi}_\bm^{[4]}\big({-}z_{11}w_{12}z_{22}{}^t\hspace{-1pt}w_{12}\big)
=\tilde{\Phi}_\bm^{\Skew(s',\BC)}\big(z_{11},w_{12}z_{22}{}^t\hspace{-1pt}w_{12}\big)
\\ &\qquad
{}\in\cP_\bm({\rm Skew}(s',\BC))_{z_{11}}\boxtimes\cP_\bm({\rm Skew}(s',\BC))_{z_{22}}\simeq V_{\bm^2}^{(s')\vee}\boxtimes V_{\bm^2}^{(s')\vee}\qquad
(\textit{Case }3)
\end{split}
\end{gather*}
hold, by the ${\rm GL}(r',\BC)\times {\rm GL}(r',\BC)$- resp. ${\rm GL}(s',\BC)\times {\rm GL}(s',\BC)$-invariance, as polynomials of~$w_{12}$,
\begin{alignat*}{3}
& \tilde{\Phi}_\bm^{(1)}\big(z_{11}w_{12}z_{22}{}^t\hspace{-1pt}w_{12}\big)
\in \cP_{2\bm}(M(r',\BC))_{w_{12}}\simeq V_{2\bm}^{(r')}\boxtimes V_{2\bm}^{(r')}\qquad&&
(\textit{Case }2),&
\\
& \tilde{\Phi}_\bm^{[4]}\big({-}z_{11}w_{12}z_{22}{}^t\hspace{-1pt}w_{12}\big)
\in \cP_{\bm^2}(M(s',\BC))_{w_{12}}\simeq V_{\bm^2}^{(s')}\boxtimes V_{\bm^2}^{(s')}\qquad&&
(\textit{Case }3)&
\end{alignat*}
hold. Next we consider {\it Cases} 4, 5. For {\it Case} 5 we have $r=2s$. Then for
\[
z_1\in\fp^+_1= \begin{cases} \Skew(r,\BC) & (\textit{Case }4), \\ \Sym(2s,\BC) & (\textit{Case }5), \end{cases} \qquad
w_2\in\fp^-_2= \begin{cases} \Sym(r,\BC) & (\textit{Case }4), \\ \Skew(2s,\BC) & (\textit{Case }5), \end{cases}
\]
we have
\begin{gather*}
t^{\lfloor r/2\rfloor}h_{\fp^+}\big(t^{-1}z_1,Q(w_2)z_1\big)^{1/2}
=t^{\lfloor r/2\rfloor}\det\big(I-t^{-1}z_1w_2z_1w_2\big)^{1/2},
\\
\tilde{\Phi}_\bm^{\fp^+_1,\fp^-_2}(z_1,w_2)=\tilde{\Phi}_\bm^{[2]}(z_1w_2z_1w_2),
\\
h_{\fp^+}(z_1,Q(w_2)z_1)^{-\mu/2}=\det(I-z_1w_2z_1w_2)^{-\mu/2}
=\sum_{\bm\in\BZ_{++}^{\lfloor r/2\rfloor}}(\mu)_{\bm,2}\tilde{\Phi}_\bm^{[2]}(z_1w_2z_1w_2),
\end{gather*}
where $\tilde{\Phi}_\bm^{[2]}$ is as in (\ref{Phi^[d]}).
Then for $w_2=I_r$ (Case 4) or $w_2=J_{2s}:=\left(\begin{smallmatrix} 0&I_s\\ -I_s&0\end{smallmatrix}\right)$ ({\it Case} 5),
by \cite[Propositions 7.6 and 8.3]{Z} we have
\begin{alignat*}{3}
&\tilde{\Phi}_\bm^{[2]}(z_1^2)\in(\cP_{2\bm}({\rm Skew}(r,\BC)))^{{\rm SO}(r,\BC)}\simeq \big(V_{2\bm^2}^{(r)\vee}\big)^{{\rm SO}(r,\BC)} \qquad &&  (\textit{Case }4), & \\
&\tilde{\Phi}_\bm^{[2]}(z_1J_{2s}z_1J_{2s})\in(\cP_{\bm^2}({\rm Sym}(2s,\BC)))^{\operatorname{Sp}(s,\BC)}\simeq \big(V_{2\bm^2}^{(2s)\vee}\big)^{\operatorname{Sp}(s,\BC)}
\qquad &&  (\textit{Case }5),&
\end{alignat*}
and by the ${\rm GL}(r,\BC)$-invariance we get
\begin{align*}
\tilde{\Phi}_\bm^{[2]}(z_1w_2z_1w_2)&\in \begin{cases}
\big(\cP_{2\bm}({\rm Skew}(r,\BC))_{z_1}\otimes \cP_{\bm^2}({\rm Sym}(r,\BC))_{w_2}\big)^{{\rm GL}(r,\BC)} & (\textit{Case }4) ,
\\
\big(\cP_{\bm^2}({\rm Sym}(2s,\BC))_{z_1}\otimes \cP_{2\bm}({\rm Skew}(2s,\BC))_{w_2}\big)^{{\rm GL}(2s,\BC)} & (\textit{Case }5) \end{cases}
\\
&\simeq \big(V_{2\bm^2}^{(r)\vee}\otimes V_{2\bm^2}^{(r)}\big)^{{\rm GL}(r,\BC)}.
\end{align*}
Finally, for {\it Case} 6, we have $r=3$, and for $z_1\in\fp^+_1=M(2,6;\BC)$, $w_2\in\fp^-_2=\Skew(6,\BC)$, since
\begin{align*}
th_{\fp^+}\big(t^{-1}z_1,Q(w_2)z_1\big)^{1/2}&=th_{\fp^+_1}(t^{-1}z_1,Q(w_2)z_1)^{1/2}
=t\det\big(I-t^{-1}z_1{}^t\hspace{-1pt}\big(J_2z_1w_2^\#\big)\big)^{1/2}
\\
&=t\det\big(I-t^{-1}z_1w_2^\#{}^t\hspace{-1pt}z_1J_2\big)^{1/2}=
\Pf\big(tJ_2+z_1w_2^\#{}^t\hspace{-1pt}z_1\big)
\\
&=t+\Pf\big(z_1w_2^\#{}^t\hspace{-1pt}z_1\big),
\end{align*}
where $w_2^\#$ is defined as in (\ref{adj6}), we have $t=-\Pf\big(z_1w_2^\#{}^t\hspace{-1pt}z_1\big)$, and
\begin{gather}
\tilde{\Phi}_m^{\fp^+_1,\fp^-_2}(z_1,w_2)=\frac{1}{m!}t^m
=\frac{1}{m!}\big({-}\Pf\big(z_1w_2^\#{}^t\hspace{-1pt}z_1\big)\big)^m,\label{Phi_mid_except}
\\
h_{\fp^+}(z_1,Q(w_2)z_1)^{-\mu/2}=\big(1+\Pf\big(z_1w_2^\#{}^t\hspace{-1pt}z_1\big)\big)^{-\mu}
=\sum_{m=0}^\infty \frac{(\mu)_m}{m!}\big({-}\Pf\big(z_1w_2^\#{}^t\hspace{-1pt}z_1\big)\big)^m.\notag
\end{gather}
Since $\frac{1}{m!}\big({-}\Pf\big(z_1w_2^\#{}^t\hspace{-1pt}z_1\big)\big)^m$ is rewritten as
\begin{align*}
\frac{1}{m!}\big({-}\Pf\big(z_1w_2^\#{}^t\hspace{-1pt}z_1\big)\big)^m
=\tilde{\Phi}_m^{[4]}\big(z_1w_2^\#{}^t\hspace{-1pt}z_1J\big)
=\tilde{\Phi}_{(m,0,0)}^{[4]}\big(w_2^\#{}^t\hspace{-1pt}z_1Jz_1\big),
\end{align*}
by the ${\rm S}({\rm GL}(2,\BC)\times {\rm GL}(6,\BC))$-invariance we conclude that
\begin{gather*}
\frac{1}{m!}\big({-}\Pf\big(z_1w_2^\#{}^t\hspace{-1pt}z_1\big)\big)^m
\\ \qquad
{}\in\big(\cP_{(m,m)}(M(2,6;\BC))_{z_1}\otimes\cP_{(m,m,0)}({\rm Skew}(6,\BC))_{w_2}\big)^{{\rm S}({\rm GL}(2,\BC)\times {\rm GL}(6,\BC))}
\\ \qquad
{}\simeq \big(\big(V_{(m,m)}^{(2)\vee}\boxtimes V_{(0,0,0,0,-m,-m)}^{(6)\vee}\big)\otimes
\big(V_{(m,m)}^{(2)}\boxtimes V_{(0,0,0,0,-m,-m)}^{(6)}\big)\big)^{{\rm S}({\rm GL}(2,\BC)\times {\rm GL}(6,\BC))}
\end{gather*}
holds.
\end{proof}

Using $\tilde{\Phi}_\bm^{\fp^\pm_1,\fp^\mp_2}(z_1,w_2)$,
for $\alpha,\beta,\gamma\in\BC$ with $\gamma\notin\bigcup_{j=1}^{\lfloor r/2\rfloor}\big(\frac{d}{2}(j-1)-\BZ_{\ge 0}\big)$ we define
\begin{gather}
{}_0F_1^{\fp^\pm_1,\fp^\mp_2}(\gamma;z_1,w_2):=\sum_{\bm\in\BZ_{++}^{\lfloor r/2\rfloor}}\frac{1}{(\gamma)_{\bm,d}}
\tilde{\Phi}_\bm^{\fp^\pm_1,\fp^\mp_2}(z_1,w_2), \label{0F1_mid}
\\
{}_2F_1^{\fp^\pm_1,\fp^\mp_2}\left( \begin{matrix}\alpha,\beta\\\gamma \end{matrix};z_1,w_2\right)
:=\sum_{\bm\in\BZ_{++}^{\lfloor r/2\rfloor}}\frac{(\alpha)_{\bm,d}(\beta)_{\bm,d}}{(\gamma)_{\bm,d}}
\tilde{\Phi}_\bm^{\fp^\pm_1,\fp^\mp_2}(z_1,w_2). \label{2F1_mid}
\end{gather}

Next we consider the expansion of $\tilde{\Phi}_\bm^{\fp^\pm_1,\fp^\mp_2}(z_1,w_2)f\big({}^t\hspace{-1pt}w_2^\itinv\big)$.
When $\varepsilon_2=1$, we have $r=r_2$. For $f(x_2)\in\cP_{(l,0,\dots,0)}\big(\fp^\pm_2\big)$ and for $\bm\in\BZ_{++}^{\lfloor r/2\rfloor}$,
$\bl\in(\BZ_{\ge 0})^{\lceil r/2\rceil}$ with $|\bl|=l$,
we define $F_{\bm,\bl}^1[f](z_1,w_2)\in\cP\big(\fp^\pm_1\big)_{z_1}\otimes\cP\big(\fp^\mp_2\big)_{w_2}$ by
\begin{gather*}
\tilde{\Phi}_\bm^{\fp^\pm_1,\fp^\mp_2}(z_1,w_2)f\big({}^t\hspace{-1pt}w_2^\itinv\big)
=\sum_{\substack{\bl\in(\BZ_{\ge 0})^{\lceil r/2\rceil} \\ |\bl|=l}}F_{\bm,\bl}^1[f](z_1,w_2),
\\
F_{\bm,\bl}^1[f](z_1,w_2)
\\ \qquad
{}\in \begin{cases}
\cP\big(\fp^\pm_1\big)_{z_1}\otimes
\cP_{\substack{(m_1,m_1-l_1,m_2,m_2-l_2,\dots, m_{r/2},m_{r/2}-l_{r/2})}}\big(\fp_2^\mp\big)_{w_2}, & r\colon\text{even},
\\
\cP\big(\fp^\pm_1\big)_{z_1}\otimes
\cP_{\substack{(m_1,m_1-l_1,m_2,m_2-l_2,\dots, m_{\lfloor r/2\rfloor},m_{\lfloor r/2\rfloor}-l_{\lfloor r/2\rfloor},
-l_{\lceil r/2\rceil})}}\big(\fp_2^\mp\big)_{w_2}, & r\colon\text{odd}.
\end{cases}
\end{gather*}
Then $F_{\bm,\bl}^1[f](z_1,w_2)\ne 0$ holds only if $0\le l_j\le m_j-m_{j+1}$ ($1\le j< \lfloor r/2\rfloor$) and
$0\le l_{\lfloor r/2\rfloor}\le m_{\lfloor r/2\rfloor}$ (only when $r$ is odd).
Similarly, when $\varepsilon_2=2$, we have $r=2r_2$. For $f(x_2)\in\cP_{\underline{1}_l}\big(\fp^\pm_2\big)$ and for $\bm\in\BZ_{++}^{r_2}$,
$\bl\in\{0,1\}^{r_2}$ with $|\bl|=l$, we define $F_{\bm,\bl}^2[f](z_1,w_2)\in\cP\big(\fp^\pm_1\big)_{z_1}\otimes\cP\big(\fp^\mp_2\big)_{w_2}$ by
\begin{gather*}
\tilde{\Phi}_\bm^{\fp^\pm_1,\fp^\mp_2}(z_1,w_2)f\big({}^t\hspace{-1pt}w_2^\itinv\big)
=\sum_{\substack{\bl\in\{0,1\}^{r_2} \\ |\bl|=l}}F_{\bm,\bl}^2[f](z_1,w_2),
\\
F_{\bm,\bl}^2[f](z_1,w_2)\in\cP\big(\fp^\pm_1\big)_{z_1}\otimes\cP_{2\bm-\bl}\big(\fp_2^\mp\big)_{w_2}.
\end{gather*}
Then $F_{\bm,\bl}^2[f](z_1,w_2)\ne 0$ holds only if $2\bm-\bl\in\BZ_+^{r_2}$, or equivalently $\bm-\bl\in\BZ_+^{r_2}$.
Then for $\mu\in\BC$ we have
\begin{align}
\det_{\fn^+}(z)^{-\mu}f(z_2)&=\det_{\fn^+_2}(z_2)^{-\varepsilon_2\mu}h_{\fp^+}\big(z_1,Q(z_2)^{-1}z_1\big)^{-\mu/2}f(z_2) \notag\\
&=\det_{\fn^+_2}(z_2)^{-\varepsilon_2\mu}\sum_{\bm\in\BZ_{++}^{\lfloor r/2\rfloor}}
(\mu)_{\bm,d}\tilde{\Phi}_\bm^{\fp^+_1,\fp^-_2}\big(z_1,{}^t\hspace{-1pt}z_2^\itinv\big)f(z_2) \notag\\
&=\det_{\fn^+_2}(z_2)^{-\varepsilon_2\mu}\sum_{\bm\in\BZ_{++}^{\lfloor r/2\rfloor}}\sum_\bl
(\mu)_{\bm,d}F_{\bm,\bl}^{\varepsilon_2}[f]\big(z_1,{}^t\hspace{-1pt}z_2^\itinv\big). \label{binom_simple}
\end{align}

\subsection{Main theorem}

We~continue to assume $\fp^+$, $\fp^+_2$ are of tube type.
When $\varepsilon_2=1$, for $f(x_2)\in\cP_{(l,0,\dots,0)}\big(\fp^\pm_2\big)$, $\mu,\nu\in\BC$
and for $z=z_1+z_2\in\Omega^\pm\subset\fn^\pm\subset\fp^\pm$, let
\begin{align}
F_1^l\left(\begin{matrix}\nu\\\mu\end{matrix};f;z\right):\hspace{-3pt}&=\det_{\fn^+_2}(z_2)^{-\nu}
\sum_{\bm\in\BZ_{++}^{\lfloor r/2\rfloor}}\sum_{\substack{\bl\in(\BZ_{\ge 0})^{\lceil r/2\rceil}\\ |\bl|=l}}
F_{\bm,\bl}^1[f]\big(z_1,{}^t\hspace{-1pt}z_2^\itinv\big) \notag\\
&\eqspace{}\times \begin{cases}
\ds \frac{(\nu)_{\bm,d}\big(\nu-\frac{d_2}{2}-(\underline{0}_{r/2-1},l)\big)_{\bm-\bl+(\underline{0}_{r/2-1},l),d}}
{\big(\mu-\frac{d_2}{2}-(\underline{0}_{r/2-1},l)\big)_{\bm-\bl+(\underline{0}_{r/2-1},l),d}}, & r\colon\text{even},
\\[1ex]
\ds \frac{(\nu)_{\bm,d}\big(\nu-\frac{d_2}{2}\big)_{\bm-\bl',d}\big(\nu-\frac{d_2}{2}(r-1)-l\big)_{l-l_{\lceil r/2\rceil}}}
{\big(\mu-\frac{d_2}{2}\big)_{\bm-\bl',d}\big(\mu-\frac{d_2}{2}(r-1)-l\big)_{l-l_{\lceil r/2\rceil}}}, & r\colon\text{odd},
\end{cases} \label{Fmunu1}
\end{align}
where for odd $r$, for $\bl=(l_1,\dots,l_{\lfloor r/2\rfloor},l_{\lceil r/2\rceil})$ let $\bl':=(l_1,\dots,l_{\lfloor r/2\rfloor})$.
Similarly, when $\varepsilon_2=2$, for $f(x_2)\in\cP_{\underline{1}_l}\big(\fp^\pm_2\big)$, $\mu,\nu\in\BC$
and for $z=z_1+z_2\in\Omega^\pm\subset\fn^\pm\subset\fp^\pm$, let
\begin{gather}
F_2^l\left(\begin{matrix}\nu\\\mu\end{matrix};f;z\right):=\det_{\fn^+_2}(z_2)^{-2\nu} \notag
\\ \qquad
{}\times\sum_{\bm\in\BZ_{++}^{r_2}}\sum_{\substack{\bl\in\{0,1\}^{r_2}\\ |\bl|=l}}\frac{(\nu)_{\bm,d}
\big(\nu+\bigl(\underline{\frac{1}{2}}_{r_2-l},\underline{-\frac{1}{2}}_l\bigr) \big)_{\bm-\bl+(\underline{0}_{r_2-l},\underline{1}_l),d}}
{\big(\mu+\bigl(\underline{\frac{1}{2}}_{r_2-l},\underline{-\frac{1}{2}}_l\bigr) \big)_{\bm-\bl+(\underline{0}_{r_2-l},\underline{1}_l),d}}
F_{\bm,\bl}^2[f]\big(z_1,{}^t\hspace{-1pt}z_2^\itinv\big). \label{Fmunu2}
\end{gather}
$F_{\varepsilon_2}^l\left(\begin{smallmatrix}\nu\\\mu\end{smallmatrix};f;z\right)$ becomes a polynomial on $\fp^\pm$
if $\nu\in-\frac{1}{\varepsilon_2}\BZ_{\ge 0}$.
Now we want to prove the following.
\begin{Theorem}\label{thm_simple}
Assume $\fp^+$, $\fp^+_2$ are of tube type.
\begin{enumerate}\itemsep=0pt
\item[$1.$] Suppose $\varepsilon_2=1$. Let $k,l\in\BZ_{\ge 0}$, $f(x_2)\in\cP_{(l,0,\dots,0)}\big(\fp^+_2\big)$.
Then for $\Re\lambda>\frac{2n}{r}-1$ and for $z=z_1+z_2\in\Omega\subset\fn^+\subset\fp^+$, we have
\begin{gather*}
\big\langle \det_{\fn^+_2}(x_2)^kf(x_2),{\rm e}^{(x|\overline{z})_{\fp^+}}\big\rangle_{\lambda,x}
\\ \qquad
{}=\frac{1}{(\lambda)_{(2k+l,\underline{2k}_{r-1}),d}}\det_{\fn^+}(z)^{-\lambda+\frac{n}{r}}
\det_{\fn^+_2}\bigg(\frac{\partial}{\partial z_2}\bigg)^k\det_{\fn^+}(z)^{\lambda+2k-\frac{n}{r}}f(z_2)
\\ \qquad
{}=C_1^{r_2}(\lambda,k,l)\det_{\fn^+}(z)^{-\lambda+\frac{n}{r}}
F_1^l\left(\begin{matrix} -\lambda-k+\frac{n}{r} \\ -\lambda-2k+\frac{n}{r} \end{matrix};f;z\right)
\\ \qquad
{}=C_1^{r_2}(\lambda,k,l)F_1^l\left(\begin{matrix} -k \\ -\lambda-2k+\frac{n}{r} \end{matrix};f;z\right)\!,
\end{gather*}
where
\[
C_1^{r_2}(\lambda,k,l)=
\begin{cases} \ds \frac{\big(\lambda+k-\frac{d}{4}r_2+\frac{d_2}{2}+(l,\underline{0}_{r_2/2-1})\big)_{\underline{k}_{r_2/2},d}}
{(\lambda)_{(2k+l,\underline{2k}_{r_2/2-1},\underline{k}_{r_2/2}),d}}, & r_2\colon\text{even},
\\[2ex]
\ds \frac{\big(\lambda-\frac{d}{2}\big\lfloor\frac{r_2}{2}\big\rfloor+\max\{2k,k+l\}\big)_{\min\{k,l\}}
\big(\lambda+k-\frac{d}{2}\big\lceil\frac{r_2}{2}\big\rceil+\frac{d_2}{2}\big)_{\underline{k}_{\lfloor r_2/2\rfloor},d}}
{(\lambda)_{(2k+l,\underline{2k}_{\lfloor r_2/2\rfloor-1},\min\{2k,k+l\},\underline{k}_{\lfloor r_2/2\rfloor}),d}}, \hspace{-52pt}& \\
& r_2\colon\text{odd}. \end{cases}
\]

\item[$2.$] Suppose $\varepsilon_2=2$. Let $k\in\BZ_{\ge 0}$, $0\le l<r_2$, $f(x_2)\in\cP_{\underline{1}_l}\big(\fp^+_2\big)$.
Then for $\Re\lambda>\frac{2n}{r}-1$ and for $z=z_1+z_2\in\Omega\subset\fn^+\subset\fp^+$, we have
\begin{gather*}
\big\langle \det_{\fn^+_2}(x_2)^kf(x_2),{\rm e}^{(x|\overline{z})_{\fp^+}}\big\rangle_{\lambda,x}
\\ \qquad
{}=\frac{1}{(\lambda)_{(\underline{k+1}_l,\underline{k}_{r-l}),d}}\det_{\fn^+}(z)^{-\lambda+\frac{n}{r}}
\det_{\fn^+_2}\bigg(\frac{1}{2}\frac{\partial}{\partial z_2}\bigg)^k\det_{\fn^+}(z)^{\lambda+k-\frac{n}{r}}f(z_2)
\\ \qquad
{}=C_2^{r_2}(\lambda,k,l)\det_{\fn^+}(z)^{-\lambda+\frac{n}{r}}
F_2^l\left(\begin{matrix} -\lambda-\frac{k}{2}+\frac{n}{r} \\ -\lambda-k+\frac{n}{r} \end{matrix};f;z\right)
\\ \qquad
{}=C_2^{r_2}(\lambda,k,l)F_2^l\left(\begin{matrix} -\frac{k}{2} \\ -\lambda-k+\frac{n}{r} \end{matrix};f;z\right)\!,
\end{gather*}
where
\[
C_2^{r_2}(\lambda,k,l)=\frac{\big(\lambda-\frac{d}{2}r_2 +\big(\underline{\big\lfloor\frac{k}{2}\big\rfloor+\frac{1}{2}}_l,
\underline{\big\lceil\frac{k}{2}\big\rceil-\frac{1}{2}}_{r_2-l}\big)
\big)_{(\underline{\lceil k/2\rceil}_l,\underline{\lfloor k/2\rfloor}_{r_2-l}),d}}
{(\lambda)_{(\underline{k+1}_l,\underline{k}_{r_2-l}, \underline{\lceil k/2\rceil}_l,\underline{\lfloor k/2\rfloor}_{r_2-l}),d}}.
\]
\end{enumerate}
\end{Theorem}

The 1st equalities are nothing but Theorem \ref{key_identity}(1). Hence by putting $-\lambda-\frac{2k}{\varepsilon_2}+\frac{n}{r}=:\mu$,
 it suffices to show
\begin{gather}
\det_{\fn^+}(z)^{\mu+\frac{2k}{\varepsilon_2}}
\det_{\fn^+_2}\bigg(\frac{\partial}{\partial z_2}\bigg)^k\det_{\fn^+}(z)^{-\mu}f(z_2) \notag
\\ \qquad
{}=\begin{cases} \ds (-1)^{kr_2}\left(\mu-(\underline{0}_{r_2-1},l)\right)_{\underline{k}_{r_2},d_2}
\det_{\fn^+}(z)^{\mu+2k}F_1^l\left(\begin{matrix} \mu+k \\ \mu \end{matrix};f;z\right), & \varepsilon_2=1,
\\
\ds (-1)^{kr_2}\left(2\mu-(\underline{0}_{r_2-l},\underline{1}_l)\right)_{\underline{k}_{r_2},d_2}
\det_{\fn^+}(z)^{\mu+k}F_2^l\left(\begin{matrix} \mu+\frac{k}{2} \\ \mu \end{matrix};f;z\right), & \varepsilon_2=2, \end{cases} \notag
\\ \qquad
{}=\begin{cases} \ds (-1)^{kr_2}\left(\mu-(\underline{0}_{r_2-1},l)\right)_{\underline{k}_{r_2},d_2}
F_1^l\left(\begin{matrix} -k \\ \mu \end{matrix};f;z\right), & \varepsilon_2=1,
\\
\ds (-1)^{kr_2}\left(2\mu-(\underline{0}_{r_2-l},\underline{1}_l)\right)_{\underline{k}_{r_2},d_2}
F_2^l\left(\begin{matrix} -\frac{k}{2} \\ \mu \end{matrix};f;z\right), & \varepsilon_2=2.
\end{cases} \label{thm'_simple}
\end{gather}
Indeed, when $\varepsilon_2=1$ and $r_2$ is even, since $r=r_2$ and $d=2d_2$ or $r=2$ hold, we have
\begin{gather*}
(-1)^{kr_2}\big(\mu-(\underline{0}_{r_2-1},l)\big)_{\underline{k}_{r_2},d_2}
=(-1)^{kr}(\mu)_{\underline{k}_{r/2},d} \bigg(\mu-\frac{d_2}{2}-(\underline{0}_{r/2-1},l)\bigg)_{\underline{k}_{r/2},d}
\\
{}=\bigg( {-}\mu - k + \frac{d}{2}\bigg(\frac{r}{2} - 1\bigg) + 1 \bigg)_{\underline{k}_{r/2},d}
\bigg( {-}\mu-k+\frac{d}{2}\bigg(\frac{r}{2}-1\bigg)+1+\frac{d_2}{2} +(l,\underline{0}_{r/2-1}) \bigg)_{\underline{k}_{r/2},d}
\\
{}=\bigg(\lambda+k-\frac{d}{2}\frac{r}{2}\bigg)_{\underline{k}_{r/2},d}
\bigg(\lambda+k-\frac{d}{2}\frac{r}{2}+\frac{d_2}{2} +(l,\underline{0}_{r/2-1})\bigg)_{\underline{k}_{r/2},d},
\end{gather*}
and hence
\begin{align*}
\frac{(-1)^{kr_2}(\mu-(\underline{0}_{r_2-1},l))_{\underline{k}_{r_2},d_2}} {(\lambda)_{(2k+l,\underline{2k}_{r-1}),d}}
&=\frac{\big(\lambda+k-\frac{d}{2}\frac{r}{2}\big)_{\underline{k}_{r/2},d}
\big(\lambda+k-\frac{d}{2}\frac{r}{2}+\frac{d_2}{2}+(l,\underline{0}_{r/2-1})\big)_{\underline{k}_{r/2},d}}
{(\lambda)_{(2k+l,\underline{2k}_{r/2-1}),d}
\big(\lambda-\frac{d}{2}\frac{r}{2}\big)_{\underline{2k}_{r/2},d}}
\\
&=\frac{\big(\lambda+k-\frac{d}{2}\frac{r}{2}
+\frac{d_2}{2}+(l,\underline{0}_{r/2-1})\big)_{\underline{k}_{r/2},d}}
{(\lambda)_{(2k+l,\underline{2k}_{r/2-1}),d}
\big(\lambda-\frac{d}{2}\frac{r}{2}\big)_{\underline{k}_{r/2},d}}
\\
&=\frac{\big(\lambda+k-\frac{d}{4}r_2+\frac{d_2}{2} +(l,\underline{0}_{r_2/2-1})\big)_{\underline{k}_{r_2/2},d}}
{(\lambda)_{(2k+l,\underline{2k}_{r_2/2-1},\underline{k}_{r_2/2}),d}}
\end{align*}
holds. Similarly, when $\varepsilon_2=1$ and $r_2$ is odd, we have
\begin{gather*}
(-1)^{kr_2}\big(\mu-(\underline{0}_{r_2-1},l)\big)_{\underline{k}_{r_2},d_2}
={}(-1)^{kr}\big(\mu-(\underline{0}_{\lfloor r/2\rfloor},l)\big)_{\underline{k}_{\lceil r/2\rceil},d}
\bigg(\mu-\frac{d_2}{2}\bigg)_{\underline{k}_{\lfloor r/2\rfloor},d}
\\ \qquad
{}=\bigg({-}\mu-k+\frac{d}{2}\bigg(\bigg\lceil\frac{r}{2}\bigg\rceil-1\bigg)+1
+(l,\underline{0}_{\lfloor r/2\rfloor})\bigg)_{\underline{k}_{\lceil r/2\rceil},d}
\\ \qquad\hphantom{=}
\ \times\bigg({-}\mu-k+\frac{d}{2}\bigg(\bigg\lfloor\frac{r}{2}\bigg\rfloor-1\bigg)+1+\frac{d_2}{2}
\bigg)_{\underline{k}_{\lfloor r/2\rfloor},d}
\\ \qquad
{}=\bigg(\lambda+k-\frac{d}{2}\bigg\lfloor\frac{r}{2}\bigg\rfloor+(l,\underline{0}_{\lfloor r/2\rfloor})\bigg)_{\underline{k}_{\lceil r/2\rceil},d}
\bigg(\lambda+k-\frac{d}{2}\bigg\lceil\frac{r}{2}\bigg\rceil +\frac{d_2}{2}\bigg)_{\underline{k}_{\lfloor r/2\rfloor},d},
\end{gather*}
and hence
\begin{gather*}
\frac{(-1)^{kr_2}\big(\mu-(\underline{0}_{r_2-1},l)\big)_{\underline{k}_{r_2},d_2}} {(\lambda)_{(2k+l,\underline{2k}_{r-1}),d}}
\\ \qquad
{}=\frac{\big(\lambda+k-\frac{d}{2}\big\lfloor\frac{r}{2}\big\rfloor+(l,\underline{0}_{\lfloor r/2\rfloor})\big)_{\underline{k}_{\lceil r/2\rceil},d}
\big(\lambda+k-\frac{d}{2}\big\lceil\frac{r}{2}\big\rceil+\frac{d_2}{2}\big)_{\underline{k}_{\lfloor r/2\rfloor},d}}
{(\lambda)_{(2k+l,\underline{2k}_{\lfloor r/2\rfloor-1}),d}
\big(\lambda-\frac{d}{2}\big\lfloor\frac{r}{2}\big\rfloor\big)_{\underline{2k}_{\lceil r/2\rceil},d}}
\\ \qquad
{}=\frac{\big(\lambda-\frac{d}{2}\big\lfloor\frac{r}{2}\big\rfloor+\max\{2k,k+l\}\big)_{\min\{k,l\}}
\big(\lambda+k-\frac{d}{2}\big\lceil\frac{r}{2}\big\rceil+\frac{d_2}{2}\big)_{\underline{k}_{\lfloor r/2\rfloor},d}}
{(\lambda)_{(2k+l,\underline{2k}_{\lfloor r/2\rfloor-1}),d}
\big(\lambda-\frac{d}{2}\big\lfloor\frac{r}{2}\big\rfloor\big)_{(\min\{2k,k+l\},\underline{k}_{\lfloor r/2\rfloor}),d}}
\\ \qquad
{}=\frac{\big(\lambda-\frac{d}{2}\big\lfloor\frac{r_2}{2}\big\rfloor+\max\{2k,k+l\}\big)_{\min\{k,l\}}
\big(\lambda+k-\frac{d}{2}\big\lceil\frac{r_2}{2}\big\rceil+\frac{d_2}{2}\big)_{\underline{k}_{\lfloor r_2/2\rfloor},d}}
{(\lambda)_{(2k+l,\underline{2k}_{\lfloor r_2/2\rfloor-1},\min\{2k,k+l\},\underline{k}_{\lfloor r_2/2\rfloor}),d}}
\end{gather*}
holds. Similarly, when $\varepsilon_2=2$, since $r=2r_2$ and $d_2=2d$ or $r_2=1$ hold, we have
\begin{align*}
\frac{(-1)^{kr_2}}{2^{kr_2}}\big(2\mu-(\underline{0}_{r_2-l}, \underline{1}_l)\big)_{\underline{k}_{r_2},d_2}
={}&\frac{1}{2^{kr_2}}\big({-}2\mu-k+d(r_2-1)+1+(\underline{1}_l, \underline{0}_{r_2-l})\big)_{\underline{k}_{r_2},2d}
\\
={}&\frac{1}{2^{kr_2}}\big(2\lambda+k-dr_2-1 +(\underline{1}_l,\underline{0}_{r_2-l})\big)_{\underline{k}_{r_2},2d}
\\
={}&\bigg(\lambda-\frac{d}{2}r_2+\bigg(\underline{\bigg\lceil\frac{k}{2}\bigg\rceil}_l,
\underline{\left\lfloor\frac{k}{2}\right\rfloor}_{r_2-l}\bigg)
\bigg)_{(\underline{\lfloor k/2\rfloor}_l,\underline{\lceil k/2\rceil}_{r_2-l}),d}
\\
{}&\times \! \bigg(\!\lambda\!-\!\frac{d}{2}r_2\!+\!\bigg(\underline{\bigg\lfloor\frac{k}{2}\bigg\rfloor\!+\!\frac{1}{2}}_l,
\underline{\left\lceil\frac{k}{2}\right\rceil-\frac{1}{2}}_{r_2-l}\bigg)\!
\bigg)_{(\underline{\lceil k/2\rceil}_l,\underline{\lfloor k/2\rfloor}_{r_2-l}),d}\!,
\end{align*}
and hence we get
\begin{gather*}
\frac{(-1)^{kr_2}\big(2\mu-(\underline{0}_{r_2-l},\underline{1}_l)\big)_{\underline{k}_{r_2},d_2}}
{2^{kr_2}(\lambda)_{(\underline{k+1}_l,\underline{k}_{r-l}),d}}
\\ \qquad
{}=\frac{1}{(\lambda)_{(\underline{k+1}_l,\underline{k}_{r_2-l}),d} \big(\lambda-\frac{d}{2}r_2\big)_{\underline{k}_{r_2},d}}
\bigg(\lambda-\frac{d}{2}r_2+\bigg(\underline{\bigg\lceil\frac{k}{2}\bigg\rceil}_l,
\underline{\bigg\lfloor\frac{k}{2}\bigg\rfloor}_{r_2-l}\bigg)
\bigg)_{(\underline{\lfloor k/2\rfloor}_l,\underline{\lceil k/2\rceil}_{r_2-l}),d}
\\ \qquad\hphantom{=}
{} \ \times \bigg(\lambda-\frac{d}{2}r_2+\bigg(\underline{\bigg\lfloor\frac{k}{2}\bigg\rfloor+\frac{1}{2}}_l,
\underline{\bigg\lceil\frac{k}{2}\bigg\rceil-\frac{1}{2}}_{r_2-l}\bigg)
\bigg)_{(\underline{\lceil k/2\rceil}_l,\underline{\lfloor k/2\rfloor}_{r_2-l}),d}
\\ \qquad
{}=\frac{\big(\lambda-\frac{d}{2}r_2+\big(\underline{\big\lfloor\frac{k}{2}\big\rfloor+\frac{1}{2}}_l,
\underline{\big\lceil\frac{k}{2}\big\rceil-\frac{1}{2}}_{r_2-l}\big)
\big)_{(\underline{\lceil k/2\rceil}_l,\underline{\lfloor k/2\rfloor}_{r_2-l}),d}}
{(\lambda)_{(\underline{k+1}_l,\underline{k}_{r_2-l}),d}
\big(\lambda-\frac{d}{2}r_2\big)_{(\underline{\lceil k/2\rceil}_l,\underline{\lfloor k/2\rfloor}_{r_2-l}),d}}
\\ \qquad
{}=\frac{\big(\lambda-\frac{d}{2}r_2+\big(\underline{\big\lfloor\frac{k}{2}\big\rfloor+\frac{1}{2}}_l,
\underline{\big\lceil\frac{k}{2}\big\rceil-\frac{1}{2}}_{r_2-l}\big)
\big)_{(\underline{\lceil k/2\rceil}_l,\underline{\lfloor k/2\rfloor}_{r_2-l}),d}}
{(\lambda)_{(\underline{k+1}_l,\underline{k}_{r_2-l},\underline{\lceil k/2\rceil}_l,\underline{\lfloor k/2\rfloor}_{r_2-l}),d}}.
\end{gather*}
To prove (\ref{thm'_simple}), let $\cB_\lambda^{\fn^+}\colon C^2(\Omega)\to C(\Omega)\otimes\fp^+$ be the Bessel operator given by
\[ \cB_\lambda^{\fn^+}f(z):=\sum_{\alpha,\beta}\frac{1}{2}P(e_\alpha,e_\beta)z\frac{\partial^2f}{\partial z_\alpha\partial z_\beta}(z)
+\lambda\sum_{\alpha}e_\alpha\frac{\partial f}{\partial z_\alpha}(z),
\]
where $\{e_\alpha\}\subset\fn^+\subset\fp^+$ is a basis of $\fn^+$, and $\{z_\alpha\}$ is the coordinate of $\fp^+$ for the dual basis
$\{e_\alpha^\vee\}\subset\fn^+\subset\fp^+$ of $\{e_\alpha\}$ with respect to $(\cdot|\cdot)_{\fn^+}$.
This $\cB_\lambda^{\fn^+}$ and $\cB_\tau^{\fp^+}(w)$ ($w\in\fp^+$) in (\ref{Bessel_op}) are related as
$(w|\cB_\lambda^{\fn^+}f)_{\fn^+}=\cB_{\chi^{-\lambda}}^{\fp^+}(w)f$.
Also, let $\Proj_1\colon \fp^+\to\fp^+_1$ be the orthogonal projection,
let $L\subset K^\BC$ be the real form given in Section~\ref{section_KKT},
and let $\Gamma_{r_2}^{d_2}$ be as (\ref{Gamma}).

\begin{Proposition}\label{prop_simple}
Let $f(z_2)$ be as before, and let $\mu,\nu\in\BC$.
\begin{enumerate}\itemsep=0pt
\item[$1.$] For $\Re\mu>\frac{1}{\varepsilon_2}\big(\frac{2n_2}{r_2}-1\big)$, $\Re\nu>\frac{1}{\varepsilon_2}\big(\frac{n_2}{r_2}-1\big)$
and for $z=z_1+z_2\in\Omega$, $a_2\in\Omega_2$ with $z_1+a_2\in\Omega$, we have
\begin{align*}
F_{\varepsilon_2}^l\left(\begin{matrix}\nu\\ \mu\end{matrix};f;z\right)={}&\begin{cases}
\ds \frac{\Gamma_{r_2}^{d_2}(\mu-(\underline{0}_{r_2-1},l))}{\Gamma_{r_2}^{d_2}(\nu-(\underline{0}_{r_2-1},l))}
\int_{\Omega_2}{\rm e}^{-(y_2|z_2)_{\fn^+_2}}\det_{\fn^+_2}(y_2)^{\nu-\mu} & (\varepsilon_2=1)
\\[2ex]
\ds \frac{\Gamma_{r_2}^{d_2}(2\mu-(\underline{0}_{r_2-l},\underline{1}_l))}
{\Gamma_{r_2}^{d_2}(2\nu-(\underline{0}_{r_2-l},\underline{1}_l))}
\int_{\Omega_2}{\rm e}^{-(y_2|z_2)_{\fn^+_2}}\det_{\fn^+_2}(y_2)^{2(\nu-\mu)}& (\varepsilon_2=2) \end{cases} \\
&{}\times\bigg(\frac{1}{(2\pi\sqrt{-1})^{n_2}}
\int_{a_2+\sqrt{-1}\fn^+_2}\!\!\!\!{\rm e}^{(x_2|y_2)_{\fn^+_2}}\det_{\fn^+}(z_1+x_2)^{-\mu}f(x_2)\,{\rm d}x_2\bigg){\rm d}y_2.
\end{align*}
\item[$2.$] For $g\in K_1^\BC\cap L$ and for $z=z_1+z_2\in\Omega$, we have
\[ F_{\varepsilon_2}^l\left(\begin{matrix}\nu\\\mu\end{matrix};f(g\cdot);z\right)
=\det_{\fn^+_2}(ge)^{\varepsilon_2\nu}F_{\varepsilon_2}^l\left(\begin{matrix}\nu\\\mu\end{matrix};f;gz\right)\!. \]
\item[$3.$] $F_{\varepsilon_2}^l\left(\begin{smallmatrix}\nu\\\mu\end{smallmatrix};f;z\right)$ satisfies the differential equation
\[ \Proj_1\left(\cB_{-\mu+2\nu+\frac{n}{r}}^{\fn^+}F_{\varepsilon_2}^l \left(\begin{matrix}\nu\\\mu\end{matrix};f;z\right)\right)=0. \]
\end{enumerate}
\end{Proposition}
\begin{proof}
(1) By (\ref{binom_simple}) we have
\begin{gather*}
\int_{\Omega_2}{\rm e}^{-(y_2|z_2)_{\fn^+_2}}\det_{\fn^+_2}(y_2)^{\varepsilon_2(\nu-\mu)}
\\ \qquad\hphantom{=}
{}\ \times\bigg(\frac{1}{(2\pi\sqrt{-1})^{n_2}}\int_{a_2+\sqrt{-1}\fn^+_2}{\rm e}^{(x_2|y_2)_{\fn^+_2}} \det_{\fn^+}(z_1+x_2)^{-\mu}f(x_2)\,{\rm d}x_2\bigg){\rm d}y_2
\\ \qquad
{}=\sum_{\bm\in\BZ_{++}^{\lfloor r/2\rfloor}}\sum_\bl
(\mu)_{\bm,d}\int_{\Omega_2}{\rm e}^{-(y_2|z_2)_{\fn^+_2}}\det_{\fn^+_2}(y_2)^{\varepsilon_2(\nu-\mu)}
\\ \qquad\hphantom{=}
{} \ \times\bigg(\frac{1}{(2\pi\sqrt{-1})^{n_2}}\int_{a_2+\sqrt{-1}\fn^+_2}{\rm e}^{(x_2|y_2)_{\fn^+_2}}\det_{\fn^+_2}(x_2)^{-\varepsilon_2\mu}
F_{\bm,\bl}^{\varepsilon_2}[f]\big(z_1,{}^t\hspace{-1pt}x_2^\itinv\big)\,{\rm d}x_2\bigg){\rm d}y_2.
\end{gather*}
When $\varepsilon_2=1$, since $r=r_2$ and $d=2d_2$ or $r=2$ hold, by Theorem \ref{Laplace} we have
\begin{gather*}
\int_{\Omega_2}{\rm e}^{-(y_2|z_2)_{\fn^+_2}}\det_{\fn^+_2}(y_2)^{\nu-\mu}
\\ \qquad\hphantom{=}
{} \ \times\bigg(\frac{1}{(2\pi\sqrt{-1})^{n_2}}\int_{a_2+\sqrt{-1}\fn^+_2}{\rm e}^{(x_2|y_2)_{\fn^+_2}}\det_{\fn^+_2}(x_2)^{-\mu}
F_{\bm,\bl}^1[f]\big(z_1,{}^t\hspace{-1pt}x_2^\itinv\big)\,{\rm d}x_2\bigg){\rm d}y_2
\\ \qquad
{}=\frac{1}{\Gamma_{r_2}^{d_2}(\mu+(m_1,m_1-l_1,m_2,m_2-l_2,\dots,m_{\lfloor r/2\rfloor},m_{\lfloor r/2\rfloor}-l_{\lfloor r/2\rfloor}(,-l_{\lceil r/2\rceil})))}
\\ \qquad\hphantom{=}
{} \ \times\int_{\Omega_2}{\rm e}^{-(y_2|z_2)_{\fn^+_2}}\det_{\fn^+_2}(y_2)^{\nu-\frac{n_2}{r_2}}
F_{\bm,\bl}^1[f](z_1,Q(\overline{e})y_2)\,{\rm d}y_2
\\ \qquad
{}=\frac{\Gamma_{r_2}^{d_2}(\nu+(m_1,m_1-l_1,m_2,m_2-l_2,\dots,m_{\lfloor r/2\rfloor},m_{\lfloor r/2\rfloor}-l_{\lfloor r/2\rfloor}(,-l_{\lceil r/2\rceil})))}
{\Gamma_{r_2}^{d_2}(\mu+(m_1,m_1-l_1,m_2,m_2-l_2,\dots,m_{\lfloor r/2\rfloor},m_{\lfloor r/2\rfloor}-l_{\lfloor r/2\rfloor}(,-l_{\lceil r/2\rceil})))}
\\ \qquad\hphantom{=}
{} \ \times\det_{\fn^+_2}(y_2)^{-\nu}F_{\bm,\bl}^1[f]\big(z_1,{}^t\hspace{-1pt}z_2^\itinv\big)
\\ \qquad
{}=\frac{\Gamma_{r_2}^{d_2}(\nu-(\underline{0}_{r_2-1},l))}{\Gamma_{r_2}^{d_2}(\mu-(\underline{0}_{r_2-1},l))}
\det_{\fn^+_2}(y_2)^{-\nu}F_{\bm,\bl}^1[f]\big(z_1,{}^t\hspace{-1pt}z_2^\itinv\big)
\\ \qquad\hphantom{=}
{} \ \times\frac{(\nu-(\underline{0}_{r_2-1},l))_{(m_1,m_1-l_1,m_2,m_2-l_2,\dots,m_{\lfloor r/2\rfloor},m_{\lfloor r/2\rfloor}-l_{\lfloor r/2\rfloor}
(,-l_{\lceil r/2\rceil}))+(\underline{0}_{r-1},l),d_2}}
{(\mu-(\underline{0}_{r_2-1},l))_{(m_1,m_1-l_1,m_2,m_2-l_2,\dots,m_{\lfloor r/2\rfloor},m_{\lfloor r/2\rfloor}-l_{\lfloor r/2\rfloor}
(,-l_{\lceil r/2\rceil}))+(\underline{0}_{r-1},l),d_2}}
\\ \qquad
{}=\frac{\Gamma_{r_2}^{d_2}(\nu-(\underline{0}_{r_2-1},l))}{\Gamma_{r_2}^{d_2}(\mu-(\underline{0}_{r_2-1},l))}
\det_{\fn^+_2}(y_2)^{-\nu}F_{\bm,\bl}^1[f]\big(z_1,{}^t\hspace{-1pt}z_2^\itinv\big)
\\ \qquad\hphantom{=}
{} \ \times \begin{cases}
\ds \frac{(\nu)_{\bm,d}\big(\nu-\frac{d_2}{2}-(\underline{0}_{r/2-1},l)\big)_{\bm-\bl +(\underline{0}_{r/2-1},l),d}}{(\mu)_{\bm,d}\big(\mu-\frac{d_2}{2} -(\underline{0}_{r/2-1},l)\big)_{\bm-\bl+(\underline{0}_{r/2-1},l),d}}, & r\colon\text{even},
\\
\ds \frac{(\nu)_{\bm,d}\big(\nu-\frac{d_2}{2}\big)_{\bm-\bl',d} \big(\nu-\frac{d_2}{2}(r-1)-l\big)_{l-l_{\lceil r/2\rceil}}}
{(\mu)_{\bm,d}\big(\mu-\frac{d_2}{2}\big)_{\bm-\bl',d}
\big(\mu-\frac{d_2}{2}(r-1)-l\big)_{l-l_{\lceil r/2\rceil}}}, &r\colon\text{odd}.
\end{cases}
\end{gather*}
Similarly, when $\varepsilon_2=2$, since $d_2=2d$ or $r_2=1$ holds and $\bl\in\{0,1\}^{r_2}$, again by Theorem \ref{Laplace} we have
\begin{gather*}
\int_{\Omega_2}{\rm e}^{-(y_2|z_2)_{\fn^+_2}}\det_{\fn^+_2}(y_2)^{2(\nu-\mu)}
\\ \qquad \hphantom{=}
{} \ \times\bigg(\frac{1}{(2\pi\sqrt{-1})^{n_2}}\int_{a_2+\sqrt{-1}\fn^+_2}{\rm e}^{(x_2|y_2)_{\fn^+_2}}\det_{\fn^+_2}(x_2)^{-2\mu}
F_{\bm,\bl}^2[f]\big(z_1,{}^t\hspace{-1pt}x_2^\itinv\big)\,{\rm d}x_2\bigg){\rm d}y_2
\\ \qquad
{}=\frac{\Gamma_{r_2}^{d_2}(2\nu+2\bm-\bl)}{\Gamma_{r_2}^{d_2}(2\mu+2\bm-\bl)}
\det_{\fn^+_2}(y_2)^{-2\nu}F_{\bm,\bl}^2[f]\big(z_1,{}^t\hspace{-1pt}z_2^\itinv\big)
\\ \qquad
{}=\frac{\Gamma_{r_2}^{d_2}(2\nu\!-\!(\underline{0}_{r_2-l},\underline{1}_l))} {\Gamma_{r_2}^{d_2}(2\mu\!-\!(\underline{0}_{r_2-l},\underline{1}_l))}
\frac{(2\nu\!-\!(\underline{0}_{r_2-l}, \underline{1}_l))_{2\bm-\bl+(\underline{0}_{r_2-l},\underline{1}_l),d_2}}
{(2\mu\!-\!(\underline{0}_{r_2-l},\underline{1}_l))_{2\bm-\bl+(\underline{0}_{r_2-l}, \underline{1}_l),d_2}}
\det_{\fn^+_2}(y_2)^{-2\nu}F_{\bm,\bl}^2[f]\big(z_1,{}^t\hspace{-1pt}z_2^\itinv\big)
\\ \qquad
{}=\frac{\Gamma_{r_2}^{d_2}(2\nu-(\underline{0}_{r_2-l},\underline{1}_l))} {\Gamma_{r_2}^{d_2}(2\mu-(\underline{0}_{r_2-l},\underline{1}_l))}\frac{(\nu)_{\bm,d}
\big(\nu+\bigl(\underline{\frac{1}{2}}_{r_2-l},\underline{-\frac{1}{2}}_l\bigr)\big)_{\bm-\bl +(\underline{0}_{r_2-l},\underline{1}_l),d}}{(\mu)_{\bm,d}
\big(\mu+\bigl(\underline{\frac{1}{2}}_{r_2-l},\underline{-\frac{1}{2}}_l\bigr)\big)_{\bm-\bl +(\underline{0}_{r_2-l},\underline{1}_l),d}}
\\ \qquad \hphantom{=}
{} \ \times
\det_{\fn^+_2}(y_2)^{-2\nu}F_{\bm,\bl}^2[f]\big(z_1,{}^t\hspace{-1pt}z_2^\itinv\big).
\end{gather*}
Hence we get the desired formulas.

(2) When $\Re\mu>\frac{1}{\varepsilon_2}\big(\frac{2n_2}{r_2}-1\big)$, $\Re\nu>\frac{1}{\varepsilon_2}\big(\frac{n_2}{r_2}-1\big)$,
in the integral expression (1), by changing the variables $x_2$ and $y_2$ to $g^{-1}x_2$ and $g^\top y_2$ respectively,
we can easily get the desired formula. By analytic continuation this holds for all $\mu$, $\nu$.

(3) Again it suffices to prove when $\Re\mu>\frac{1}{\varepsilon_2}\big(\frac{2n_2}{r_2}-1\big)$,
$\Re\nu>\frac{1}{\varepsilon_2}\big(\frac{n_2}{r_2}-1\big)$.
Let $\{e_{1,\alpha}\}\subset\fp^+_1\cap\fn^+\subset\fp^+_1$ and $\{e_{2,\alpha}\}\subset\fn^+_2\subset\fp^+_2$ be
orthonormal bases with respect to $(\cdot|\cdot)_{\fn^+}$ and $(\cdot|\cdot)_{\fn^+_2}$ respectively, so that
$\big\{e_{1,\alpha},\frac{1}{\varepsilon_2}e_{2,\alpha}\big\},\{e_{1,\alpha},e_{2,\alpha}\}\subset\fn^+\subset\fp^+$
are mutually dual with respect to $(\cdot|\cdot)_{\fn^+}$.
Let $\{z_{1,\alpha},z_{2,\alpha}\}$ be the coordinates for $\{e_{1,\alpha},e_{2,\alpha}\}$.
Then the projection of the Bessel operator is written as
\begin{gather*}
\Proj_1(\cB_\lambda^{\fn^+})
=\frac{1}{2}\sum_{\alpha,\beta}P(e_{1,\alpha},e_{1,\beta})z_1\frac{\partial^2}{\partial z_{1,\alpha}\partial z_{1,\beta}}
+\frac{1}{\varepsilon_2}\sum_{\alpha,\beta}P(e_{1,\alpha},e_{2,\beta})z_2\frac{\partial^2}{\partial z_{1,\alpha}\partial z_{2,\beta}}
\\ \hphantom{\Proj_1(\cB_\lambda^{\fn^+})=}
{}+\frac{1}{2\varepsilon_2^2}\sum_{\alpha,\beta}P(e_{2,\alpha},e_{2,\beta})z_1
\frac{\partial^2}{\partial z_{2,\alpha}\partial z_{2,\beta}}
+\lambda\sum_\alpha e_{1,\alpha}\frac{\partial}{\partial z_{1,\alpha}}.
\end{gather*}
Let $\det_{\fn^+}(z)^{-\mu}f(z_2)=:g(z)$. When we operate the 3rd term to the integral expression (1), we have
\begin{gather*}
\frac{1}{2\varepsilon_2^2}\sum_{\alpha,\beta}P(e_{2,\alpha},e_{2,\beta})z_1\frac{\partial^2}{\partial z_{2,\alpha}\partial z_{2,\beta}}
\\ \qquad\hphantom{=}
{} \ \times\int_{\Omega_2}{\rm e}^{-(y_2|z_2)_{\fn^+_2}}\det_{\fn^+_2}(y_2)^{\varepsilon_2(\nu-\mu)}
\bigg(\int_{a_2+\sqrt{-1}\fn^+_2}{\rm e}^{(x_2|y_2)_{\fn^+_2}}g(z_1,x_2)\,{\rm d}x_2\bigg){\rm d}y_2
\\ \qquad
{}=\frac{1}{\varepsilon_2^2}\int_{\Omega_2}(P(y_2)z_1){\rm e}^{-(y_2|z_2)_{\fn^+_2}}\det_{\fn^+_2}(y_2)^{\varepsilon_2(\nu-\mu)}
\bigg(\int_{a_2+\sqrt{-1}\fn^+_2}{\rm e}^{(x_2|y_2)_{\fn^+_2}}g(z_1,x_2)\,{\rm d}x_2\bigg){\rm d}y_2
\\ \qquad
{}=\int_{\Omega_2}{\rm e}^{-(y_2|z_2)_{\fn^+_2}}\det_{\fn^+_2}(y_2)^{\varepsilon_2(\nu-\mu)}
\\ \qquad\hphantom{=}
{} \ \times\biggl(\int_{a_2+\sqrt{-1}\fn^+_2}\biggl(\frac{1}{2\varepsilon_2^2}\sum_{\alpha,\beta}P(e_{2,\alpha},e_{2,\beta})z_1
\frac{\partial^2}{\partial x_{2,\alpha}\partial x_{2,\beta}}{\rm e}^{(x_2|y_2)_{\fn^+_2}}\biggr)\,g(z_1,x_2)\,{\rm d}x_2\biggr){\rm d}y_2
\\ \qquad
{}=\int_{\Omega_2}{\rm e}^{-(y_2|z_2)_{\fn^+_2}}\det_{\fn^+_2}(y_2)^{\varepsilon_2(\nu-\mu)}
\\ \qquad\hphantom{=}
{} \ \times\biggl(\int_{a_2+\sqrt{-1}\fn^+_2}{\rm e}^{(x_2|y_2)_{\fn^+_2}}\,
\biggl(\frac{1}{2\varepsilon_2^2}\sum_{\alpha,\beta}P(e_{2,\alpha},e_{2,\beta})z_1
\frac{\partial^2}{\partial x_{2,\alpha}\partial x_{2,\beta}}g(z_1,x_2)\biggr){\rm d}x_2\biggr){\rm d}y_2.
\end{gather*}
Similarly, the 2nd term is computed as
\begin{gather*}
\frac{1}{\varepsilon_2}\sum_{\alpha,\beta}P(e_{1,\alpha},e_{2,\beta})z_2\frac{\partial^2}{\partial z_{1,\alpha}\partial z_{2,\beta}}
\\ \qquad\hphantom{=}
{} \ \times\int_{\Omega_2}{\rm e}^{-(y_2|z_2)_{\fn^+_2}}\det_{\fn^+_2}(y_2)^{\varepsilon_2(\nu-\mu)}
\bigg(\int_{a_2+\sqrt{-1}\fn^+_2}{\rm e}^{(x_2|y_2)_{\fn^+_2}}g(z_1,x_2)\,{\rm d}x_2\bigg){\rm d}y_2
\\ \qquad
{}=-\frac{1}{\varepsilon_2}\sum_{\alpha}
\int_{\Omega_2}(P(e_{1,\alpha},y_2)z_2){\rm e}^{-(y_2|z_2)_{\fn^+_2}}\det_{\fn^+_2}(y_2)^{\varepsilon_2(\nu-\mu)}
\\ \qquad\hphantom{=}
{} \ \times\bigg(\int_{a_2+\sqrt{-1}\fn^+_2}{\rm e}^{(x_2|y_2)_{\fn^+_2}}\frac{\partial g}{\partial z_{1,\alpha}}(z_1,x_2)\,{\rm d}x_2\bigg){\rm d}y_2
\\ \qquad
{}=\frac{1}{\varepsilon_2}\sum_{\alpha,\beta}
\int_{\Omega_2}\bigg(P(e_{1,\alpha},y_2)e_{2,\beta}\frac{\partial}{\partial y_{2,\beta}}{\rm e}^{-(y_2|z_2)_{\fn^+_2}}\bigg)\det_{\fn^+_2}(y_2)^{\varepsilon_2(\nu-\mu)}
\\ \qquad\hphantom{=}
{} \ \times\bigg(\int_{a_2+\sqrt{-1}\fn^+_2}{\rm e}^{(x_2|y_2)_{\fn^+_2}}\frac{\partial g}{\partial z_{1,\alpha}}(z_1,x_2)\,{\rm d}x_2\bigg){\rm d}y_2
\\ \qquad
{}=-\frac{1}{\varepsilon_2}\sum_{\alpha,\beta}
\int_{\Omega_2}{\rm e}^{-(y_2|z_2)_{\fn^+_2}}\,\biggl(\frac{\partial}{\partial y_{2,\beta}}P(e_{1,\alpha},y_2)e_{2,\beta}
\det_{\fn^+_2}(y_2)^{\varepsilon_2(\nu-\mu)}
\\ \qquad\hphantom{=}
{} \ \times\bigg(\int_{a_2+\sqrt{-1}\fn^+_2}{\rm e}^{(x_2|y_2)_{\fn^+_2}}\frac{\partial g}{\partial z_{1,\alpha}}(z_1,x_2)\,{\rm d}x_2\bigg)\biggr){\rm d}y_2
\\ \qquad
{}=-\frac{1}{\varepsilon_2}\sum_{\alpha}\int_{\Omega_2}{\rm e}^{-(y_2|z_2)_{\fn^+_2}}\det_{\fn^+_2}(y_2)^{\varepsilon_2(\nu-\mu)}\,\biggl(\int_{a_2+\sqrt{-1}\fn^+_2}
\biggl(\sum_\beta P(e_{1,\alpha},e_{2,\beta})e_{2,\beta}
\\ \qquad\hphantom{=}
\ {}+P(e_{1,\alpha},y_2)x_2+\varepsilon_2(\nu-\mu)P(e_{1,\alpha},y_2)y_2^\itinv\biggr)\,{\rm e}^{(x_2|y_2)_{\fn^+_2}}
\frac{\partial g}{\partial z_{1,\alpha}}(z_1,x_2)\,{\rm d}x_2\biggr){\rm d}y_2
\\ \qquad
{}=-\frac{1}{\varepsilon_2}\sum_{\alpha}
\int_{\Omega_2}{\rm e}^{-(y_2|z_2)_{\fn^+_2}}\det_{\fn^+_2}(y_2)^{\varepsilon_2(\nu-\mu)}\,\biggl(\int_{a_2+\sqrt{-1}\fn^+_2} \biggl(\sum_\beta \frac{\partial}{\partial x_{2,\beta}}P(e_{1,\alpha},e_{2,\beta})
\\ \qquad\hphantom{=}
{} \ \times x_2 {\rm e}^{(x_2|y_2)_{\fn^+_2}}
\!+\!2\varepsilon_2(\nu\!-\!\mu)e_{1,\alpha}{\rm e}^{(x_2|y_2)_{\fn^+_2}}\!\biggr)
\frac{\partial g}{\partial z_{1,\alpha}}(z_1,x_2){\rm d}x_2\!\biggr){\rm d}y_2
\\ \qquad
{}=\int_{\Omega_2}{\rm e}^{-(y_2|z_2)_{\fn^+_2}}\det_{\fn^+_2}(y_2)^{\varepsilon_2(\nu-\mu)}\,
\biggl(\int_{a_2+\sqrt{-1}\fn^+_2}{\rm e}^{(x_2|y_2)_{\fn^+_2}}
\\ \qquad\hphantom{=}
{} \ \times\biggl(\frac{1}{\varepsilon_2}\sum_{\alpha,\beta}
P(e_{1,\alpha},e_{2,\beta})x_2\frac{\partial^2}{\partial z_{1,\alpha}\partial x_{2,\beta}}
+2(\mu-\nu)\sum_{\alpha}e_{1,\alpha}\frac{\partial}{\partial z_{1,\alpha}}\biggr)\,g(z_1,x_2)\,{\rm d}x_2\biggr){\rm d}y_2,
\end{gather*}
where we have used \cite[Proposition III.4.2(ii)]{FK} at the 4th equality. Therefore we have
\begin{gather*}
\Proj_1\big(\cB_{-\mu+2\nu+\frac{n}{r},z}^{\fn^+}\big)
\int_{\Omega_2}{\rm e}^{-(y_2|z_2)_{\fn^+_2}}\det_{\fn^+_2}(y_2)^{\varepsilon_2(\nu-\mu)}
\bigg(\!\int_{a_2+\sqrt{-1}\fn^+_2}{\rm e}^{(x_2|y_2)_{\fn^+_2}}g(z_1,x_2)\,{\rm d}x_2\!\bigg){\rm d}y_2
\\
{}=\int_{\Omega_2}\!\!{\rm e}^{-(y_2|z_2)_{\fn^+_2}}\det_{\fn^+_2}(y_2)^{\varepsilon_2(\nu-\mu)}
\bigg(\!\int_{a_2+\sqrt{-1}\fn^+_2} \!\!{\rm e}^{(x_2|y_2)_{\fn^+_2}}\Proj_1(\cB_{\mu+\frac{n}{r},z_1+x_2}^{\fn^+})g(z_1,x_2)\,{\rm d}x_2\!\bigg){\rm d}y_2.
\end{gather*}
Now by \cite[Proof of Proposition XV.2.4]{FK} we have
\[
\cB_{\mu+\frac{n}{r}}^{\fn^+}g(z)=\cB_{\mu+\frac{n}{r}}^{\fn^+}\det_{\fn^+}(z)^{-\mu}f(z_2)
=\det_{\fn^+}(z)^{-\mu}\cB_{-\mu+\frac{n}{r}}^{\fn^+}f(z_2),
\]
and by Corollary~\ref{FKCor}, (\ref{diff_eq_Bessel}) and Proposition \ref{prop_simple_prelim}, for $\Re\lambda>\frac{2n}{r}-1$ we have
\begin{align*}
\Proj_1(\cB_\lambda)f(z_2)
&=\Proj_1(\cB_\lambda)\times
\begin{cases} (\lambda)_l \bigl\langle f(x_2),{\rm e}^{(x|\overline{z})_{\fp^+}}\bigr\rangle_{\lambda,x}
 &(\varepsilon_2=1)
 \\
(\lambda)_{\underline{1}_l,d} \bigl\langle f(x_2),{\rm e}^{(x|\overline{z})_{\fp^+}}\bigr\rangle_{\lambda,x} & (\varepsilon_2=2) \end{cases} \\ &=0,
\end{align*}
and by analytic continuation this holds for all $\lambda\in\BC$. Hence
$\Proj_1\bigl(\cB_{-\mu+2\nu+\frac{n}{r}}^{\fn^+}F_{\varepsilon_2}^l \big(\begin{smallmatrix}\nu\\\mu\end{smallmatrix};f;z\big)\bigr)=0$ holds.
\end{proof}
\begin{proof}[Proof of Theorem \ref{thm_simple}]
First we prove the 1st equality of (\ref{thm'_simple}). For $\Re\mu>\frac{1}{\varepsilon_2}\big(\frac{2n_2}{r_2}-1\big)$,
by the Laplace and the inverse Laplace transforms and Proposition \ref{prop_simple}(1) we have
\begin{gather*}
\det_{\fn^+_2}\bigg(\frac{\partial}{\partial z_2}\bigg)^k\det_{\fn^+}(z)^{-\mu}f(z_2)
=\det_{\fn^+_2}\bigg(\frac{\partial}{\partial z_2}\bigg)^k
\int_{\Omega_2}{\rm e}^{-(y_2|z_2)_{\fn^+_2}}
\\ \qquad\hphantom{=}
{} \ \times\bigg(\frac{1}{(2\pi\sqrt{-1})^{n_2}}\int_{a_2+\sqrt{-1}\fn^+_2}{\rm e}^{(x_2|z_2)_{\fn^+}}
\det_{\fn^+}(z_1+x_2)^{-\mu}f(x_2)\,{\rm d}x_2\bigg){\rm d}y_2
\\ \qquad
{}=\int_{\Omega_2}{\rm e}^{-(y_2|z_2)_{\fn^+_2}}\det_{\fn^+_2}(-y_2)^k
\\ \qquad\hphantom{=}
{} \ \times\bigg(\frac{1}{(2\pi\sqrt{-1})^{n_2}}\int_{a_2+\sqrt{-1}\fn^+_2}{\rm e}^{(x_2|z_2)_{\fn^+}}
\det_{\fn^+}(z_1+x_2)^{-\mu}f(x_2)\,{\rm d}x_2\bigg){\rm d}y_2
\\ \qquad
{}=\begin{cases} \ds (-1)^{kr_2}\frac{\Gamma_{r_2}^{d_2}(\mu+k-(\underline{0}_{r_2-1},l))} {\Gamma_{r_2}^{d_2}(\mu-(\underline{0}_{r_2-1},l))}
F_1^l\left(\begin{matrix} \mu+k \\ \mu \end{matrix};f;z\right)\!, & \varepsilon_2=1, \\
\ds (-1)^{kr_2}\frac{\Gamma_{r_2}^{d_2}(2\mu+k-(\underline{0}_{r_2-l},\underline{1}_l))}
{\Gamma_{r_2}^{d_2}(2\mu-(\underline{0}_{r_2-l},\underline{1}_l))}
F_2^l\left(\begin{matrix} \mu+\frac{k}{2} \\ \mu \end{matrix};f;z\right)\!, &\varepsilon_2=2, \end{cases}
\\ \qquad{}
=\begin{cases} \ds (-1)^{kr_2}\left(\mu-(\underline{0}_{r_2-1},l)\right)_{\underline{k}_{r_2},d_2}
F_1^l\left(\begin{matrix} \mu+k \\ \mu \end{matrix};f;z\right)\!, & \varepsilon_2=1,
\\
\ds (-1)^{kr_2}\left(2\mu-(\underline{0}_{r_2-l},\underline{1}_l)\right)_{\underline{k}_{r_2},d_2}
F_2^l\left(\begin{matrix} \mu+\frac{k}{2} \\ \mu \end{matrix};f;z\right)\!, & \varepsilon_2=2. \end{cases}
\end{gather*}
Hence the 1st equality of (\ref{thm'_simple}) and the 2nd equalities of Theorem \ref{thm_simple} are proved.
Next we consider the 3rd equalities of Theorem \ref{thm_simple}. By Proposition \ref{prop_simple}(2), (3), we have
\begin{gather*}
F_{\varepsilon_2}^l\left(\begin{matrix} -\frac{k}{\varepsilon_2} \\[3pt] -\lambda-\frac{2k}{\varepsilon_2}+\frac{n}{r} \end{matrix};
f(g\cdot);z\right)
=\det_{\fn^+_2}(ge)^{-k}F_{\varepsilon_2}^l\left(\begin{matrix} -\frac{k}{\varepsilon_2} \\[3pt]
-\lambda-\frac{2k}{\varepsilon_2}+\frac{n}{r} \end{matrix};f;gz\right)\!, \qquad g\in K_1^\BC,
\\
\Proj_1\left(\cB_\lambda^{\fn^+}F_{\varepsilon_2}^l
\left(\begin{matrix} -\frac{k}{\varepsilon_2} \\[3pt] -\lambda-\frac{2k}{\varepsilon_2}+\frac{n}{r} \end{matrix};f;gz\right)\right)=0,
\end{gather*}
and $F_{\varepsilon_2}^l\left(\begin{smallmatrix} -\frac{k}{\varepsilon_2} \\ -\lambda-\frac{2k}{\varepsilon_2}+\frac{n}{r} \end{smallmatrix};
f;z\right)$ is a polynomial on $\fp^+$. Hence the linear map
\[ \det_{\fn^+_2}(x_2)^kf(x_2)\longmapsto
F_{\varepsilon_2}^l\left(\begin{matrix} -\frac{k}{\varepsilon_2} \\[3pt] -\lambda-\frac{2k}{\varepsilon_2}+\frac{n}{r} \end{matrix};f;z\right) \]
sits in
\[
\begin{cases}
\Hom_{K_1}\big(\cP_{(k+l,\underline{k}_{r_2-1})}(\fp^+_2),\cP(\fp^+) \cap\mathrm{Sol}(\Proj_1(\cB_\lambda^{\fn^+}))\big),
&\varepsilon_2=1,
\\
\Hom_{K_1}\big(\cP_{(\underline{k+1}_l,\underline{k}_{r_2-l})}(\fp^+_2),\cP(\fp^+) \cap\mathrm{Sol}(\Proj_1(\cB_\lambda^{\fn^+}))\big),
&\varepsilon_2=2.
\end{cases}
\]
By Theorem \ref{thm_Fmethod} these spaces are 1-dimensional if $\lambda>\frac{2n}{r}-1$, and hence by (\ref{diff_eq_Bessel}),
this is proportional to $\bigl\langle \det_{\fn^+_2}(x_2)^kf(x_2),{\rm e}^{(x|\overline{z})_{\fp^+}}\bigr\rangle_{\lambda,x}$.
By comparing the values at $z_1=0$, we get the desired formulas.
\end{proof}

(\ref{thm'_simple}) holds for all $\mu\in\BC$. Especially if we put $\lambda=\frac{n}{r}-j$, or $\mu=j-\frac{2k}{\varepsilon_2}$, then we have
\begin{equation}\label{identity_simple}
F_{\varepsilon_2}^l\left(\begin{matrix} -\frac{k}{\varepsilon_2} \\[3pt] j-\frac{2k}{\varepsilon_2} \end{matrix};f;z\right)
=\det_{\fn^+}(z)^j F_{\varepsilon_2}^l\left(\begin{matrix} j-\frac{k}{\varepsilon_2} \\[3pt] j-\frac{2k}{\varepsilon_2} \end{matrix};f;z\right)\!,
\end{equation}
and if $j\in\BZ_{\ge 0}$, $j\le \lfloor k/\varepsilon_2\rfloor$, then
$F_{\varepsilon_2}^l\left(\begin{smallmatrix} -\frac{k}{\varepsilon_2} \\ j-\frac{2k}{\varepsilon_2} \end{smallmatrix};f;z\right)$,
$F_{\varepsilon_2}^l\left(\begin{smallmatrix} j-\frac{k}{\varepsilon_2} \\ j-\frac{2k}{\varepsilon_2} \end{smallmatrix};f;z\right)$
and $\det_{\fn^+}(z)^j$ are polynomials. If $l=0$ or $r=2$, then this also follows from (\ref{Kummer2}) and the argument below.

Next we consider special cases. If $l=0$ and $f(x_2)=1$, then we have
\[
F_{\bm,\underline{0}_{\lceil r/2\rceil}}^{\varepsilon_2}[1](z_1,w_2)=\tilde{\Phi}_\bm^{\fp^+_1,\fp^-_2}(z_1,w_2).
\]
Similarly, if $\varepsilon_2=1$ and $r=2$ with general $l$, then we have
\[
F_{m,l}^1[f](z_1,w_2)=\tilde{\Phi}_m^{\fp^+_1,\fp^-_2}(z_1,w_2)f\big({}^t\hspace{-1pt}w_2^\itinv\big)
=\frac{1}{m!}\left(\frac{1}{2}(z_1|Q(w_2)z_1)_{\fp^+}\right)^mf\big({}^t\hspace{-1pt}w_2^\itinv\big).
\]
Therefore, the results become simpler for these cases by using the notations in (\ref{0F1_mid}), (\ref{2F1_mid}).
\begin{Corollary}\label{cor_scalar_simple}
Assume $\fp^+$, $\fp^+_2$ are of tube type.
\begin{enumerate}\itemsep=0pt
\item[$1.$] Suppose $\varepsilon_2=1$. Then for $k\in\BZ_{\ge 0}$, $\Re\lambda>\frac{2n}{r}-1$ and for $z=z_1+z_2\in\fp^+$, we have
\begin{align*}
\big\langle \det_{\fn^+_2}(x_2)^k,{\rm e}^{(x|\overline{z})_{\fp^+}}\big\rangle_{\lambda,x}
={}&\frac{\big(\lambda+k-\frac{d}{2}\big\lceil\frac{r_2}{2}\big\rceil +\frac{d_2}{2}\big)_{\underline{k}_{\lfloor r_2/2\rfloor},d}}
{(\lambda)_{(\underline{2k}_{\lfloor r_2/2\rfloor},\underline{k}_{\lceil r_2/2\rceil}),d}}
\det_{\fn^+_2}(z_2)^k
\\
&\times{}_2F_1^{\fp^+_1,\fp^-_2}\left( \begin{matrix}-k,-k-\frac{d_2}{2} \\
-\lambda-2k+\frac{n}{r}-\frac{d_2}{2}\end{matrix};z_1,{}^t\hspace{-1pt}z_2^\itinv\right)\!.
\end{align*}
\item[$2.$] Suppose $\varepsilon_2=1$, $r=2$. Then for $k,l\in\BZ_{\ge 0}$, $f(x_2)\in\cP_{(l,0)}(\fp^+_2)$, $\Re\lambda>n-1$ and for
$z=z_1+z_2\in\fp^+$, we have
\begin{align*}
\big\langle \det_{\fn^+_2}(x_2)^kf(x_2),{\rm e}^{(x|\overline{z})_{\fp^+}}\big\rangle_{\lambda,x}
={}&\frac{\big(\lambda+k+l-\frac{d-d_2}{2}\big)_k}{(\lambda)_{(2k+l,k),d}}\det_{\fn^+_2}(z_2)^kf(z_2) \\
&{}\times\,{}_2F_1\left(\begin{matrix}-k,-k-l-\frac{d_2}{2} \\
-\lambda-2k-l+\frac{n}{2}-\frac{d_2}{2}\end{matrix};\frac{\big(z_1|Q(z_2)^{-1}z_1\big)_{\fp^+}}{2}\right)\!.
\end{align*}
\item[$3.$] Suppose $\varepsilon_2=2$. Then for $k\in\BZ_{\ge 0}$, $\Re\lambda>\frac{2n}{r}-1$ and for $z=z_1+z_2\in\fp^+$, we have
\begin{align*}
\big\langle \det_{\fn^+_2}(x_2)^k,{\rm e}^{(x|\overline{z})_{\fp^+}}\big\rangle_{\lambda,x}
={}&\frac{\big(\lambda+\big\lceil \frac{k}{2}\big\rceil-\frac{d}{2}r_2-\frac{1}{2}\big)_{\underline{\lfloor k/2\rfloor}_{r_2},d}}
{(\lambda)_{(\underline{k}_{r_2},\underline{\lfloor k/2\rfloor}_{r_2}),d}}
\det_{\fn^+_2}(z_2)^k
\\
&\times{}_2F_1^{\fp^+_1,\fp^-_2}\left(\begin{matrix} -\frac{k}{2},-\frac{k-1}{2} \\
-\lambda-k+\frac{n}{r}+\frac{1}{2} \end{matrix};z_1,{}^t\hspace{-1pt}z_2^\itinv\right)\!.
\end{align*}
\end{enumerate}
\end{Corollary}
Therefore by Remark \ref{rem_Heckman-Opdam},
the inner product $\big\langle \det_{\fn^+_2}(x_2)^k,{\rm e}^{(x|\overline{z})_{\fp^+}}\big\rangle_{\lambda,x}$ is
given by using a~Heck\-man--Opdam's multivariate hypergeometric polynomial of type $BC_{\lfloor r/2\rfloor}$.

\subsection{Computation of poles}

By Theorem \ref{thm_simple}, we can compute the poles of $\big\langle F(x_2),{\rm e}^{(x|\overline{z})_{\fp^+}}\big\rangle_{\lambda,x}$
for $F(x_2)\in\cP_{(k+l,\underline{k}_{a-1})}\big(\fp^+_2\big)$ or $F(x_2)\in\cP_{(\underline{k+1}_l,\underline{k}_{a-l})}\big(\fp^+_2\big)$
with $1\le a\le r_2$, without the assumption that $\fp^+$ and $\fp^+_2$ are of tube type.

\begin{Corollary}\label{cor_pole_simple}\quad
\begin{enumerate}\itemsep=0pt
\item[$1.$] Suppose $\varepsilon_2=1$. Let $k,l\in\BZ_{\ge 0}$, $2\le a\le r_2$, $F(x_2)\in\cP_{(k+l,\underline{k}_{a-1})}\big(\fp^+_2\big)$.
Then as functions of~$\lambda$,
\begin{gather*}
(\lambda)_{(2k+l,\underline{2k}_{a/2-1},\underline{k}_{a/2}),d}
\big\langle F(x_2),{\rm e}^{(x|\overline{z})_{\fp^+}}\big\rangle_{\lambda,x} \hspace{30.5mm} \text{if}\quad a\colon \text{even},
\\
(\lambda)_{(2k+l,\underline{2k}_{\lfloor a/2\rfloor-1},\min\{2k,k+l\},\underline{k}_{\lfloor a/2\rfloor}),d}
\big\langle F(x_2),{\rm e}^{(x|\overline{z})_{\fp^+}}\big\rangle_{\lambda,x}\qquad
\text{if}\quad a\colon\text{odd}
\end{gather*}
are holomorphically continued for all $\lambda\in\BC$. Moreover, when $F(x_2)\ne 0$, if $a$ is even or $k=0$ or $l=0$,
then these give non-zero polynomials in $z\in\fp^+$ for all $\lambda\in\BC$.
In addition, for $z_1=0$ we have
\begin{gather*}
\big\langle F(x_2),{\rm e}^{(x_2|\overline{z_2})_{\fp^+}}\big\rangle_{\lambda,x}
\\ \qquad
{}=\begin{cases}
\ds \frac{\big(\lambda+k-\frac{d}{4}a+\frac{d_2}{2}+(l,\underline{0}_{a/2-1})\big)
_{\underline{k}_{a/2},d}}{(\lambda)_{(2k+l,\underline{2k}_{a/2-1},\underline{k}_{a/2}),d}}F(z_2), & a\colon\text{even},
\\%[10pt]
\ds \frac{\big(\lambda-\frac{d}{2}\big\lfloor\frac{a}{2}\big\rfloor+\max\{2k,k+l\}\big)_{\min\{k,l\}}
\big(\lambda+k-\frac{d}{2}\big\lceil\frac{a}{2}\big\rceil+\frac{d_2}{2}\big)_{\underline{k}_{\lfloor a/2\rfloor},d}}
{(\lambda)_{(2k+l,\underline{2k}_{\lfloor a/2\rfloor-1},\min\{2k,k+l\},\underline{k}_{\lfloor a/2\rfloor}),d}}F(z_2), \hspace{-40pt}& \\
& a\colon\text{odd}. \end{cases}
\end{gather*}

\item[$2.$] Suppose $\varepsilon_2=2$. Let $k\in\BZ_{\ge 0}$, $0\le l< a\le r_2$, $F(x_2)\in\cP_{(\underline{k+1}_l,\underline{k}_{a-l})}\big(\fp^+_2\big)$.
Then as a~func\-tion of $\lambda$,
\[
(\lambda)_{(\underline{k+1}_l,\underline{k}_{a-l},\underline{\lceil k/2\rceil}_l,\underline{\lfloor k/2\rfloor}_{a-l}),d}
\big\langle F(x_2),{\rm e}^{(x|\overline{z})_{\fp^+}}\big\rangle_{\lambda,x}
\]
is holomorphically continued for all $\lambda\in\BC$, and when $F(x_2)\ne 0$,
this gives a non-zero polynomial in $z\in\fp^+$ for all $\lambda\in\BC$.
In addition, for $z_1=0$ we have
\begin{gather*}
\big\langle F(x_2),{\rm e}^{(x_2|\overline{z_2})_{\fp^+}}\big\rangle_{\lambda,x}
=\frac{\big(\lambda-\frac{d}{2}a+\big(\underline{\big\lfloor\frac{k}{2}\big\rfloor+\frac{1}{2}}_l,
\underline{\left\lceil\frac{k}{2}\right\rceil-\frac{1}{2}}_{a-l}\big)\big)
_{(\underline{\lceil k/2\rceil}_l,\underline{\lfloor k/2\rfloor}_{a-l}),d}}
{(\lambda)_{(\underline{k+1}_l,\underline{k}_{a-l},\underline{\lceil k/2\rceil}_l,\underline{\lfloor k/2\rfloor}_{a-l}),d}}F(z_2).
\end{gather*}
\end{enumerate}
\end{Corollary}

\begin{proof}
(1) Let $\varepsilon_2=1$. First we assume $\fp^+$, $\fp^+_2$ are of tube type, $a=r_2$ and let $F(x_2)=\det_{\fn^+_2}(x_2)^kf(x_2)$.
We~put $\mu:=-\lambda-2k+\frac{n}{r}$. Then by Theorem \ref{thm_simple}(1), we have
\begin{gather*}
(\lambda)_{(2k+l,\underline{2k}_{a/2-1},\underline{k}_{a/2}),d}
\big\langle \det_{\fn^+_2}(x_2)^kf(x_2),{\rm e}^{(x|\overline{z})_{\fp^+}}\big\rangle_{\lambda,x}
\\ \qquad
{}=(-1)^{kr/2}\bigg(\mu-\frac{d_2}{2}-(\underline{0}_{r/2-1},l)\bigg)_{\underline{k}_{r/2},d}
F_1^l\left(\begin{matrix}-k\\\mu\end{matrix};f;z\right)
\\ \qquad
{}=(-1)^{kr/2}\bigg(\mu-\frac{d_2}{2}-(\underline{0}_{r/2-1},l)\bigg)_{\underline{k}_{r/2},d} \det_{\fn^+_2}(z_2)^k
\\ \qquad\hphantom{=}
{}\ \times \sum_{\bm\in\BZ_{++}^{r/2}}\sum_{\substack{\bl\in(\BZ_{\ge 0})^{r/2}\\ |\bl|=l}}
\frac{(-k)_{\bm,d}\big({-}k-\frac{d_2}{2} -(\underline{0}_{r/2-1},l)\big)_{\bm-\bl+(\underline{0}_{r/2-1},l),d}}
{\big(\mu-\frac{d_2}{2}-(\underline{0}_{r/2-1},l)\big)_{\bm-\bl+(\underline{0}_{r/2-1},l),d}} F_{\bm,\bl}^1[f]\big(z_1,{}^t\hspace{-1pt}z_2^\itinv\big)
\end{gather*}
for even $r=r_2$, and
\begin{gather*}
(\lambda)_{(2k+l,\underline{2k}_{\lfloor a/2\rfloor-1},\min\{2k,k+l\},\underline{k}_{\lfloor a/2\rfloor}),d}
\big\langle \det_{\fn^+_2}(x_2)^kf(x_2),{\rm e}^{(x|\overline{z})_{\fp^+}}\big\rangle_{\lambda,x}
\\ \qquad
{}=(-1)^{k\lfloor r/2\rfloor+\min\{k,l\}}\bigg(\mu-\frac{d_2}{2}\bigg)_{\underline{k}_{\lfloor r/2\rfloor},d}\bigg(\mu-\frac{d_2}{2}(r-1)-l\bigg)_{\min\{k,l\}}F_1^l \left(\begin{matrix}-k\\\mu\end{matrix};f;z\right)
\\ \qquad
{}=(-1)^{k\lfloor r/2\rfloor+\min\{k,l\}}\bigg(\mu-\frac{d_2}{2}\bigg)_{\underline{k}_{\lfloor r/2\rfloor},d}
\bigg(\mu-\frac{d_2}{2}(r-1)-l\bigg)_{\min\{k,l\}}\det_{\fn^+_2}(z_2)^k
\\ \qquad\hphantom{=}
{} \ \times\sum_{\bm\in\BZ_{++}^{\lfloor r/2\rfloor}}\sum_{\substack{\bl\in(\BZ_{\ge 0})^{\lceil r/2\rceil}\\ |\bl|=l}}
\frac{(-k)_{\bm,d}\big({-}k-\frac{d_2}{2}\big)_{\bm-\bl',d}\big({-}k-\frac{d_2}{2}(r-1)-l\big)_{l-l_{\lceil r/2\rceil}}}
{\big(\mu-\frac{d_2}{2}\big)_{\bm-\bl',d}\big(\mu-\frac{d_2}{2}(r-1)-l\big)_{l-l_{\lceil r/2\rceil}}}
\\ \qquad\hphantom{=}
{} \ \times F_{\bm,\bl}^1[f]\big(z_1,{}^t\hspace{-1pt}z_2^\itinv\big)
\end{gather*}
for odd $r=r_2$. Then $(-k)_{\bm,d}F_{\bm,\bl}^1[f](z_1,{}^t\hspace{-1pt}z_2^\itinv)$ is non-zero only if
\[
\begin{cases}
\bm\le \underline{k}_{r/2}, \quad \bm-\bl+(\underline{0}_{r/2-1},l)\le \underline{k}_{r/2}, & r\colon\text{even},
\\
\bm\le \underline{k}_{\lfloor r/2\rfloor}, \quad \bm-\bl'\le \underline{k}_{\lfloor r/2\rfloor}, \quad
l-l_{\lceil r/2\rceil}=|\bl'|\le \min\{k,l\}, & r\colon\text{odd},
\end{cases}
\]
where $\bk\le\bl$ means $k_j\le l_j$ for all $j$, and $\bl=(\bl',l_{\lceil r/2\rceil})$ for odd $r$.
Hence the above formulas are holomorphic for all $\mu\in\BC$.
Moreover, all non-zero $F_{\bm,\bl}^1[f]$ are linearly independent, and if $r$ is even or $k=0$ or $l=0$, then the coefficient of
\[ F_{\underline{k}_{\lfloor r/2\rfloor},(\underline{0}_{\lceil r/2\rceil-1},l)}^1[f]\big(z_1,{}^t\hspace{-1pt}z_2^\itinv\big)
=\tilde{\Phi}_{\underline{k}_{\lfloor r/2\rfloor}}^{\fp^+_1,\fp^-_2}\big(z_1,{}^t\hspace{-1pt}z_2^\itinv\big)f(z_2) \]
is non-zero for all $\mu\in\BC$.
Hence the above formulas are non-zero for all $\mu\in\BC$ when $f(x_2)\ne 0$ and $r$ is even or $k=0$ or $l=0$.
The values at $z_1=0$ are clear.

Next we consider general $\fp^+$, $\fp^+_2$ and general $2\le a\le r_2$.
We~take a tripotent $e'\in\fp^+_2$ of rank $a$, and let $\fp^{+\prime}:=\fp^+(e')_2$, $\fp^{+\prime}_2:=\fp^+_2\cap\fp^{+\prime}$.
For $z\in\fp^+$, let $z'\in\fp^{+\prime}$ be the projection of $z$ onto $\fp^{+\prime}$.
We~take $f(x_2')\in\cP_{(l,0,\dots,0)}(\fp^{+\prime}_2)$,
and let $F(x_2)=\det_{\fn^{+\prime}_2}(x_2')^kf(x_2')\in\cP_{(k+l,\underline{k}_{a-1})}(\fp^{+\prime}_2)
\subset\cP_{(k+l,\underline{k}_{a-1})}(\fp^+_2)$.
Then by Corollary~\ref{inner_sub}, as a function of $\lambda$,
\begin{gather*}
(\lambda)_{(2k+l,\underline{2k}_{\lfloor a/2\rfloor-1},(\min\{2k,k+l\},)\underline{k}_{\lfloor a/2\rfloor}),d}
\big\langle F(x_2),{\rm e}^{(x|\overline{z})_{\fp^+}}\big\rangle_{\lambda,x,\fp^+}
\\ \qquad
{}=(\lambda)_{(2k+l,\underline{2k}_{\lfloor a/2\rfloor-1},(\min\{2k,k+l\},)\underline{k}_{\lfloor a/2\rfloor}),d}
\big\langle \det_{\fn^{+\prime}_2}(x_2')^kf(x_2'),{\rm e}^{(x'|\overline{z'})_{\fp^{+\prime}}}\big\rangle_{\lambda,x',\fp^{+\prime}}
\end{gather*}
is entirely holomorphic on $\BC$, and by the $K_1$-equivariance, this holds for all $F(x_2)\!\in\!\cP_{(\!k{+}l,\underline{k}_{a{-}1}\!)}\!\big(\fp^+_2\!\big)$.
The proof of the results for $z_1=0$ is similar.

(2) Next let $\varepsilon_2=1$. Again we assume $\fp^+$, $\fp^+_2$ are of tube type, $a=r_2$ and let $F(x_2)=\det_{\fn^+_2}(x_2)^kf(x_2)$.
We~put $\mu:=-\lambda-k+\frac{n}{r}$. Then by Theorem \ref{thm_simple}(2), we have
\begin{gather*}
(\lambda)_{(\underline{k+1}_l,\underline{k}_{a-l},\underline{\lceil k/2\rceil}_l,\underline{\lfloor k/2\rfloor}_{a-l}),d}
\big\langle F(x_2),{\rm e}^{(x|\overline{z})_{\fp^+}}\big\rangle_{\lambda,x}
\\ \qquad
{}=(-1)^{\lceil k/2\rceil l+\lfloor k/2\rfloor (r_2-l)}\bigg(\mu+\bigg(\underline{\frac{1}{2}}_{r_2-l},\underline{-\frac{1}{2}}_l\bigg)
\bigg)_{(\underline{\lfloor k/2\rfloor}_{r_2-l},\underline{\lceil k/2\rceil}_l),d}
F_2^l\left(\begin{matrix}-k/2\\\mu\end{matrix};f;z\right)
\\ \qquad
{}=(-1)^{\lceil k/2\rceil l+\lfloor k/2\rfloor (r_2-l)}\left(\mu+\left(\underline{\frac{1}{2}}_{r_2-l},\underline{-\frac{1}{2}}_l\right)
\right)_{(\underline{\lfloor k/2\rfloor}_{r_2-l},\underline{\lceil k/2\rceil}_l),d}\det_{\fn^+_2}(z_2)^k \\
\qquad\hphantom{=}
{} \ \times\sum_{\bm\in\BZ_{++}^{r_2}}\sum_{\substack{\bl\in\{0,1\}^{r_2}\\ |\bl|=l}}\frac{\big({-}\frac{k}{2}\big)_{\bm,d}
\big({-}\frac{k}{2}+\bigl(\underline{\frac{1}{2}}_{r_2-l},\underline{-\frac{1}{2}}_l\bigr)
\big)_{\bm-\bl+(\underline{0}_{r_2-l},\underline{1}_l),d}}
{\big(\mu+\bigl(\underline{\frac{1}{2}}_{r_2-l}, \underline{-\frac{1}{2}}_l\bigr)\big)_{\bm-\bl+(\underline{0}_{r_2-l},\underline{1}_l),d}}
F_{\bm,\bl}^2[f]\big(z_1,{}^t\hspace{-1pt}z_2^\itinv\big).
\end{gather*}
Then $\big({-}\frac{k}{2}\big)_{\bm,d}\big({-}\frac{k}{2}+\bigl(\underline{\frac{1}{2}}_{r_2-l}, \underline{-\frac{1}{2}}_l\bigr)\big)_{\bm-\bl+(\underline{0}_{r_2-l},\underline{1}_l),d} F_{\bm,\bl}^2[f]\big(z_1,{}^t\hspace{-1pt}z_2^\itinv\big)$ is non-zero only if
\[ \bm\le\left(\underline{\left\lceil\frac{k}{2}\right\rceil}_l,\underline{\left\lfloor\frac{k}{2}\right\rfloor}_{r_2-l}\right)\!, \qquad
\bm-\bl+(\underline{0}_{r_2-l},\underline{1}_l)
\le\left(\underline{\left\lfloor\frac{k}{2}\right\rfloor}_{r_2-l},\underline{\left\lceil\frac{k}{2}\right\rceil}_l\right)\!. \]
Hence the above formula is holomorphic for all $\mu\in\BC$.
Moreover, all non-zero $F_{\bm,\bl}^2[f]$ are linearly independent, and the coefficient of
\[
\begin{cases} F_{\underline{k/2}_{r_2},(\underline{0}_{r_2-l},\underline{1}_l)}^2[f]\big(z_1,{}^t\hspace{-1pt}z_2^\itinv\big), &k\colon\text{even},
\\
F^2_{(\underline{\lceil k/2\rceil}_l,\underline{\lfloor k/2\rfloor}_{r_2-l}),(\underline{1}_l,\underline{0}_{r_2-l})}[f]\big(z_1,{}^t\hspace{-1pt}z_2^\itinv\big), & k\colon\text{odd} \end{cases}
\]
is non-zero for all $\mu\in\BC$. If $f(x_2)\ne 0$, when $k$ is even,
\[
F_{\underline{k/2}_{r_2},(\underline{0}_{r_2-l},\underline{1}_l)}^2[f](z_1,w_2) =\tilde{\Phi}_{\underline{k/2}_{r_2}}^{\fp^+_1,\fp^-_2}(z_1,w_2)f\big({}^t\hspace{-1pt}w_2^\itinv\big)\ne 0
\]
holds, and when $k$ is odd,
\[
F^2_{(\underline{\lceil k/2\rceil}_l,\underline{\lfloor k/2\rfloor}_{r_2-l}),(\underline{1}_l,\underline{0}_{r_2-l})}[f](z_1,w_2)\ne 0
\]
also holds by Lemma \ref{lem_nonzero_simple} given later.
Hence the above formula is non-zero for all $\mu\in\BC$. The values at $z_1=0$ are clear.
For general $\fp^+$, $\fp^+_2$ and general $a\le r_2$, this is also proved as the case~(1).
\end{proof}

To complete the proof of Corollary~\ref{cor_pole_simple}(2), we prove the following.
\begin{Lemma}\label{lem_nonzero_simple}
Let $\varepsilon_2=2$ and let $f(x_2)\in\cP_{\underline{1}_l}\big(\fp^+_2\big)$, $f(x_2)\ne 0$.
Then for $k\in\BZ_{\ge 0}$, $\bm\in\BZ_{++}^l$ with $m_l\ge 1$,
$F_{(\bm,\underline{0}_{r_2-l})+\underline{k}_{r_2},(\underline{1}_l,\underline{0}_{r_2-l})}^2[f](z_1,w_2)$ gives a non-zero polynomial.
\end{Lemma}
\begin{proof}
When $\varepsilon_2=2$, $\fp^+_1$ automatically is of tube type, and for $k\in\BZ_{\ge 0}$, $\bm\in\BZ_{++}^{r_2}$, we have
\begin{gather*}
\tilde{\Phi}_{\bm+\underline{k}_{r_2}}^{\fp^+_1,\fp^-_2}(z_1,w_2)
=c_{k,\bm}\tilde{\Phi}_{\underline{k}_{r_2}}^{\fp^+_1,\fp^-_2}(z_1,w_2)\tilde{\Phi}_\bm^{\fp^+_1,\fp^-_2}(z_1,w_2), \\
F_{\bm+\underline{k}_{r_2},\bl}^2[f](z_1,w_2)
=c_{k,\bm}\tilde{\Phi}_{\underline{k}_{r_2}}^{\fp^+_1,\fp^-_2}(z_1,w_2)F_{\bm,\bl}^2[f](z_1,w_2)
\end{gather*}
for some $c_{k,\bm}\ne 0$. Hence we may assume $k=0$. By the definition of $F_{\bm,\bl}^2[f]$, we have
\begin{gather*}
\tilde{\Phi}_\bm^{\fp^+_1,\fp^-_2}(z_1,w_2)f\big({}^t\hspace{-1pt}w_2^\itinv\big)\det_{\fn^-_2}(w_2)
=\sum_{\substack{\bl\in\{0,1\}^{r_2} \\ |\bl|=r_2-l}}F_{\bm,\underline{1}_{r_2}-\bl}^2[f](z_1,w_2)\det_{\fn^-_2}(w_2),
\\
F_{\bm,\underline{1}_{r_2}-\bl}^2[f](z_1,w_2)\det_{\fn^-_2}(w_2) \in\cP(\fp^+_1)_{z_1}\otimes\cP_{2\bm+\bl}(\fp_2^-)_{w_2}.
\end{gather*}
Now let $\bm\in\BZ_{++}^l$ with $m_l\ge 1$. If we operate $\det_{\fn^-_2}\big(\frac{\partial}{\partial w_2}\big)$, then since only $\bl=(\underline{0}_l,\underline{1}_{r_2-l})$ term remains, by \cite[Proposition~VII.1.6]{FK} we have
\begin{gather*}
\det_{\fn^-_2}\bigg(\frac{\partial}{\partial w_2}\bigg)
\tilde{\Phi}_{(\bm,\underline{0}_{r_2-l})}^{\fp^+_1,\fp^-_2}(z_1,w_2)f\big({}^t\hspace{-1pt}w_2^\itinv\big)\det_{\fn^-_2}(w_2) \\ \qquad
{}=\bigg(\frac{n_2}{r_2}+(2\bm-\underline{1}_l,\underline{0}_{r_2-l})\bigg)_{\underline{1}_{r_2},d_2}
F_{(\bm,\underline{0}_{r_2-l}),(\underline{1}_l,\underline{0}_{r_2-l})}^2[f](z_1,w_2).
\end{gather*}
Now we take mutually orthogonal tripotents $e',e''\in\fn^+_2\subset\fp^+_2$ of rank $l$ and $r_2-l$ respectively such that $e'+e''=e$ holds.
Let $\fp^\pm_2(e')_2=\fp^\pm_2(e'')_0=:\fp^{\pm\prime}_2$, $\fp^\pm_2(e')_0=\fp^\pm_2(e'')_2=:\fp^{\pm\prime\prime}_2$,
and for $w_2\in\fp^\pm_2$, let $w_2'\in\fp^{\pm\prime}_2$, $w_2''\in\fp^{\pm\prime\prime}$ be the orthogonal projections.
Also let $\fp^\pm_1\cap(\fp^\pm(e')_2)=:\fp^{\pm\prime}_1$. First we consider the case $f(x_2)=\det_{\fn^{+\prime}_2}(x_2')$.
Then if $z_1=z_1'\in\fp^{+\prime}_1$, by (\ref{det_proj}) we have
\begin{gather*}
\det_{\fn^-_2}\bigg(\frac{\partial}{\partial w_2}\bigg) \tilde{\Phi}_{(\bm,\underline{0}_{r_2-l})}^{\fp^+_1,\fp^-_2}(z_1',w_2)
\det_{\fn^{+\prime}_2}\big(\big({}^t\hspace{-1pt}w_2^\itinv\big)'\big)\det_{\fn^-_2}(w_2)
\\ \qquad
{}=\det_{\fn^{-\prime}_2}\bigg(\frac{\partial}{\partial w_2'}\bigg)\det_{\fn^{-\prime\prime}_2}\bigg(\frac{\partial}{\partial w_2''}\bigg)
\tilde{\Phi}_\bm^{\fp^{+\prime}_1,\fp^{-\prime}_2}(z_1',w_2')\det_{\fn^{-\prime\prime}_2}(w_2'')
\\ \qquad
{}=\bigg(\frac{n_2'}{l}-1+\bm\bigg)_{\underline{1}_l,d_2} \bigg(\frac{n_2''}{r_2-l}\bigg)_{\underline{1}_{r_2-l},d_2}
\tilde{\Phi}_\bm^{\fp^{+\prime}_1,\fp^{-\prime}_2}(z_1',w_2')\det_{\fn^{-\prime}_2}(w_2')^{-1}\ne 0,
\end{gather*}
where $n_2':=\dim\fp^{+\prime}_2$, $n_2'':=\dim\fp^{+\prime\prime}_2$, and hence $F_{(\bm,\underline{0}_{r_2-l}),(\underline{1}_l,\underline{0}_{r_2-l})}^2[\det_{\fn^{+\prime}_2}]\ne 0$ holds. For general $f(x_2)\ne 0$, $F_{(\bm,\underline{0}_{r_2-l}),(\underline{1}_l,\underline{0}_{r_2-l})}^2[f]\ne 0$
follows from the $K_1^\BC$-equivariance of $F_{\bm,\bl}^2$.
\end{proof}

Especially, by (\ref{inner_Fubini}), for $F(x_2)\!\in\!\cP_{(k+l,\underline{k}_{a-1})}\big(\fp^+_2\big)$ (when $\varepsilon_2=1$) or
for $F(x_2)\!\in\!\cP_{(\underline{k+1}_l,\underline{k}_{a-l})}\big(\fp^+_2\big)$ (when $\varepsilon_2=2$) we have
\begin{align*}
\Vert F(x_2)\Vert_{\lambda,x,\fp^+}^2&=\big\langle\big\langle F(x_2),{\rm e}^{(x_2|\overline{z_2})_{\fp^+}}\big\rangle_{\lambda,x,\fp^+},F(z_2)\big\rangle_{F,z,\fp^+}
\\
&=\bigg\langle\big\langle F(x_2),{\rm e}^{(x_2|\overline{z_2})_{\fp^+}}\big\rangle_{\lambda,x,\fp^+},
F\bigg(\frac{z_2}{\varepsilon_2}\bigg)\bigg\rangle_{F,z_2,\fp^+_2}
\\[2ex]
&=\begin{cases}
\ds \frac{\big(\lambda+k\!-\frac{d}{4}a\!+\frac{d_2}{2}+(l,\underline{0}_{a/2-1})\big)
_{\underline{k}_{a/2},d}}{(\lambda)_{(2k+l,\underline{2k}_{a/2-1},\underline{k}_{a/2}),d}}\Vert F(z_2)\Vert_{F,z_2,\fp^+_2}^2,
\quad\varepsilon_2=1,\ a\colon\text{even},
\\
\ds \frac{\big(\lambda-\frac{d}{2}\big\lfloor\frac{a}{2}\big\rfloor+\max\{2k,k+l\}\big)_{\min\{k,l\}}
\big(\lambda\!+\!k\!-\!\frac{d}{2}\big\lceil\frac{a}{2}\big\rceil\!+\!\frac{d_2}{2}\big)_{\underline{k}_{\lfloor a/2\rfloor},d}}
{(\lambda)_{(2k+l,\underline{2k}_{\lfloor a/2\rfloor-1},\min\{2k,k+l\},\underline{k}_{\lfloor a/2\rfloor}),d}}
 \\ \hspace{56mm}
{}\times \Vert F(z_2)\Vert_{F,z_2,\fp^+_2}^2,\quad \varepsilon_2=1,\ a\colon\text{odd},
 \\
\ds \frac{\big(\lambda-\frac{d}{2}a+\big(\underline{\big\lfloor\frac{k}{2}\big\rfloor+\frac{1}{2}}_l,
\underline{\left\lceil\frac{k}{2}\right\rceil-\frac{1}{2}}_{a-l}\big)\big)
_{(\underline{\lceil k/2\rceil}_l,\underline{\lfloor k/2\rfloor}_{a-l}),d}}
{(\lambda)_{(\underline{k+1}_l,\underline{k}_{a-l},\underline{\lceil k/2\rceil}_l,\underline{\lfloor k/2\rfloor}_{a-l}),d}}
\bigg\Vert F\bigg(\frac{z_2}{\sqrt{2}}\bigg)\bigg\Vert_{F,z_2,\fp^+_2}^2,
\\ \hspace{95mm}
 \varepsilon_2=2.
\end{cases}\!\!\!
\end{align*}
Moreover, if $\fp^+_2$ is of tube type, $a=r_2$, $l=0$ and $F(x_2)=\det_{\fn^+_2}(x_2)^k$,
then since
\[
\big\Vert \det_{\fn^+_2}(z_2)^k\big\Vert^2_{F,z_2,\fp^+_2} =\bigg(\frac{n_2}{r_2}\bigg)_{\underline{k}_{r_2},d_2}
\]
holds (see \cite[Proposition XI.4.1(ii)]{FK}), we get
\[
\big\Vert \det_{\fn^+_2}(x_2)^k\big\Vert_{\lambda,x,\fp^+}^2=
\begin{cases}
\ds \frac{\big(\frac{n_2}{r_2}\big)_{\underline{k}_{r_2},d_2}
\big(\lambda+k-\frac{d}{2}\big\lceil\frac{r_2}{2}\big\rceil+\frac{d_2}{2}\big)_{\underline{k}_{\lfloor r_2/2\rfloor},d}}
{(\lambda)_{\underline{2k}_{\lfloor r_2/2\rfloor},d}
\big(\lambda-\frac{d}{2}\big\lfloor\frac{r_2}{2}\big\rfloor\big)_{\underline{k}_{\lceil r_2/2\rceil},d}}, &\varepsilon_2=1,
\\[3ex]
\ds \frac{\big(\frac{n_2}{r_2}\big)_{\underline{k}_{r_2},d_2}
\big(\lambda+\big\lceil\frac{k}{2}\big\rceil-\frac{d}{2}r_2-\frac{1}{2}\big)_{\underline{\lfloor k/2\rfloor}_{r_2},d}}
{2^{kr_2}(\lambda)_{\underline{k}_{r_2},d}\big(\lambda-\frac{d}{2}r_2\big)_{\underline{\lfloor k/2\rfloor}_{r_2},d}}, &\varepsilon_2=2.
\end{cases}
\]

\subsection{Individual cases}

In this section we observe the computation of inner products (Corollary~\ref{cor_scalar_simple}) under the explicit realization of $\fp^+$.
For {\it Case} 1, $\fp^+=\BC^n=\BC^{n'}\oplus\BC^{n''}$ with $n''\ge 3$, we write $k+l=:k_1$, $k=:k_2$,
$q(x_2)^kf(x_2)=:F(x_2)\in\cP_{(k_1,k_2)}(\BC^{n''})$. Then for $z=(z_1,z_2)\in\BC^{n'}\oplus\BC^{n''}$,
since $\frac{1}{2}(z_1|Q(z_2)^{-1}z_1)_{\fp^+}=-\frac{q(z_1)}{q(z_2)}$ holds,
by Corollary~\ref{cor_scalar_simple}(2) we have
\begin{gather}
\big\langle F(x_2),{\rm e}^{2q(x,\overline{z})}\big\rangle_{\lambda,x}
=\frac{\big(\lambda+k_1-\frac{n'}{2}\big)_{k_2}}{(\lambda)_{k_1+k_2}\big(\lambda\!-\!\frac{n-2}{2}\big)_{k_2}}
{}_2F_1\left(\begin{matrix} -k_2,-k_1-\frac{n''}{2}+1\\-\lambda\!-\!k_1\!-\!k_2\!+\!\frac{n'}{2}\!+\!1\end{matrix};
-\frac{q(z_1)}{q(z_2)}\right)F(z_2). \label{rank2}
\end{gather}
Similarly, when $n''=1$, for $F(x_2)=x_2^k\in\cP_k(\BC)$, by Corollary~\ref{cor_scalar_simple}(3) we have
\begin{align*}
\big\langle x_2^k,{\rm e}^{2q(x,\overline{z})}\big\rangle_{\lambda,x}
&=\frac{\big(\lambda+\big\lceil\frac{k}{2}\big\rceil-\frac{n-1}{2}\big)_{\lfloor k/2\rfloor}}
{(\lambda)_k\big(\lambda-\frac{n-2}{2}\big)_{\lfloor k/2\rfloor}}
{}_2F_1\left(\begin{matrix} -\frac{k}{2},-\frac{k-1}{2}\\-\lambda-k+\frac{n+1}{2}\end{matrix};
-\frac{q(z_1)}{q(z_2)}\right)z_2^k
\\
&=\frac{\big(\lambda+\big\lceil\frac{k}{2}\big\rceil-\frac{n'}{2}\big)_{\lfloor k/2\rfloor}}
{(\lambda)_k\big(\lambda-\frac{n-2}{2}\big)_{\lfloor k/2\rfloor}}
{}_2F_1\left(\begin{matrix} -\left\lceil\frac{k}{2}\right\rceil+\frac{1}{2},-\left\lfloor\frac{k}{2}\right\rfloor\\
-\lambda-k+\frac{n'}{2}+1\end{matrix};-\frac{q(z_1)}{q(z_2)}\right)z_2^k.
\end{align*}
Then by putting $\cP_k(\BC)=:\cP_{(\lceil k/2\rceil,\lfloor k/2\rfloor)}\big(\BC^1\big)$, this result is included in (\ref{rank2}).
Similarly, when $n''=2$, by Theorem \ref{explicit_nonsimple}(1), for $(k_1,k_2)\in(\BZ_{\ge 0})^2$, $F(x_2)\in\cP_{(k_1,k_2)}\big(\BC^2\big)$ we have
\[
\big\langle F(x_2),{\rm e}^{2q(x,\overline{z})}\big\rangle_{\lambda,x}
=\frac{\big(\lambda+k_1-\frac{n-2}{2}\big)_{k_2}} {(\lambda)_{k_1+k_2}\big(\lambda-\frac{n-2}{2}\big)_{k_2}}
{}_2F_1\left(\begin{matrix} -k_1,-k_2\\-\lambda-k_1-k_2+\frac{n'}{2}+1\end{matrix};-\frac{q(z_1)}{q(z_2)}\right)F(z_2),
\]
and this is also included in (\ref{rank2}).

Next, for {\it Case} 2, $\fp^+=\Sym(r,\BC)=\Sym(r',\BC)\oplus M(r',r'';\BC)\oplus \Sym(r'',\BC)$,
we write $z=\left(\begin{smallmatrix}z_{11}&z_{12}\\{}^t\hspace{-1pt}z_{12}&z_{22}\end{smallmatrix}\right)$.
When $r'=r''$, we regard $\fp^+=\Sym(2r',\BC)$ as a Jordan algebra by using the maximal tripotent
$e=\left(\begin{smallmatrix}0&I_{r'}\\I_{r'}&0\end{smallmatrix}\right)$. Then since
$\tilde{\Phi}_\bm^{\fp^+_1,\fp^-_2}\big(z_1,{}^t\hspace{-1pt}z_2^\itinv\big) =\tilde{\Phi}_\bm^{(1)}\big(z_{11}{}^t\hspace{-1pt}z_{12}^{-1}z_{22}z_{12}^{-1}\big)$
holds (see~(\ref{Phi_mid_Case23})), by Corollary~\ref{cor_scalar_simple}(3) we have
\begin{gather*}
\big\langle \det(x_{12})^k,{\rm e}^{\tr(xz^*)}\big\rangle_{\lambda,x}
\\ \qquad
{}=\frac{\big(\lambda+\big\lceil\frac{k}{2}\big\rceil-\frac{r'+1}{2}\big)_{\underline{\lfloor k/2\rfloor}_{r'},1}}
{(\lambda)_{\underline{k}_{r'},1}\big(\lambda-\frac{r'}{2}\big)_{\underline{\lfloor k/2\rfloor}_{r'},1}}
\det(z_{12})^k {}_2F_1^{(1)}\left( \begin{matrix}-\frac{k}{2},-\frac{k-1}{2}\\ -\lambda-k+r'+1 \end{matrix};
z_{11}{}^t\hspace{-1pt}z_{12}^{-1}z_{22}z_{12}^{-1}\right)\!,
\end{gather*}
where ${}_2F_1^{(1)}$ is as in (\ref{2F1^(d)}). This result is generalized to $r'\ne r''$ case. We~ assume $r'\le r''$,
and we write $z_{12}=(z_{12}',z_{12}'')\in M(r',\BC)\oplus M(r',r''-r';\BC)$,
$z_{22}=\left(\begin{smallmatrix} z_{22}'&z_{22}''\\{}^t\hspace{-1pt}z_{22}''&z_{22}''' \end{smallmatrix}\right)
\in\Sym(r',\BC)\oplus M(r',r''-r';\BC)\oplus\Sym(r''-r',\BC)$,
and $z'=\left(\begin{smallmatrix}z_{11}&z_{12}'\\{}^t\hspace{-1pt}z_{12}'&z_{22}'\end{smallmatrix}\right)$.
Then for $w_{12}=(I_{r'},0)\in M(r',r'';\BC)$, by Corollary~\ref{inner_sub} we have
\begin{gather*}
\big\langle \det(x_{12}{}^t\hspace{-1pt}w_{12})^k,{\rm e}^{\tr(xz^*)}\big\rangle_{\lambda,x,\Sym(r,\BC)}
=\big\langle \det(x_{12}')^k,{\rm e}^{\tr(x'(z')^*)}\big\rangle_{\lambda,x',\Sym(2r',\BC)}
\\ \qquad
{}=\frac{\big(\lambda+\big\lceil\frac{k}{2}\big\rceil-\frac{r'+1}{2}\big)_{\underline{\lfloor k/2\rfloor}_{r',1}}}
{(\lambda)_{\underline{k}_{r'},1}\big(\lambda-\frac{r'}{2}\big)_{\underline{\lfloor k/2\rfloor}_{r'},1}}
\det(z_{12}')^k{}_2F_1^{(1)}\left( \begin{matrix}-\frac{k}{2},-\frac{k-1}{2}\\ -\lambda-k+r'+1 \end{matrix};
z_{11}{}^t\hspace{-1pt}(z_{12}')^{-1}z_{22}'(z_{12}')^{-1}\right)
\\ \qquad
{}=\frac{\big(\lambda+\big\lceil\frac{k}{2}\big\rceil-\frac{r'+1}{2}\big)_{\underline{\lfloor k/2\rfloor}_{r'},1}}
{(\lambda)_{\underline{k}_{r'},1} \big(\lambda-\frac{r'}{2}\big)_{\underline{\lfloor k/2\rfloor}_{r'},1}}
\det\big(z_{12}{}^t\hspace{-1pt}w_{12}\big)^k
\\ \qquad\hphantom{=}
{} \ \times{}_2F_1^{(1)}\left(\begin{matrix}-\frac{k}{2},-\frac{k-1}{2}\\ -\lambda-k+r'+1 \end{matrix};
z_{11}\big(w_{12}{}^t\hspace{-1pt}z_{12}\big)^{-1}w_{12}z_{22} {}^t\hspace{-1pt}w_{12}\big(z_{12}{}^t\hspace{-1pt}w_{12}\big)^{-1}\right)\!.
\end{gather*}
Then by the ${\rm GL}(r',\BC)\times {\rm GL}(r'',\BC)$-equivariance, this holds for all $w_{12}\in M(r',r'';\BC)$.
Similarly, the result for {\it Case} 3 is generalized to $s'\ne s''$ case.

Next, for {\it Case} 6, for $z_1\in\fp^+_1=M(2,6;\BC)$, $z_2\in\fp^+_2=\Skew(6,\BC)$, by (\ref{Phi_mid_except}) we have
\[
\tilde{\Phi}_m^{\fp^+_1,\fp^-_2}\big(z_1,{}^t\hspace{-1pt}z_2^\itinv\big)
=\frac{1}{m!}\big({-}\Pf\big(z_1\big({}^t\hspace{-1pt}z_2^{-1}\big)^\#{}^t\hspace{-1pt}z_1\big)\big)^m
=\frac{1}{m!}\bigg(\frac{\Pf\big(z_1z_2{}^t\hspace{-1pt}z_1\big)}{\Pf(z_2)}\bigg)^m\!,
\]
and hence
\begin{gather*}
\big\langle \Pf(x_2)^k,{\rm e}^{(x|\overline{z})_{\fp^+}}\big\rangle_{\lambda,x}
=\frac{(\lambda+k-6)_k}{(\lambda)_{(2k,k,k),8}}\Pf(z_2)^k
{}_2F_1\left(\begin{matrix} -k,-k-2 \\ -\lambda-2k+7 \end{matrix};\frac{\Pf\big(z_1z_2{}^t\hspace{-1pt}z_1\big)}{\Pf(z_2)}\right)
\end{gather*}
holds. We~summarize the results for {\it Cases} 1--6. Here ${}_2F_1^{(d)}$ and ${}_2F_1^{[d]}$ are as in (\ref{2F1^(d)}) and (\ref{2F1^[d]}).

\begin{Theorem}\label{explicit_simple}\quad
\begin{enumerate}\itemsep=0pt
\item[$1.$] Let $\big(\fp^+,\fp^+_1,\fp^+_2\big)=(\BC^n,\BC^{n'},\BC^{n''})$ with $n=n'+n''$. Then for
\[
\begin{cases} (k_1,k_2)\in\BZ_{++}^2, &n''\ge 3,
\\
(k_1,k_2)\in(\BZ_{\ge 0})^2, & n''=2,
\\
(k_1,k_2)=\big(\big\lceil\frac{k}{2}\big\rceil,\big\lfloor\frac{k}{2}\big\rfloor\big), \quad k\in\BZ_{\ge 0}, &n''=1, \end{cases}
\]
for $F(x_2)\in\cP_{(k_1,k_2)}(\fp^+_2)$ and for $z=z_1+z_2\in\fp^+$, we have
\begin{gather*}
\big\langle F(x_2),{\rm e}^{2q(x,\overline{z})}\big\rangle_{\lambda,x} =\frac{\big(\lambda+k_1-\frac{n'}{2}\big)_{k_2}} {(\lambda)_{k_1+k_2}\big(\lambda-\frac{n-2}{2}\big)_{k_2}}
{}_2F_1\left(\begin{matrix} -k_2,-k_1-\frac{n''}{2}+1
\\-\lambda\!-\!k_1\!-\!k_2\!+\!\frac{n'}{2}\!+\!1\end{matrix};
-\frac{q(z_1)}{q(z_2)}\right)F(z_2).
\end{gather*}
\item[$2.$] Let $\big(\fp^+,\fp^+_{11},\fp^\pm_{12},\fp^+_{22}\big)=({\rm Sym}(r,\BC),\Sym(r',\BC),M(r',r'';\BC),\Sym(r'',\BC))$ with $r=r'+r''$, $r'\le r''$.
Then for $k\in\BZ_{\ge 0}$ and for $z=\left(\begin{smallmatrix}z_{11}&z_{12}\\{}^t\hspace{-1pt}z_{12}&z_{22}\end{smallmatrix}\right)\in\fp^+$,
$w_{12}\in\fp^-_{12}$, we have
\begin{gather*}
\big\langle \det\big(x_{12}{}^t\hspace{-1pt}w_{12}\big)^k,{\rm e}^{\tr(xz^*)}\big\rangle_{\lambda,x}
=\frac{\big(\lambda+\big\lceil\frac{k}{2}\big\rceil-\frac{r'+1}{2}\big)_{\underline{\lfloor k/2\rfloor}_{r'},1}}
{(\lambda)_{\underline{k}_{r'},1}\big(\lambda-\frac{r'}{2}\big)_{\underline{\lfloor k/2\rfloor}_{r'},1}}
\det\big(z_{12}{}^t\hspace{-1pt}w_{12}\big)^k
\\ \qquad
{}\times{}_2F_1^{(1)}\left(\begin{matrix}-\frac{k}{2},-\frac{k-1}{2}\\ -\lambda-k+r'+1 \end{matrix};
z_{11}\big(w_{12}{}^t\hspace{-1pt}z_{12}\big)^{-1}w_{12}z_{22}{}^t\hspace{-1pt}w_{12}\big(z_{12}{}^t\hspace{-1pt}w_{12}\big)^{-1}\right)\!.
\end{gather*}
\item[$3.$] Let $\big(\fp^+,\fp^+_{11},\fp^\pm_{12},\fp^+_{22}\big)=({\rm Skew}(s,\BC),\Skew(s',\BC),M(s',s'';\BC),\Skew(s'',\BC))$ with $s=s'+s''$, $s'\le s''$.
Then for $k\in\BZ_{\ge 0}$ and for $z=\left(\begin{smallmatrix}z_{11}&z_{12}\\-{}^t\hspace{-1pt}z_{12}&z_{22}\end{smallmatrix}\right)\in\fp^+$, $w_{12}\in\fp^-_{12}$,
we have
\begin{gather*}
\big\langle \det\big(x_{12}{}^t\hspace{-1pt}w_{12}\big)^k,{\rm e}^{\frac{1}{2}\tr(xz^*)}\big\rangle_{\lambda,x}
=\frac{\big(\lambda+k-2\big\lceil\frac{s'}{2}\big\rceil+1\big)_{\underline{k}_{\lfloor s'/2\rfloor},4}}
{(\lambda)_{\underline{2k}_{\lfloor s'/2\rfloor},4}
\big(\lambda-2\big\lfloor \frac{s'}{2}\big\rfloor\big)_{\underline{k}_{\lceil s'/2\rceil},4}}\det\big(z_{12}{}^t\hspace{-1pt}w_{12}\big)^k
\\ \qquad
{}\times{}_2F_1^{[4]}\left( \begin{matrix} -k,-k-1\\ -\lambda-2k+2s'-2 \end{matrix};
-z_{11}\big(w_{12}{}^t\hspace{-1pt}z_{12}\big)^{-1}w_{12}z_{22}{}^t \hspace{-1pt}w_{12}\big(z_{12}{}^t\hspace{-1pt}w_{12}\big)^{-1}\right)\!.
\end{gather*}
\item[$4.$] Let $\big(\fp^+,\fp^+_1,\fp^+_2\big)=(M(r,\BC),\Skew(r,\BC),\Sym(r,\BC)\big)$. Then for $k\in\BZ_{\ge 0}$ and for $z=z_1+z_2\in\fp^+$, we have
\begin{align*}
\big\langle \det(x_2)^k,{\rm e}^{\tr(xz^*)}\big\rangle_{\lambda,x}
={}&\frac{\big(\lambda+k-\big\lceil\frac{r}{2}\big\rceil+\frac{1}{2}\big)_{\underline{k}_{\lfloor r/2\rfloor},2}}
{(\lambda)_{\underline{2k}_{\lfloor r/2\rfloor},2}
\big(\lambda-\big\lfloor\frac{r}{2}\big\rfloor\big)_{\underline{k}_{\lceil r/2\rceil},2}}
\\
&\times\det(z_2)^k
{}_2F_1^{[2]}\left( \begin{matrix} -k,-k-\frac{1}{2} \\ -\lambda-2k+r-\frac{1}{2} \end{matrix};z_1z_2^{-1}z_1z_2^{-1}\right)\!.
\end{align*}
\item[$5.$] Let $\big(\fp^+,\fp^+_1,\fp^+_2\big)=(M(2s,\BC),\Sym(2s,\BC),\Skew(2s,\BC))$. Then for $k\in\BZ_{\ge 0}$ and for $z=z_1+z_2\in\fp^+$, we have
\begin{align*}
\big\langle \Pf(x_2)^k,{\rm e}^{\tr(xz^*)}\big\rangle_{\lambda,x}
={}&\frac{\big(\lambda+\big\lceil\frac{k}{2}\big\rceil-s-\frac{1}{2}\big)_{\underline{\lfloor k/2\rfloor}_s,2}}
{(\lambda)_{\underline{k}_{s},2}(\lambda-s)_{\underline{\lfloor k/2\rfloor}_{s},2}}
\\
&\times\Pf(z_2)^k{}_2F_1^{[2]}\left(
\begin{matrix}-\frac{k}{2},-\frac{k-1}{2}\\ -\lambda-k+2s+\frac{1}{2} \end{matrix};z_1z_2^{-1}z_1z_2^{-1}\right)\!.
\end{align*}
\item[$6.$] Let $\big(\fp^+,\fp^+_1,\fp^+_2\big)=\big({\rm Herm}(3,\BO)^\BC,M(2,6;\BC),\Skew(6,\BC)\big)$. Then for $k\in\BZ_{\ge 0}$ and for $z=z_1+z_2\in\fp^+$, we have
\begin{align*}
\big\langle \Pf(x_2)^k,{\rm e}^{(x|\overline{z})_{\fp^+}}\big\rangle_{\lambda,x}
=\frac{(\lambda+k-6)_k}{(\lambda)_{2k}(\lambda\!-\!4)_k(\lambda\!-\!8)_k}\!\Pf(z_2)^k
{}_2F_1\!\left( \begin{matrix} -k,-k-2 \\ -\lambda\!-\!2k\!+\!7 \end{matrix};\frac{\Pf\big(z_1z_2{}^t\hspace{-1pt}z_1\big)}{\Pf(z_2)}\right)\!.
\end{align*}
\end{enumerate}
\end{Theorem}

\section{Poles of inner products and submodules}\label{section_submod}

In this section we apply the results on the poles of $\big\langle F(x_2),{\rm e}^{(x|\overline{z})_{\fp^+}}\big\rangle_\lambda$ to determining
which $\big(\fg,\widetilde{K}\big)$-submodule of $({\rm d}\tau_\lambda,\cO_\lambda(D)_{\widetilde{K}})$ contains
the $\big(\fg_1,\widetilde{K}_1\big)$-submodule generated by $F(x_2)$.
For $\lambda>p-1$, $f=\sum_\bm f_\bm$, $g=\sum_\bm g_\bm\in\cH_\lambda(D)_{\widetilde{K}}=\cP(\fp^+)=\bigoplus_{\bm\in\BZ_{++}^r}\cP_\bm(\fp^+)$,
by \cite{FK0}, \cite[Part~III, Corollary V.3.9]{FKKLR} the inner product is given by
\begin{gather*}
\langle f,g\rangle_\lambda=\sum_{\bm\in\BZ_{++}^r}\frac{1}{(\lambda)_{\bm,d}}\langle f_\bm,g_\bm\rangle_F
=\sum_{\bm\in\BZ_{++}^r}\frac{1}{\prod_{j=1}^r\left(\lambda-\frac{d}{2}(j-1)\right)_{m_j}}\langle f_\bm,g_\bm\rangle_F.
\end{gather*}
This is meromorphically continued for all $\lambda\in\BC$, and is positive definite on $\cP(\fp^+)$ for $\lambda>\frac{d}{2}(r-1)$.
If $\lambda$ is a pole, then $\cP(\fp^+)$ becomes reducible
as a $\big(\fg,\widetilde{K}\big)$-module, and for $i=1,2,\dots,r$, $\lambda\in\frac{d}{2}(i-1)-\BZ_{\ge 0}$,
\[ M_i(\lambda)=M_i^\fg(\lambda):=\bigoplus_{\substack{\bm\in\BZ_{++}^{r}\\ m_i\le\frac{d}{2}(i-1)-\lambda}}\cP_\bm(\fp^+)\subset
\cP(\fp^+)=\cO_\lambda(D)_{\widetilde{K}} \]
is a $\big(\fg,\widetilde{K}\big)$-submodule of $\cO_\lambda(D)_{\widetilde{K}}$.
Especially, the submodule $M_i\big(\frac{d}{2}(i-1)\big)$ $(1\le i\le r)$ and the quotient $\cO_\lambda(D)_{\widetilde{K}}/M_r(\lambda)$
are infinitesimally unitary (see \cite{FK0}).
For $f(x)\in\cP(\fp^+)$, $\bm\in\BZ_{++}^r$, if
$(\lambda)_{\bm,d}\langle f,g\rangle_\lambda$ is holomorphic for all $\lambda\in\BC$ for any $g\in\cP(\fp^+)$,
or equivalently by (\ref{inner_Fubini}), if
\begin{equation}\label{poles}
(\lambda)_{\bm,d}\big\langle f(x),{\rm e}^{(x|\overline{z})_{\fp^+}}\big\rangle_{\lambda,x}
\end{equation}
is holomorphic for all $\lambda\in\BC$, then for $1\le i\le r$,
\[
f(x)\in M_i(\lambda) \qquad \text{holds if} \quad \lambda\in \frac{d}{2}(i-1)-m_i-\BZ_{\ge 0},
\]
and moreover, for any Lie subalgebra $\fg_1\subset\fg$,
\[
{\rm d}\tau_\lambda(\cU(\fg_1))f(x)\subset M_i(\lambda) \qquad \text{holds if} \quad \lambda\in \frac{d}{2}(i-1)-m_i-\BZ_{\ge 0}.
\]
In addition, if (\ref{poles}) gives a non-zero polynomial in $z\in\fp^+$ for all $\lambda\in\BC$, then the converse also holds.

For example, let $\fp^+=\BC^n$ with $n\ge 3$. Then we have $d=n-2$, and as an $\big(\mathfrak{so}(2,n),\widetilde{\rm SO}(2)\times {\rm SO}(n)\big)$-module,
$\cO_\lambda(D)_{\widetilde{K}}$ is reducible if and only if $\lambda\in -\BZ_{\ge 0}\cup \big(\frac{n-2}{2}-\BZ_{\ge 0}\big)$.
For these $\lambda$ we have the submodules
\begin{alignat*}{3}
& \cO_\lambda(D)_{\widetilde{K}}\supset M_2(\lambda)\supset M_1(\lambda)\supset\{0\}, \qquad&& n\colon\text{even},\quad\lambda\in\BZ,\quad\lambda\le 0,&
\\
& \cO_\lambda(D)_{\widetilde{K}}\supset M_2(\lambda)\supset\{0\},\qquad &&
n\colon\text{even},\quad\lambda\in\BZ,\quad1\le\lambda\le{\textstyle \frac{n-2}{2}},&
\\
& \cO_\lambda(D)_{\widetilde{K}}\supset M_1(\lambda)\supset\{0\}, \qquad && n\colon\text{odd},\quad\lambda\in\BZ,\quad\lambda\le 0,&
\\
& \cO_\lambda(D)_{\widetilde{K}}\supset M_2(\lambda)\supset\{0\},\qquad&&
n\colon\text{odd},\quad\lambda\in\BZ+{\textstyle \frac{1}{2}},\quad\lambda\le{\textstyle \frac{n-2}{2}}.&
\end{alignat*}
We~consider $\BC^n=\BC^{n'}\oplus\BC^{n''}$. If $n''\ge 3$, then for $(k_1,k_2)\in\BZ_{++}^2$, $F(x_2)\in\cP_{(k_1,k_2)}\big(\BC^{n''}\big)$,
$F(x_2)\ne 0$, by Corollary~\ref{cor_pole_simple}(1) or Theorem \ref{explicit_simple}(1),
\[
(\lambda)_{(k_1+k_2,k_2),n-2}\big\langle F(x_2),{\rm e}^{2q(x,\overline{z})}\big\rangle_{\lambda,x} \]
is non-zero holomorphic for all $\lambda\in\BC$. If $n''=2$, then for $(k_1,k_2)\in(\BZ_{\ge 0})^2$, $F(x_2)\in\cP_{(k_1,k_2)}\big(\BC^2\big)$,
$F(x_2)\ne 0$, by Corollary~\ref{cor_pole_nonsimple} or Theorem \ref{explicit_nonsimple}(1),
\[
(\lambda)_{(k_1+k_2,\min\{k_1,k_2\}),n-2}\big\langle F(x_2),{\rm e}^{2q(x,\overline{z})}\big\rangle_{\lambda,x}
\]
is non-zero holomorphic for all $\lambda\in\BC$. If $n''=1$, then for $k\in\BZ_{\ge 0}$, $F(x_2)\in\cP_k(\BC)$,
$F(x_2)\ne 0$, by Corollary~\ref{cor_pole_simple}(2) or Theorem \ref{explicit_simple}(1),
\[
(\lambda)_{(k,\lfloor k/2\rfloor),n-2}\big\langle F(x_2),{\rm e}^{2q(x,\overline{z})}\big\rangle_{\lambda,x}
\]
is non-zero holomorphic for all $\lambda\in\BC$. Similarly, we consider
\[
\BC^{2s} = \big(\big(1,\sqrt{-1}\big)\BC\big)^s\oplus\big(\big(1,-\sqrt{-1}\big)\BC\big)^s
\simeq M(1,s;\BC)\oplus M(1,s;\BC).
\]
Then for $k\in\BZ_{\ge 0}$, $F(x_2)\in\cP_k(M(1,s;\BC))$, $F(x_2)\ne 0$, by~(\ref{easycase}),
\[
(\lambda)_k\big\langle F(x_2),{\rm e}^{2q(x,\overline{z})}\big\rangle_{\lambda,x}
\]
is non-zero holomorphic for all $\lambda\in\BC$. Therefore the following holds.

\begin{Theorem}\qquad\samepage
\begin{enumerate}\itemsep=0pt
\item[$1.$] Let $n''\ge 3$, $(k_1,k_2)\in\BZ_{++}^2$. Then for $i=1,2$,
\[
{\rm d}\tau_\lambda(\cU(\mathfrak{so}(2,n')\oplus\mathfrak{so}(n'')))\cP_{(k_1,k_2)}\big(\BC^{n''}\big)\subset M_i^{\mathfrak{so}(2,n)}(\lambda)
\]
holds if and only if
\[
\lambda\in \begin{cases} -(k_1+k_2)-\BZ_{\ge 0}, & i=1, \\ \frac{n-2}{2}-k_2-\BZ_{\ge 0}, & i=2. \end{cases}
\]
\item[$2.$] Let $n''=2$, $(k_1,k_2)\in(\BZ_{\ge 0})^2$. Then for $i=1,2$,
\[
{\rm d}\tau_\lambda(\cU(\mathfrak{so}(2,n')\oplus\mathfrak{so}(2)))\cP_{(k_1,k_2)}\big(\BC^2\big)\subset M_i^{\mathfrak{so}(2,n)}(\lambda)
\]
holds if and only if
\[
\lambda\in \begin{cases}
-(k_1+k_2)-\BZ_{\ge 0}, &i=1, \\ \frac{n-2}{2}-\min\{k_1,k_2\}-\BZ_{\ge 0}, & i=2.
\end{cases}
\]
\item[$3.$] Let $n''=1$, $k\in\BZ_{\ge 0}$. Then for $i=1,2$,
\[
{\rm d}\tau_\lambda(\cU(\mathfrak{so}(2,n')))\cP_k(\BC)\subset M_i^{\mathfrak{so}(2,n)}(\lambda)
\]
holds if and only if
\[
\lambda\in \begin{cases} -k-\BZ_{\ge 0}, & i=1, \\ \frac{n-2}{2}-\left\lfloor\frac{k}{2}\right\rfloor-\BZ_{\ge 0}, & i=2.
\end{cases}
\]
\item[$4.$] Suppose $n=2s$ and let $k\in\BZ_{\ge 0}$. Then for $i=1,2$,
\[
{\rm d}\tau_\lambda(\cU(\mathfrak{u}(1,s)))\cP_k(M(1,s;\BC))\subset M_i^{\mathfrak{so}(2,n)}(\lambda)
\]
holds if and only if
\[
\lambda\in \begin{cases} -k-\BZ_{\ge 0}, & i=1, \\ \frac{n-2}{2}-\BZ_{\ge 0}, & i=2. \end{cases}
\]
\end{enumerate}
\end{Theorem}

Next we consider $\fp^+=\Sym(r,\BC)$. Then we have $d=1$, and for $r\ge 2$, as an $\big(\mathfrak{sp}(r,\BR), \widetilde{{\rm U}}(r)\big)$-module,
$\cO_\lambda(D)_{\widetilde{K}}$ is reducible if and only if $\lambda\in\frac{1}{2}\BZ$, $\lambda\le \frac{1}{2}(r-1)$.
For these $\lambda$ we have the submodules
\begin{alignat*}{3}
& \cO_\lambda(D)_{\widetilde{K}}\supset M_{2\left\lceil\frac{r}{2}\right\rceil-1}(\lambda)
\supset M_{2\left\lceil\frac{r}{2}\right\rceil-3}(\lambda)
\supset\cdots\supset M_{\max\{2\lambda,0\}+1}(\lambda)\supset\{0\},\qquad&& \lambda\in\BZ,&
\\
&\cO_\lambda(D)_{\widetilde{K}}\supset M_{2\left\lfloor\frac{r}{2}\right\rfloor}(\lambda)
\supset M_{2\left\lfloor\frac{r}{2}\right\rfloor-2}(\lambda)
\supset\cdots\supset M_{\max\{2\lambda,1\}+1}(\lambda)\supset\{0\},\qquad && \lambda\in\BZ+{\textstyle \frac{1}{2}}.&
\end{alignat*}
We~consider
$\Sym(r,\BC)=\Sym(r',\BC)\oplus M(r',r'';\BC)\oplus\Sym(r'',\BC)$.
Then for $1\le a\le r'$, $k\in\BZ_{\ge 0}$, $\bl\in\BZ_{++}^{r''}$ and for
$F(x_{11},x_{22})\in\cP_{\underline{k}_a}({\rm Sym}(r',\BC))\boxtimes\cP_{\bl}({\rm Sym}(r'',\BC))$, $F(x_{11},x_{22})\ne 0$,
by Corollary~\ref{cor_pole_nonsimple},
\[
(\lambda)_{(\underline{k}_{a}+\bl',\min\{\underline{k}_{r''}+\bl'',\bl\}),1}
\big\langle F(x_{11},x_{22}),{\rm e}^{\tr(xz^*)}\big\rangle_{\lambda,x}
\]
is holomorphic for all $\lambda\in\BC$, where $\bl'\in\BZ_{++}^a$, $\bl''\in\BZ_{++}^{r''}$ are defined as in (\ref{lprime}),
and this is non-zero if $l_{a+1}=0$ or $k=0$ or $\bl=\underline{l}_{a'}$ ($l\in\BZ_{\ge 0}$, $1\le a'\le r''$).
Similarly, if $r'\le r''$, then for $0\le l< a\le r'$, $k\in\BZ_{\ge 0}$ and for
$F(x_{12})\in\cP_{(\underline{k+1}_l,\underline{k}_{a-l})}(M(r',r'';\BC))$, $F(x_{12})\ne 0$, by Corollary~\ref{cor_pole_simple}(2),
\[
(\lambda)_{(\underline{k+1}_l,\underline{k}_{a-l},\underline{\lceil k/2\rceil}_l,\underline{\lfloor k/2\rfloor}_{a-l}),1}
\big\langle F(x_{12}),{\rm e}^{\tr(xz^*)}\big\rangle_{\lambda,x}
\]
is non-zero holomorphic for all $\lambda\in\BC$. Therefore the following holds.

\begin{Theorem}\quad
\begin{enumerate}\itemsep=0pt
\item[$1.$] Let $1\le a\le r'$, $k\in\BZ_{\ge 0}$, $\bl\in\BZ_{++}^{r''}$. Then for $1\le i\le r$,
\[
{\rm d}\tau_\lambda(\cU(\mathfrak{u}(r',r'')))\cP_{\underline{k}_a}({\rm Sym}(r',\BC))\boxtimes\cP_{\bl}({\rm Sym}(r'',\BC))
\subset M_i^{\mathfrak{sp}(r,\BR)}(\lambda)
\]
holds if $a\ge r''$ and
\[
\lambda\in \begin{cases}
\frac{1}{2}(i-1)-(k+l_i)-\BZ_{\ge 0}, & 1\le i\le r'',
\\
\frac{1}{2}(i-1)-k-\BZ_{\ge 0}, & r''+1\le i\le a,
\\
\frac{1}{2}(i-1)-\min\{k,l_{i-a}\}-\BZ_{\ge 0}, & a+1\le i\le a+r'',
\\
\frac{1}{2}(i-1)-\BZ_{\ge 0}, & a+r''+1\le i\le r, \end{cases}
\]
or if $a<r''$ and
\[
\lambda\in \begin{cases}
\frac{1}{2}(i-1)-(k+l_i)-\BZ_{\ge 0}, &1\le i\le a,
\\
\frac{1}{2}(i-1)-\min\{k+l_i,l_{i-a}\}-\BZ_{\ge 0}, & a+1\le i\le r'',
\\
\frac{1}{2}(i-1)-\min\{k,l_{i-a}\}-\BZ_{\ge 0}, & r''+1\le i\le a+r'',
\\
\frac{1}{2}(i-1)-\BZ_{\ge 0}, &a+r''+1\le i\le r. \end{cases}
\]
If $l_{a+1}=0$ or $k=0$ or $\bl=\underline{l}_{a'}$ $(l\in\BZ_{\ge 0}$, $1\le a'\le r'')$, then the converse also holds.
\item[$2.$] Suppose $r'\le r''$ and let $0\le l< a\le r'$, $k\in\BZ_{\ge 0}$. Then for $1\le i\le r$,
\[
{\rm d}\tau_\lambda(\cU(\mathfrak{sp}(r',\BR)\oplus\mathfrak{sp}(r'',\BR))) \cP_{(\underline{k+1}_l,\underline{k}_{a-l})}(M(r',r'';\BC))
\subset M_i^{\mathfrak{sp}(r,\BR)}(\lambda)
\]
holds if and only if
\[
\lambda\in \begin{cases}
\frac{1}{2}(i-1)-(k+1)-\BZ_{\ge 0}, &1\le i\le l,
\\
\frac{1}{2}(i-1)-k-\BZ_{\ge 0}, &l+1\le i\le a,
\\
\frac{1}{2}(i-1)-\left\lceil \frac{k}{2}\right\rceil -\BZ_{\ge 0}, &a+1\le i\le a+l, \\
\frac{1}{2}(i-1)-\left\lfloor \frac{k}{2}\right\rfloor -\BZ_{\ge 0}, &a+l+1\le i\le 2a, \\
\frac{1}{2}(i-1)-\BZ_{\ge 0}, &2a+1\le i\le r. \end{cases}
\]
\end{enumerate}
\end{Theorem}

Next we consider $\fp^+=M(q,s;\BC)$. Then we have $d=2$, and as an $(\mathfrak{su}(q,s),{\rm S}({\rm U}(q)\times {\rm U}(s))^\sim)$-module,
$\cO_\lambda(D)_{\widetilde{K}}$ is reducible if and only if $\lambda\in\BZ$, $\lambda\le \min\{q,s\}-1$.
For these $\lambda$ we have the submodules
\[ \cO_\lambda(D)_{\widetilde{K}}\supset M_{\min\{q,s\}}(\lambda)\supset M_{\min\{q,s\}-1}(\lambda)
\supset\cdots\supset M_{\max\{\lambda,0\}+1}(\lambda)\supset\{0\}. \]
We~consider $M(q,s;\BC)=M(q',s';\BC)\oplus M(q',s'';\BC)\oplus M(q'',s';\BC)\oplus M(q'',s'';\BC)$ with $q''s''\ne 0$.
If $q's'\ne 0$, then for $1\le a\le \min\{q',s'\}$, $k\in\BZ_{\ge 0}$, $\bl\in\BZ_{++}^{\min\{q'',s''\}}$ and for
$F(x_{11},x_{22})\in\cP_{\underline{k}_a}(M(q',s';\BC))\boxtimes\cP_{\bl}(M(q'',s'';\BC))$, $F(x_{11},x_{22})\ne 0$,
by Corollary~\ref{cor_pole_nonsimple},
\[
(\lambda)_{(\underline{k}_{a}+\bl',\min\{\underline{k}_{\min\{q'',s''\}}+\bl'',\bl\}),2}
\big\langle F(x_{11},x_{22}),{\rm e}^{\tr(xz^*)}\big\rangle_{\lambda,x}
\]
is holomorphic for all $\lambda\in\BC$, where $\bl'\in\BZ_{++}^a$, $\bl''\in\BZ_{++}^{\min\{q'',s''\}}$ are defined as in (\ref{lprime}),
and this is non-zero if $l_{a+1}=0$ or $k=0$ or $\bl=\underline{l}_{a'}$ ($l\in\BZ_{\ge 0}$, $1\le a'\le \min\{q'',s''\}$).
If $q's'=0$, then the poles are determined by (\ref{easycase}), and this case is contained in the above formula with $a=0$.
Similarly, when $q=s=r$, we consider $M(r,\BC)=\Skew(r,\BC)\oplus\Sym(r,\BC)$.
Then for $k,l\in\BZ_{\ge 0}$, $2\le a\le r$ and for $F(x_2)\in\cP_{(k+l,\underline{k}_{a-1})}({\rm Sym}(r,\BC))$, $F(x_2)\ne 0$,
by Corollary~\ref{cor_pole_simple}(1),
\begin{alignat*}{3}
& (\lambda)_{(2k+l,\underline{2k}_{a/2-1},\underline{k}_{a/2}),2}
\big\langle F(x_2),{\rm e}^{\tr(xz^*)}\big\rangle_{\lambda,x} \qquad && \text{if}\quad a\colon\text{even},&
\\
& (\lambda)_{(2k+l,\underline{2k}_{\lfloor a/2\rfloor-1},\min\{2k,k+l\},\underline{k}_{\lfloor a/2\rfloor}),2}
\big\langle F(x_2),{\rm e}^{\tr(xz^*)}\big\rangle_{\lambda,x} \qquad && \text{if}\quad a\colon\text{odd}&
\end{alignat*}
are holomorphic for all $\lambda\in\BC$, and these are non-zero if $a$ is even or $k=0$ or $l=0$.
Similarly, for $0\le l< a\le \lfloor r/2\rfloor$, $k\in\BZ_{\ge 0}$ and for
$F(x_2)\in\cP_{(\underline{k+1}_l,\underline{k}_{a-l})}({\rm Skew}(r,\BC))$, $F(x_2)\ne 0$, by Corollary~\ref{cor_pole_simple}(2),
\[
(\lambda)_{(\underline{k+1}_l,\underline{k}_{a-l},\underline{\lceil k/2\rceil}_l,\underline{\lfloor k/2\rfloor}_{a-l}),2}
\big\langle F(x_2),{\rm e}^{\tr(xz^*)}\big\rangle_{\lambda,x}
\]
is non-zero holomorphic for all $\lambda\in\BC$. Therefore the following holds.

\begin{Theorem}\quad
\begin{enumerate}\itemsep=0pt
\item[$1.$] Let $0\le a\le \min\{q',s'\}$, $k\in\BZ_{\ge 0}$, $\bl\in\BZ_{++}^{\min\{q'',s''\}}$. Then for $1\le i\le \min\{q,s\}$,
\[
{\rm d}\tau_\lambda(\cU(\mathfrak{s}(\mathfrak{u}(q',s'')\oplus\mathfrak{u}(q'',s'))))
\cP_{\underline{k}_a}(M(q',s';\BC))\boxtimes\cP_{\bl}(M(q'',s'';\BC))\subset M_i^{\mathfrak{su}(q,s)}(\lambda)
\]
holds if $a\ge \min\{q'',s''\}$ and
\[
\lambda\in \begin{cases}
i-1-(k+l_i)-\BZ_{\ge 0}, & 1\le i\le \min\{q'',s''\},
\\
i-1-k-\BZ_{\ge 0}, &\min\{q'',s''\}+1\le i\le a,
\\
i-1-\min\{k,l_{i-a}\}-\BZ_{\ge 0}, &a+1\le i\le a+\min\{q'',s''\},
\\
i-1-\BZ_{\ge 0}, & a+\min\{q'',s''\}+1\le i\le \min\{q,s\}, \end{cases}
\]
or if $a<\min\{q'',s''\}$ and
\[
\lambda\in \begin{cases}
i-1-(k+l_i)-\BZ_{\ge 0}, & 1\le i\le a, \\
i-1-\min\{k+l_i,l_{i-a}\}-\BZ_{\ge 0}, & a+1\le i\le \min\{q'',s''\}, \\
i-1-\min\{k,l_{i-a}\}-\BZ_{\ge 0}, & \min\{q'',s''\}+1\le i\le a+\min\{q'',s''\}, \\
i-1-\BZ_{\ge 0}, & a+\min\{q'',s''\}+1\le i\le \min\{q,s\}. \end{cases}
\]
If $l_{a+1}=0$ or $k=0$ or $\bl=\underline{l}_{a'}$ $(l\in\BZ_{\ge 0}$, $1\le a'\le \min\{q'',s''\})$,
then the converse also holds.
\item[$2.$] Suppose $q=s=r$ and let $k,l\in\BZ_{\ge 0}$, $2\le a\le r$. Then for $1\le i\le r$,
\[
{\rm d}\tau_\lambda(\cU(\mathfrak{so}^*(2r)))\cP_{(k+l,\underline{k}_{a-1})}({\rm Sym}(r,\BC))
\subset M_i^{\mathfrak{su}(r,r)}(\lambda)
\]
holds if
\[
\lambda\in \begin{cases}
i-1-(2k+l)-\BZ_{\ge 0}, &i=1, \\
i-1-2k-\BZ_{\ge 0}, &2\le i\le \lfloor a/2\rfloor, \\
i-1-\min\{2k,k+l\}-\BZ_{\ge 0}, &a\colon\text{odd},\ i=\lceil a/2\rceil, \\
i-1-k-\BZ_{\ge 0}, &\lceil a/2\rceil+1\le i\le a, \\
i-1-\BZ_{\ge 0}, &a+1\le i\le r. \end{cases}
\]
If $a$ is even or $k=0$ or $l=0$, then the converse also holds.
\item[$3.$] Again suppose $q=s=r$ and let $0\le l< a\le \lfloor r/2\rfloor$, $k\in\BZ_{\ge 0}$. Then for $1\le i\le r$,
\[
{\rm d}\tau_\lambda(\cU(\mathfrak{sp}(r,\BR)))\cP_{(\underline{k+1}_l,\underline{k}_{a-l})}({\rm Skew}(r,\BC))
\subset M_i^{\mathfrak{su}(r,r)}(\lambda)
\]
holds if and only if
\[
\lambda\in \begin{cases}
i-1-(k+1)-\BZ_{\ge 0}, & 1\le i\le l, \\
i-1-k-\BZ_{\ge 0}, & l+1\le i\le a, \\
i-1-\left\lceil \frac{k}{2}\right\rceil -\BZ_{\ge 0}, & a+1\le i\le a+l, \\
i-1-\left\lfloor \frac{k}{2}\right\rfloor -\BZ_{\ge 0}, & a+l+1\le i\le 2a, \\
i-1-\BZ_{\ge 0}, &2a+1\le i\le r. \end{cases}
\]
\end{enumerate}
\end{Theorem}

Next we consider $\fp^+=\Skew(s,\BC)$. Then we have $d=4$, and as an $\big(\mathfrak{so}^*(2s),\widetilde{{\rm U}}(s)\big)$-module,
$\cO_\lambda(D)_{\widetilde{K}}$ is reducible if and only if $\lambda\in\BZ$, $\lambda\le 2(\lfloor s/2\rfloor-1)$.
For these $\lambda$ we have the submodules
\[
\cO_\lambda(D)_{\widetilde{K}}\supset M_{\left\lfloor\frac{s}{2}\right\rfloor}(\lambda)
\supset M_{\left\lfloor\frac{s}{2}\right\rfloor-1}(\lambda)
\supset\cdots\supset M_{\max\left\{\left\lceil\frac{\lambda}{2}\right\rceil,0\right\}+1}(\lambda)\supset\{0\}.
\]
We~consider $\Skew(s,\BC)=\Skew(s',\BC)\oplus M(s',s'';\BC)\oplus\Skew(s'',\BC)$ with $s''\ge 2$.
If $s'\ge 2$, then for $1\le a\le \lfloor s'/2\rfloor$, $k\in\BZ_{\ge 0}$, $\bl\in\BZ_{++}^{\lfloor s''/2\rfloor}$ and for
$F(x_{11},x_{22})\in\cP_{\underline{k}_a}({\rm Skew}(s',\BC))\boxtimes\cP_{\bl}({\rm Skew}(s'',\BC))$, $F(x_{11},x_{22})\ne 0$,
by Corollary~\ref{cor_pole_nonsimple},
\[
(\lambda)_{(\underline{k}_{a}+\bl',\min\{\underline{k}_{\lfloor s''/2\rfloor}+\bl'',\bl\}),4}
\big\langle F(x_{11},x_{22}),{\rm e}^{\frac{1}{2}\tr(xz^*)}\big\rangle_{\lambda,x}
\]
is holomorphic for all $\lambda\in\BC$, where $\bl'\in\BZ_{++}^a$, $\bl''\in\BZ_{++}^{\lfloor s''/2\rfloor}$ are defined as in (\ref{lprime}),
and this is non-zero if $l_{a+1}=0$ or $k=0$ or $\bl=\underline{l}_{a'}$ ($l\in\BZ_{\ge 0}$, $1\le a'\le \lfloor s''/2\rfloor$).
Similarly, if $2\le s'\le s''$, then for $k,l\in\BZ_{\ge 0}$, $2\le a\le s'$ and for
$F(x_{12})\in\cP_{(k+l,\underline{k}_{a-1})}(M(s',s'';\BC))$, $F(x_{12})\ne 0$, by Corollary~\ref{cor_pole_simple}(1),
\begin{alignat*}{3}
& (\lambda)_{(\underline{2k}_{a/2},\underline{k}_{a/2})+(l,0,\dots,0),4}
\big\langle F(x_{12}),{\rm e}^{\frac{1}{2}\tr(xz^*)}\big\rangle_{\lambda,x},\qquad&&
\text{if }a\colon\text{even},&
\\
& (\lambda)_{(\underline{2k}_{\lfloor a/2\rfloor},\min\{2k,k+l\},\underline{k}_{\lfloor a/2\rfloor})+(l,0,\dots,0),4}
\big\langle F(x_{12}),{\rm e}^{\frac{1}{2}\tr(xz^*)}\big\rangle_{\lambda,x}, \qquad&&
\text{if }a\colon\text{odd}&
\end{alignat*}
are holomorphic for all $\lambda\in\BC$, and these are non-zero if $a$ is even or $k=0$ or $l=0$.
If $s'=1$, then the poles are determined by (\ref{easycase}),
and these cases are contained in the above formulas with $a=0$ for the former case, and with $a=1$, $k=0$ for the latter case.
Therefore the following holds.
\begin{Theorem}\quad
\begin{enumerate}\itemsep=0pt
\item[$1.$] Let $0\le a\le \lfloor s'/2\rfloor$, $k\in\BZ_{\ge 0}$, $\bl\in\BZ_{++}^{\lfloor s''/2\rfloor}$. Then for $1\le i\le \lfloor s/2\rfloor$,
\[
{\rm d}\tau_\lambda(\cU(\mathfrak{u}(s',s'')))\cP_{\underline{k}_a} ({\rm Skew}(s',\BC))\boxtimes\cP_{\bl}({\rm Skew}(s'',\BC))
\subset M_i^{\mathfrak{so}^*(2s)}(\lambda)
\]
holds if $a\ge \lfloor s''/2\rfloor$ and
\[
\lambda\in \begin{cases}
2(i-1)-(k+l_i)-\BZ_{\ge 0}, & 1\le i\le \lfloor s''/2\rfloor, \\
2(i-1)-k-\BZ_{\ge 0}, & \lfloor s''/2\rfloor+1\le i\le a, \\
2(i-1)-\min\{k,l_{i-a}\}-\BZ_{\ge 0}, & a+1\le i\le a+\lfloor s''/2\rfloor, \\
2(i-1)-\BZ_{\ge 0}, & a+\lfloor s''/2\rfloor+1\le i\le \lfloor s/2\rfloor, \end{cases}
\]
or if $a<\lfloor s''/2\rfloor$ and
\[
\lambda\in \begin{cases}
2(i-1)-(k+l_i)-\BZ_{\ge 0}, & 1\le i\le a, \\
2(i-1)-\min\{k+l_i,l_{i-a}\}-\BZ_{\ge 0}, & a+1\le i\le \lfloor s''/2\rfloor, \\
2(i-1)-\min\{k,l_{i-a}\}-\BZ_{\ge 0}, & \lfloor s''/2\rfloor+1\le i\le a+\lfloor s''/2\rfloor, \\
2(i-1)-\BZ_{\ge 0}, & a+\lfloor s''/2\rfloor+1\le i\le \lfloor s/2\rfloor. \end{cases}
\]
If $l_{a+1}=0$ or $k=0$ or $\bl=\underline{l}_{a'}$ $(l\in\BZ_{\ge 0}$, $1\le a'\le \lfloor s''/2\rfloor)$,
then the converse also holds.
\item[$2.$] Suppose $s'\le s''$ and let $k,l\in\BZ_{\ge 0}$, $1\le a\le s'$. If $a=1$ we put $k=0$. Then for $1\le i\le \lfloor s/2\rfloor$,
\[
{\rm d}\tau_\lambda(\cU(\mathfrak{so}^*(2s')\oplus\mathfrak{so}^*(2s'')))\cP_{(k+l,\underline{k}_{a-1})} (M(s',s'';\BC))
\subset M_i^{\mathfrak{so}^*(2s)}(\lambda)
\]
holds if
\[
\lambda\in \begin{cases}
2(i-1)-(2k+l)-\BZ_{\ge 0}, & i=1, \\
2(i-1)-2k-\BZ_{\ge 0}, & 2\le i\le \lfloor a/2\rfloor, \\
2(i-1)-\min\{2k,k+l\}-\BZ_{\ge 0}, & a\ge 3\colon\text{odd},\ i=\lceil a/2\rceil, \\
2(i-1)-k-\BZ_{\ge 0}, & \lceil a/2\rceil+1\le i\le a, \\
2(i-1)-\BZ_{\ge 0}, & a+1\le i\le \lfloor s/2\rfloor. \end{cases}
\]
If $a$ is even or $k=0$ or $l=0$, then the converse also holds.
\end{enumerate}
\end{Theorem}

Next we consider $\fp^+=\Herm(3,\BO)^\BC$. Then $d=8$, and as an $\big(\mathfrak{e}_{7(-25)},\widetilde{{\rm U}}(1)\times E_6\big)$-module,
$\cO_\lambda(D)_{\widetilde{K}}$ is reducible if and only if $\lambda\in\BZ$, $\lambda\le 8$.
For these $\lambda$ we have the submodules
\begin{alignat*}{3}
& \cO_\lambda(D)_{\widetilde{K}}\supset M_3(\lambda)\supset M_2(\lambda)\supset M_1(\lambda)\supset\{0\},\qquad && \lambda\in\BZ,\quad \lambda\le 0,&
\\
& \cO_\lambda(D)_{\widetilde{K}}\supset M_3(\lambda)\supset M_2(\lambda)\supset\{0\}, \qquad && \lambda\in\BZ,\quad 1\le\lambda\le 4,&
\\
& \cO_\lambda(D)_{\widetilde{K}}\supset M_3(\lambda)\supset\{0\}, \qquad &&
\lambda\in\BZ,\quad 5\le\lambda\le 8.&
\end{alignat*}
We~consider
\begin{gather*}
\Herm(3,\BO)^\BC=\BC\oplus M(1,2;\BO)^\BC\oplus \Herm(2,\BO)^\BC
=M(2,6;\BC)\oplus\Skew(6,\BC).
\end{gather*}
Then for $k\in\BZ_{\ge 0}$, $\bl\in\BZ_{++}^2$ and for $F(x_{11},x_{22})\in\cP_k(\BC)\boxtimes\cP_{(l_1,l_2)}\big({\rm Herm}(2,\BO)^\BC\big)$,
$F(x_{11},x_{22})\ne 0$, by Corollary~\ref{cor_pole_nonsimple},
\[
(\lambda)_{(k+l_1,\min\{k+l_2,l_1\},l_2),8}
\big\langle F(x_{11},x_{22}),{\rm e}^{(x|\overline{z})_{\fp^+}}\big\rangle_{\lambda,x}
\]
is holomorphic for all $\lambda\in\BC$, and this is non-zero if $k=0$ or $\bl=\underline{l}_{a'}$ ($l\in\BZ_{\ge 0}$, $a'=1,2$).
Similarly, for $k,l\in\BZ_{\ge 0}$ and for $F(x_2)\in\cP_{(k+l,k,k)}({\rm Skew}(6,\BC))$, $F(x_2)\ne 0$,
by Corollary~\ref{cor_pole_simple}(1),
\[
(\lambda)_{(2k+l,\min\{2k,k+l\},k),8}\big\langle F(x_2),{\rm e}^{(x|\overline{z})_{\fp^+}}\big\rangle_{\lambda,x}
\]
is holomorphic for all $\lambda\in\BC$, and this is non-zero if $k=0$ or $l=0$.
Similarly, for $\fp^+_2=\Skew(6,\BC)$, $M(2,6;\BC)$ or $M(1,2;\BO)^\BC$,
for $(k_1,k_2)\in\BZ_{++}^2$ and for $F(x_2)\in\cP_{(k_1,k_2)}(\fp_2^+)$, $F(x_2)\ne 0$, again by Corollary~\ref{cor_pole_simple}(1),
\[
(\lambda)_{(k_1+k_2,k_2),8}\big\langle F(x_2),{\rm e}^{(x|\overline{z})_{\fp^+}}\big\rangle_{\lambda,x}
\]
is non-zero holomorphic for all $\lambda\in\BC$. Therefore the following holds.

\begin{Theorem}\quad
\begin{enumerate}\itemsep=0pt
\item[$1.$] Let $k\in\BZ_{\ge 0}$, $\bl\in\BZ_{++}^2$. Then for $i=1,2,3$,
\[
{\rm d}\tau_\lambda(\cU(\mathfrak{u}(1)\oplus\mathfrak{e}_{6(-14)}))\cP_k(\BC) \boxtimes\cP_{(l_1,l_2)}\big({\rm Herm}(2,\BO)^\BC\big)
\subset M_i^{\mathfrak{e}_{7(-25)}}(\lambda)
\]
holds if
\[
\lambda\in \begin{cases} -(k+l_1)-\BZ_{\ge 0}, &\quad i=1, \\
4-\min\{k+l_2,l_1\}-\BZ_{\ge 0}, &\quad i=2, \\
8-l_2-\BZ_{\ge 0}, &\quad i=3. \end{cases}
\]
If $k=0$ or $\bl=(l,0)$, $(l,l)$ $(l\in\BZ_{\ge 0})$, then the converse also holds.
\item[$2.$] Let $k,l\in\BZ_{\ge 0}$. Then for $i=1,2,3$,
\[
{\rm d}\tau_\lambda(\cU(\mathfrak{su}(2,6)))\cP_{(k+l,k,k)}({\rm Skew}(6,\BC)) \subset M_i^{\mathfrak{e}_{7(-25)}}(\lambda)
\]
holds if
\[
\lambda\in
\begin{cases} -(2k+l)-\BZ_{\ge 0}, & i=1, \\
4-\min\{2k,k+l\}-\BZ_{\ge 0}, & i=2, \\
8-k-\BZ_{\ge 0}, & i=3. \end{cases}
\]
If $k=0$ or $l=0$, then the converse also holds.
\item[$3.$] Let $(k_1,k_2)\in\BZ_{++}^2$. Then for $i=1,2,3$,
\[
\begin{cases} {\rm d}\tau_\lambda(\cU(\mathfrak{su}(2,6)))\cP_{(k_1,k_2,0)}({\rm Skew}(6,\BC)) \\
{\rm d}\tau_\lambda(\cU(\mathfrak{su}(2)\oplus\mathfrak{so}^*(12)))\cP_{(k_1,k_2)}(M(2,6;\BC)) \\
{\rm d}\tau_\lambda(\cU(\mathfrak{sl}(2,\BR)\oplus\mathfrak{so}(2,10)))\cP_{(k_1,k_2)}\big(M(1,2;\BO)^\BC\big) \end{cases}
\subset M_i^{\mathfrak{e}_{7(-25)}}(\lambda)
\]
holds if and only if
\[
\lambda\in \begin{cases} -(k_1+k_2)-\BZ_{\ge 0}, &i=1, \\ 4-k_2-\BZ_{\ge 0}, &i=2, \\ 8-\BZ_{\ge 0}, & i=3. \end{cases}
\]
\end{enumerate}
\end{Theorem}

Finally we consider $\fp^+=M(1,2;\BO)^\BC$. Then we have $d=6$, and as an $\big(\mathfrak{e}_{6(-14)},\widetilde{{\rm U}}(1)\times \operatorname{Spin}(10)\big)$-module,
$\cO_\lambda(D)_{\widetilde{K}}$ is reducible if and only if $\lambda\in\BZ$, $\lambda\le 3$.
For these $\lambda$ we have the submodules
\begin{alignat*}{3}
& \cO_\lambda(D)_{\widetilde{K}}\supset M_2(\lambda)\supset M_1(\lambda)\supset\{0\},\qquad && \lambda\in\BZ,\quad\lambda\le 0,&
\\
& \cO_\lambda(D)_{\widetilde{K}}\supset M_2(\lambda)\supset\{0\},\qquad &&
\lambda\in\BZ,\quad 1\le\lambda\le 3.&
\end{alignat*}
We~consider
\begin{align*}
M(1,2;\BO)^\BC&=\BO^\BC\oplus\BO^\BC\simeq \BC^8\oplus\BC^8 =\BC\oplus \Skew(5,\BC)\oplus M(1,5;\BC) \\
&=M(2,4;\BC)\oplus M(4,2;\BC).
\end{align*}
Then for $(k_1,k_2)\in\BZ_{++}^2$ and for $F(x_2)\in\cP_{(k_1,k_2)}\big(\BC^8\big)$, $F(x_2)\ne 0$, by~(\ref{easycase}),
\[
(\lambda)_{(k_1,k_2),6}\big\langle F(x_2),{\rm e}^{(x|\overline{z})_{\fp^+}}\big\rangle_{\lambda,x}
\]
is non-zero holomorphic for all $\lambda\in\BC$. Similarly, for $k,l\in\BZ_{\ge 0}$ and for $F(x_{11},x_{22})\in\cP_k(\BC)\boxtimes
\cP_l(M(1,5;\BC))$, $F(x_{11},x_{22})\ne 0$, by Corollary~\ref{cor_pole_nonsimple} or Theorem \ref{explicit_nonsimple}(6),
\[
(\lambda)_{(k+l,\min\{k,l\}),6}
\big\langle F(x_{11},x_{22}),{\rm e}^{(x|\overline{z})_{\fp^+}}\big\rangle_{\lambda,x}
\]
is non-zero holomorphic for all $\lambda\in\BC$. Similarly, for $\fp^+_2=M(4,2;\BC)$ or $\Skew(5,\BC)$, for $(k_1,k_2)\in\BZ_{++}^2$ and for
$F(x_2)\in\cP_{(k_1,k_2)}\big(\fp_2^+\big)$, $F(x_2)\ne 0$, by Corollary~\ref{cor_pole_simple}(1)
\[
(\lambda)_{(k_1+k_2,k_2),6}\big\langle F(x_2),{\rm e}^{(x|\overline{z})_{\fp^+}}\big\rangle_{\lambda,x}
\]
is non-zero holomorphic for all $\lambda\in\BC$. Therefore we have the following.
\begin{Theorem}\quad
\begin{enumerate}\itemsep=0pt
\item[$1.$] Let $(k_1,k_2)\in\BZ_{++}^2$. Then for $i=1,2$,
\[
{\rm d}\tau_\lambda(\cU(\mathfrak{u}(1)\oplus\mathfrak{so}(2,8)))\cP_{(k_1,k_2)}\big(\BC^8\big) \subset M_i^{\mathfrak{e}_{6(-14)}}(\lambda)
\]
holds if and only if
\[
\lambda\in \begin{cases} -k_1-\BZ_{\ge 0}, & i=1, \\ 3-k_2-\BZ_{\ge 0}, & i=2. \end{cases}
\]
\item[$2.$] Let $k,l\in\BZ_{\ge 0}$. Then for $i=1,2$,
\[
{\rm d}\tau_\lambda(\cU(\mathfrak{u}(1)\oplus\mathfrak{so}^*(10)))\cP_k(\BC)\boxtimes\cP_l(M(1,5;\BC))
\subset M_i^{\mathfrak{e}_{6(-14)}}(\lambda)
\]
holds if and only if
\[
\lambda\in \begin{cases} -(k+l)-\BZ_{\ge 0}, & i=1, \\ 3-\min\{k,l\}-\BZ_{\ge 0}, & i=2. \end{cases}
\]
\item[$3.$] Let $(k_1,k_2)\in\BZ_{++}^2$. Then for $i=1,2$,
\[
\begin{cases} {\rm d}\tau_\lambda(\cU(\mathfrak{su}(2,4)\oplus\mathfrak{su}(2)))\cP_{(k_1,k_2)}(M(4,2;\BC)) \\
{\rm d}\tau_\lambda(\cU(\mathfrak{sl}(2,\BR)\oplus\mathfrak{su}(1,5)))\cP_{(k_1,k_2)}({\rm Skew}(5,\BC)) \end{cases}
\subset M_i^{\mathfrak{e}_{6(-14)}}(\lambda)
\]
holds if and only if
\[
\lambda\in \begin{cases} -(k_1+k_2)-\BZ_{\ge 0}, &i=1, \\ 3-k_2-\BZ_{\ge 0}, &i=2. \end{cases}
\]
\end{enumerate}
\end{Theorem}

\section{Intertwining operators}\label{section_SBO}

In this section we apply the results on the inner products to the construction of intertwining operators between holomorphic discrete series
representations.

Let $G$ be a connected Hermitian Lie group, with the complexification $G^\BC$.
We~fix a Cartan involution on $G$, and let $K\subset G$, $K^\BC\subset G^\BC$ be the corresponding
maximal compact subgroup and its complexification.
Let $G_1\subset G$ be a connected subgroup, with the complexification $G_1^\BC\subset G^\BC$,
such that $(G,G_1)$ is a symmetric pair of holomorphic type.
Without loss of generality we may assume that $G_1$ is stable under the Cartan involution of $G$,
and let $K_1\subset G_1$, $K^\BC_1\subset G^\BC_1$ be the corresponding maximal compact subgroup and its complexification.
Their corresponding Lie algebras are denoted by the corresponding lowercase frakturs,
and let $\fg^\BC=\fp^+\oplus\fk^\BC\oplus\fp^-$, $\fg^\BC_1=\fp^+_1\oplus\fk^\BC_1\oplus\fp^-_1$ be the gradings with respect to
the adjoint action of $\fz\big(\fk^\BC\big)$. Then $\fp^\pm$, $\fp^\pm_1$ have Jordan triple system structures.
Let $\fp^\pm_2=\big(\fp^\pm_1\big)^\bot\subset\fp^\pm$, and let $D\subset\fp^+$, $D_1\subset\fp^+_1$ be the bounded symmetric domains,
so that $G/K\simeq D$, $G_1/K_1\simeq D_1$ hold.

For $\lambda>\frac{d}{2}(r-1)$, we consider the unitary representation $\cH_\lambda(D)\subset\cO(D)$
of the universal covering group $\widetilde{G}$ with the inner product $\langle\cdot,\cdot\rangle_\lambda$.
Then for the $\widetilde{K}$-finite part we have $\cH_\lambda(D)_{\widetilde{K}}=\cP(\fp^+)$, and this is holomorphic discrete if $\lambda>p-1$.
Similarly, for a unitary representation $(\rho,W)=\big(\chi_1^{-\varepsilon_1\lambda}\otimes\rho_0,W\big)
=\big(\big(\chi^{-\lambda}|_{\widetilde{K}_1}\big)\otimes\rho_0,W\big)$ of~$\widetilde{K}_1$, we consider the unitary representation
$\cH_{\varepsilon_1\lambda}(D_1,W)\subset\cO(D_1,W)$ of $\widetilde{G}_1$ if it exists.
Then $\cH_{\varepsilon_1\lambda}(D_1,W)$ appears in the decomposition of $\cH_\lambda(D)|_{\widetilde{G}_1}$
if and only if $(\rho_0,W)$ appears in the decomposition of $\cP(\fp^+_2)$, and then there exists a symmetry breaking operator
\[
\cF^\downarrow_{\lambda,\rho_0}\colon\ \cH_\lambda(D)|_{\widetilde{G}_1}\longrightarrow \cH_{\varepsilon_1\lambda}(D_1,W), \]
which is unique up to scalar multiple for $\lambda>\frac{d}{2}(r-1)$ by the multiplicity-freeness of $\cH_\lambda(D)|_{\widetilde{G}_1}$.
By Theorem \ref{thm_intertwining}(1), this operator is given by substituting the differential operator to the polynomial
$F^\downarrow_{\lambda,\rho_0}(z)=F^\downarrow_{\lambda,\rho_0}(z_1,z_2)\in\cP(\fp^-,W)$,
\[
F^\downarrow_{\lambda,\rho_0}(z):=\big\langle {\rm e}^{(x|z)_{\fp^+}},\rK(\overline{x_2})^*\big\rangle_{\lambda,\fp^+,x}
=\big\langle \rK(x_2),{\rm e}^{(x|\overline{z})_{\fp^-}},\big\rangle_{\lambda,\fp^-,x},
\]
where $\rK(x_2)\in\cP\bigl(\fp^-_2,\Hom_\BC\big(\chi^{-\lambda},W\otimes\chi_1^{-\varepsilon_1\lambda}\big) \bigr)^{\widetilde{K}_1^\BC}\simeq (\cP(\fp^-_2)\otimes W)^{K_1^\BC}$.
Therefore, the construction of symmetry breaking operators are reduced to Theorems~\ref{thm_nonsimple} and~\ref{thm_simple}.

\subsection[Case $\fp^+_2$ is non-simple]{Case $\boldsymbol{\fp^+_2}$ is non-simple}

First we consider the cases such that $\fp^\pm_2=\fp^\pm_{11}\oplus\fp^\pm_{22}$ is non-simple, and write $\fp^\pm_1=\fp^\pm_{12}$.
We~assume~$\fp^\pm_{11}$ is of tube type. That is, we consider
\begin{gather}
\big(\fp^\pm,\fp^\pm_{11},\fp^\pm_{12},\fp^\pm_{22}\big) \notag \\ \quad
{}= \begin{cases}
\big(\BC^{d+2},\BC,\BC^{d},\BC\big) & (\textit{Case }1), \\
\big({\rm Sym}(r,\BC),\Sym(r',\BC),M(r',r'';\BC),\Sym(r'',\BC)\big) & (\textit{Case }2), \\
\big(M(q,s;\BC),M(r',r';\BC),M(r',s'';\BC)\oplus M(q'',r';\BC),M(q'',s'';\BC)\big) & (\textit{Case }3), \\
\big({\rm Skew}(s,\BC),\Skew(2r',\BC),M(2r',s'';\BC),\Skew(s'',\BC)\big) & (\textit{Case }4), \\
\big({\rm Herm}(3,\BO)^\BC,\BC,M(1,2;\BO)^\BC,\Herm(2,\BO)^\BC\big) & (\textit{Case }5), \\
\big(M(1,2;\BO)^\BC,\BC,\Skew(5,\BC),M(1,5;\BC)\big) & (\textit{Case }6) \end{cases}\!\!\!
\label{list_nonsimple}
\end{gather}
with $r=r'+r''$ for {\it Case} 2, $q=r'+q''$, $s=r'+s''$ for {\it Case} 3, and $s=2r'+s''$ for {\it Case} 4.
Then the corresponding symmetric pairs are
\[
(G,G_1)= \begin{cases}
({\rm SO}_0(2,d+2),{\rm SO}_0(2,d)\times {\rm SO}(2)) & (\textit{Case }1), \\
(\operatorname{Sp}(r,\BR),{\rm U}(r',r'')) & (\textit{Case }2), \\
({\rm U}(q,s),{\rm U}(r',s'')\times {\rm U}(q'',r')) & (\textit{Case }3), \\
({\rm SO}^*(2s),{\rm U}(2r',s'')) & (\textit{Case }4), \\
(E_{7(-25)},{\rm U}(1)\times E_{6(-14)}) & (\textit{Case }5), \\
(E_{6(-14)},{\rm U}(1)\times {\rm SO}^*(10)) & (\textit{Case }6) \end{cases}
\]
(up to covering). Here, for {\it Case} 3 we consider ${\rm U}(q,s)$ instead of ${\rm SU}(q,s)$.
Let $\dim\fp^+=:n$, $\dim\fp^+_{11}=:n'$, $\dim\fp^+_{22}=:n''$, $\rank\fp^+=:r$, $\rank\fp^+_{11}=:r'$, $\rank\fp^+_{22}=:r''$,
let $d$ be the number defined in (\ref{str_const}), and take $\varepsilon_{12}\in\{1,2\}$ such that
\[
(x_{12}|\overline{y_{12}})_{\fp^+}=\varepsilon_{12}(x_{12}|\overline{y_{12}})_{\fp^+_{12}}, \qquad x_{12},y_{12}\in\fp^+_{12},
\]
so that we have
\[
(r,r',r'',d,\varepsilon_{12}) =\begin{cases}
(2,1,1,d,1) & (\textit{Case }1,\; d\ge 2), \\
(2,1,1,1,2) & (\textit{Case }1,\; d=1), \\
(r,r',r'',1,2) & (\textit{Case }2), \\
(\min\{q,s\},r',\min\{q'',s''\},2,1) & (\textit{Case }3), \\
(\lfloor s/2\rfloor, r', \lfloor s''/2\rfloor,4,1) & (\textit{Case }4), \\
(3,1,2,8,1) & (\textit{Case }5), \\
(2,1,1,6,1) & (\textit{Case }6). \end{cases}
\]
For $j=1,2$, let $\fk_{jj}:=\big[\fp^+_{jj},\fp^-_{jj}\big]\cap\fk$, let $K_{jj}$ be the corresponding connected subgroup of $G_1$,
and let $\Sigma_{jj}\subset\fp^+_{jj}$ be the Bergman--Shilov boundary of the bounded symmetric domain,
with the normalized $K_{jj}$-invariant measure, so that $\Sigma_{jj}\simeq K_{jj}/K_{L_{jj}}M_{jj}$ holds
for some $K_{L_{jj}}M_{jj}\subset K_{jj}$.

As an example, we consider {\it Case} 2. Then according to the decomposition of $\cP\big(\fp^+_{11}\big)\boxtimes\cP\big(\fp^+_{22}\big)$
under $K_1={\rm U}(r')\times {\rm U}(r'')$,
\[
\cP\big(\fp^+_{11}\big)\boxtimes\cP\big(\fp^+_{22}\big)
=\bigoplus_{\bk\in\BZ_{++}^{r'}}\bigoplus_{\bl\in\BZ_{++}^{r''}}\cP_\bk\big(\fp^+_{11}\big)\boxtimes\cP_\bl\big(\fp^+_{22}\big)
\simeq \bigoplus_{\bk\in\BZ_{++}^{r'}}\bigoplus_{\bl\in\BZ_{++}^{r''}} V_{2\bk}^{(r')\vee}\boxtimes V_{2\bl}^{(r'')},
\]
$\cH_\lambda(D_{\operatorname{Sp}(r,\BR)})$ for $\lambda>\frac{r-1}{2}$ is decomposed under $\widetilde{G}_1$ as
\[
\cH_\lambda(D_{\operatorname{Sp}(r,\BR)})|_{\widetilde{G}_1} \simeq\hsum_{\bk\in\BZ_{++}^{r'}}\hsum_{\bl\in\BZ_{++}^{r''}}
\cH_{\lambda+\lambda}\big(D_{{\rm U}(r',r'')},V_{2\bk}^{(r')\vee}\boxtimes V_{2\bl}^{(r'')}\big),
\]
where $\cH_{\lambda+\lambda}(D_{{\rm U}(r',r'')},V'\boxtimes V'')$ is defined by using the character (\ref{char_U(q,s)})
of $\widetilde{{\rm U}}(r')\times\widetilde{{\rm U}}(r'')\subset \widetilde{{\rm U}}(r',r'')$.
In the following let $\bk=\underline{k}_{r'}$, and consider the symmetry breaking operator
\begin{align*}
\cF_{\lambda,k,\bl}^\downarrow\colon\ \cH_\lambda(D_{\operatorname{Sp}(r,\BR)})\longrightarrow{}&
\cH_{\lambda+\lambda}\big(D_{{\rm U}(r',r'')},V_{\underline{2k}_{r'}}^{(r')\vee}\boxtimes V_{2\bl}^{(r'')}\big)
\\
\simeq {}&\cH_{(\lambda+2k)+\lambda}\big(D_{{\rm U}(r',r'')},\BC\boxtimes V_{2\bl}^{(r'')}\big).
\end{align*}
We~take a vector-valued polynomial
$\rK_\bl(x_{22})\in\bigl(\cP_\bl(\fp^-_{22})\otimes V_{2\bl}^{(r'')}\bigr)^{K_{22}}$, normalized such that
\[
\big\Vert (v,\rK_\bl(\overline{x_{22}}))_{V_{2\bl}^{(r'')}}\big\Vert_{L^2(\Sigma_{22}),x_{22}}^2= |v|_{V_{2\bl}^{(r'')}}^2,
\qquad v\in V_{2\bl}^{(r'')},
\]
or equivalently,
\[
\big|\big\langle f(x_{22}), \overline{\rK_\bl(\overline{x_{22}})}\big\rangle_{L^2(\Sigma_{22}),x_{22}}\big|_{V_{2\bl}^{(r'')}}^2
=\Vert f(x_{22})\Vert_{L^2(\Sigma_{22}),x_{22}}^2, \qquad f(x_{22})\in \cP_\bl\big(\fp^+_{22}\big).
\]
Using this, we define the polynomials
$\tilde{\rK}_{\bl,\bm,\bn}^{\fp^-_{22}}(x_{22},y_{22})\in\cP_\bm(\fp^-_{22})\otimes\cP_\bn(\fp^-_{22})\otimes V_{2\bl}^{(r'')}$ by
\[ \sum_{\bm,\bn\in\BZ_{++}^{r''}}\tilde{\rK}_{\bl,\bm,\bn}^{\fp^-_{22}}(x_{22},y_{22})=\rK_\bl(x_{22}+y_{22}). \]
Then the polynomial $F^\downarrow_{\lambda,\rho_0}(z)=F^\downarrow_{\lambda,k,\bl}(z)\in\cP\big(\fp^-,V_{2\bl}^{(r'')}\big)$ is computed by using Theorem \ref{thm_nonsimple}. We~renormalize this as
\begin{align*}
F^\downarrow_{\lambda,k,\bl}(z):\!={}&\frac{1}{\big(\frac{n'}{r'}\big)_{\underline{k}_{r'},d} \big(\frac{n''}{r''}\big)_{\bl,d}
\big(\lambda+k-\frac{d}{2}r'\big)_{\bl,d}} \\
&\times\det_{\fn^-_{11}}(z_{11})^k\sum_{\bm,\bn\in\BZ_{++}^{r''}}(-k)_{\bm,d} \bigg(\lambda+k-\frac{d}{2}r'\bigg)_{\bn,d}
\tilde{\rK}_{\bl,\bm,\bn}^{\fp^-_{22}}\big(Q(z_{12}){}^t\hspace{-1pt}z_{11}^\itinv,z_{22}\big)
\\
={}&\frac{\det_{\fn^-_{11}}(z_{11})^k}
{\big(\frac{n'}{r'}\big)_{\underline{k}_{r'},d}\big(\frac{n''}{r''}\big)_{\bl,d}}
\sum_{\bm,\bn\in\BZ_{++}^{r''}}\frac{(-k)_{\bm,d}}{\big(\lambda+k-\frac{d}{2}r'+\bn\big)_{\bl-\bn,d}}
\tilde{\rK}_{\bl,\bm,\bn}^{\fp^-_{22}}\big(Q(z_{12}){}^t\hspace{-1pt}z_{11}^\itinv,z_{22}\big),
\end{align*}
with $d=1$. Then the linear map
$\cF^\downarrow_{\lambda,k,\bl}=F^\downarrow_{\lambda,k,\bl}\big(\frac{\partial}{\partial x}\big)\bigr|_{x_{11}=0,x_{22}=0}$
becomes a symmetry breaking operator, and for $f(x_{22})\in\cP_\bl(\fp^+_{22})$, by \cite[Part III, Corollary V.3.5]{FKKLR} we have
\begin{gather*}
\big(\cF^\downarrow_{\lambda,k,\bl}\big(\det_{\fn^+}(x_{11})^kf(x_{22})\big)\big)(x_{12})
\\ \qquad
{}=\frac{1}{\big(\frac{n'}{r'}\big)_{\underline{k}_{r'},d}\big(\frac{n''}{r''}\big)_{\bl,d}}
\det_{\fn^-_{11}}\bigg(\frac{\partial}{\partial x_{11}}\bigg)^k\rK_\bl\bigg(\frac{\partial}{\partial x_{22}}\bigg)
\det_{\fn^+_{11}}(x_{11})^kf(x_{22})\bigg |_{x_{11}=0,x_{22}=0}
\\ \qquad
{}=\frac{1}{\big(\frac{n'}{r'}\big)_{\underline{k}_{r'},d}\big(\frac{n''}{r''}\big)_{\bl,d}}
\big\langle\det_{\fn^+_{11}}(x_{11})^kf(x_{22}),\det_{\fn^+_{11}}(x_{11})^k \overline{\rK_\bl(\overline{x_{22}})}
\big\rangle_{F,\fp^+_{11}\oplus\fp^+_{22},x_{11},x_{22}}
\\ \qquad
{}=\big\langle f(x_{22}),\overline{\rK_\bl(\overline{x_{22}})}\big\rangle_{L^2(\Sigma_{22}),x_{22}}.
\end{gather*}
Hence by the normalization of $\rK_\bl(x_{22})$ we have
\begin{gather*}
\big\Vert\cF^\downarrow_{\lambda,k,\bl}\big(\det_{\fn^+}(x_{11})^kf(x_{22})\big)(x_{12})
\big\Vert_{\cH_{(\lambda+2k)+\lambda}(D_{{\rm U}(r',r'')},\BC\boxtimes V_{2\bl}^{(r'')})}^2
\\ \qquad
{}=\big|\big\langle f(x_{22}),\overline{\rK_\bl(\overline{x_{22}})}\big\rangle_{L^2(\Sigma_{22}),x_{22}} \big|_{V_{2\bl}^{(r'')}}^2
=\big\Vert f(x_{22})\big\Vert_{L^2(\Sigma_{22}),x_{22}}^2.
\end{gather*}
On the other hand, by (\ref{norm_nonsimple}) and again by \cite[Part III, Corollary V.3.5]{FKKLR} we have
\[
\big\Vert \det_{\fn^+}(x_{11})^kf(x_{22})\big\Vert_{\lambda,\fp^+,x}^2
=\frac{\big(\frac{n'}{r'}\big)_{\underline{k}_{r'},d}\big(\frac{n''}{r''}\big)_{\bl,d}
\big(\lambda-\frac{d}{2}r'\big)_{\underline{k}_{r''}+\bl,d}}
{(\lambda)_{\underline{k}_r+\bl,d}\big(\lambda-\frac{d}{2}r'\big)_{\bl,d}}\big\Vert f(x_{22})\big\Vert_{L^2(\Sigma_{22}),x_{22}}^2.
\]
Now since $\cF^\downarrow_{\lambda,k,\bl}$ is an isometry up to constant on the subspace of $\cH_\lambda(D_{\operatorname{Sp}(r,\BR)})$
isomorphic to $\cH_{(\lambda+2k)+\lambda}\big(D_{{\rm U}(r',r'')},\BC\boxtimes V_{2\bl}^{(r'')}\big)$,
and identically zero on its orthogonal complement, we have
\[
\big\Vert \cF^\downarrow_{\lambda,k,\bl}\big\Vert_{\op}^{-2}
=\frac{\big(\frac{n'}{r'}\big)_{\underline{k}_{r'},d}\big(\frac{n''}{r''}\big)_{\bl,d}
\big(\lambda-\frac{d}{2}r'\big)_{\underline{k}_{r''}+\bl,d}}
{(\lambda)_{\underline{k}_r+\bl,d}\big(\lambda-\frac{d}{2}r'\big)_{\bl,d}}.
\]
Similar results also hold for the other cases. Now we summarize the results so far.
\begin{Theorem}\label{SBO_nonsimple1}
Let $\big(\fp^\pm,\fp^\pm_{11},\fp^\pm_{12},\fp^\pm_{22}\big)$ be as in $(\ref{list_nonsimple})$.
For $\bl\in\BZ_{++}^{r''}$, let $V_\bl$ be the representation of~$K_{22}$ isomorphic to $\cP_\bl(\fp^+_{22})$, and
we take a vector-valued polynomial $\rK_\bl(x_{22})\in\bigl(\cP_\bl(\fp^-_{22})\otimes V_\bl\bigr)^{K_{22}}$ such that
\[
\big\Vert (v,\rK_\bl(\overline{x_{22}}))_{V_\bl}\big\Vert_{L^2(\Sigma_{22}),x_{22}}^2
=|v|_{V_{\bl}}^2, \qquad v\in V_{\bl}.
\]
We~define the polynomials
$\tilde{\rK}_{\bl,\bm,\bn}^{\fp^-_{22}}(x_{22},y_{22})\in\cP_\bm(\fp^-_{22})\otimes\cP_\bn(\fp^-_{22})\otimes V_\bl$ by
\[ \sum_{\bm,\bn\in\BZ_{++}^{r''}}\tilde{\rK}_{\bl,\bm,\bn}^{\fp^-_{22}}(x_{22},y_{22})=\rK_\bl(x_{22}+y_{22}). \]
Using these, for $\lambda>\frac{d}{2}(r-1)$, $k\in\BZ_{\ge 0}$, we define
$F^\downarrow_{\lambda,k,\bl}(z)=F^\downarrow_{\lambda,k,\bl}(z_{11},z_{12},z_{22})\in\cP(\fp^-,V_\bl)$ by
\[ F^\downarrow_{\lambda,k,\bl}(z):=\frac{\det_{\fn^-_{11}}(z_{11})^k}
{\left(\frac{n'}{r'}\right)_{\underline{k}_{r'},d}\left(\frac{n''}{r''}\right)_{\bl,d}}
\sum_{\bm,\bn\in\BZ_{++}^{r''}}\frac{(-k)_{\bm,d}}{\left(\lambda+k-\frac{d}{2}r'+\bn\right)_{\bl-\bn,d}}
\tilde{\rK}_{\bl,\bm,\bn}^{\fp^-_{22}}\big(Q(z_{12}){}^t\hspace{-1pt}z_{11}^\itinv,z_{22}\big). \]
Then the linear map
\begin{gather*}
\cF^\downarrow_{\lambda,k,\bl}\colon \left\{ \hspace{-4pt} \begin{array}{lll}
\cH_\lambda(D_{{\rm SO}_0(2,d+2)})\hspace{-7pt}&\longrightarrow\cH_{\lambda+k+l}(D_{{\rm SO}_0(2,d)})\boxtimes \BC_{-k+l} &(\textit{Case }1),
\\
\cH_\lambda(D_{\operatorname{Sp}(r,\BR)})&\longrightarrow\cH_{(\lambda+2k)+\lambda}\big(D_{{\rm U}(r',r'')},\BC\boxtimes V_{2\bl}^{(r'')}\big) &(\textit{Case }2),
\\
\cH_{\lambda_1+\lambda_2}(D_{{\rm U}(q,s)})\hspace{-7pt} &\longrightarrow \cH_{(\lambda_1+k)+\lambda_2}\big(D_{{\rm U}(r',s'')},\BC\boxtimes V_\bl^{(s'')}\big) &
\\
&\qquad{}\hboxtimes \cH_{\lambda_1+(\lambda_2+k)}\big(D_{{\rm U}(q'',r')},V_\bl^{(q'')\vee}\boxtimes \BC\big) &(\textit{Case }3),
\\
\cH_\lambda(D_{{\rm SO}^*(2s)})&\longrightarrow
\cH_{\left(\frac{\lambda}{2}+k\right)+\frac{\lambda}{2}}\big(D_{{\rm U}(2r',s'')},\BC\boxtimes V_{\bl^2}^{(s'')}\big) &(\textit{Case }4), \\
\cH_\lambda(D_{E_{7(-25)}})&\longrightarrow \cH_{\lambda+k+\frac{|\bl|}{2}}\big(D_{E_{6(-14)}},V_{(l_1-l_2,0,0,0,0)}^{[10]\vee}\big)
\boxtimes \chi_{{\rm U}(1)}^{-\lambda+2k-2|\bl|}\! &(\textit{Case }5), \\
\cH_\lambda(D_{E_{6(-14)}})&\longrightarrow
\cH_{\lambda+k}\big(D_{{\rm SO}^*(10)},V_{\left(\frac{l}{2},\frac{l}{2},\frac{l}{2},\frac{l}{2},-\frac{l}{2}\right)}^{(5)\vee}\big)
\boxtimes \chi_{{\rm U}(1)}^{-\lambda+3k-3l} &(\textit{Case }6), \end{array} \right.
\\
(\cF^\downarrow_{\lambda,k,\bl}f)(x_{12})=F^\downarrow_{\lambda,k,\bl}
\bigg(\frac{\partial}{\partial x}\bigg)f(x)\bigg|_{x_{11}=0,x_{22}=0}
\end{gather*}
intertwines the $\widetilde{G}_1$-action, where we put $\lambda_1+\lambda_2=\lambda$ for {\it Case} $3$, and its operator norm is given~by
\[
\Vert \cF^\downarrow_{\lambda,k,\bl}\Vert^{-2}_{\op}
=\frac{\left(\frac{n'}{r'}\right)_{\underline{k}_{r'},d}\left(\frac{n''}{r''}\right)_{\bl,d}
\left(\lambda-\frac{d}{2}r'\right)_{\underline{k}_{r''}+\bl,d}}
{(\lambda)_{\underline{k}_r+\bl,d}\left(\lambda-\frac{d}{2}r'\right)_{\bl,d}}.
\]
\end{Theorem}

For holographic operators $\cF^\uparrow_{\lambda,k,\bl}$, see \cite[Theorems 5.7(2), 5.10 and 5.19(1)]{N2}.
Next we assume~$\fp^\pm_{22}$ is also of tube type, that is, we consider {\it Case} 1, {\it Case} 2, {\it Case} 3 with $q''=s''=r''$, {\it Case} 4 with $s''=2r''$, and {\it Case} 5.
Then by Theorem \ref{thm_nonsimple} and (\ref{key_nonsimple}) we have
\begin{align*}
F_{\lambda,k,\bl}^\downarrow(z)
={}&\frac{1}{\big(\frac{n'}{r'}\big)_{\underline{k}_{r'},d}\big(\frac{n''}{r''}\big)_{\bl,d}}
\frac{\big(\lambda-\frac{d}{2}r'\big)_{\bl,d}}
{\big(\lambda-\frac{d}{2}r'\big)_{\underline{k}_{r''},d}\big(\lambda+k-\frac{d}{2}r'\big)_{\bl,d}}
\\
&\times\det_{\fn^-}(z)^{-\lambda+\frac{n}{r}}
\det_{\fn^-_{22}}\bigg(\frac{\partial}{\partial z_{22}}\bigg)^k\det_{\fn^-}(z)^{\lambda+k-\frac{n}{r}}\rK_\bl(z_{22})
\\
={}&\frac{\det_{\fn^-}(z)^{-\lambda+\frac{n}{r}}}{\big(\frac{n'}{r'}\big)_{\underline{k}_{r'},d} \big(\frac{n''}{r''}\big)_{\bl,d}
\big(\lambda-\frac{d}{2}r'+\bl\big)_{\underline{k}_{r''},d}}
\det_{\fn^-_{22}}\bigg(\frac{\partial}{\partial z_{22}}\bigg)^k\det_{\fn^-}(z)^{\lambda+k-\frac{n}{r}}\rK_\bl(z_{22}).
\end{align*}
Moreover, if $\bl=\underline{l}_{r''}$, then as in Corollary~\ref{cor_scalar_nonsimple}, the intertwining operators become easier.
Here ${}_0F_1^{\fp^+_{11}}$ and ${}_2F_1^{\fp^+_{11}}$ are as in (\ref{0F1^p}) and (\ref{2F1^p}).

\begin{Theorem}\label{SBO_nonsimple2}
Assume $\fp^\pm$, $\fp^\pm_{11}$, $\fp^\pm_{22}$ are of tube type.
For $\lambda>\frac{d}{2}(r-1)$, $k,l\in\BZ_{\ge 0}$,
we define $F^\uparrow_{\lambda,k,l}(x_{11},x_{22};z_{12})\in\cO\big(\fp^+_{11}\times\fp^+_{22}\times\fp^-_{12}\big)$ by
\[ F^\uparrow_{\lambda,k,l}(x_{11},x_{22};z_{12}):=\det_{\fn^+_{11}}(x_{11})^k\det_{\fn^+_{22}}(x_{22})^l
{}_0F_1^{\fp^+_{11}}\bigg(\lambda+k+l-\frac{d-2}{2};x_{11},Q(z_{12})x_{22}\bigg),
\]
and $F^\downarrow_{\lambda,k,\bl}(z)=F^\downarrow_{\lambda,k,\bl}(z_{11},z_{12},z_{22})\in\cP(\fp^-)$ by
\[
F^\downarrow_{\lambda,k,l}(z):=\frac{\det_{\fn^-_{11}}(z_{11})^k\det_{\fn^-_{22}}(z_{22})^l}
{\left(\frac{n'}{r'}\right)_{\underline{k}_{r'},d}\left(\frac{n''}{r''}\right)_{\underline{l}_{r''},d}}
{}_2F_1^{\fp^+_{11}}\left(\begin{matrix} -k,-l\\-\lambda-k-l+\frac{n}{r}\end{matrix};
{}^t\hspace{-1pt}z_{11}^\itinv,Q(z_{12}){}^t\hspace{-1pt}z_{22}^\itinv\right)\!.
\]
Then the linear map
\begin{gather*}
\cF^\uparrow_{\lambda,k,l}\colon \left\{ \hspace{-4pt} \begin{array}{lll}
\bigl(\cH_{\lambda+k+l}(D_{{\rm SO}_0(2,d)})\boxtimes \BC_{-k+l}\bigr)_{\widetilde{K}_1}
&\longrightarrow\cH_\lambda(D_{{\rm SO}_0(2,d+2)})_{\widetilde{K}} &(\textit{Case }1), \\
\cH_{(\lambda+2k)+(\lambda+2l)}(D_{{\rm U}(r',r'')})_{\widetilde{K}_1} &\longrightarrow \cH_\lambda(D_{\operatorname{Sp}(r,\BR)})_{\widetilde{K}} & (\textit{Case }2), \\
\bigl(\cH_{(\lambda_1+k)+(\lambda_2+l)}(D_{{\rm U}(r',r'')}) &&\\
\hspace{17pt}{}\hboxtimes \cH_{(\lambda_1+l)+(\lambda_2+k)}(D_{{\rm U}(r'',r')})\bigr)_{\widetilde{K}_1} \hspace{-7pt}
&\longrightarrow \cH_{\lambda_1+\lambda_2}(D_{{\rm U}(r,r)})_{\widetilde{K}} & (\textit{Case }3), \\
\cH_{\left(\frac{\lambda}{2}+k\right)+\left(\frac{\lambda}{2}+l\right)}(D_{{\rm U}(2r',2r'')})_{\widetilde{K}_1}
&\longrightarrow \cH_\lambda(D_{{\rm SO}^*(4r)})_{\widetilde{K}} & (\textit{Case }4), \\
\bigl(\cH_{\lambda+k+l}(D_{E_{6(-14)}})\boxtimes \chi_{{\rm U}(1)}^{-\lambda+2k-4l}\bigr)_{\widetilde{K}_1} \hspace{-7pt}
&\longrightarrow \cH_\lambda(D_{E_{7(-25)}})_{\widetilde{K}} & (\textit{Case }5), \end{array} \right. \\
(\cF^\uparrow_{\lambda,k,l}f)(x)
=F^\uparrow_{\lambda,k,l}\left(x_{11},x_{22};\frac{1}{\varepsilon_{12}}\frac{\partial}{\partial x_{12}}\right)f(x_{12}),
\end{gather*}
intertwines the $\big(\fg_1,\widetilde{K}_1\big)$-action, and the linear map
\begin{gather*}
\cF^\downarrow_{\lambda,k,l}\colon \left\{ \hspace{-4pt} \begin{array}{lll}
\cH_\lambda(D_{{\rm SO}_0(2,d+2)})\hspace{-7pt}&\longrightarrow\cH_{\lambda+k+l}(D_{{\rm SO}_0(2,d)})\boxtimes \BC_{-k+l} &(\textit{Case }1), \\
\cH_\lambda(D_{\operatorname{Sp}(r,\BR)})&\longrightarrow\cH_{(\lambda+2k)+(\lambda+2l)}(D_{{\rm U}(r',r'')}) &(\textit{Case }2), \\
\cH_{\lambda_1+\lambda_2}(D_{{\rm U}(r,r)})\hspace{-7pt} &\longrightarrow \cH_{(\lambda_1+k)+(\lambda_2+l)}(D_{{\rm U}(r',r'')}) &\\
&\hspace{35pt}{}\hboxtimes \cH_{(\lambda_1+l)+(\lambda_2+k)}(D_{{\rm U}(r'',r')}) &(\textit{Case }3), \\
\cH_\lambda(D_{{\rm SO}^*(4r)})&\longrightarrow
\cH_{\left(\frac{\lambda}{2}+k\right)+\left(\frac{\lambda}{2}+l\right)}(D_{{\rm U}(2r',2r'')}) &(\textit{Case }4), \\
\cH_\lambda(D_{E_{7(-25)}})&\longrightarrow \cH_{\lambda+k+l}(D_{E_{6(-14)}})
\boxtimes \chi_{{\rm U}(1)}^{-\lambda+2k-4l} &(\textit{Case }5), \end{array} \right. \\
\big(\cF^\downarrow_{\lambda,k,l}f\big)(x_{12})
=F^\downarrow_{\lambda,k,l}\bigg(\frac{\partial}{\partial x}\bigg)f(x)\bigg|_{x_{11}=0,x_{22}=0}
\end{gather*}
intertwines the $\widetilde{G}_1$-action, where we put $\lambda_1+\lambda_2=\lambda$ for {\it Case} $3$. Their operator norms are given~by
\begin{align*}
\big\Vert \cF^\uparrow_{\lambda,k,l}\big\Vert^2_{\op}=\big\Vert \cF^\downarrow_{\lambda,k,l}\big\Vert^{-2}_{\op}
&=\frac{\big(\frac{n'}{r'}\big)_{\underline{k}_{r'},d}\big(\frac{n''}{r''}\big)_{\underline{l}_{r''},d}
\big(\lambda+l-\frac{1}{2}r'\big)_{\underline{k}_{r''},d}}
{(\lambda)_{\underline{k+l}_{r''},d}\big(\lambda-\frac{1}{2}r''\big)_{\underline{k}_{r'},d}} \\
&=\frac{\big(\frac{n'}{r'}\big)_{\underline{k}_{r'},d}\big(\frac{n''}{r''}\big)_{\underline{l}_{r''},d}
\big(\lambda+k-\frac{1}{2}r''\big)_{\underline{l}_{r'},d}}
{(\lambda)_{\underline{k+l}_{r'},d}\big(\lambda-\frac{1}{2}r'\big)_{\underline{l}_{r''},d}}.
\end{align*}
\end{Theorem}
Here we normalize $\frac{\partial}{\partial x}$ and $\frac{\partial}{\partial x_{12}}$ with respect to the bilinear forms
$(\cdot|\cdot)_{\fp^+}\colon\fp^+\times\fp^-\to\BC$ and $(\cdot|\cdot)_{\fp^+_{12}}\colon\fp^+_{12}\times\fp^-_{12}\to\BC$ respectively,
so that $\frac{\partial}{\partial x}\bigr|_{\fp^+_{12}}=\frac{1}{\varepsilon_{12}}\frac{\partial}{\partial x_{12}}$ holds.

For {\it Case} 3, the above result is generalized to the case such that $\fp^+_{11}=M(q',s';\BC)$, $\fp^+_{22}=M(q'',s'';\BC)$ are of non-tube type,
as in Theorem \ref{explicit_nonsimple}(3).
When $q'\le s'$, let $\cP_{\underline{k}_{q'}}(M(q',s';\BC))_R$ be the representation of ${\rm GL}(s',\BC)$ given by
\begin{equation}\label{PMR}
g.f(x_{11}):=f(x_{11}g), \qquad g\in {\rm GL}(s',\BC),\quad f\in\cP_{\underline{k}_{q'}}(M(q',s';\BC)),
\end{equation}
so that $\cP_{\underline{k}_{q'}}(M(q',s';\BC))_R\simeq V_{\underline{k}_{q'}}^{(s')}$ holds.
Similarly, when $q'\ge s'$, let $\cP_{\underline{k}_{s'}}(M(q',s';\BC))_L$ be the representation of ${\rm GL}(q',\BC)$ given by
\[ g.f(x_{11}):=f\big(g^{-1}x_{11}\big), \qquad g\in {\rm GL}(q',\BC),\quad f\in\cP_{\underline{k}_{s'}}(M(q',s';\BC)), \]
so that $\cP_{\underline{k}_{s'}}(M(q',s';\BC))_L\simeq V_{\underline{k}_{s'}}^{(q')\vee}$ holds.
We~consider the Fischer inner product $\langle\cdot,\cdot\rangle_F$ on $\cP_{\underline{k}_{q'}}(M(q',s';\BC))_R$ and
$\cP_{\underline{k}_{s'}}(M(q',s';\BC))_L$. Especially if $q'=s'$, then the identification $\BC_{-k}\to\cP_{\underline{k}_{q'}}(M(q',\BC))$,
$1\mapsto\det(x_{11})^k$ has the operator norm $\sqrt{(q')_{\underline{k}_{q'},2}}$.
Then the intertwining operators are given as follows. Here ${}_0F_1^{(2)}$ and ${}_2F_1^{(2)}$ are as in (\ref{0F1^(d)}) and (\ref{2F1^(d)}).
\begin{Theorem}\label{SBO_nonsimple3}
We~consider $\big(\fp^\pm,\fp^\pm_{11},\fp^\pm_{12},\fp^\pm_{21},\fp^\pm_{22}\big)=(M(q,s;\BC),M(q',s';\BC),M(q',s'';\BC),\linebreak
M(q'',s';\BC),M(q'',s'';\BC))$ with $q=q'+q''$, $s=s'+s''$. We~assume $q'\le s'$.
For $\lambda=\lambda_1+\lambda_2>\min\{q,s\}-1$, $k,l\in\BZ_{\ge 0}$, we define
$F^\uparrow_{\lambda,k,l}(x_{11},x_{22};z_{12},z_{21},y_{11},y_{22})
\in\cO\big(\fp^+_{11}\times\fp^+_{22}\times\fp^-_{12}\times\fp^-_{21}\times\fp^-_{11}\times\fp^-_{22}\big)$ by
\begin{gather*}
F^\uparrow_{\lambda,k,l}(x_{11},x_{22};z_{12},z_{21},y_{11},y_{22})
\\ \qquad
{}:=\begin{cases}
\ds \frac{\det\big(x_{11}{}^t\hspace{-1pt}y_{11}\big)^k\det\big(x_{22}{}^t\hspace{-1pt}y_{22}\big)^l}
{(q')_{\underline{k}_{q'},2}(q'')_{\underline{l}_{q''},2}}
{}_0F_1^{(2)}\big(\lambda+k+l;x_{11}{}^t\hspace{-1pt}z_{21}x_{22}{}^t\hspace{-1pt}z_{12}\big), & q''\le s'',
\\[3ex]
\ds \frac{\det\big(x_{11}{}^t\hspace{-1pt}y_{11}\big)^k\det\big({}^t\hspace{-1pt}y_{22}x_{22}\big)^l}
{(q')_{\underline{k}_{q'},2}(s'')_{\underline{l}_{s''},2}}
{}_0F_1^{(2)}\big(\lambda+k+l;x_{11}{}^t\hspace{-1pt}z_{21}x_{22}{}^t\hspace{-1pt}z_{12}\big), & q''\ge s'',
\end{cases}
\end{gather*}
and $F^\downarrow_{\lambda,k,l}(z;w_{11},w_{22})
=F^\downarrow_{\lambda,k,l}\big(\!\left(\begin{smallmatrix}z_{11}&z_{12}\\z_{21}&z_{22} \end{smallmatrix}\right);w_{11},w_{22}\big)
\in\cP\big(\fp^-\times\fp^+_{11}\times\fp^+_{22}\big)$ by
\begin{gather*}
F^\downarrow_{\lambda,k,l}(z;w_{11},w_{22})
\\ \qquad
{}:=\begin{cases}
\ds \frac{\det\big(z_{11}{}^t\hspace{-1pt}w_{11}\big)^k\det\big(z_{22}{}^t\hspace{-1pt}w_{22}\big)^l}
{(q')_{\underline{k}_{q'},2}(q'')_{\underline{l}_{q''},2}}
\vspace{1mm}\\ \
\ds \times\, {}_2F_1^{(2)}\left(\begin{matrix} -k,-l \\ -\lambda\!-\!k\!-\!l\!+\!q'\!+\!q''\end{matrix};
\big(z_{11}{}^t\hspace{-1pt}w_{11}\big)^{-1}z_{12}{}^t \hspace{-1pt}w_{22}\big(z_{22}{}^t\hspace{-1pt}w_{22}\big)^{-1}z_{21}{}^t\hspace{-1pt}w_{11}\right)\!,
& q''\le s'',
\vspace{1mm}\\
\ds \frac{\det\big(z_{11}{}^t\hspace{-1pt}w_{11}\big)^k\det\big({}^t\hspace{-1pt}w_{22}z_{22}\big)^l}
{(q')_{\underline{k}_{q'},2}(s'')_{\underline{l}_{s''},2}}
\vspace{1mm}\\ \
\ds \times\, {}_2F_1^{(2)}\left(\begin{matrix} -k,-l \\ -\lambda\!-\!k\!-\!l\!+\!q'\!+\!s''\end{matrix};
\big(z_{11}{}^t\hspace{-1pt}w_{11}\big)^{-1}z_{12}\big({}^t\hspace{-1pt}w_{22}z_{22}\big)^{-1}{}^t \hspace{-1pt}w_{22}z_{21}{}^t\hspace{-1pt}w_{11}\right)\!,
& q''\ge s''. \end{cases}
\end{gather*}
Then the linear map
\begin{gather*}
\cF^\uparrow_{\lambda,k,l}\colon \
\begin{cases}
\bigl(\cH_{(\lambda_1+k)+\lambda_2}(D_{{\rm U}(q',s'')},\BC\boxtimes\cP_{\underline{l}_{q''}}(\fp^+_{22})_R)
\\ \qquad
{}\hboxtimes\cH_{(\lambda_1+l)+\lambda_2}(D_{{\rm U}(q'',s')},\BC \boxtimes\cP_{\underline{k}_{q'}}(\fp^+_{11})_R)\bigr)_{\widetilde{K}_1},\hspace{10mm} q''\le s'',
\\[2ex]
\bigl(\cH_{(\lambda_1+k)+(\lambda_2+l)}(D_{{\rm U}(q',s'')})
\hboxtimes
\cH_{\lambda_1+\lambda_2}(D_{{\rm U}(q'',s')},\cP_{\underline{l}_{s''}}(\fp^+_{22})_L\boxtimes\cP_{\underline{k}_{q'}}(\fp^+_{11})_R)
\bigr)_{\widetilde{K}_1},
\\ \hspace{92mm}
q''\ge s''
\end{cases}
\\ \qquad
{}\longrightarrow\cH_{\lambda_1+\lambda_2}(D_{{\rm U}(q,s)})_{\widetilde{K}},
\\
\big(\cF^\uparrow_{\lambda,k,l}f\big)(x)=F^\uparrow_{\lambda,k,l}\bigg(x_{11},x_{22};
\frac{\partial}{\partial x_{12}},\frac{\partial}{\partial x_{21}},
\frac{\partial}{\partial w_{11}},\frac{\partial}{\partial w_{22}}\bigg)f(x_{12},x_{21},w_{11},w_{22})
\bigg|_{w_{11}=0,\,w_{22}=0}
\end{gather*}
intertwines the $\big(\mathfrak{u}(q',s'')\oplus\mathfrak{u}(q'',s'),
\widetilde{{\rm U}}(q')\times\widetilde{{\rm U}}(s'')\times\widetilde{{\rm U}}(q'')\times\widetilde{{\rm U}}(s')\big)$-action,
and the linear map
\begin{gather*}
\cF^\downarrow_{\lambda,k,l}\colon \ \cH_{\lambda_1+\lambda_2}(D_{{\rm U}(q,s)})\!\longrightarrow \!
\begin{cases}
\cH_{(\lambda_1+k)+\lambda_2}\big(D_{{\rm U}(q',s'')},\BC\boxtimes\cP_{\underline{l}_{q''}}(\fp^+_{22})_R\big)
\\ \quad
{}\hboxtimes\cH_{(\lambda_1+l)+\lambda_2}\big(D_{{\rm U}(q'',s')},\BC\boxtimes\cP_{\underline{k}_{q'}}(\fp^+_{11})_R\big), \hspace*{12mm} q''\le s'',
\\
\cH_{(\lambda_1+k)+(\lambda_2+l)}\big(D_{{\rm U}(q',s'')}\big)
\\ \quad
{}\hboxtimes
\cH_{\lambda_1+\lambda_2}\big(D_{{\rm U}(q'',s')},\cP_{\underline{l}_{s''}}(\fp^+_{22})_L \boxtimes\cP_{\underline{k}_{q'}}(\fp^+_{11})_R\big), \ \
q''\ge s'',
 \end{cases}\!\!\!\!\!\!
\\[1ex]
\big(\cF^\downarrow_{\lambda,k,l}f\big)(x_{12},x_{21},w_{11},w_{22})
=F^\downarrow_{\lambda,k,l}\bigg(\frac{\partial}{\partial x};w_{11},w_{22}\bigg)f(x)\bigg|_{x_{11}=0,\,x_{22}=0}
\end{gather*}
intertwines the $\widetilde{{\rm U}}(q',s'')\times\widetilde{{\rm U}}(q'',s')$-action. Their operator norms are given by
\begin{gather*}
\big\Vert \cF^\uparrow_{\lambda,k,l}\big\Vert^2_{\op}=\big\Vert \cF^\downarrow_{\lambda,k,l}\big\Vert^{-2}_{\op}
=\begin{cases}
\ds \frac{(\lambda+l-q')_{\underline{k}_{q''},2}}
{(\lambda)_{\underline{k+l}_{q''},2}(\lambda-q'')_{\underline{k}_{q'},2}}
=\frac{(\lambda+k-q'')_{\underline{l}_{q'},2}}
{(\lambda)_{\underline{k+l}_{q'},2}(\lambda-q')_{\underline{l}_{q''},2}}, & q''\le s'',
\\[2ex]
\ds \frac{(\lambda+l-q')_{\underline{k}_{s''},2}}
{(\lambda)_{\underline{k+l}_{s''},2}(\lambda-s'')_{\underline{k}_{q'},2}}
=\frac{(\lambda+k-s'')_{\underline{l}_{q'},2}}
{(\lambda)_{\underline{k+l}_{q'},2}(\lambda-q')_{\underline{l}_{s''},2}}, & q''\ge s''.
\end{cases}
\end{gather*}
\end{Theorem}
Except for the results on operator norms, Theorems \ref{SBO_nonsimple1}, \ref{SBO_nonsimple2} and \ref{SBO_nonsimple3} hold
without the assumption $\lambda>\frac{d}{2}(r-1)$, if $\lambda$ is not a pole and if we replace $\cH$ with $\cO$.

We~return to the case $\fp^+$, $\fp^+_{11}$, $\fp^+_{22}$ are of tube type. If $\lambda=\frac{n}{r}-a$ with $a=1,2,\dots,\min\{k,l_{r''}\}$,
then by Corollary~\ref{cor_identity_nonsimple} we have
\begin{align*}
F_{\frac{n}{r}-a,k,\bl}^\downarrow(z)
&=\frac{\det_{\fn^-}(z)^{a}}{\big(\frac{n'}{r'}\big)_{\underline{k}_{r'},d}\big(\frac{n''}{r''}\big)_{\bl,d}
\big(\frac{n''}{r''}-a+\bl\big)_{\underline{k}_{r''},d}}
\det_{\fn^-_{22}}\left(\frac{\partial}{\partial z_{22}}\right)^k\det_{\fn^-}(z)^{k-a}\rK_\bl(z_{22})
\\
&=\frac{1}{\big(\frac{n'}{r'}\big)_{\underline{k}_{r'},d}\big(\frac{n''}{r''}\big)_{\bl,d}
\big(\frac{n''}{r''}\!+\!\bl\big)_{\underline{k-a}_{r''},d}}
\det_{\fn^-_{22}}\bigg(\frac{\partial}{\partial z_{22}}\bigg)^{k-a}\!\!\!\!\!\!\!\det_{\fn^-}(z)^k \det_{\fn^-_{22}}(z_{22})^{-a}\rK_\bl(z_{22})
\\
&=\frac{\big(\frac{n'}{r'}\big)_{\underline{k-a}_{r'},d}\big(\frac{n''}{r''}\big)_{\bl-\underline{a}_{r''},d}}
{\big(\frac{n'}{r'}\big)_{\underline{k}_{r'},d}\big(\frac{n''}{r''}\big)_{\bl,d}}
\det_{\fn^-}(z)^aF_{\frac{n}{r}+a,k-a,\bl-\underline{a}_{r''}}^\downarrow(z)
\\
&=\frac{(-1)^{ar}}{(-k)_{\underline{a}_{r'},d}(-\bl^\vee)_{\underline{a}_{r''},d}}
\det_{\fn^-}(z)^aF_{\frac{n}{r}+a,k-a,\bl-\underline{a}_{r''}}^\downarrow(z),
\end{align*}
where $\bl^\vee=(l_{r''},\dots,l_1)$.
Next we additionally assume $r'\le r''$ and $\bl=\underline{l}_{r''}$.
Then at $\lambda=\frac{d}{2}r''-k-l+a$ with $a=1,2,\dots,\min\{k,l\}$, by Proposition \ref{prop_hypergeom}(2) we have
\begin{gather*}
\lim_{\lambda\to\frac{d}{2}r''-k-l+a}
\bigg({-}\lambda-k-l+\frac{n}{r}\bigg)_{\underline{a}_{r'},d}F^\downarrow_{\lambda,k,l}(z)
\\ \qquad
{}=\lim_{\lambda\to\frac{d}{2}r''-k-l+a}
\frac{\big({-}\lambda-k-l+\frac{n}{r}\big)_{\underline{a}_{r'},d}}
{\big(\frac{n'}{r'}\big)_{\underline{k}_{r'},d}\big(\frac{n''}{r''}\big)_{\underline{l}_{r''},d}}
\det_{\fn^-_{11}}(z_{11})^k\det_{\fn^-_{22}}(z_{22})^l
\\ \qquad\hphantom{=}
{} \ \times{}_2F_1^{\fp^+_{11}}\left(\begin{matrix} -k,-l\\-\lambda-k-l+\frac{n}{r}\end{matrix};
{}^t\hspace{-1pt}z_{11}^\itinv,Q(z_{12}){}^t\hspace{-1pt}z_{22}^\itinv\right)
\\ \qquad
{}=\frac{(-k)_{\underline{a}_{r'},d}(-l)_{\underline{a}_{r'},d}}
{\big(\frac{n'}{r'}\big)_{\underline{k}_{r'},d}\big(\frac{n''}{r''}\big)_{\underline{l}_{r''},d}
\big(\frac{n'}{r'}\big)_{\underline{a}_{r'},d}}\det_{\fn^-_{11}}(z_{11})^k\det_{\fn^-_{22}}(z_{22})^l
\\ \qquad\hphantom{=}
{} \ \times\det_{\fn^+_{11}}({}^t\hspace{-1pt}z_{11}^\itinv)^a\det_{\fn^-_{11}}\left(Q(z_{12}){}^t\hspace{-1pt}z_{22}^\itinv\right)^a
{}_2F_1^{\fp^+_{11}}\left(\begin{matrix} -k+a,-l+a\\\frac{n'}{r'}+a\end{matrix};
{}^t\hspace{-1pt}z_{11}^\itinv,Q(z_{12}){}^t\hspace{-1pt}z_{22}^\itinv\right)
\\ \qquad
{}=\frac{\big(\frac{n'}{r'}+k-a\big)_{\underline{a}_{r'},d} \big(\frac{n'}{r'}+l-a\big)_{\underline{a}_{r'},d}}
{\big(\frac{n'}{r'}\big)_{\underline{k}_{r'},d}\big(\frac{n''}{r''}\big)_{\underline{l}_{r''},d}
\big(\frac{n'}{r'}\big)_{\underline{a}_{r'},d}}\det_{\fn^-_{11}}(z_{11})^{k-a}\det_{\fn^-_{22}}(z_{22})^l \\ \qquad\hphantom{=}
{} \ \times\det_{\fn^-_{11}}\big(Q(z_{12}){}^t\hspace{-1pt}z_{22}^\itinv\big)^a
{}_2F_1^{\fp^+_{11}}\left(\begin{matrix} -k+a,-l+a\\\frac{n'}{r'}+a\end{matrix};
{}^t\hspace{-1pt}z_{11}^\itinv,Q(z_{12}){}^t\hspace{-1pt}z_{22}^\itinv\right)\!.
\end{gather*}
Especially, when $\fp^+_{12}$ is also of tube type, let $\dim\fp^+_{12}=:n_{12}$, $\rank\fp^+_{12}=:r_{12}$.
Then $r'=r''$ and $\frac{d}{2}r''=\frac{n_{12}}{\varepsilon_{12}r_{12}}$ hold.
Let $\det_{\fn^-_{12}}(z_{12})$ be the polynomial on $\fp^-_{12}$ satisfying
\[
\det_{\fn^-_{11}}(Q(z_{12})x_{22})
=\det_{\fn^-_{12}}(z_{12})^{\varepsilon_{12}}\det_{\fn^+_{22}}(x_{22}), \qquad
z_{12}\in\fp^-_{12},\quad x_{22}\in\fp^+_{22}.
\]
Then we have
\begin{gather*}
\lim_{\lambda\to\frac{n_{12}}{\varepsilon_{12}r_{12}}-k-l+a}
\bigg({-}\lambda-k-l+\frac{n}{r}\bigg)_{\underline{a}_{r'},d}F^\downarrow_{\lambda,k,l}(z)
\\ \qquad
{}=\frac{\det_{\fn^-_{11}}(z_{11})^{k-a}\det_{\fn^-_{22}}(z_{22})^{l-a} \det_{\fn^-_{12}}(z_{12})^{\varepsilon_{12}a}}
{\big(\frac{n'}{r'}\big)_{\underline{k-a}_{r'},d}\big(\frac{n''}{r''}\big)_{\underline{l-a}_{r''},d}
\big(\frac{n'}{r'}\big)_{\underline{a}_{r'},d}}
{}_2F_1^{\fp^+_{11}}\left(\begin{matrix} -k\!+\!a,-l\!+\!a\\\frac{n'}{r'}+a\end{matrix};
{}^t\hspace{-1pt}z_{11}^\itinv,Q(z_{12}){}^t\hspace{-1pt}z_{22}^\itinv\!\right)
\\ \qquad
{}=\frac{1}{\big(\frac{n'}{r'}\big)_{\underline{a}_{r'},d}}
\det_{\fn^-_{12}}(z_{12})^{\varepsilon_{12}a}F^\downarrow_{\frac{n_{12}}{\varepsilon_{12}r_{12}}-k-l+a,k-a,l-a}(z).
\end{gather*}
Therefore the following holds.

\begin{Theorem}\label{thm_func_id_nonsimple}\quad
\begin{enumerate}\itemsep=0pt
\item[$1.$] Assume $\fp^+$, $\fp^+_{11}$, $\fp^+_{22}$ are of tube type. If $\lambda=\frac{n}{r}-a$, $a=1,2,\dots,\allowbreak\min\{k,l_{r''}\}$,
then we have
\[
 \cF^\downarrow_{\frac{n}{r}-a,k,\bl}
=\frac{(-1)^{ar}}{(-k)_{\underline{a}_{r'},d}(-\bl^\vee)_{\underline{a}_{r''},d}}
\cF^\downarrow_{\frac{n}{r}+a,k-a,\bl-\underline{a}_{r''}}\circ \det_{\fn^-}\bigg(\frac{\partial}{\partial x}\bigg)^a.
\]
\item[$2.$] Assume $\fp^+_{12}$ is also of tube type, and let $\bl=\underline{l}_{r''}$.
At $\lambda=\frac{n_{12}}{\varepsilon_{12}r_{12}}-k-l+a$, $a=1,2,\allowbreak\dots,\min\{k,l\}$, we have
\begin{gather*}
\lim_{\lambda\to\frac{n_{12}}{\varepsilon_{12}r_{12}}-k-l+a}
\bigg({-}\lambda-k-l+\frac{n}{r}\bigg)_{\underline{a}_{r'},d}\cF^\downarrow_{\lambda,k,l}
\\ \qquad
{}=\frac{1}{\big(\frac{n'}{r'}\big)_{\underline{a}_{r'},d}}
\det_{\fn^-_{12}}\bigg(\frac{1}{\varepsilon_{12}}\frac{\partial}{\partial x_{12}}\bigg)^{\varepsilon_{12}a}
\circ\cF^\downarrow_{\frac{n_{12}}{\varepsilon_{12}r_{12}}-k-l+a,k-a,l-a}.
\end{gather*}
\end{enumerate}
\end{Theorem}
This theorem is regarded as an analogue of the functional identities \cite[Corollaries 12.7(2) and~12.8(3)]{KS1}.
Here, for individual cases, when we take a maximal tripotent $e\in\fp^+_{11}\oplus\fp^+_{22}$ suitably,
$\det_{\fn^-}(z)$ and $\det_{\fn^-_{12}}(z_{12})$ (for $r'=r''$ case) are given by
\[
\det_{\fn^-}(z)=q(z), \qquad \det_{\fn^-_{12}}(z_{12})^{\varepsilon_{12}}=-q(z_{12})
\]
for {\it Case} 1,
\[
\det_{\fn^-}\begin{pmatrix}z_{11}&z_{12}\\{}^t\hspace{-1pt}z_{12}&z_{22}\end{pmatrix}
=\det\begin{pmatrix}z_{11}&z_{12}\\{}^t\hspace{-1pt}z_{12}&z_{22}\end{pmatrix}\!, \qquad
\det_{\fn^-_{12}}\begin{pmatrix}0&z_{12}\\{}^t\hspace{-1pt}z_{12}&0\end{pmatrix}=\det(z_{12})
\]
for {\it Case} 2,
\[
\det_{\fn^-}\begin{pmatrix}z_{11}&z_{12}\\z_{21}&z_{22}\end{pmatrix}
=\det\begin{pmatrix}z_{11}&z_{12}\\z_{21}&z_{22}\end{pmatrix}\!, \qquad
\det_{\fn^-_{12}}\begin{pmatrix}0&z_{12}\\z_{21}&0\end{pmatrix}=\det(z_{12})\det(z_{21})
\]
for {\it Case} 3,
\[
\det_{\fn^-}\begin{pmatrix}z_{11}&z_{12}\\-{}^t\hspace{-1pt}z_{12}&z_{22}\end{pmatrix}
=\Pf\begin{pmatrix}z_{11}&z_{12}\\-{}^t\hspace{-1pt}z_{12}&z_{22}\end{pmatrix}\!, \qquad
\det_{\fn^-_{12}}\begin{pmatrix}0&z_{12}\\-{}^t\hspace{-1pt}z_{12}&0\end{pmatrix}=\det(z_{12})
\]
for {\it Case} 4, and
\[
\det_{\fn^-}\begin{pmatrix}z_{11}&z_{12}\\{}^t\hspace{-1pt}\hat{z}_{12}&z_{22}\end{pmatrix}
=\det\begin{pmatrix}z_{11}&z_{12}\\{}^t\hspace{-1pt}\hat{z}_{12}&z_{22}\end{pmatrix}
\]
for {\it Case} 5. By \cite[Theorem 6.4]{A}, if $\fp^+$ is of tube type, then for $a\in\BZ_{>0}$,
\begin{equation}\label{G-intertwine}
\det_{\fn^-}\bigg(\frac{\partial}{\partial x}\bigg)^a\colon\ \cO_{\frac{n}{r}-a}(D)\longrightarrow\cO_{\frac{n}{r}+a}(D)
\end{equation}
intertwines the $\widetilde{G}$-action (see also \cite[Propositions~1.2 and 2.2]{J0}, \cite[Lemma 7.1, equations~(8.6)--(8.8)]{Sh}), and similar for $\fp^+_{12}$. Especially, the both sides of Theorem \ref{thm_func_id_nonsimple}(1), (2) intertwine the $\widetilde{G}_1$-action.

\subsection[Case $\fp^+_2$ is simple]{Case $\boldsymbol{\fp^+_2}$ is simple}

Next we consider the cases such that $\fp^\pm_2$ is simple and $\fp^\pm$, $\fp^\pm_2$ are of tube type.
That is, we consider
\begin{equation}\label{list_simple}
(\fp^\pm,\fp^\pm_1,\fp^\pm_2)= \begin{cases}
(\BC^n,\BC^{n'},\BC^{n''}), \quad n=n'+n'',\quad n''\ge 3 & (\textit{Case }1), \\
(\BC^n,\BC^{n-1},\BC), \quad n\ge 3 & (\textit{Case }1'), \\
({\rm Sym}(2r',\BC),\Sym(r',\BC)\oplus\Sym(r',\BC),M(r',\BC)) & (\textit{Case }2), \\
({\rm Skew}(2s',\BC),\Skew(s',\BC)\oplus\Skew(s',\BC),M(s',\BC)) & (\textit{Case }3), \\
(M(r,\BC),\Skew(r,\BC),\Sym(r,\BC)) & (\textit{Case }4), \\
(M(2s,\BC),\Sym(2s,\BC),\Skew(2s,\BC)) & (\textit{Case }5), \\
\big({\rm Herm}(3,\BO)^\BC,M(2,6;\BC),\Skew(6,\BC)\big) & (\textit{Case }6). \end{cases}
\end{equation}
Then the corresponding symmetric pairs are
\[ (G,G_1)= \begin{cases}
({\rm SO}_0(2,n),{\rm SO}_0(2,n')\times {\rm SO}(n'')) & (\textit{Case }1), \\
({\rm SO}_0(2,n),{\rm SO}_0(2,n-1)) & (\textit{Case }1'), \\
(\operatorname{Sp}(2r',\BR),\operatorname{Sp}(r',\BR)\times \operatorname{Sp}(r',\BR)) & (\textit{Case }2), \\
({\rm SO}^*(4s'),{\rm SO}^*(2s')\times {\rm SO}^*(2s')) & (\textit{Case }3), \\
({\rm SU}(r,r),{\rm SO}^*(2r)) & (\textit{Case }4), \\
({\rm SU}(2s,2s),\operatorname{Sp}(2s,\BR)) & (\textit{Case }5), \\
(E_{7(-25)},{\rm SU}(2,6)) & (\textit{Case }6) \end{cases} \]
(up to covering). Let $\dim\fp^+=:n$, $\dim\fp^+_2=:n_2$, $\rank\fp^+=:r$, $\rank\fp^+_2=:r_2$,
let $d$, $d_1$, $d_2$ be the numbers defined in (\ref{str_const}) for $\fp^+$, $\fp^+_1$, $\fp^+_2$ respectively, and take $\varepsilon_1,\varepsilon_2\in\{1,2\}$ such that
\[
(x_j|\overline{y_j})_{\fp^+}=\varepsilon_j(x_j|\overline{y_j})_{\fp^+_j}, \qquad
x_j,y_j\in\fp^+_j,\qquad j=1,2.
\]
Also let $\delta_1$, $\delta_2$ be the numbers defined by
\[
\delta_j:=\begin{cases} 0 & \big(\fp^+_j\colon\text{not simple}\big), \\ d_j & \big(\fp^+_j\colon\text{simple},\; \varepsilon_j=1\big), \\
-1 & \big(\fp^+_j\colon\text{simple},\; \varepsilon_j=2\big), \end{cases} \qquad j=1,2.
\]
Then we have
\[
(r,r_2,d,d_2,\varepsilon_1,\varepsilon_2,\delta_1,\delta_2)= \begin{cases}
(2,2,n-2,n''-2,1,1,n'-2,n''-2) & (\textit{Case }1,\; n'\ge 2), \\
(2,2,n-2,n''-2,2,1,-1,n''-2) & (\textit{Case }1,\; n'=1), \\
(2,1,n-2,-,1,2,n-3,-1) & (\textit{Case }1'), \\
(2r',r',1,2,1,2,0,-1) & (\textit{Case }2), \\
(s',s',4,2,1,1,0,2) & (\textit{Case }3), \\
(r,r,2,1,2,1,-1,1) & (\textit{Case }4), \\
(2s,s,2,4,1,2,1,-1) & (\textit{Case }5), \\
(3,3,8,4,1,1,2,4) & (\textit{Case }6). \end{cases}
\]
We~note that if $r_2=1$, then $d_2$ is not determined uniquely, and any number is allowed.
Let $\Sigma_2\subset\fp^+_2$ be the Bergman--Shilov boundary of the bounded symmetric domain, with the normalized $K_1$-invariant measure,
so that $\Sigma_2\simeq K_1/K_{L_2}$ holds for some $K_{L_2}\subset K_1$. For these cases,
for $(\rho_0,W)\simeq \cP_{(k+l,\underline{k}_{r_2-1})}\big(\fp^+_2\big)=\cP_{(l,0,\dots,0)}(\fp^+_2)\det_{\fn^+_2}(x_2)^k$ ($\varepsilon_2=1$) or
for $(\rho_0,W)\simeq \cP_{(\underline{k+1}_l,\underline{k}_{r_2-l})}\big(\fp^+_2\big)=\cP_{\underline{1}_l}\big(\fp^+_2\big)\det_{\fn^+_2}(x_2)^k$
($\varepsilon_2=2$),
$F^\downarrow_{\lambda,\rho_0}(z)=F^\downarrow_{\lambda,k,l}(z)$ is computable by using Theorem~\ref{thm_simple}.

\begin{Theorem}\label{SBO_simple1}
Let $\big(\fp^\pm,\fp^\pm_1,\fp^\pm_2\big)$ be as in $(\ref{list_simple})$.
For $l\in\BZ_{\ge 0}$ $(\varepsilon_2=1)$ or $l\in\{0,1,\dots,r_2-1\}$ $(\varepsilon_2=2)$,
let $V_l$ be the representation of $K_1$ isomorphic to $\cP_{(l,0,\dots,0)}\big(\fp^+_2\big)$ $(\varepsilon_2=1)$
or $\cP_{\underline{1}_l}(\fp^+_2)$ $(\varepsilon_2=2)$, and
we take a vector-valued polynomial $\rK_l(x_2)\in\bigl(\cP(\fp^-_2)\otimes V_l\bigr)^{K_1}$ such that
\[
\big\Vert (v,\rK_l(\overline{x_2}))_{V_l}\big\Vert_{L^2(\Sigma_2),x_2}^2=|v|_{V_l}^2, \qquad v\in V_l.
\]
For $\lambda>\frac{d}{2}(r-1)$, $k\in\BZ_{\ge 0}$,
we define $F^\downarrow_{\lambda,k,l}(z)=F^\downarrow_{\lambda,k,l}(z_1,z_2)\in\cP(\fp^-,V_l)$ by
\[
F^\downarrow_{\lambda,k,l}(z):= \begin{cases}
\ds \frac{1}{\big(\frac{n_2}{r_2}\big)_{(k+l,\underline{k}_{r_2-1}),d_2}}
F_1^l\left(\begin{matrix} -k \\ -\lambda-2k+\frac{n}{r} \end{matrix};\rK_l;z\right)\!, & \varepsilon_2=1, \\[18pt]
\ds \frac{2^{kr_2+l}}{\big(\frac{n_2}{r_2}\big)_{(\underline{k+1}_l,\underline{k}_{r_2-l}),d_2}}
F_2^l\left(\begin{matrix} -\frac{k}{2} \\ -\lambda-k+\frac{n}{r} \end{matrix};\rK_l;z\right)\!, & \varepsilon_2=2, \end{cases}
\]
by using $F_{\varepsilon_2}^l\left(\begin{smallmatrix}\nu\\\mu\end{smallmatrix};f;z\right)$ in \eqref{Fmunu1}, \eqref{Fmunu2}.
Then the linear map
\begin{gather*}
\cF^\downarrow_{\lambda,k,l}\colon \left\{ \hspace{-4pt} \begin{array}{lll}
\cH_\lambda(D_{{\rm SO}_0(2,n)})\hspace{-7pt} &\longrightarrow \cH_{\lambda+2k+l}\big(D_{{\rm SO}_0(2,n')}\big)\boxtimes V_{(l,0,\dots,0)}^{[n'']\vee}
& (\textit{Case }1), \\
\cH_\lambda(D_{{\rm SO}_0(2,n)})\hspace{-7pt} &\longrightarrow \cH_{\lambda+k}\big(D_{{\rm SO}_0(2,n-1)}\big) & (\textit{Case }1'), \\
\cH_\lambda(D_{\operatorname{Sp}(2r',\BR)})\hspace{-7pt} &\longrightarrow \cH_{\lambda+k}\big(D_{\operatorname{Sp}(r',\BR)},V_{\underline{1}_l}^{(r')\vee}\big) \\
&\hspace{29pt}{}\hboxtimes\cH_{\lambda+k}\big(D_{\operatorname{Sp}(r',\BR)},V_{\underline{1}_l}^{(r')\vee}\big) & (\textit{Case }2), \\
\cH_\lambda(D_{{\rm SO}^*(4s')})\hspace{-7pt} &\longrightarrow \cH_{\lambda+2k}\big(D_{{\rm SO}^*(2s')},V_{(l,0,\dots,0)}^{(s')\vee}\big) \\
&\hspace{29pt}{}\hboxtimes\cH_{\lambda+2k}\big(D_{{\rm SO}^*(2s')},V_{(l,0,\dots,0)}^{(s')\vee}\big) & (\textit{Case }3), \\
\cH_\lambda(D_{{\rm SU}(r,r)})\hspace{-7pt} &\longrightarrow \cH_{2\lambda+4k}\big(D_{{\rm SO}^*(2r)},V_{(2l,0,\dots,0)}^{(r)\vee}\big) & (\textit{Case }4), \\
\cH_\lambda(D_{{\rm SU}(2s,2s)})\hspace{-7pt} &\longrightarrow \cH_{\lambda+k}\big(D_{\operatorname{Sp}(2s,\BR)},V_{\underline{1}_{2l}}^{(2s)\vee}\big) & (\textit{Case }5), \\
\cH_\lambda(D_{E_{7(-25)}})\hspace{-7pt} &\longrightarrow \cH_{\lambda+2k}\big(D_{{\rm SU}(2,6)},\BC\boxtimes V_{(l,l,l,l,0,0)}^{(6)}\big)
& (\textit{Case }6), \end{array} \right. \\
\big(\cF^\downarrow_{\lambda,k,l}f\big)(x_1)
=\left.F^\downarrow_{\lambda,k,l}\left(\frac{\partial}{\partial x}\right)f(x)\right|_{x_2=0}
\end{gather*}
intertwines the $\widetilde{G}_1$-action, and its operator norm is given by
\begin{gather*}
\Vert \cF^\downarrow_{\lambda,k,l}\Vert^{-2}_{\op}
\\ \qquad
{}=\begin{cases}
\ds \frac{\big(\frac{n_2}{r_2}\big)_{(k+l,\underline{k}_{r_2-1}),d_2}
\big(\lambda+k-\frac{d}{4}r_2+\frac{d_2}{2}+(l,\underline{0}_{r_2/2-1})\big)_{\underline{k}_{r_2/2},d}}
{(\lambda)_{(2k+l,\underline{2k}_{r_2/2-1},\underline{k}_{r_2/2}),d}},
& \hspace{-45pt}\varepsilon_2=1,\;r_2\colon\text{even},
\\
\ds \frac{\big(\frac{n_2}{r_2}\big)_{(k+l,\underline{k}_{r_2-1}),d_2}
\big(\lambda\!-\!\frac{d}{2}\big\lfloor\frac{r_2}{2}\big\rfloor\!+\!\max\{2k,k+l\}\big)_{\min\{k,l\}}
\big(\lambda\!+\!k\!-\!\frac{d}{2}\big\lceil\frac{r_2}{2} \big\rceil\!+\!\frac{d_2}{2}\big)_{\underline{k}_{\lfloor r_2/2\rfloor},d}}
{(\lambda)_{(2k+l,\underline{2k}_{\lfloor r_2/2\rfloor-1},\min\{2k,k+l\},\underline{k}_{\lfloor r_2/2\rfloor}),d}},
 \hspace{-55pt}& \\ & \hspace{-40pt}\varepsilon_2=1,\;r_2\colon\text{odd},
 \\
\ds \frac{\big(\frac{n_2}{r_2}\big)_{(\underline{k+1}_l,\underline{k}_{r_2-l}),d_2}
\big(\lambda\!-\!\frac{d}{2}r_2\!+\!\big(\underline{\big\lfloor\frac{k}{2}\big\rfloor\!+\!\frac{1}{2}}_l,
\underline{\big\lceil\frac{k}{2}\big\rceil\!-\!\frac{1}{2}}_{r_2-l}\big)\big)
_{(\underline{\lceil k/2\rceil}_l,\underline{\lfloor k/2\rfloor}_{r_2-l}),d}}
{2^{kr_2+l}(\lambda)_{(\underline{k+1}_l,\underline{k}_{r_2-l},\underline{\lceil k/2\rceil}_l,\underline{\lfloor k/2\rfloor}_{r_2-l}),d}},
&\varepsilon_2=2. \end{cases}
\end{gather*}
\end{Theorem}
For holographic operators $\cF^\uparrow_{\lambda,k,l}$, see \cite[Theorems 5.7, 5.12, 5.17 and~5.19]{N2}.
Especially if~$l=0$, then as in Corollary~\ref{cor_scalar_simple}, the intertwining operators become easier.
Here ${}_0F_1^{\fp^-_1,\fp^+_2}$ and~${}_2F_1^{\fp^-_1,\fp^+_2}$ are as in (\ref{0F1_mid}) and (\ref{2F1_mid}).

\begin{Theorem}\label{SBO_simple2}
Let $\big(\fp^\pm,\fp^\pm_1,\fp^\pm_2\big)$ be as in \eqref{list_simple}. For $\lambda>\frac{d}{2}(r-1)$, $k\in\BZ_{\ge 0}$,
we define $F^\uparrow_{\lambda,k,0}(x_2;z_1)\in\cO\big(\fp^+_2\times\fp^-_1\big)$ by
\[
F^\uparrow_{\lambda,k,0}(x_2;z_1):=\det_{\fn^+_2}(x_2)^k
{}_0F_1^{\fp^-_1,\fp^+_2}\bigg(\lambda+\frac{2k}{\varepsilon_2}-\frac{\delta_1}{2};z_1,x_2\bigg),
\]
and $F^\downarrow_{\lambda,k,0}(z)=F^\downarrow_{\lambda,k,0}(z_1,z_2)\in\cP(\fp^-)$ by
\[ F^\downarrow_{\lambda,k,0}(z):= \frac{\det_{\fn^-_2}(z_2)^k}{\big(\frac{n_2}{r_2}\big)_{\underline{k}_{r_2},d_2}} \times
\begin{cases}
\ds {}_2F_1^{\fp^-_1,\fp^+_2}\left(\begin{matrix} -k,-k-\frac{d_2}{2}\\-\lambda-2k+\frac{n}{r}-\frac{d_2}{2}\end{matrix};
z_1,{}^t\hspace{-1pt}z_2^\itinv\right)\!, & \varepsilon_2=1,
\\
\ds 2^{kr_2}{}_2F_1^{\fp^-_1,\fp^+_2}\left(\begin{matrix} -\frac{k}{2},-\frac{k-1}{2}\\-\lambda-k+\frac{n}{r}+\frac{1}{2}\end{matrix};
z_1,{}^t\hspace{-1pt}z_2^\itinv\right)\!, & \varepsilon_2=2. \end{cases}
\]
Then the linear map
\begin{gather*}
\cF^\uparrow_{\lambda,k,0}\colon\
\left\{ \hspace{-4pt} \begin{array}{lll}
\bigl(\cH_{\lambda+2k}\big(D_{{\rm SO}_0(2,n')}\big)\boxtimes \BC\bigr)_{\widetilde{K}_1}
&\longrightarrow \cH_\lambda\big(D_{{\rm SO}_0(2,n)}\big)_{\widetilde{K}} & (\text{Case }1), \\
\cH_{\lambda+k}\big(D_{{\rm SO}_0(2,n-1)}\big)_{\widetilde{K}_1} &\longrightarrow \cH_\lambda\big(D_{{\rm SO}_0(2,n)}\big)_{\widetilde{K}} & (\text{Case }1'), \\
\bigl(\cH_{\lambda+k}\big(D_{{\rm Sp}(r',\BR)}\big)\hboxtimes\cH_{\lambda+k}\big(D_{{\rm Sp}(r',\BR)}\big)\bigr)_{\widetilde{K}_1}\hspace{-7pt}
&\longrightarrow \cH_\lambda\big(D_{{\rm Sp}(2r',\BR)}\big)_{\widetilde{K}} & (\text{Case }2), \\
\bigl(\cH_{\lambda+2k}\big(D_{{\rm SO}^*(2s')}\big)\hboxtimes\cH_{\lambda+2k}\big(D_{{\rm SO}^*(2s')}\big)\bigr)_{\widetilde{K}_1}\hspace{-7pt}
&\longrightarrow \cH_\lambda\big(D_{{\rm SO}^*(4s')}\big)_{\widetilde{K}} & (\text{Case }3), \\
\cH_{2\lambda+4k}\big(D_{{\rm SO}^*(2r)}\big)_{\widetilde{K}_1} &\longrightarrow \cH_\lambda\big(D_{{\rm SU}(r,r)}\big)_{\widetilde{K}} & (\text{Case }4), \\
\cH_{\lambda+k}\big(D_{{\rm Sp}(2s,\BR)}\big)_{\widetilde{K}_1} &\longrightarrow \cH_\lambda\big(D_{{\rm SU}(2s,2s)}\big)_{\widetilde{K}} & (\text{Case }5), \\
\cH_{\lambda+2k}\big(D_{{\rm SU}(2,6)}\big)_{\widetilde{K}_1} &\longrightarrow \cH_\lambda\big(D_{E_{7(-25)}}\big)_{\widetilde{K}} & (\text{Case }6),
\end{array} \right.
\\
\big(\cF^\uparrow_{\lambda,k,0}f\big)(x)
=F^\uparrow_{\lambda,k,0}\bigg(x_2;\frac{1}{\varepsilon_1}\frac{\partial}{\partial x_1}\bigg)f(x_1)
\end{gather*}
intertwines the $\big(\fg_1,\widetilde{K}_1\big)$-action, and the linear map
\begin{gather*}
\cF^\downarrow_{\lambda,k,0}\colon\
\left\{ \hspace{-4pt} \begin{array}{lll}
\cH_\lambda\big(D_{{\rm SO}_0(2,n)}\big)\hspace{-7pt} &\longrightarrow \cH_{\lambda+2k}\big(D_{{\rm SO}_0(2,n')}\big)\boxtimes \BC & (\text{Case }1), \\
\cH_\lambda\big(D_{{\rm SO}_0(2,n)}\big)\hspace{-7pt} &\longrightarrow \cH_{\lambda+k}\big(D_{{\rm SO}_0(2,n-1)}\big) & (\text{Case }1'), \\
\cH_\lambda\big(D_{{\rm Sp}(2r',\BR)}\big)\hspace{-7pt} &\longrightarrow \cH_{\lambda+k}\big(D_{{\rm Sp}(r',\BR)}\big)\hboxtimes\cH_{\lambda+k}\big(D_{{\rm Sp}(r',\BR)}\big)
& (\text{Case }2), \\
\cH_\lambda\big(D_{{\rm SO}^*(4s')}\big)\hspace{-7pt} &\longrightarrow \cH_{\lambda+2k}\big(D_{{\rm SO}^*(2s')}\big)\hboxtimes\cH_{\lambda+2k}\big(D_{{\rm SO}^*(2s')}\big)
& (\text{Case }3), \\
\cH_\lambda\big(D_{{\rm SU}(r,r)}\big)\hspace{-7pt} &\longrightarrow \cH_{2\lambda+4k}\big(D_{{\rm SO}^*(2r)}\big) & (\text{Case }4), \\
\cH_\lambda\big(D_{{\rm SU}(2s,2s)}\big)\hspace{-7pt} &\longrightarrow \cH_{\lambda+k}\big(D_{{\rm Sp}(2s,\BR)}\big) & (\text{Case }5), \\
\cH_\lambda\big(D_{E_{7(-25)}}\big)\hspace{-7pt} &\longrightarrow \cH_{\lambda+2k}\big(D_{{\rm SU}(2,6)}\big) & (\text{Case }6), \end{array} \right.
\\
\big(\cF^\downarrow_{\lambda,k,0}f\big)(x_1)
=\left.F^\downarrow_{\lambda,k,0}\left(\frac{\partial}{\partial x}\right)f(x)\right|_{x_2=0}
\end{gather*}
intertwines the $\widetilde{G}_1$-action. Their operator norms are given by
\[
\big\Vert \cF^\uparrow_{\lambda,k,0}\big\Vert^2_{\op}=\big\Vert \cF^\downarrow_{\lambda,k,0}\big\Vert^{-2}_{\op}
=\begin{cases}
\ds \frac{\big(\frac{n_2}{r_2}\big)_{\underline{k}_{r_2},d_2}
\big(\lambda+k-\frac{d}{2}\big\lceil\frac{r_2}{2}\big\rceil+\frac{d_2}{2}\big)_{\underline{k}_{\lfloor r_2/2\rfloor},d}}
{(\lambda)_{\underline{2k}_{\lfloor r_2/2\rfloor},d}
\big(\lambda-\frac{d}{2}\big\lfloor\frac{r_2}{2}\big\rfloor\big)_{\underline{k}_{\lceil r_2/2\rceil},d}}, & \varepsilon_2=1,
\\
\ds \frac{\big(\frac{n_2}{r_2}\big)_{\underline{k}_{r_2},d_2}
\big(\lambda+\big\lceil\frac{k}{2}\big\rceil-\frac{d}{2}r_2-\frac{1}{2}\big)_{\underline{\lfloor k/2\rfloor}_{r_2},d}}
{2^{kr_2}(\lambda)_{\underline{k}_{r_2},d}\big(\lambda-\frac{d}{2}r_2\big)_{\underline{\lfloor k/2\rfloor}_{r_2},d}}, & \varepsilon_2=2.
\end{cases}
\]
\end{Theorem}
Here we normalize $\frac{\partial}{\partial x}$ and $\frac{\partial}{\partial x_1}$ with respect to the bilinear forms
$(\cdot|\cdot)_{\fp^+}\colon\fp^+\times\fp^-\to\BC$ and $(\cdot|\cdot)_{\fp^+_1}\colon\fp^+_1\times\fp^-_1\to\BC$ respectively,
so that $\frac{\partial}{\partial x}\bigr|_{\fp^+_1}=\frac{1}{\varepsilon_1}\frac{\partial}{\partial x_1}$ holds.

For {\it Cases} 2, 3, the above results are generalized to the cases $\fp^+_2=M(r',r'';\BC)$, $M(s',s'';\BC)$ are of non-tube type,
as in Theorem \ref{explicit_simple}(2), (3). Here ${}_0F_1^{(1)}$, ${}_0F_1^{[4]}$, ${}_2F_1^{(1)}$ and ${}_2F_1^{[4]}$ are as in
(\ref{0F1^(d)})--(\ref{2F1^[d]}), and $\cP_{\underline{k}_{r'}}(M(r',r'';\BC))_R$ is the representation of ${\rm GL}(r'',\BC)$ given in (\ref{PMR})
with the Fischer inner product.
\begin{Theorem}\label{SBO_simple3}
Let
\[ \big(\fp^\pm,\fp^\pm_{11},\fp^\pm_{12},\fp^\pm_{22}\big)= \begin{cases}
({\rm Sym}(r,\BC),\Sym(r',\BC),M(r',r'';\BC),\Sym(r'',\BC)) & (\textit{Case }2), \\
({\rm Skew}(s,\BC),\Skew(s',\BC),M(s',s'';\BC),\Skew(s'',\BC)) & (\textit{Case }3), \end{cases} \]
with $r=r'+r''$, $r'\le r''$ or $s=s'+s''$, $s'\le s''$ respectively.
For $\lambda>\frac{r-1}{2}$ or $\lambda>2\big(\big\lfloor\frac{s}{2}\big\rfloor-1\big)$ respectively and for $k\in\BZ_{\ge 0}$, we define
$F^\uparrow_{\lambda,k,0}(x_{12};z_{11},z_{22},y_{12})\in\cP\big(\fp^+_{12}\times\fp^-_{11} \times\fp^-_{22}\times\fp^-_{12}\big)$ by
\[
F^\uparrow_{\lambda,k,0}(x_{12};z_{11},z_{22},y_{12}):=
\begin{cases}
\ds \frac{\det(x_{12}{}^t\hspace{-1pt}y_{12})^k}{(r')_{\underline{k}_{r'},2}}
{}_0F_1^{(1)}\big(\lambda+k;z_{11}x_{12}z_{22}{}^t\hspace{-1pt}x_{12}\big) & (\textit{Case }2),
\\[2ex]
\ds \frac{\det(x_{12}{}^t\hspace{-1pt}y_{12})^k}{(s')_{\underline{k}_{s'},2}}
{}_0F_1^{[4]}\big(\lambda+2k;-z_{11}x_{12}z_{22}{}^t\hspace{-1pt}x_{12}\big) & (\textit{Case }3), \end{cases}
\]
and $F^\downarrow_{\lambda,k,0}(z;w_{12})=F^\downarrow_{\lambda,k,0}
\left(\left(\begin{smallmatrix}z_{11}&z_{12}\\ \pm{}^t\hspace{-1pt}z_{12}&z_{22}\end{smallmatrix}\right);w_{12}\right)\in\cP\big(\fp^-\times\fp^+_{12}\big)$ by
\begin{gather*}
F^\downarrow_{\lambda,k,0}(z;w_{12})
\\ \qquad
:=\begin{cases} \ds \frac{2^{kr'}\det(z_{12}{}^t\hspace{-1pt}w_{12})^k}{(r')_{\underline{k}_{r'},2}}
{}_2F_1^{(1)}\!\left( \begin{matrix}-\frac{k}{2},-\frac{k-1}{2}\\ -\lambda\!-\!k\!+\!r'\!+\!1 \end{matrix};
z_{11}(w_{12}{}^t\hspace{-1pt}z_{12})^{-1}w_{12}z_{22}{}^t\hspace{-1pt}w_{12}(z_{12}{}^t \hspace{-1pt}w_{12})^{-1}\right)
\\
\hspace{330pt}(\textit{Case }2),
\\
\ds \frac{\det(z_{12}{}^t\hspace{-1pt}w_{12})^k}{(s')_{\underline{k}_{s'},2}}
{}_2F_1^{[4]}\!\left( \begin{matrix} -k,-k-1\\ -\lambda\!-\!2k\!+\!2s'\!-\!2 \end{matrix};
-z_{11}(w_{12}{}^t\hspace{-1pt}z_{12})^{-1}w_{12}z_{22}{}^t\hspace{-1pt}w_{12}(z_{12}{}^t \hspace{-1pt}w_{12})^{-1}\right)
\\
\hspace{330pt}(\textit{Case }3). \end{cases}\!\!\!
\end{gather*}
Then the linear maps
\begin{gather*}
\cF^\uparrow_{\lambda,k,0}\colon \begin{cases}
\bigl(\cH_{\lambda+k}(D_{\operatorname{Sp}(r',\BR)})\hboxtimes\cH_\lambda(D_{\operatorname{Sp}(r'',\BR)},\cP_{\underline{k}_{r'}}(M(r',r'';\BC))_R)\bigr)_{\widetilde{K}_1}
\longrightarrow \cH_\lambda(D_{\operatorname{Sp}(r,\BR)})_{\widetilde{K}} \\
\hspace{337pt} (\textit{Case }2),\hspace{-12pt} \\
\bigl(\cH_{\lambda+2k}(D_{{\rm SO}^*(2s')})\hboxtimes\cH_\lambda(D_{{\rm SO}^*(2s'')},\cP_{\underline{k}_{s'}}(M(s',s'';\BC))_R)\bigr)_{\widetilde{K}_1}
\longrightarrow \cH_\lambda(D_{{\rm SO}^*(2s)})_{\widetilde{K}}\hspace{-10pt} \\
\hspace{337pt} (\textit{Case }3),\hspace{-12pt} \end{cases} \\
\big(\cF^\uparrow_{\lambda,k,0}f\big)(x)=F^\uparrow_{\lambda,k,0}\bigg(x_{12};\frac{\partial}{\partial x_{11}},\frac{\partial}{\partial x_{22}},
\frac{\partial}{\partial w_{12}}\bigg)f(x_{11},x_{22},w_{12})\bigg|_{w_{12}=0}
\end{gather*}
intertwine the $(\mathfrak{sp}(r',\BR)\oplus \mathfrak{sp}(r'',\BR),\widetilde{{\rm U}}(r')\times\widetilde{{\rm U}}(r''))$- or
$(\mathfrak{so}^*(2s')\oplus\mathfrak{so}^*(2s''),\widetilde{{\rm U}}(s')\times\widetilde{{\rm U}}(s''))$-action respectively,
and the linear maps
\begin{gather*}
\cF^\downarrow_{\lambda,k,0}\colon\ \begin{cases}
\cH_\lambda(D_{\operatorname{Sp}(r,\BR)})
\longrightarrow \cH_{\lambda+k}(D_{\operatorname{Sp}(r',\BR)})\hboxtimes
\cH_\lambda(D_{\operatorname{Sp}(r'',\BR)},\cP_{\underline{k}_{r'}}(M(r',r'';\BC))_R) \\
\hspace{310pt}(\textit{Case }2),
\\
\cH_\lambda(D_{{\rm SO}^*(2s)})
\longrightarrow \cH_{\lambda+2k}(D_{{\rm SO}^*(2s')})\hboxtimes\cH_\lambda(D_{{\rm SO}^*(2s'')},\cP_{\underline{k}_{s'}}(M(s',s'';\BC))_R) \\
\hspace{310pt}(\textit{Case }3), \end{cases}
\\
\big(\cF^\downarrow_{\lambda,k,0}f\big)(x_{11},x_{22},w_{12})
=F^\downarrow_{\lambda,k,0}\bigg(\frac{\partial}{\partial x};w_{12}\bigg)f(x)
\bigg|_{x_{12}=0}
\end{gather*}
intertwine the $\widetilde{\operatorname{Sp}}(r',\BR)\times \widetilde{\operatorname{Sp}}(r'',\BR)$- or
$\widetilde{{\rm SO}^*}(2s')\times\widetilde{{\rm SO}^*}(2s'')$-action respectively. Their operator norms are given by
\[
\big\Vert \cF^\uparrow_{\lambda,k,0}\big\Vert^2_{\op}=\big\Vert \cF^\downarrow_{\lambda,k,0}\big\Vert^{-2}_{\op}= \begin{cases}
\ds \frac{\big(\lambda+\big\lceil\frac{k}{2}\big\rceil-\frac{r'+1}{2}\big)_{\underline{\lfloor k/2\rfloor}_{r'},1}}
{2^{kr'}(\lambda)_{\underline{k}_{r'},1}\big(\lambda-\frac{r'}{2}\big)_{\underline{\lfloor k/2\rfloor}_{r'},1}} & (\textit{Case }2),
\\[2ex]
\ds \frac{\big(\lambda+k-2\big\lceil\frac{s'}{2}\big\rceil+1\big)_{\underline{k}_{\lfloor s'/2\rfloor},4}}
{(\lambda)_{\underline{2k}_{\lfloor s'/2\rfloor},4}
\big(\lambda-2\big\lfloor \frac{s'}{2}\big\rfloor\big)_{\underline{k}_{\lceil s'/2\rceil},4}} & (\textit{Case }3).
\end{cases}
\]
\end{Theorem}
Except for the results on operator norms, Theorems \ref{SBO_simple1}, \ref{SBO_simple2} and \ref{SBO_simple3} hold
without the assumption $\lambda>\frac{d}{2}(r-1)$,
if $\lambda$ is not a pole and if we replace $\cH$ with $\cO$.
\begin{Remark}\label{rem_previous}
For {\it Case} $1'$, since $F^\downarrow_{\lambda,k,0}(z)$ is written as
\begin{align*}
F^\downarrow_{\lambda,k,0}(z)&=\frac{2^k}{k!}
(z_2)^k{}_2F_1\left( \begin{matrix}-\frac{k}{2},-\frac{k-1}{2}\\
-\lambda-k+\frac{n}{2}+\frac{1}{2} \end{matrix};-\frac{q(z_1)}{(z_2)^2}\right)
\\
&=\frac{2^k}{k!}(z_2)^k\sum_{m=0}^{\lfloor k/2\rfloor}\frac{(-k)_{2m}}{\left(-\lambda-k+\frac{n-1}{2}+1\right)_mm!}
\left(-\frac{q(z_1)}{(2z_2)^2}\right)^m
\\
&=\frac{(-q(z_1))^{k/2}}{\left(\lambda-\frac{n-1}{2}\right)_k}
\sum_{m=0}^{\lfloor k/2\rfloor}\frac{(-1)^m\left(\lambda-\frac{n-1}{2}\right)_{k-m}}
{(k-2m)!m!}\bigg(\frac{2z_2}{\sqrt{-q(z_1)}}\bigg)^{k-2m}
\\
&=\frac{(-q(z_1))^{k/2}}{\left(\lambda-\frac{n-1}{2}\right)_k}
C_k^{\lambda-\frac{n-1}{2}}\bigg(\frac{2z_2}{\sqrt{-q(z_1)}}\bigg)
\end{align*}
by using the Gegenbauer polynomial $C_k^\alpha(t)$, the symmetry breaking operator $\cF^\downarrow_{\lambda,k,0}$ coincides with
the holomorphic version of Juhl's operator (see \cite{Ju}, \cite[Theorem 6.3]{KP2}) up to constant multiple.
Similarly, for {\it Case} $2$ with $r'=1$, since $F^\downarrow_{\lambda,k,0}(z;w_{12})$ is again written by using the Gegenbauer polynomial
$C^{\lambda-1}_k(t)$, the symmetry breaking operator $\cF^\downarrow_{\lambda,k,0}$ coincides with the result by Kobayashi--Pevzner
\cite[Theorem 7.1]{KP2} up to constant multiple. Moreover, for {\it Case} $2$ with $r'=r''$, combining with Proposition \ref{prop_hypergeom}(1),
$F^\downarrow_{\lambda,k,0}(z)$ coincides with the result by Ibukiyama--Kuzumaki--Ochiai \cite[Theorem~7.5]{IKO} up to constant multiple,
and $\cF^\downarrow_{\lambda,k,0}$ coincides with the differential operator in \cite{IKO} which is related to the Siegel modular forms.
In addition, similar functions for {\it Cases} $1$ and $1'$ also appear in the restriction of complementary series representations to some subgroups
as inclusion maps from discrete summands, since these maps are given by the Fourier transforms of the symmetry breaking operators
(Kobayashi--Speh \cite[Chapter 15]{KS1}, M\"ollers--Oshima \cite{MO}).
\end{Remark}

We~return to the case $\fp^+$, $\fp^+_2$ are of tube type.
When $\varepsilon_2=1$, if $\lambda=\frac{n}{r}-a$ with $a=1,2,\dots,k$,
then by (\ref{identity_simple}) we have
\begin{align*}
\bigg(\frac{n_2}{r_2}\bigg)_{(k+l,\underline{k}_{r_2-1}),d_2}F^\downarrow_{\frac{n}{r}-a,k,l}(z)
&=F_1^l\left(\begin{matrix} -k \\ a-2k \end{matrix};\rK_l;z\right)
=\det_{\fn^-}(z)^aF_1^l\left(\begin{matrix} a-k \\ a-2k \end{matrix};\rK_l;z\right) \\
&=\bigg(\frac{n_2}{r_2}\bigg)_{(k-a+l,\underline{k-a}_{r_2-1}),d_2} \det_{\fn^-}(z)^aF^\downarrow_{\frac{n}{r}+a,k-a,l}(z),
\end{align*}
and when $\varepsilon_2=2$, if $\lambda=\frac{n}{r}-a$ with $a=1,2,\dots,\left\lfloor\frac{k}{2}\right\rfloor$, then we have
\begin{gather*}
\frac{1}{2^{kr_2+l}}\bigg(\frac{n_2}{r_2}\bigg)_{(\underline{k+1}_l,\underline{k}_{r_2-l}),d_2} F^\downarrow_{\frac{n}{r}-a,k,l}(z)
\\
\qquad{} =F_2^l\left(\begin{matrix} -\frac{k}{2} \\ a-k \end{matrix};\rK_l;z\right)
=\det_{\fn^-}(z)^aF_2^l\left(\begin{matrix} a-\frac{k}{2} \\ a-k \end{matrix};\rK_l;z\right)
\\
\qquad =\frac{1}{2^{(k-2a)r_2+l}}\bigg(\frac{n_2}{r_2}\bigg)_{(\underline{k-2a+1}_l,\underline{k-2a}_{r_2-l}),d_2}
\det_{\fn^-}(z)^aF^\downarrow_{\frac{n}{r}+a,k-2a,l}(z).
\end{gather*}
Therefore we have
\[
F^\downarrow_{\frac{n}{r}-a,k,l}(z)
=\begin{cases}
\ds \frac{(-1)^{ar_2}}{(-k-(\underline{0}_{r_2-1},l))_{\underline{a}_{r_2},d_2}}
\det_{\fn^-}(z)^aF^\downarrow_{\frac{n}{r}+a,k-a,l}(z), & \varepsilon_2=1,
\\
\ds \frac{2^{2ar_2}}{(-k-(\underline{0}_{r_2-l},\underline{1}_l))_{\underline{2a}_{r_2},d_2}}
\det_{\fn^-}(z)^aF^\downarrow_{\frac{n}{r}+a,k-2a,l}(z), & \varepsilon_2=2.
\end{cases}
\]
Next we assume $\fp^+_1$ is also of tube type, so that $r$ is even. Let $\dim\fp^+_1=:n_1$, $\rank\fp^+_1=:r_1$,
and let $\det_{\fn^-_1}(z_1)$ be the polynomial on $\fp^-_1$ satisfying
\[
t^{r/2}h_{\fp^+}\big(Q(x_2)z_1,-t^{-1}z_1\big)^{1/2}\big|_{t=0}
=\det_{\fn^-_1}(z_1)^{\varepsilon_1}\det_{\fn^+_2}(x_2)^{\varepsilon_2}, \qquad x_2\in\fp^+_2,\quad
z_1\in\fp^-_1.
\]
Suppose $l=0$, and for $a=1,2,\dots,\lfloor k/\varepsilon_2\rfloor$ we consider
\[
\lim_{\lambda\to \frac{n_1}{\varepsilon_1r_1}-\frac{2k}{\varepsilon_2}+a}
\bigg({-}\lambda-\frac{2k}{\varepsilon_2}+\frac{n}{r}-\frac{\delta_2}{2}\bigg)_{\underline{a}_{r/2},d} F^\downarrow_{\lambda,k,0}(z).
\]
By case-by-case consideration, we have
\[
\frac{n}{r}-\frac{\delta_2}{2}-\bigg(\frac{d}{2}\bigg(\frac{r}{2}-1\bigg)+1\bigg)
=\frac{d}{4}r-\frac{\delta_2}{2}=\frac{n_1}{\varepsilon_1r_1}.
\]
Therefore by Proposition~\ref{prop_hypergeom}(2) we have
\begin{gather*}
\lim_{\lambda\to \frac{n_1}{\varepsilon_1r_1}-\frac{2k}{\varepsilon_2}+a}
\bigg({-}\lambda-\frac{2k}{\varepsilon_2}+\frac{n}{r}-\frac{\delta_2}{2}\bigg)_{\underline{a}_{r/2},d} F^\downarrow_{\lambda,k,0}(z)
\\ \qquad
{}=\lim_{\lambda\to \frac{n_1}{\varepsilon_1r_1}-\frac{2k}{\varepsilon_2}+a}
\frac{\varepsilon_2^{kr_2} \big({-}\lambda-\frac{2k}{\varepsilon_2}+\frac{n}{r}-\frac{\delta_2}{2}\big)_{\underline{a}_{r/2},d}}
{\big(\frac{n_2}{r_2}\big)_{\underline{k}_{r_2},d_2}}\det_{\fn^-_2}(z_2)^k
\\ \qquad\hphantom{=}
{} \ \times {}_2F_1^{\fp^-_1,\fp^+_2}\left(\begin{matrix} -\frac{k}{\varepsilon_2},-\frac{k}{\varepsilon_2}-\frac{\delta_2}{2}\\
-\lambda-\frac{2k}{\varepsilon_2}+\frac{n}{r}-\frac{\delta_2}{2}\end{matrix};
z_1,{}^t\hspace{-1pt}z_2^\itinv\right)
\\ \qquad
{}=\frac{\varepsilon_2^{kr_2}\big({-}\frac{k}{\varepsilon_2}\big)_{\underline{a}_{r/2},d}
\big({-}\frac{k}{\varepsilon_2}-\frac{\delta_2}{2}\big)_{\underline{a}_{r/2},d}}
{\big(\frac{n_2}{r_2}\big)_{\underline{k}_{r_2},d_2} \big(\frac{d}{2}\big(\frac{r}{2}-1\big)+1\big)_{\underline{a}_{r/2},d}}
\det_{\fn^-_2}(z_2)^k
\det_{\fn^-_1}(z_1)^{\varepsilon_1a}\det_{\fn^-_2}(z_2)^{-\varepsilon_2a}
\\ \qquad\hphantom{=}
{} \ \times{}_2F_1^{\fp^-_1,\fp^+_2}\left(\begin{matrix} -\frac{k}{\varepsilon_2}+a,-\frac{k}{\varepsilon_2}+a-\frac{\delta_2}{2}\\[2pt]
\frac{d}{2}\big(\frac{r}{2}-1\big)+1+a \end{matrix};z_1,{}^t\hspace{-1pt}z_2^\itinv\right)
\\ \qquad
{}=\frac{\varepsilon_2^{(k-2a)r_2}(-k)_{\underline{\varepsilon_2a}_{r_2},d_2}}
{\big(\frac{n_2}{r_2}\big)_{\underline{k}_{r_2},d_2} \big(\frac{d}{2}\big(\frac{r}{2}-1\big)+1\big)_{\underline{a}_{r/2},d}}
\det_{\fn^-_1}(z_1)^{\varepsilon_1a}\det_{\fn^-_2}(z_2)^{k-\varepsilon_2a}
\\ \qquad\hphantom{=}
{} \ \times{}_2F_1^{\fp^-_1,\fp^+_2}\left(\begin{matrix} -\frac{k}{\varepsilon_2}+a,-\frac{k}{\varepsilon_2}+a-\frac{\delta_2}{2}\\
a-\frac{n_1}{\varepsilon_1r_1}+\frac{n}{r}-\frac{\delta_2}{2} \end{matrix};z_1,{}^t\hspace{-1pt}z_2^\itinv\right)
\\ \qquad
{}=\frac{(-k)_{\underline{\varepsilon_2a}_{r_2},d_2} \big(\frac{n_2}{r_2}\big)_{\underline{k-\varepsilon_2a}_{r_2},d_2}}
{\big(\frac{n_2}{r_2}\big)_{\underline{k}_{r_2},d_2} \big(\frac{d}{2}\big(\frac{r}{2}-1\big)+1\big)_{\underline{a}_{r/2},d}}
\det_{\fn^-_1}(z_1)^{\varepsilon_1a}F^\downarrow_{\frac{n_1}{\varepsilon_1r_1} -\frac{2k}{\varepsilon_2}+a,k-\varepsilon_2a,0}(z)
\\ \qquad
{}=\frac{1}{\left(\frac{d}{2}\left(\frac{r}{2}-1\right)+1\right)_{\underline{a}_{r/2},d}}
\det_{\fn^-_1}(z_1)^{\varepsilon_1a}F^\downarrow_{\frac{n_1}{\varepsilon_1r_1}-\frac{2k}{\varepsilon_2}+a,k-\varepsilon_2a,0}(z).
\end{gather*}
Similarly, we consider the case $\varepsilon_2=1$ and $r=2$, with general $l$.
Then $\frac{n}{2}-\frac{d_2}{2}-1=\frac{n_1}{2}$ holds, and for $a=1,2,\dots,k$ we have
\begin{gather*}
\lim_{\lambda\to \frac{n_1}{2}-2k-l+a}\bigg({-}\lambda-2k-l+\frac{n}{2}-\frac{d_2}{2}\bigg)_aF^\downarrow_{\lambda,k,l}(z) \\ \qquad
{}=\lim_{\lambda\to \frac{n_1}{2}-2k-l+a}\frac{\big({-}\lambda-2k-l+\frac{n}{2}-\frac{d_2}{2}\big)_a}
{\big(\frac{n_2}{2}\big)_{(k+l,k),d_2}}\det_{\fn^-_2}(z_2)^k
\\ \qquad\hphantom{=}
{} \ \times{}_2F_1\!\left(\begin{matrix}-k,-k-l-\frac{d_2}{2}
\\
-\lambda-2k-l+\frac{n}{2}-\frac{d_2}{2}\end{matrix};\frac{\big(z_1|Q(z_2)^{-1}z_1\big)_{\fp^-}}{2}\right)\rK_l(z_2) \\ \qquad
{}=\frac{(-k)_a\big({-}k-l-\frac{d_2}{2}\big)_a}
{\big(\frac{n_2}{2}\big)_{k+l}(1)_ka!}\det_{\fn^-_2}(z_2)^k
\\ \qquad\hphantom{=}
{} \ \times\det_{\fn^-_1}(z_1)^{\varepsilon_1a}\det_{\fn^-_2}(z_2)^{-a}{}_2F_1\!\left(\! \begin{matrix}-k+a,-k-l+a-\frac{d_2}{2} \\
a+1\end{matrix};\frac{\big(z_1|Q(z_2)^{-1}z_1\big)_{\fp^-}}{2}\!\right)\!\rK_l(z_2)
\\ \qquad
{}=\frac{1}{\left(\frac{n_2}{2}\right)_{k+l-a}(1)_{k-a}a!}\det_{\fn^-_1}(z_1)^{\varepsilon_1a}
\\ \qquad\hphantom{=}
{} \ \times\det_{\fn^-_2}(z_2)^{k-a}{}_2F_1\!\left( \begin{matrix}-k+a,-k-l+a-\frac{d_2}{2}
\\
a-\frac{n_1}{2}+\frac{n}{2}-\frac{d_2}{2}\end{matrix};\frac{\big(z_1|Q(z_2)^{-1}z_1\big)_{\fp^-}}{2}\right)\rK_l(z_2) \\ \qquad
{}=\frac{1}{a!}\det_{\fn^-_1}(z_1)^{\varepsilon_1a}F^\downarrow_{\frac{n_1}{2}-2k-l+a,k-a,l}(z).
\end{gather*}

Therefore the following holds.
\begin{Theorem}\label{thm_func_id_simple}\quad
\begin{enumerate}\itemsep=0pt
\item[$1.$] Assume $\fp^+$, $\fp^+_2$ are of tube type. If $\lambda=\frac{n}{r}-a$ with
$a=1,2,\dots,\lfloor k/\varepsilon_2\rfloor$, then we have
\[
 \cF^\downarrow_{\frac{n}{r}-a,k,l}=\begin{cases}
\ds \frac{(-1)^{ar_2}}{(-k-(\underline{0}_{r_2-1},l))_{\underline{a}_{r_2},d_2}}
\cF^\downarrow_{\frac{n}{r}+a,k-a,l}\circ\det_{\fn^-}\bigg(\frac{\partial}{\partial x}\bigg)^a, & \varepsilon_2=1,
\\[2ex]
\ds \frac{2^{2ar_2}}{(-k-(\underline{0}_{r_2-l},\underline{1}_l))_{\underline{2a}_{r_2},d_2}}
\cF^\downarrow_{\frac{n}{r}+a,k-2a,l}\circ\det_{\fn^-}\bigg(\frac{\partial}{\partial x}\bigg)^a, & \varepsilon_2=2. \end{cases}
\]
\item[$2.$] Assume $\fp^+_1$ is also of tube type, and let $l=0$. Then at $\lambda=\frac{n_1}{\varepsilon_1r_1}-\frac{2k}{\varepsilon_2}+a$
with $a=1,2,\dots,\lfloor k/\varepsilon_2\rfloor$, we have
\begin{gather*}
\lim_{\lambda\to \frac{n_1}{\varepsilon_1r_1}-\frac{2k}{\varepsilon_2}+a}
\left(-\lambda-\frac{2k}{\varepsilon_2}+\frac{n}{r}-\frac{\delta_2}{2}\right)_{\underline{a}_{r/2},d} \cF^\downarrow_{\lambda,k,0}
\\ \qquad
{}=\frac{1}{\big(\frac{d}{2}\big(\frac{r}{2}-1\big)+1\big)_{\underline{a}_{r/2},d}}
\det_{\fn^-_1}\bigg(\frac{1}{\varepsilon_1}\frac{\partial}{\partial x_1}\bigg)^{\varepsilon_1a}
\circ\cF^\downarrow_{\frac{n_1}{\varepsilon_1r_1}-\frac{2k}{\varepsilon_2}+a,k-\varepsilon_2a,0}.
\end{gather*}
\item[$3.$] Again assume $\fp^+_1$ is also of tube type, and $\varepsilon_2=1$, $r=2$. Then at $\lambda=\frac{n_1}{2}-2k-l+a$
with $a=1,2,\dots,k$, we have
\begin{gather*}
\lim_{\lambda\to \frac{n_1}{2}-2k-l+a}\bigg({-}\lambda-2k-l+\frac{n}{2}-\frac{d_2}{2}\bigg)_a\cF^\downarrow_{\lambda,k,l} \\ \qquad
{}=\frac{1}{a!}\det_{\fn^-_1}\bigg(\frac{1}{\varepsilon_1}\frac{\partial}{\partial x_1}\bigg)^{\varepsilon_1a}
\circ\cF^\downarrow_{\frac{n_1}{2}-2k-l+a,k-a,l}.
\end{gather*}
\end{enumerate}
\end{Theorem}

This theorem is again regarded as an analogue of the functional identities \cite[Corolla\-ries~12.7(2) and~12.8(3)]{KS1}.
Here, for individual cases, when we take a maximal tripotent $e\in\fp^+_2$ suitably, $\det_{\fn^-}(z)$ and $\det_{\fn_1^-}(z_1)$ are given by
\[
\det_{\fn^-}(z)=q(z), \qquad \det_{\fn^-_1}(z_1)^{\varepsilon_1}=-q(z_1)
\]
for {\it Cases} 1 and $1'$,
\begin{gather*}
\det_{\fn^-}\begin{pmatrix}z_{11}&z_{12}\\{}^t\hspace{-1pt}z_{12}&z_{22}\end{pmatrix}
=\det\left(\begin{pmatrix}z_{11}&z_{12}\\{}^t\hspace{-1pt}z_{12}&z_{22}\end{pmatrix} \begin{pmatrix}0&I\\I&0\end{pmatrix}\right)
=(-1)^{r'}\det\begin{pmatrix}z_{11}&z_{12}\\{}^t\hspace{-1pt}z_{12}&z_{22}\end{pmatrix}\!,
\\
\det_{\fn^-_1}\begin{pmatrix}z_{11}&0\\0&z_{22}\end{pmatrix}=\det(z_{11})\det(z_{22})
\end{gather*}
for {\it Case} 2,
\[
\det_{\fn^-}\begin{pmatrix}z_{11}&z_{12}\\-{}^t\hspace{-1pt}z_{12}&z_{22}\end{pmatrix}
=\Pf\begin{pmatrix}z_{11}&z_{12}\\-{}^t\hspace{-1pt}z_{12}&z_{22}\end{pmatrix}\!, \qquad
\det_{\fn^-_1}\begin{pmatrix}z_{11}&0\\0&z_{22}\end{pmatrix}=\Pf(z_{11})\Pf(z_{22})
\]
for {\it Case} 3 (with $s'=2r'$ for $\det_{\fn^-_1}(z_1)$),
\[
\det_{\fn^-}(z_1+z_2)=\det(z_1+z_2), \qquad \det_{\fn^-_1}(z_1)=\sqrt{-1}^s\Pf(z_1)
\]
for {\it Case} 4 (with $r=2s$ for $\det_{\fn^-_1}(z_1)$),
\[
\det_{\fn^-}(z_1+z_2)=\det(z_1+z_2), \qquad \det_{\fn^-_1}(z_1)=(-1)^s\det(z_1)
\]
for {\it Case} 5, and
\[
\det_{\fn^-}((z_1,z_2))=\det((z_1,z_2))
\]
under the identification (\ref{iotaj}) for {\it Case} 6. Again by \cite[Theorem 6.4]{A}, (\ref{G-intertwine}) intertwines the $\widetilde{G}$-action if $\fp^+$ is of tube type,
and hence the both sides of Theorem \ref{thm_func_id_simple}(1)--(3) intertwine the $\widetilde{G}_1$-action.

\subsection*{Acknowledgments}

The author would like to thank Professor T.~Kobayashi for a lot of helpful advice on this research,
and also thank Professor H.~Ochiai for a lot of helpful comments on this paper.
This work was supported by Grant-in-Aid for JSPS Fellows Grant Number JP20J00114.

\LastPageEnding

\end{document}